\documentclass[11pt,a4paper]{article}
\usepackage{amsmath}
\usepackage{amsfonts}
\usepackage{amssymb}
\usepackage{mathrsfs}
\usepackage{mathdots}
\usepackage{bm}

\title{An elementary recursive bound for effective Positivstellensatz and
Hilbert 17-th problem
}
\author{Henri Lombardi\thanks{
Laboratoire of Math\'ematiques
(UMR CNRS 6623)
 UFR des Sciences et Techniques,
Universit\'e of Franche-Comt\'e 
25030 Besan\c{c}on cedex
FRANCE. E-mail: lombardi@math.univ-fcomte.fr
}\\
 Daniel Perrucci \thanks{Departamento de Matem\'atica, 
 FCEN, Universidad de Buenos Aires and IMAS CONICET-UBA,
ARGENTINA. Partially supported by the grants UBACYT 20020120100133 and PIP 099/11 CONICET. 
E-mail: perrucci@dm.uba.ar
}
 \\
 Marie-Fran\c{c}oise Roy\thanks{IRMAR (UMR CNRS 6625),
 Universit\'{e} de Rennes 1,
 Campus de Beaulieu 35042 Rennes cedex FRANCE. E-mail: marie-francoise.roy@univ-rennes1.fr
\newline \textbf{MSC Classification:} 12D15, 14P99, 13J30.
 } }

\newtheorem{theorem}{Theorem}[subsection]
\newtheorem{Tlemma}[theorem]{Technical Lemma}
\newtheorem{lemma}[theorem]{Lemma}

\newtheorem{definition}[theorem]{Definition}
\newtheorem{proposition}[theorem]{Proposition}
\newtheorem{remark}[theorem]{Remark}
\newtheorem{example}[theorem]{Example}
\newtheorem{notation}[theorem]{Notation}
\newenvironment{proof}[1]{
\trivlist \item[\hskip\labelsep{\bf#1}]}{\hfill\mbox{$\Box$}
\endtrivlist}
\newcommand{\hide}[1]{}

\newcommand \A {\mathbf{A}}
\newcommand \K {\mathbf{K}}

\newcommand \N {\mathbb{N}}

\newcommand \R {\mathbf{R}}
\newcommand \D {\mathbf{D}}
\newcommand \F {\mathbf{F}}
\renewcommand \L {{\bf L}}

\newcommand \cD {\mathcal{D}}
\newcommand \cF {\mathcal{F}}
\newcommand \cH {\mathcal{H}}

\newcommand \cP {\mathcal{P}}
\newcommand \cQ {\mathcal{Q}}
\newcommand \scM {\mathscr{M}}
\newcommand \scI {\mathscr{Z}}
\newcommand \scP {\mathscr{N}}
\newcommand \scZ {\mathscr{Z}}
\newcommand \scN {\mathscr{N}}
\newcommand \Gram {\mathrm{Gram}}
\newcommand \fra[2] {{\textstyle{\frac{#1}{#2}}}}
\newcommand \sign{\mathrm{sign}}
\newcommand \inv{\mathrm{inv}}
\newcommand {\s}{\sign}

\newcommand \Id{{\rm Id}}

\newcommand \Res{{\rm Res}}

\newcommand \Der{{\rm Der}}
\newcommand \Thom{{\rm Thom}}

\newcommand \Quot{{\rm Quot}}
\newcommand \re{_{\rm Re}}
\newcommand \im{_{\rm Im}}

\newcommand \lda{\left\downarrow\ }
\newcommand \rda{\ \right\downarrow}

\date{version: \today}

\begin{document}  

\maketitle
\begin{abstract}  
We prove an elementary recursive bound on the degrees 
for Hilbert 17-th problem.
More precisely we 
express a nonnegative polynomial as a sum of squares of rational functions, 
and we obtain as degree estimates for the numerators and denominators the following tower of five exponentials
$$2^{
2^{
2^{d^{4^{k}}}
}
}
$$ where
$d$ is the degree and $k$ is the number of variables of the
input polynomial.
Our method is based on the proof of an elementary recursive bound on the degrees for 
Stengle's Positivstellensatz. More precisely we give an algebraic certificate of the emptyness of the 
realization of a system of sign conditions and we obtain
as degree bounds for this certificate  a tower of five exponentials, 
namely
$$
2^{
2^{\left(2^{\max\{2,d\}^{4^{k}}}+ s^{2^{k}}\max\{2, d\}^{16^{k}{\rm bit}(d)} \right)}
}$$
where
$d$ is a bound on the degrees, $s$ is the number of polynomials and $k$ is the number of variables of the
input polynomials.

\end{abstract}

\tableofcontents

\section{Introduction}\label{sec:Introduction}
\setcounter{equation}{0}

Throughout this paper, we denote by $\N$ the set of nonnegative integers, by $\N_*$ 
the set of positive integers, by $\mathbb{R}$ the field of real numbers,  by $\K$ an ordered field,
by $\K_+$ the set of positive elements of $\K$ 
and by $\R$ a real 
closed extension of $\K$.

\subsection{Hilbert 17-th problem}\label{pstvz_and_hilb17}

Hilbert 17-th problem asks whether a real multivariate polynomial taking only nonnegative values
is a sum of squares of rational functions (\cite{Hilb_ref_f}, \cite{Hilb_ref_g}, \cite{Hilb_ref_e}). 
E. Artin gave a positive answer proving the 
following statement \cite{Artin}.

\begin{theorem}[Hilbert 17-th problem]
\label{artin}
Let $P\in \mathbb{R}[x_1, \dots, x_k]$. If $P$ takes only nonnegative values in $\mathbb{R}^k$, then $P$ is a sum of squares 
in $\mathbb{R}(x_1, \dots, x_k)$.
\end{theorem}

\subsection{Positivstellensatz}
\setcounter{equation}{0}

In order to give the statement of the Positivstellensatz, we 
will deal with 
finite 
conjunctions of equalities, strict inequalities and nonstrict inequalities  
on
polynomials in $\K[x]$, where 
$x = (x_1,\ldots, x_k)$ is a set of variables.

\begin{definition}\label{def1.1.2}
\noindent A {\em system of sign conditions} ${\cal F}$ in $\K[x]$ is a list of three finite
(possibly empty) subsets $[{\cal F}_{\neq}, \, {\cal F}_\ge,\, {\cal F}_=]$
of $\K[x]$, representing the 
conjunction
$$
\left \{
\begin{array}{cc} 
P
\ne 0 &  \hbox{for } P \in {\cal F}_{\ne},     \\[1mm] 
P
\ge 0 &  \hbox{for } P \in {\cal F}_{\ge},   \\[1mm]    
P
 = 0 &  \hbox{for } P \in {\cal F}_{=}. 
\end{array} \right.
$$
\end{definition}

Since
the condition $P
\le 0$ is equivalent to $-P
\ge 0$,
the condition $P
> 0$ is equivalent to $P
\ne 0 \wedge P
\ge 0$ and  
the condition $P
 <0$ is equivalent to $P
 \ne 0 \wedge -P 
 \ge 0$,
any finite conjunction of equalities, strict inequalities and nonstrict inequalities can be represented by
a system of sign conditions as in Definition  \ref{def1.1.2}. 
Throughout this paper, by a slight abuse of language, we will refer to such more general conjunctions 
as systems of sign conditions, when we should be strictly speaking referring to the systems of sign 
conditions associated to such conjunctions.

If $P\in \K[x]$ and $\xi=(\xi_1, \dots, \xi_k)\in \L^k$ where $\L$ is a field extension of $\K$,
we denote by $P(\xi) \in \L$ the result of the substitution of $x$ by $\xi$.

\begin{definition}
For an ordered extension $\L$ of $\K$,
the \emph{realization in $\L$ of} a system of sign conditions ${\cal F}$  in $\K[x]$
is the set 
$$
{\rm Real}({\cal F}, \L) = 
\{ \xi \in \L^k  \ | \bigwedge_{P\in {\cal F}_{\ne}} P(\xi)\not=0, \bigwedge_{P\in {\cal F}_{\ge}} P(\xi)\ge 0,
\bigwedge_{P\in {\cal F}_{=}} P(\xi)=0\}.
$$
If ${\rm Real}({\cal F}, \L)$ is the empty set, 
we say that ${\cal F}$ is \emph{unrealizable} in $\L$. 
\end{definition}

Stengle's Positivstellensatz, which we will refer from now on simply as the 
Positivstellensatz, states that if a system $\cF$ 
is unrealizable in 
$\R$, there is an 
algebraic identity which certifies this
fact. To describe such an identity, we introduce the following notation and definitions.

\begin{notation}
Let ${\cal P}$ be a finite subset of $\K[x]$. We denote by
\begin{itemize}
\item ${\cal P}^2 $ the set of squares of elements of ${\cal P}$,
\item ${\scM}({\cal P})$ the multiplicative monoid generated by ${\cal
P}$,
\item  ${\scN}({\cal P})_{\K[x]}$ the nonnegative cone generated by ${\cal P}$
in $\K[x]$,
which is  the set of elements of  type $\sum_{1 \le i \le m} \omega_i V_i^2 \cdot N_i$
with 
$\omega_i \in \K_+$,   $V_i \in \K[x]$ and $N_i \in {\scM}({\cal P})$ for $1 \le i \le m$,
\item ${\scZ}({\cal P})_{\K[x]}$ the ideal generated by ${\cal P}$ in $\K[x]$.
\end{itemize}
When the ring $\K[x]$ is clear from the context, we simply write 
${\scN}({\cal P})$ for ${\scN}({\cal P})_{\K[x]}$ and ${\scZ}({\cal P})$ for ${\scZ}({\cal P})_{\K[x]}$. 
\end{notation}

\begin{definition}
A system of sign conditions ${\cal F}$ in $\K[x]$
is 
{\em incompatible}
if there is an algebraic identity  
\begin{equation}\label{def_incomp}
S + N + Z = 0
\end{equation}
with $S\in {\scM}({\cal F}_{\neq}^2)$, $N \in
{\scN}({\cal F}_\ge )_{\K[x]}$ and $Z \in {\scZ}({\cal F}_=)_{\K[x]}$.
The 
identity $(\ref{def_incomp})$ is called an {\em incompatibility} of ${\cal F}$.
We use the notation
$$\lda {\cal F} \rda_{\K[x]}
$$
to mean that an incompatibility of ${\cal F}$ is provided.
We denote  simply
$$\lda {\cal F} \rda
$$
when the ring $\K[x]$ is clear from the context.
The polynomials $S$, $N$ and $Z$ are called the \emph{monoid}, \emph{cone} and \emph{ideal  
part} of the incompatibility.

\end{definition}

An incompatibility (\ref{def_incomp}) 
of $\cF $
is a certificate that ${\cal F}$ is unrealizable in any
ordered extension $\L$ of $\K$. 
Indeed, suppose
that there exsists $\xi \in {\rm Real}({\cal F}, \L)$.
Then
$$
S(\xi)>0, \ N(\xi) \ge 0, \  \hbox{and} \  Z(\xi)=0, 
$$
which is impossible since $
S+N+Z=0.
$
\begin{example}
\label{exsquare-square}
The identity
\begin{equation}
\label{square-square}
P^2-P^2=0
\end{equation}
is an incompatibility of ${\cal F}_1 =[\{P\},\emptyset,\{P\}]$, 
since $P^2 \in \scM(\{P\}^2)$,  $0 \in \scN(\emptyset)$ and 
$-P^2 \in \scZ(\{P\})$.
For simplicity we write
$$
\lda  P \ne  0, \  P=0 \rda
$$
to mean $
\lda {\cal F}_1 \rda
$. 
 
The identity (\ref{square-square}) is  also and incompatibility of 
${\cal F}_2 =[\{P\},\{P,-P\},\emptyset]$, 
since $P^2 \in \scM(\{P\}^2)$,
$-P^2 \in \scN(\{P, -P\})$ and $0 \in \scZ(\emptyset)$. 
For simplicity, and following the procedure explained before so that every 
system of sign conditions is as in Definition \ref{def1.1.2},
we write
$$\lda P > 0, \ P \le 0 \rda$$
to mean $
\lda {\cal F}_2 \rda
$. 

Similarly, the identity (\ref{square-square})
also shows that 
$$
\lda P > 0, \ P=0 \rda, \quad \lda P < 0, \ P=0   \rda, \quad
\lda P < 0, \ P \ge 0   \rda \quad 
\hbox{and} \quad \lda P > 0, \ P < 0\rda.
$$
\end{example}

\begin{notation} Let $\cF = [\cF_{\ne}, \, \cF_{\ge}, \, \cF_{=}]$
and $\cF'= [\cF'_{\ne}, \, \cF'_{\ge}, \, \cF'_{=}]$ be systems of sign conditions in 
$\K[x]$. We denote by $\cF, \ \cF'$ the system 
$[\cF_{\ne} \cup \cF'_{\ne}, \, \cF_{\ge} \cup \cF'_{\ge} , \, \cF_{=}\cup \cF'_{=}]$. 
\end{notation}

Note that  $\lda \cF \rda$ implies  $\lda {\cal F}, \ \cF' \rda$.

A major concern in this paper are degrees of incompatibilities in the Positivstellensatz. 
To deal with them, we introduce below the following definitions.  

\begin{definition}

Let $\cP$ be a finite set in $\K[x]$. 
\begin{itemize}
\item For $N = \sum_{1 \le i \le m}\omega_i V_i^2 \cdot N_i \in {\scN}({\cal P})$, 
with $\omega_i \in \K_+$, $V_i \in \K[x]$ and $N_i \in {\scM}({\cal P})$ 
for $1 \le i \le m$,
we say that $\omega_i  V_i^2 \cdot N_i$ are the \emph{components} of $N$ in ${\scN}({\cal P})$. 
\item For $Z=\sum_{1 \le i \le m} W_i \cdot P_i \in {\scZ}({\cal P})$ 
with $W_i \in \K[x]$ and $P_i \in {\cal P}$
 for $1 \le i \le m$, we say that $W_i \cdot  P_i$ are the 
\emph{components} of $Z$ in ${\scZ}({\cal P})$.
\end{itemize}
\end{definition}

Note that $N \in {\scN}({\cal P})$ and $Z \in {\scZ}({\cal P})$ 
can be written as a sum of components in many different ways. So, when we talk of the components of $N$
or $Z$, the ones we refer to should be clear from the context.

\begin{definition}\label{defDegI} Let $\cF$ by a system of sign conditions in $\K[x]$. 
The {\em degree of the incompatibility}
\begin{equation}\label{eq_def_degree}
S + N + Z = 0
\end{equation}
with $S \in {\scM}( {\cal F}_{\neq}^2)$,
$N 
\in {\scN}({\cal P})$, 
and 
$Z 
 \in {\scZ}({\cal F}_=) $
is the maximum of the 
total degrees in $x$ of $S$, the
components of $N$ and the components
of $Z$.
For a subset of variables $w \subset x$, the degree in $w$ of the 
incompatibility (\ref{eq_def_degree}) 
is the maximum of the 
total
degrees in $w$ of $S$, the
components of $N$ and the components
of $Z$.
\end{definition}

Contrary to the common convention, 
we consider the 
degree of the zero polynomial as $0$. In this way, we have for instance
the incompatibility 
$ 0 = 0 $ of degree $0$
which proves $\lda  0  \ne 0 \rda$.

The Positivstellensatz is the following theorem.

\begin{theorem}[Positivstellensatz]\label{thPST}
Let 
${\cal F}$
be a system of
sign conditions in $\K[x]$.
The following are equivalent:
\begin{enumerate}
\item ${\cal F}$ is unrealizable in $\R$,
\item ${\cal F}$ is unrealizable in every ordered extension of $\K$,
\item ${\cal F}$ is incompatible.
\end{enumerate}
\end{theorem}
$3.\implies 2.$ and $2.\implies 1.$ are clear, the difficult part is to prove
$1.\implies 3.$

This statement comes from \cite{Ste} (see also \cite{BCR,Du,Efr,Kri,Ris}).

An immediate consequence of the Positivstellensatz (Theorem \ref{thPST}) is the Real Nullstellensatz. 

\begin{theorem}[Real Nullstellensatz]
Let $P, P_1, \dots, P_s \in 
\K[x_1, \dots, x_k]$.
 If 
$P$ vanishes on the common zero set of $P_1, \dots, P_s$ in $\R^k$, 
there is an identity
$$
P^{2e} + N  = Z
$$
with $N \in {\scN}(\emptyset)_{\K[x_1,\ldots,x_k]}$, 
and 
$Z  \in {\scZ}({P_1,\ldots,P_s})_{\K[x_1,\ldots,x_k]} $
\end{theorem}
\begin{proof}{Proof.} Apply Theorem \ref{thPST} (Positivstellensatz) to the system of sign conditions $[\{P\},\emptyset,\{P_1,\ldots,P_s\}]$, 
which corresponds to
$P \ne 0, P_1 = 0, \dots, P_s = 0$.
\end{proof}

As another consequence of the Positivstellensatz (Theorem \ref{thPST}), we have
an improved version of Hilbert 17-th problem due to Stengle \cite{Ste}.

\begin{theorem}[Improved Hilbert 17-th problem]\label{th:Hilbert_by_Stengle}
\label{stengle}
Let $P\in\K[x_1,\ldots,x_k]$ be a polynomial of degree $d$. If $P$ is nonnegative in $\R^k$, then 
$$
P =  \sum_{i}\omega_i \frac{P_i^2}{Q^2}
$$
with $\omega_i \in \K_+,  P_i \in \K[x_1,\ldots,x_k], Q \in \K[x_1,\ldots,x_k]$, $Q$
vanishing only at points where $P$ vanishes.
\end{theorem}

\begin{proof}{Proof.} Since $P$ is nonnegative in $\R^k$, 
by  Theorem \ref{thPST} (Positivstellensatz) applied 
to the system
$[\{P\},\{-P\},\emptyset]$, which corresponds to the sign condition 
$P < 0$, we have an identity
$$
P^{2e}+ N_1 - N_2 \cdot P=0
$$
with $e \in \N$ and 
$N_1, N_2 \in \scN(\emptyset)_{\K[x_1,\ldots,x_k]}$. 
Therefore
\begin{equation}\label{eq:Pos_impl_Hilb}
P=\frac{N_2 \cdot  P^2}{P^{2e}+ N_1}=\frac{N_2 \cdot P^2  \cdot (P^{2e}+ N_1)} {(P^{2e}+ N_1)^2}.
\end{equation}
The result follows by expanding the numerator of the last expression in (\ref{eq:Pos_impl_Hilb}). 
\end{proof}

\subsection{Historical background on constructive proofs and degree bounds}

In order to compare different degree bounds, in this section we use the notions of primitive recursive function
and elementary recursive function (see \cite[Chapter 1]{Rose}).

With respect to Hilbert 17-th problem, Artin's proof of Theorem \ref{artin} is non-constructive 
and uses Zorn's lemma.  
Later, Kreisel and Daykin produced the first constructive proofs \cite{Kr,Kr2,Day,Del} of this result, 
providing primitive recursive degree bounds.

For the Positivstellensatz, the original proofs were also non-constructive and 
used Zorn's lemma. The first constructive proof was given in 
\cite{Lom.a,Lom.b,Lom.c}, and it is
based on the translation into algebraic identities of  Cohen-H\"ormander's quantifier elimination
algorithm
\cite{Coh,Hor,BCR}. Following this construction,
primitive recursive degree estimates for the incompatibility of the input system were obtained
in \cite{Lom.e}. In order to state this result precisely, we introduce the 
following notation. 

\begin{notation}\label{not:set_of_pol_under_eq_rel}
Let ${\cal F} = [{\cal F}_{\neq}, \; {\cal F}_\ge, \; {\cal F}_=]$ be a system of sign conditions
in $\K[x]$. We denote by  $|{\cal F}|$ a subset of
 ${\cal F}_{\neq}  \cup {\cal F}_\ge \cup {\cal F}_=$
 such that for every  $P\in {\cal F}_{\neq}  \cup {\cal F}_\ge \cup {\cal F}_=$
one and only one element of $\{P,-P\}$ is in $|{\cal F}|$.
\end{notation}

The first  known degree bound for the Positivstellensatz is the following result
(see \cite[Th\'eor\`eme 26]{Lom.e}), which is, in fact, still the only known
degree bound till now. 

\begin{theorem}[Positivstellensatz with primitive recursive
degree estimates] 
Let ${\cal F}$ be a system of sign conditions
in $\K[x_1, \dots, x_k]$, such that
 $ \#|{\cal F}| = s$ and the degree of every polynomial 
in ${\cal F}$ is bounded by $d$.
If ${\rm Real}({\cal F}, \R)$ is empty, 
one can
compute an incompatibility $\lda {\cal F} \rda$ 
with degree bounded by
$$
 2^{  2^{\iddots^{d{\rm log}(d) + {\rm log log}(s) + c}}   }  
$$
where $c$ is a universal constant and the height of the exponential tower is $k+4$. 
\end{theorem}

A different constructive proof of the Real Nullstellensatz and Hilbert 17-th problem was given in 
\cite{Sch}, providing also primitive recursive degree bounds 
for the incompatibility it produces.

On the other hand, lower bounds on the degrees for the Positivstellensatz are given in \cite{GV2}, 
where for $k \ge 2$, an example is provided
of an incompatible system ${\cal F}$ in $\K[x_1, \dots, x_k]$ with $|{\cal F}| = k$ and the degree
of every polynomial in  ${\cal F}$ bounded by $2$, such that every incompatibility of the system 
has degree at least $2^{k-2}$. Concerning Hilbert 17-th problem, 
an example of a 
nonnegative polynomial of degree $4$ in $k$ variables, such that
in any decomposition as a sum of squares of rational functions, the degree of  
some denominator 
is bounded from below by a linear function in $k$, appears in \cite{BlGoPf}.

The huge gap between the 
best known lower bound on the degrees for the Positivstellensatz, 
which is singly exponential, and the best upper 
bound on the degrees known up to now, which is primitive recursive,
is in strong contrast with the state of the art for the Hilbert Nullstellensatz. For this result, 
elementary recursive upper degree bounds are already known since \cite{Herm}.
Indeed, it is easy using resultants to obtain a doubly exponential bound on the degree of a 
Nullstellensatz identity
\cite{VW,BPRbook}.
More recent and sophisticated results give singly exponential degree estimates \cite{Brown,CGH,Kol,Jel}, which are
known to be optimal.

\subsection{Our results}

The aim of this paper
is to 
provide for the first time 
elementary recursive estimates on the degrees of the polynomials involved in the 
Positivstellensatz, Real Nullstellensatz and Hibert 17th problem. 
The existence of such bounds is a long-standing open question.

\begin{notation}
We denote by ${\rm bit}(d)$ the number of bits of the natural number $d$, defined by
$${\rm bit}(d)=
\left \{ \begin{array}{ll}
          1 & \hbox{if } d = 0, \cr
	    k & \hbox{if } d \ne 0 \hbox{ and } 2^{k-1}\le d < 2^k.
         \end{array}
\right.
$$
\end{notation}

We can state now the main results of this paper.

\begin{theorem}[Positivstellensatz with elementary recursive
degree estimates]\label{Theoremfinal}
Let ${\cal F}$ be a system of sign conditions
in $\K[x_1, \dots, x_k]$, such that
 $ \#|{\cal F}| = s$ and the degree of every polynomial 
in ${\cal F}$ is bounded by $d$.
If ${\rm Real}({\cal F}, \R)$ is empty, 
one can
compute an incompatibility $\lda {\cal F} \rda$ 
with degree bounded by
$$
2^{
2^{\left(2^{\max\{2,d\}^{4^{k}}}+s^{2^{k}}\max\{2, d\}^{16^{k}{\rm bit}(d)} \right)}.
}$$
\end{theorem}

As 
an immediate consequence
of Theorem \ref{Theoremfinal} we also get the following result. 

\begin{theorem}[Real Nullstellensatz with elementary recursive
degree estimates]
Let $P, P_1, \dots, P_s$ $\in 
\K[x_1, \dots, x_k]$ with degree bounded by $d$. If 
$P$ vanishes on the common zero set of $P_1, \dots, P_s$ in $\R^k$, 
there is an identity
$$
P^{2e} + N  = Z
$$
with $N \in {\scN}(\emptyset)_{\K[x_1,\ldots,x_k]}$, 
and 
$Z  \in {\scZ}({P_1,\ldots,P_s})_{\K[x_1,\ldots,x_k]}$
of degree bounded by
$$
2^{
2^{\left(2^{\max\{2,d\}^{4^{k}}}+(s+1)^{2^{k}}\max\{2, d\}^{16^{k}{\rm bit}(d)} \right)}.
}$$
\end{theorem}

The final main theorem of this paper is the following result. 

\begin{theorem} [Hilbert 17-th problem with elementary recursive
degree estimates]
\label{Theoremfinal17}
Let $P\in\K[x_1,\ldots,x_k]$ be a polynomial of degree $d$. If $P$ is nonnegative in $\R^k$, then 
$$
P =  \sum_{i}\omega_i \frac{P_i^2}{Q^2}
$$
with $\omega_i \in \K_+,  P_i \in \K[x_1,\ldots,x_k], Q \in \K[x_1,\ldots,x_k]$, $Q$
vanishing only at points where $P$ vanishes
and
$\deg P^2_i$
for every $i$ and
$\deg Q^2$
bounded by $$2^{
2^{
2^{d^{4^{k}}}
}
}
.$$
\end{theorem}

We sketch now a 
very brief description of 
the strategy we follow in our proof of Theorem \ref{Theoremfinal} and Theorem 
\ref{Theoremfinal17}. If a system of sign conditions
$\cF$ in $\K[x]$ 
is unrealizable in $\R$, we want to construct an incompatibility 
of $\cF$.
The idea is to transform a proof of the fact that
$\cF$ 
is unrealizable into a construction of an incompatibility.
This was already the strategy used by \cite{Lom.b, Lom.e};
the proof that $\cF$ is unrealizable
was using  Cohen-H\"ormander quantifier elimination method \cite{Coh,Hor,BCR}
and was giving primitive recursive bounds for the final incompatibility.

In the current paper, the proof that  $\cF$ 
is unrealizable has to be  based on more powerful
tools than  Cohen-H\"ormander quantifier elimination method  to obtain elementary recursive
degree bounds, 
but it also has to 
remain \emph{on the algebraic side}, 
so that we are able to turn it into a construction of an incompatibility.

The methods to prove the unrealizability in $\R$ of a system $\cF$ are composed of many steps. 
Therefore, we
need to know how to turn each of this steps into the construction of a new incompatibility. 
This is in general a very hard task and requires 
transforming standard and rather abstract proofs
into very concrete proofs, in a way such that 
the outcome is so transparent that 
it becomes possible to read these new proofs
as algebraic certificates or as constructions of algebraic certificates from other ones. 
More explicitly, in order to construct incompatibilities, we first need to associate to a well-chosen
existing proof of the preceeding results, some specific algebraic identities.
Then, using the key notions of weak inference and weak existence  coming from \cite{Lom.e},
we have to show how to 
translate these modified proofs into constructions of incompatibilities.
This translation is far from straightforward,
relies heavily on the selected proof 
and the associated algebraic identities, and, as said before, should be 
done at each step for the corresponding specific result, most of the times in a different way.
Indeed, the methods we develop here to consctruct incompatibilities associated to some 
well known results in mathematics may actually be 
of interest independent of our main results.

Since a single step of a proof that a system 
$\cF$
is 
unrealizable in $\R$ which cannot be transformed into the construction of an incompatibility is enough to 
ruin the whole construction, it is clear that 
the choice of the 
 method we use to prove that $\cF$ is unrealizable, 
taking into account 
which steps
compose this method,
is of major importance.

Proving that that a system of sign conditions
$\cF$
is 
unrealizable in $\R$  
is an instance of the Existential Theory of the Reals,
which is a special case of Quantifier Elimination of the theory of real
closed fields.
Most of the proofs of Quantifier Elimination are based on the elimination of variables one after the other. 
More recent methods eliminate in one step a block of variables
\cite{Gri88,GV,R92,BPRQE,DH,BPRbook}.

The first proofs of Quantifier Elimination for the reals by Tarski, Seidenberg, Cohen or Hormander \cite{Tarski51,Seidenberg54,Coh,Hor} were all 
providing primitive recursive 
algorithms.
The situation changed with the Cylindrical Algebraic Decomposition  method \cite{Col,Loj} and elementary recursive algorithms where obtained \cite{Mo}. 
Cylindrical Algebraic Decomposition, 
is in fact doubly exponential (see for example \cite{BPRbook}). 

Singly exponential degree bounds, 
have been obtained for the Existential Theory of the Reals  \cite{Gri88,GV,R92,BPRQE,DH,BPRbook}
by eliminating in one step 
the block of existential
variables.
But these 
singly exponential 
results are based on the critical point method which seems \emph{too geometric} 
to be translated into algebraic identities,
and this is why we choose to use the technique of 
elimination of one variable after the other.
In order to obtain our main results we need a method such 
that for each of its steps we are able to produce an incompatibility, and therefore
we are led to design a suitable 
new elimination method
with this property.
This new elimination method produces a new purely algebraic proof of Quantifier Elimination which is elementary recursive \cite{PerR2}.

Our proof translates into constructions of incompatibilities several main ingredients. Some of them are 
classical mathematical facts, but many of them come from much more recent results in computer algebra. These main ingredients are:
\begin{itemize}
\item the Intermediate Value Theorem for polynomials,
 \item Laplace's proof of the Fundamental Theorem of Algebra, 
 \item Hermite's quadratic form, for real root counting with polynomial constraints,
  \item subresultants whose signs are determining the signature of Hermite's quadratic form,
 \item Sylvester's inertia law,
 \item Thom's lemma characterizing real algebraic numbers by sign conditions, and sign determination,
  \item a 
  new elimination method
 reducing one by one the number of variables to consider.
\end{itemize}

Finally, for any unrealizable system of sign conditions
we are able to construct an explicit incompatibility 
and prove that the  degree bound of this incompatibility is elementary recursive.
More precisely the five levels of exponentials in Theorem \ref{Theoremfinal} and Theorem 
\ref{Theoremfinal17} come from the following facts
\begin{itemize}
\item 
eliminating all variables one after another produces univariate polynomials of doubly exponential degree,
 \item Laplace's proof of the Fundamental Theorem of Algebra introduces a polynomial of 
 exponential 
 degree,
 \item the construction of incompatibilities for the Intermediate Value Theorem produces algebraic identities of doubly exponential degrees.
\end{itemize}

Applying Laplace's proof of the 
Fundamental Theorem of Algebra to a univariate polynomial of doubly 
exponential degree, coming from 
the elimination process
produces 
a polynomial of triple 
exponential degree.
Since the Intermediate Value Theorem adds two more exponents to the degree of the final incompatibility, we obtain by our method a tower of five exponents.

We are lucky enough that  the other ingredients of our construction do not increase the height of the tower above five exponentials.
Full details will be provided in the paper.

\subsection{Organization of the paper}

Since the paper is very long, a 
significant
effort is made to keep the organization simple.

In Section \ref{Weak.inferences} we
describe the concepts of weak inference and weak existence and we include many
lemmas showing examples of them, with degree estimates, which correspond each
to a very simple mathematical fact. These lemmas are used a large number of times in 
the rest of the paper and can be considered
as the basis steps we use to obtain our results. 

From Section \ref{section_ivt} to \ref{sect_elim_of_one_var} we develop weak inference versions of  
different theorems. 
In Section 
\ref{section_ivt} we give a weak inference version of the Intermediate
Value Theorem for polynomials. 
In Section
\ref{section_fta} we give a weak inference version of the classical 
Laplace's proof of the Fundamental Theorem of Algebra and finally get a weak 
inference version of the factorization of a real polynomial into factors of degrees one and two.  
In Section 
\ref{sect_sylv_inertia_law}, which is independent from Section \ref{section_ivt}  and Section
\ref{section_fta}, we obtain 
incompatibilities expressing the impossibility for a polynomial 
to have a
number of real roots in  conflict with the rank and signature of its Hermite's quadratic form, 
through an incompatibility version of  Sylvester's Inertia Law. 
In Section \ref{sect_elim_of_one_var} we 
show how to eliminate a 
variable in a family of polynomials under weak inference form. 
As said before, all these results may be of interest independent of our main results. 
Finally, in Section \ref{secMainTh} we prove Theorem \ref{Theoremfinal} and Theorem 
\ref{Theoremfinal17}.

Each of Sections \ref{section_ivt} to \ref{sect_elim_of_one_var}
contains a final theorem which is the only
result from the section which is used outside the section, and it is used only in one of the 
remaining sections, as illustrated in the following diagram. 
$$\begin{array}{ccccccc}
\hbox{Section }\ref{section_ivt} & \rightarrow & \hbox{Section }\ref{section_fta} \cr
& & & \searrow \cr
&&&& \hbox{Section }\ref{sect_elim_of_one_var} & \rightarrow 
\hbox{Section }\ref{secMainTh} \cr
& &  & \nearrow \cr
& & \hbox{Section }\ref{sect_sylv_inertia_law} \cr
\end{array}
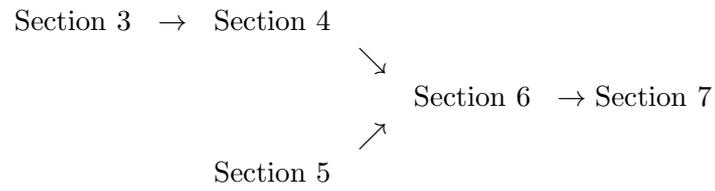
$$

A final annex provides the details of the proofs of several technical lemmas comparing 
the values of numerical functions which we use in our degree estimates.

\paragraph{Acknowledgements}
We would like to thank the anonymous referee for his/her relevant remarks. Special thanks to Saugata Basu for his linguistic advice.

\section{Weak inference and weak existence} \label{Weak.inferences}
\setcounter{equation}{0}

In this section we describe the concepts of weak inference (Definition \ref{def2.1.1})
and weak existence (Definition \ref{defWE})
introduced in \cite{Lom.e},
improving and making more precise results from \cite{Lom.d} (see also \cite{clr}).
These are mechanisms to construct new incompatibilities 
from other ones already available.
Most of the work we do in the paper is to
develop weak inference and weak existence versions of known mathematical and algorithmical results,
and obtain the corresponding degree estimates; therefore, these notions are central to our work.
Several examples of the use of these notions, which play a role in the other sections of the paper, are provided, the most important being
the case by case reasoning (see Subsection
\ref{casebycase}).

\subsection{Weak inference}

The idea behind the concept of weak inference is the following: 
let $\cF$, $\cF_1$, \dots, $\cF_m$, be systems of sign conditions  in $\K[u]=\K[u_1,\dots,u_n]$.
Suppose that we know that for every $
\vartheta
=(\vartheta_1,\dots,\vartheta_n) \in \R^n$
if the system $\cF$ is satisfied at $\vartheta$, then at least one of the systems
$\cF_1$, \dots, $\cF_m$ 
is also satisfied at $\vartheta$. If we are given initial incompatibilities  
$\lda {\cF}_1, \ {\cH}\rda_{\K[v]}$, \dots, $\lda {\cF}_m,  \ {\cH}\rda_{\K[v]}, v \supset u$, this means that 
all the systems 
$[{\cF}_1, \  {\cH}]$, \dots, $[{\cF}_m, \  {\cH}]$ are unrealizable. Then we can conclude that the
system $[\cF, \ \cH]$ is also unrealizable in $\R$ and we would like an incompatibility
$\lda \cF, \ \cH \rda_{\K[v]}$ to certify this fact. A weak inference is an explicit way to construct this 
final incompatibility 
from the given initial ones.

\begin{definition}[Weak Inference] \label{def2.1.1} Let $\cF, \cF_1, \dots, \cF_m$
be systems of sign conditions  in $\K[u]$.  
A {\em weak inference}
$$ \cF \ \ \ \vdash \ \ \  \bigvee_{1 \le j \le m} {\cF}_j 
$$
is a construction that, 
for any system of sign conditions $\cH$ in $\K[v]$ with $v\supset u$,
and any incompatibilities  
$$\lda {\cF}_1,  \ {\cH}\rda_{\K[v]}, \dots, \lda {\cF}_m, \  {\cH}\rda_{\K[v]}$$
called the \emph{initial} incompatibilities,
produces  
an incompatibility 
$$\lda \cF,  \ \cH \rda_{\K[v]}$$
called the \emph{final} incompatibility.
\end{definition}

Whenever we prove a weak inference, we also provide a description of the monoid part 
and a bound for the degree 
in the final incompatibility.
This information is necessary to obtain the degree bound in our main results.

\subsubsection{Basic rules}

In the following lemmas we give  some simple examples of weak inferences, most of them involving a one-term disjunction to the right 
(that is with $m=1$ in Definition \ref{def2.1.1}).

\begin{lemma}\label{lemma_basic_sign_rule_1} Let $P_1, P_2, \ldots, P_m \in \K[u]$. Then

\begin{eqnarray}
 P_1 > 0  &\vdash& P_1 \ge 0,  \label{lemma_basic_sign_rule:1}\\
P_1 > 0 &\vdash& P_1 \ne 0, \label{lemma_basic_sign_rule:1.5}\\
&\vdash & P^2_1 \ge 0, \label{lemma_basic_sign_rule:2} \\
 P_1 \ne 0  & \vdash &P_1^2 > 0,\label{lemma_basic_sign_rule:3} \\
 P_1 = 0  & \vdash  &P_1 \cdot P_2 = 0,  \label{lemma_basic_sign_rule:4}\\
\bigwedge_{1\le j \le m}  P_j \ne 0  & \vdash & \prod_{1\le j \le m} P_j \ne 0, \label{lemma_basic_sign_rule:5}\\
\bigwedge_{1\le j \le m}  P_j \ge 0  & \vdash & \prod_{1\le j \le m} P_j \ge 0, \label{lemma_basic_sign_rule:6} \\
\bigwedge_{1\le j \le m}  P_j > 0 & \vdash & \prod_{1\le j \le m} P_j > 0. \label{lemma_basic_sign_rule:7}
\end{eqnarray}

Moreover, in all cases, the initial incompatibility serves as the final incompatibility. 
\end{lemma}

\begin{proof}{Proof.}
Since the proof of all the items is very similar, we only prove (\ref{lemma_basic_sign_rule:7})
which is the least obvious one. 
Consider the initial incompatibility
$$
S\cdot \Big(\prod_{1\le j \le m} P_i\Big)^{2e} + N_0 + N_1 \cdot \prod_{1\le j \le m} P_i + Z  = 0
$$
with 
$S \in \scM(\cH_{\ne}^2)$, $N_0, N_1 \in \scN(\cH_{\ge})$ and $Z \in \scZ(\cH_{=})$,
where $\cH = [\cH_{\ne}, \, \cH_{\ge}, \, \cH_{=}]$ is a system of sign conditions
in $\K[v]$ with $v \supset u$. This proves the claim since 
 $$S\cdot \Big(\prod_{1\le j \le m} P_i\Big)^{2e}=S\cdot \prod_{1\le j \le m} P_i^{2e} \in 
\scM( (\cH_{\ne} \cup \{P_1, \dots, P_m\})^2),$$  
 $N_0 + N_1 \cdot \prod_{1\le i \le m} P_i \in \scN(\cH_{\ge} \cup \{P_1,\ldots, P_m\} )$ and  
$ Z \in  \scZ(\cH_{=}).$
\end{proof}

\begin{lemma}\label{lemma_basic_sign_rule_by_scalar} Let $\alpha \in \K, P \in \K[u]$. 

\noindent If $\alpha > 0$,
\begin{eqnarray}
P \ge 0   &\vdash&  \alpha P \ge 0, \label{lemma_basic_sign_rule_by_scalar:1} \\
P > 0   &\vdash&  \alpha P >0 .\label{lemma_basic_sign_rule_by_scalar:1bis}
\end{eqnarray}
If $\alpha < 0$,
\begin{eqnarray}
P \ge 0   &\vdash&  \alpha P \le 0, \label{lemma_basic_sign_rule_by_scalar:2} \\
P > 0   &\vdash&  \alpha P <0.\label{lemma_basic_sign_rule_by_scalar:2bis}
\end{eqnarray}
For any $\alpha $,
\begin{eqnarray}
P = 0   &\vdash&  \alpha P = 0 .\label{lemma_basic_sign_rule_by_scalar:3}
\end{eqnarray}

Moreover, in all cases, up to a division by an element of $\K_+$,  the initial incompatibility serves as the final incompatibility. 
\end{lemma}
\begin{proof}{Proof.} Immediate.
\end{proof}

\begin{lemma} \label{lemma_greatereq_and_lowereq_then_eq}
Let $P \in \K[u]$. Then
$$
P \ge 0, \ P \le 0 \ \ \ \vdash \ \ \  P = 0.
$$

If we have an initial incompatibility in $\K[v]$ where $v\supset u$ with monoid part 
$S$
and degree in $w \subset v$ bounded by $\delta_w$,
the final incompatibility has the same monoid part 
and degree in $w$ bounded by 
$\delta_w + \max\{\delta_w - \deg_w P, 0\}$.
\end{lemma}

\begin{proof}{Proof.}
Consider the initial incompatibility
$$
S + N + Z + W \cdot P = 0
$$
with 
$S \in \scM(\cH_{\ne}^2),$ $N \in \scN(\cH_{\ge})$, $Z \in \scZ(\cH_{=})$ and $W \in \K[v]$,
where $\cH = [\cH_{\ne}, \, \cH_{\ge}, \, \cH_{=}]$ is a system of sign conditions in $\K[v]$.
If $W$ is the zero polynomial there is nothing to do; otherwise 
we rewrite the initial incompatibility as 
$$\textstyle
S + N + \frac14 (1+W)^2 \cdot P + \frac14 (1-W)^2 \cdot (-P)  + Z = 0.
$$
This proves the claim since
$S \in \scM(\cH_{\ne}^2)$, $N + \frac14 (1+W)^2 \cdot P + \frac14 (1-W)^2 \cdot (-P) \in 
\scN(\cH_{\ge} \cup \{P, -P\})$ and $Z \in \scZ(\cH_{=})$.
The degree bound follows easily. 
\end{proof}

\begin{lemma}\label{lemma_sum_of_pos_and_zer_is_pos} Let $P_1, \dots, P_m \in \K[u]$. Then
\begin{eqnarray}
\bigwedge_{1 \le j \le m}  P_j = 0   &\vdash& \sum_{1 \le j \le m}P_j = 0, \label{lemma_sum_of_pos_and_zer_is_pos:1}\\
\bigwedge_{1 \le j \le m'}  P_j \ge 0, \ \bigwedge_{m' + 1 \le j \le m}  P_j = 0 &\vdash& \sum_{1 \le j \le m}P_j \ge 0. \label{lemma_sum_of_pos_and_zer_is_pos:2}
\end{eqnarray}

In both cases, if we have an initial incompatibility
 in $\K[v]$ where $v\supset u$ with monoid part  $S$ and degree in $w \subset v$ 
bounded by $\delta_w$, the final incompatibility has the same monoid part 
and degree in $w$ bounded by 
$\delta_w + \max\{\deg_w  P_j \ | \ 1 \le j \le m\} - \deg_w \sum_{1 \le j \le m}P_j .$
\end{lemma}
\begin{proof}{Proof.} 
We first prove item 
\ref{lemma_sum_of_pos_and_zer_is_pos:1}. 
Consider the initial incompatibility
$$
S + N +  Z + W \cdot \sum_{1 \le j \le m} P_j  = 0
$$
with 
$S \in \scM(\cH_{\ne}^2), N \in \scN(\cH_{\ge})$, $Z \in \scZ(\cH_{=})$ and $W \in \K[v]$,
where $\cH = [\cH_{\ne}, \, \cH_{\ge}, \, \cH_{=}]$ is a system of sign conditions
in $\K[v]$. 
We rewrite 
this equation as 
$$
S + N +  Z + \sum_{1 \le j \le m} W \cdot P_j = 0.
$$
This proves the claim since 
 $S \in \scM( \cH_{\ne} ^2)$,  
 $N \in \scN(\cH_{\ge})$ and  
$ Z + \sum_{1 \le j \le m} W \cdot P_j \in  \scZ(\cH_{=} \cup \{P_{1}, \dots, P_m \}).$
The degree bound follows easily.

Now we prove item 
\ref{lemma_sum_of_pos_and_zer_is_pos:2}. 
Consider the initial incompatibility
$$
S + N_0 + N_1 \cdot \sum_{1 \le j \le m} P_j + Z  = 0
$$
with 
$S \in \scM(\cH_{\ne}^2), N_0, N_1 \in \scN(\cH_{\ge})$ and $Z \in \scZ(\cH_{=})$,
where $\cH = [\cH_{\ne}, \, \cH_{\ge}, \, \cH_{=}]$ is a system of sign conditions
in $\K[v]$. 
We rewrite  
this equation as 
$$
S + N_0 + \sum_{1 \le j \le m'} N_1 \cdot P_j + Z + \sum_{m'+1 \le j \le m} N_1 \cdot P_j= 0.
$$
This proves the claim since 
 $S \in \scM( \cH_{\ne} ^2)$,  
 $N_0 + \sum_{1 \le j \le m'} N_1 \cdot P_j \in \scN(\cH_{\ge} \cup \{P_1, \dots, P_{m'}\} )$ and  
$ Z + \sum_{m'+1 \le j \le m} N_1 \cdot P_j \in  \scZ(\cH_{=} \cup \{P_{m'+1}, \dots, P_m \}).$
The degree bound follows easily.   
\end{proof}

\begin{lemma}\label{lemma_sum_ne_e_e} Let $P_1, \dots, P_m \in \K[u]$. Then
$$
P_1 \ne 0, \ \bigwedge_{2 \le j \le m}  P_j = 0
\ \ \ \vdash \ \ \  
\sum_{1 \le j \le m}P_j \ne 0.
$$ 

If we have an initial incompatibility in $\K[v]$ where $v\supset u$ with monoid part 
$
S(\sum_{1 \le j \le m}P_j)^{2e}
$
and  degree in $w \subset v$ bounded by $\delta_w$,
the final incompatibility has monoid part 
$S \cdot P_1^{2e}$
and degree in $w$ bounded by 
$\delta_w + 2e
\Big(\max\{\deg_w  P_j \ | \ 1 \le j \le m\} - \deg_w  \sum_{1 \le j \le m}P_j\Big).$
\end{lemma}

\begin{proof}{Proof.} 
Consider the initial incompatibility
$$
S \cdot \Big(\sum_{1 \le j \le m} P_j\Big)^{2e} + N  + Z  = 0
$$
with 
$S \in \scM(\cH_{\ne}^2), N \in \scN(\cH_{\ge})$ and $Z \in \scZ(\cH_{=})$,
where $\cH = [\cH_{\ne}, \, \cH_{\ge}, \, \cH_{=}]$ is a system of sign conditions
in $\K[v]$. 
We rewrite  
this equation as 
$$
S \cdot P_1^{2e}  + N + Z + Z_2 = 0
$$
where $Z_2 \in  \scZ(\{P_{2}, \dots, P_m \})$ is the sum of all the terms in the expansion of 
$S \cdot (\sum_{1 \le j \le m} P_j)^{2e}$ which involve at least one of  $P_{2}, 
\dots, P_m$.
This proves the claim since  
 $S \cdot P_1^{2e} \in \scM( (\cH_{\ne} \cup \{P_1\}) ^2)$,  
 $N \in \scN(\cH_{\ge})$ and  
$ Z + Z_2  \in  \scZ(\cH_{=} \cup \{P_{2}, \dots, P_m \}).$
The degree bound follows easily.   
\end{proof}

\begin{lemma}\label{lemma_sum_of_pos_is_pos} Let $P_1, \dots, P_m \in \K[u]$. Then
$$
P_1 > 0, \ \bigwedge_{2 \le j \le m'}  P_j \ge 0, \ \bigwedge_{m' + 1 \le j \le m}  P_j = 0  
\ \ \ \vdash \ \ \  
\sum_{1 \le j \le m}P_j > 0.
$$ 

If we have an initial incompatibility in $\K[v]$ where $v\supset u$ with monoid part 
$
S \cdot \big(\sum_{1 \le j \le m}P_j\big)^{2e}
$
and  degree in $w \subset v$ bounded by $\delta_w$,
the final incompatibility has monoid part 
$S \cdot P_1^{2e}$
and degree in $w$ bounded by 
$\delta_w + \max\{1, 2e\}
\Big(\max\{\deg_w  P_j \ | \ 1 \le j \le m\} - \deg_w  \sum_{1 \le j \le m}P_j\Big).$
\end{lemma}

\begin{proof}{Proof.} 
Consider the initial incompatibility
$$
S \cdot \Big(\sum_{1 \le j \le m} P_j\Big)^{2e} + N_0 + N_1 \cdot \sum_{1 \le j \le m} P_j + Z  = 0
$$
with 
$S \in \scM(\cH_{\ne}^2), N_0, N_1 \in \scN(\cH_{\ge})$ and $Z \in \scZ(\cH_{=})$,
where $\cH = [\cH_{\ne}, \, \cH_{\ge}, \, \cH_{=}]$ is a system of sign conditions
in $\K[v]$. 
We rewrite 
this equation as 
$$
S \cdot P_1^{2e}  + N_0 + N_2 + \sum_{1 \le j \le m'} N_1 \cdot P_j + Z + Z_2 + \sum_{m'+1 \le j \le m} N_1 \cdot P_j = 0,
$$
where $N_2 \in \scN( \{P_1, \dots, P_{m'} \})$ is the sum of all the terms in the expansion of 
$S \cdot (\sum_{1 \le j \le m} P_j)^{2e}$ which do not involve any of $P_{m'+1}, 
\dots, P_m$ with the exception of the term $S \cdot P_1^{2e}$
and $Z_2 \in  \scZ(\{P_{m'+1}, \dots, P_m \})$ is the sum of all the terms in the expansion of 
$S\cdot (\sum_{1 \le j \le m} P_j)^{2e}$ which involve at least one of $P_{m'+1}, 
\dots, P_m$.
This proves the claim since  
 $S \cdot P_1^{2e} \in \scM( (\cH_{\ne} \cup \{P_1\}) ^2)$,  
 $N_0 + N_2 + \sum_{1 \le j \le m'} N_1 \cdot P_j  \in \scN(\cH_{\ge} \cup \{P_1, \dots, P_{m'}\} )$ and  
$ Z + Z_2 + \sum_{m'+1 \le j \le m} N_1 \cdot P_j \in  \scZ(\cH_{=} \cup \{P_{m'+1}, \dots, P_m \}).$
The degree bound follows easily.   
\end{proof}

\begin{lemma}\label{lem:comb_lin_zero_zero} Let 
$m_1, \dots, m_n \in \N_*$ and $P_{j,k}, Q_{j,k} \in \K[u]$ for $1 \le j \le m_k, 1 \le k \le n$. Then 
$$
\bigwedge_{1 \le k \le n, \atop
1 \le j \le m_k} P_{j,k} = 0 
\ \ \ \vdash \ \ \  
\bigwedge_{1 \le k \le n} \sum_{1 \le j \le m_k} P_{j,k}\cdot Q_{j,k} = 0.
$$ 

If we have an initial incompatibility in $\K[v]$ where $v\supset u$ with monoid part 
$S$
and degree in $w \subset v$ bounded by $\delta_w$,
the final incompatibility has the same monoid part 
and degree in $w$ bounded by 
$$\delta_w + \max\Big\{\max\{\deg_w P_{j,k}\cdot Q_{j,k} \ | \ 1 \le j \le m_k\} - 
\deg_w \sum_{1 \le j \le m_k} P_{j,k}\cdot Q_{j,k} \ | \ 1 \le k \le n\Big\}.$$
\end{lemma}

\begin{proof}{Proof.} Follows from Lemmas \ref{lemma_basic_sign_rule_1} (item 
\ref{lemma_basic_sign_rule:4}) and an easy adaptation of the proof of 
Lemma \ref{lemma_sum_of_pos_and_zer_is_pos}
(item \ref{lemma_sum_of_pos_and_zer_is_pos:1}). 
\end{proof}

\begin{lemma} \label{lem_prod_le_g_then_fact_le}
Let $P_1, P_2 \in \K[u]$. Then
$$
P_1 \cdot P_2 \ge 0, \ P_2 > 0  \ \ \ \vdash \ \ \  P_1 \ge 0.
$$

If we have an initial incompatibility in $\K[v]$ where $v\supset u$ with monoid part 
$S$
and degree in $w \subset v$ bounded by $\delta_w$,
the final incompatibility has monoid part 
$S\cdot P_2^2$
and degree in $w$ bounded by 
$
\delta_w + 2\deg_w P_2
$.
\end{lemma}

\begin{proof}{Proof.}
Consider the initial incompatibility
$$
S + N_0 + N_1\cdot P_1 + Z  = 0
$$
with 
$S \in \scM(\cH_{\ne}^2)$, $N_0, N_1 \in \scN(\cH_{\ge})$ and $Z \in \scZ(\cH_{=})$,
where $\cH = [\cH_{\ne}, \, \cH_{\ge}, \, \cH_{=}]$ is a system of sign conditions
in $\K[v]$.
We multiply this equation by $P_2^2$ and we obtain
$$
S\cdot P_2^2 + N_0 \cdot P_2^2 + N_1\cdot P_1\cdot P_2^2 + Z\cdot P_2^2  = 0.
$$
This proves the claim since 
$S \cdot P_2^2 \in \scM((\cH_{\ne}\cup\{P_2\})^2)$, 
$N_0\cdot P_2^2 + N_1\cdot P_1 \cdot P_2^2 \in \scN(\cH_{\ge}\cup \{P_1 \cdot P_2, P_2\})$
and $Z \cdot P_2^2 \in \scZ(\cH_{=})$.
The degree bound follows easily. 
\end{proof}

\begin{lemma} \label{lem_prod_g_g_then_fact_g}
Let $P_1, P_2 \in \K[u]$. Then
$$
P_1 \cdot P_2 > 0, \ P_2 > 0  \ \ \ \vdash \ \ \  P_1 > 0.
$$

If we have an initial incompatibility in $\K[v]$ where $v\supset u$ with monoid part 
$S \cdot P_1^{2e}$
and degree in $w \subset v$ 
bounded by $\delta_w$,
the final incompatibility has monoid part 
$S \cdot P_2^2$ if $e = 0$ and $S\cdot (P_1 \cdot P_2)^{2e}$ if $e \ge 1$
and degree in $w$ bounded by 
$
\delta_w + 2\max\{1, e\}\deg_w P_2
$ in both cases.
\end{lemma}

\begin{proof}{Proof.}
Consider the initial incompatibility
$$
S\cdot P_1^{2e} + N_0 + N_1\cdot P_1 + Z  = 0
$$
with 
$S \in \scM(\cH_{\ne}^2)$, $N_0, N_1 \in \scN(\cH_{\ge})$ and $Z \in \scZ(\cH_{=})$,
where $\cH = [\cH_{\ne}, \, \cH_{\ge}, \, \cH_{=}]$ is a system of sign conditions
in $\K[v]$. If $e = 0$, we proceed as in the proof of Lemma \ref{lem_prod_le_g_then_fact_le}.
If $e \ge 1$, we multiply this equation by $P_2^{2e}$ and we obtain
$$
S\cdot (P_1 \cdot P_2)^{2e} + N_0 \cdot P_2^{2e} + N_1 \cdot P_1 \cdot P_2^{2e} + Z\cdot P_2^{2e}  = 0.
$$
This proves the claim since 
$S \cdot (P_1 \cdot P_2)^{2e} \in \scM((\cH_{\ne}\cup\{P_1 \cdot P_2, P_2\})^2)$, 
$N_0 \cdot P_2^{2e} + N_1 \cdot P_1 \cdot P_2^{2e} \in \scN(\cH_{\ge}\cup \{P_1 \cdot  P_2, P_2\})$
and $Z \cdot P_2^{2e} \in \scZ(\cH_{=})$.
The degree bound follows easily. 
\end{proof}

\begin{lemma}\label{lemma_sum_prod_pos_then_pos}
Let $P_1, P_2 \in \K[u]$. Then
$$P_1 + P_2 > 0, \ P_1 \cdot P_2 \ge 0 \ \ \ \vdash \ \ \ P_1 \ge 0,\ P_2 \ge 0.$$

If we have an initial incompatibility in $\K[v]$ where $v \supset u$ with monoid part
$S$ and degree in $w \subset v$ bounded by $\delta_w$,
the final incompatibility has monoid part 
$S \cdot (P_1 + P_2)^{2}
$
and degree in $w$ bounded by 
$
\delta_w +  2\max\{\deg_w  P_1, \deg_w  P_2\}
$.
\end{lemma}

\begin{proof}{Proof.}
Consider the initial incompatibility
$$
S + N_0 + N_1 \cdot P_1 + N_2 \cdot P_2 + N_3 \cdot P_1 \cdot P_2+ Z  = 0
$$
with 
$S \in \scM(\cH_{\ne}^2)$, $N_0, N_1, N_2, N_3 \in \scN(\cH_{\ge})$ and $Z \in \scZ(\cH_{=})$,
where $\cH = [\cH_{\ne}, \, \cH_{\ge}, \, \cH_{=}]$ is a system of sign conditions
in $\K[v]$. We multiply this equation by $(P_1 + P_2)^2$ and we rewrite it as 
$$
S \cdot (P_1 + P_2)^2 + N_0 \cdot (P_1 + P_2)^2 + N_1 \cdot P_1^2 \cdot (P_1 + P_2)  + N_2 \cdot P_2^2 \cdot (P_1 + P_2) + 
$$
$$ 
 + 
 (N_1 + N_2) \cdot (P_1 + P_2) \cdot P_1 \cdot P_2
+ N_3 \cdot (P_1 + P_2)^2 \cdot P_1 \cdot P_2 +
Z \cdot (P_1 + P_2)^2  = 0.
$$
This proves the claim since
$S \cdot (P_1 + P_2)^2 \in \scM((\cH_{\ne}\cup\{P_1 + P_2\})^2)$, 
$
N_0 \cdot (P_1 + P_2)^2 + N_1 \cdot P_1^2 \cdot (P_1 + P_2)  + N_2 \cdot P_2^2 \cdot (P_1 + P_2) 
 + 
 (N_1 + N_2) \cdot (P_1 + P_2) \cdot P_1 \cdot P_2 + N_3 \cdot (P_1 + P_2)^2 \cdot P_1 \cdot P_2
\in \scN(\cH_{\ge}\cup \{P_1 + P_2,  P_1 \cdot P_2 \})$
and $Z \cdot (P_1 + P_2)^2 \in \scZ(\cH_{=})$.
The degree bound follows easily. 
\end{proof}

\begin{lemma} \label{lem_prod_zero_at_least_one_fact_zero}
Let $P_1, \dots, P_m \in \K[u]$. Then
$$
\prod_{1 \le j \le m}P_j = 0  \ \ \ \vdash \ \ \  \bigvee_{1 \le j \le m}P_j = 0.
$$

If we have initial incompatibilities in $\K[v]$ where $v\supset u$ 
with monoid part 
$S_j$ and degree in $w \subset v$ bounded by $\delta_{w, j}$, 
the final incompatibility has monoid part 
$
\prod_{1 \le j \le m}S_j
$
and degree in $w$ bounded by 
$\sum_{1 \le j \le m}\delta_{w, j}.$
\end{lemma}

\begin{proof}{Proof.}
Consider for $1 \le j \le m$ the initial incompatibility
$$
S_j + N_j + Z_j + W_j \cdot P_j  = 0
$$
with 
$S_j \in \scM(\cH_{\ne}^2)$, $N_j \in \scN(\cH_{\ge})$,  $Z_j  \in \scZ(\cH_{=})$
and $W_j \in \K[v]$, 
where $\cH = [\cH_{\ne}, \, \cH_{\ge}, \, \cH_{=}]$ is a system of sign conditions
in $\K[v]$. 
We pass $W_j \cdot P_j$ to the right hand side 
in the 
initial incompatibility,
we multiply all the results and we pass $(-1)^m\prod_{1 \le j \le m}W_j \cdot P_j$ 
to the left hand side. We obtain 
$$
\prod_{1 \le j \le m}S_j + N + Z + (-1)^{m+1} \prod_{1 \le j \le m}W_j \cdot P_j = 0 
$$
where $N\in \scN(\cH_{\ge})$ is the sum of all the terms in the expansion of 
$\prod_{1 \le j \le m}(S_j + N_j)$ with the exception of the term $\prod_{1 \le j \le m}S_j$
and 
$Z \in \scZ(\cH_{=})$ is the sum of all the terms in the expansion of 
$\prod_{1 \le j \le m}(S_j + N_j + Z_j)$ which involve at least one of $Z_1, \dots, Z_m$. 
This proves the claim since
$\prod_{1 \le j \le m}S_j \in \scM(\cH_{\ne}^2)$, 
$
N
\in \scN(\cH_{\ge})$
and $Z + (-1)^{m+1} \prod_{1 \le j \le m}W_jP_j \in \scZ(\cH_{=} \cup \{\prod_{1 \le j \le m}P_j\})$.
The degree bound follows easily. 
\end{proof}

\subsubsection{Sums of squares}

The following remark states a very useful algebraic identity.

\begin{remark} \label{lemCS}
Let $\A$ be a commutative ring and $A_1, \dots, A_m,$ $B_1, \dots, B_m\in\A$. Consider
the sum of squares
$$
{\rm N}(A_1, \dots, A_m, B_1, \dots, B_m) = 
\sum_{\sigma  \in \{-1, 1\}^m, \atop
\sigma \ne (1, \dots, 1)} 
\Big(\sum_{1\le j \le m} \sigma(j) A_jB_j \Big)^2 
+ 2^m\sum_{1 \le j, j' \le m, \, j \ne j'} (A_jB_{j'})^2.
$$
Then
\begin{equation}\label{equation_identity_sum_of_squares}
\Big(\sum_{1\le j \le m} A_jB_j \Big)^2 + {\rm N}(A_1, \dots, A_m, B_1, \dots, B_m) = 2^m 
\sum_{1\le j \le m} A_j^2
\cdot 
\sum_{1\le j \le m} B_j^2.
\end{equation}
\end{remark}

We can now prove some more weak inferences. 

\begin{lemma}
\label{sos_non_pos_disjunct}
Let $P_1, \dots, P_m\in\K[u]$. Then
$$
\sum_{1 \le j \le m}  P_j^2  = 0
\ \ \ \vdash \ \ \   \bigwedge_{1 \le j \le m} P_j = 0. 
$$ 

If we have an initial incompatibility 
in $\K[v]$ where $v\supset u$ with monoid part 
$S$
and degree in $w \subset v$ bounded by $\delta_w$,
the final incompatibility has monoid part 
$S^{2}$
and degree in $w$ bounded by 
$$ 2\Big(\delta_w + \max\{\deg_w P_j \ | \ 1 \le j \le m\} - \min\{\deg_w  P_j \ | \ 1 \le j \le m\}\Big).$$
\end{lemma}
\begin{proof}{Proof.}
Consider the initial incompatibility 
$$S+N+Z+\sum_{1\le j \le m}W_j\cdot P_j=0$$
with 
$S \in \scM(\cH_{\ne}^2)$, $N \in \scN(\cH_{\ge})$,  
$Z \in \scZ(\cH_{=})$ and $W_j\in \K[v]$ for $1 \le j \le m$,
where $\cH$ $=$ $[\cH_{\ne}, \, \cH_{\ge}, \, \cH_{=}]$ is a system of sign conditions
in $\K[v]$. First, we pass $\sum W_j \cdot P_j$
to the right hand side, we raise to the square, we add  
${\rm N}(W_1, \dots, W_m, P_1, \dots, P_m)$ defined as in Remark \ref{lemCS} and we substitute
using (\ref{equation_identity_sum_of_squares}). Then we pass  
$2^m  \sum W_j^2
\cdot \sum P_j^2
$ to the left hand side and we obtain
$$ 
S^{2}+N_1 + {\rm N}(W_1, \dots, W_m, P_1, \dots, P_m) + Z_1 - 2^m
\sum_{1\le j \le m} W_j^2
\cdot
\sum_{1\le j \le m} P_j^2
= 0
$$
where $N_1 = 2N\cdot S  + N^2$ and $Z_1 = 2Z \cdot S  + 2Z \cdot N + Z^2$.
This proves the claim since  
$S^{2} \in \scM(\cH_{\ne}^2)$, $N_1 + {\rm N}(W_1, \dots, W_m, P_1, \dots, P_m) \in \scN(\cH_{\ge})$ and 
 $Z_1 - 2^m\sum W_j^2\cdot
\sum P_j^2 \in \scZ(\cH_{=}  \cup \{ \sum P_j^2 \} )$. 
The degree bound follows easily taking into account that 
$\deg_w \sum P_j^2 = 2 \max \{\deg_w  P_j\}$. 
\end{proof}

\begin{lemma}\label{comb_ne_some_elem_ne}
Let $P_1, \dots, P_m, Q_1, \dots, Q_m \in \K[u]$. Then
$$
\sum_{1 \le j \le m}P_j \cdot Q_j \ne 0 \ \ \ \vdash \ \ \ \sum_{1 \le j \le m}P_j^2  \ne 0.
$$

If we have an initial incompatibility in $\K[v]$ where $v \supset u$ with monoid part 
$
S \cdot (\sum_{1 \le j \le m}P_j^2 )^{2e}
$ 
and degree in $w \subset v$ bounded by $\delta_w$,
the final 
incompatibility has monoid part
$
S \cdot (\sum_{1 \le j \le m}P_j \cdot Q_j)^{4e}
$ 
and degree in $w$ bounded by 
$\delta_w + 4e\max\{\deg_w Q_j \ | \ 1 \le j \le m\}.$
\end{lemma}
\begin{proof}{Proof.}
Consider the initial incompatibility
$$
S \cdot \Big(\sum_{1 \le j \le m}P_j^2 \Big)^{2e} + N + Z  = 0
$$
with 
$S \in \scM(\cH_{\ne}^2)$, $N \in \scP(\cH_{\ge})$ and $Z \in \scI(\cH_{=})$,
where $\cH = [\cH_{\ne}, \, \cH_{\ge}, \, \cH_{=}]$ is a system of sign conditions in $\K[v]$.
We multiply this equation  by $2^{2me}(\sum Q_j^2 )^{2e}$, we substitute using
(\ref{equation_identity_sum_of_squares}) and we obtain
$$
S\cdot \Big(\sum_{1 \le j \le m}P_j \cdot Q_j \Big)^{4e} + N_1 + 
2^{2me}N \cdot \Big(\sum_{1\le j \le m} Q_j^2 \Big)^{2e} + 2^{2me}Z\cdot \Big(\sum_{1\le j \le m} Q_j^2 \Big)^{2e}  = 0
$$
where $N_1$ is the  
sum of all the terms in the expansion of 
$S \cdot ((\sum_{1\le j \le m} P_j\cdot Q_j)^2 + {\rm N}(P_1, \dots,  P_m, Q_1, \dots, Q_m))^{2e}$ 
with the exception of the term $S \cdot (\sum_{1\le j \le m} P_j \cdot Q_j)^{4e}$.
This proves the claim since 
$
S\cdot (\sum_{1\le j \le m} P_j \cdot Q_j )^{4e} \in \scM((\cH_{\ne} \cup \{ \sum_{1\le j \le m} P_j \cdot Q_j\})^2)$, 
$N_1 + 2^{2me}N \cdot (\sum_{1\le j \le m} Q_j^2 )^{2e} \in \scN(\cH_{\ge})$ and $
2^{2me}Z\cdot (\sum  Q_j^2 )^{2e} \in \scZ(\cH_{=})
$.
The degree bound follows easily.
\end{proof}

\subsubsection{Case by case reasoning}
\label{casebycase}

We will refer to the weak inferences in the following lemmas as ``case by case reasoning'', 
which
enable us to consider separately the different possible sign conditions in each case.

\begin{lemma} \label{CasParCas_1}
Let $P \in \K[u]$. Then
$$
\vdash \ \ \ P \ne 0 \ \, \vee \ \, P = 0.
$$ 

If we have initial incompatibilities in $\K[v]$ where $v\supset u$ with monoid part 
$S_{1}\cdot P^{2e}$ and $S_{2}$
and degree in $w \subset v$ bounded by $\delta_{w, 1}$
and $\delta_{w, 2}$,  
the final incompatibility has monoid part 
$S_{1}\cdot S_{2}^{2e}$
and degree in $w$ bounded by 
$\delta_{w, 1} + 2e(\delta_{w, 2} - \deg_w P)$.
\end{lemma}

\begin{proof} {Proof.}
Consider the initial incompatibilities
\begin{equation}\label{init_inc_neq_lema_cas_par_cas_1}
S_1\cdot P^{2e} + N_1 + Z_1   = 0
\end{equation}
and
\begin{equation}\label{init_inc_eq_lema_cas_par_cas_1}
S_2 + N_2 + Z_2 + W\cdot P = 0
\end{equation}
with 
$S_1, S_2 \in \scM(\cH_{\ne}^2)$, $N_1, N_2 \in \scN(\cH_{\ge})$, $Z_1, Z_2 \in \scZ(\cH_{=})$
and $W \in \K[v]$,
where $\cH$ $=$ $[\cH_{\ne}, \, \cH_{\ge}, \, \cH_{=}]$ is a system of sign conditions in 
$\K[v]$.
If $e = 0$ we take (\ref{init_inc_neq_lema_cas_par_cas_1}) as the final incompatibility. If $e\ne 0$ we
proceed as follows. 
We pass $W \cdot P$ to the right hand side in (\ref{init_inc_eq_lema_cas_par_cas_1}), we raise both
sides to the $(2e)$-th power and we multiply the result by $S_1$. We obtain 
\begin{equation}\label{aux_inc_eq_lema_cas_par_cas_1}
S_1 \cdot S_2^{2e} + N_3 + Z_3 = S_1\cdot W^{2e} \cdot P^{2e}  
\end{equation}
where $N_3 \in \scN(\cH_{\ge})$ is the sum of all the terms in the expansion of 
$S_1\cdot (S_2 + N_2 + Z_2)^{2e}$ which do not involve $Z_2$ with the exception of the term $S_1\cdot S_2^{2e}$
and $Z_3 \in \scZ(\cH_{=})$ is the sum of all the terms in the expansion of 
$S_1\cdot (S_2 + N_2 + Z_2)^{2e}$ which involve $Z_2$.
If $W$ is the zero polynomial, we take (\ref{aux_inc_eq_lema_cas_par_cas_1}) as 
the final incompatibility. Otherwise, 
we multiply (\ref{init_inc_neq_lema_cas_par_cas_1}) by $W^{2e}$, we substitute 
$S_1\cdot W^{2e} \cdot P^{2e}$ using 
(\ref{aux_inc_eq_lema_cas_par_cas_1}) and we obtain 
$$
S_1\cdot S_2^{2e} +  N_1\cdot W^{2e} + N_3 + Z_1\cdot W^{2e} + Z_3 = 0.
$$
This proves the claim since 
$S_1\cdot S_2^{2e} \in \scM(\cH_{\ne}^2)$, 
$
N_1\cdot W^{2e} + N_3
\in \scN(\cH_{\ge})$
and $Z_1\cdot W^{2e} + Z_3 \in \scZ(\cH_{=})$.
The degree bound follows easily. 
\end{proof}

\begin{lemma} \label{CasParCas_2}
Let $P \in \K[u]$. Then
$$
P \ne  0  \ \ \  \vdash \ \ \ P > 0 \ \,  \vee \ \, P < 0.
$$ 

If we have initial incompatibilities in $\K[v]$ where $v\supset u$ with monoid part 
$S_{1}\cdot P^{2e_1}$ and $S_{2}\cdot P^{2e_2}$
and degree in $w \subset v$ bounded by $\delta_{w, 1}$
and $\delta_{w, 2}$,  
the final incompatibility has monoid part 
$S_{1}\cdot S_{2}\cdot P^{2(e_1 + e_2)}$
and degree in $w$ bounded by 
$\delta_{w, 1} + \delta_{w, 2}$.
\end{lemma}

\begin{proof}{Proof.}
Consider the initial incompatibilities
\begin{equation}\label{init_inc_ge_lema_cas_par_cas_2}
S_1\cdot P^{2e_1} + N_1 + N'_1\cdot P + Z_1   = 0
\end{equation}
and
\begin{equation}\label{init_inc_le_lema_cas_par_cas_2}
S_2\cdot P^{2e_2} + N_2 - N'_2\cdot P + Z_2   = 0
\end{equation}
with 
$S_1, S_2 \in \scM(\cH_{\ne}^2)$, $N_1, N'_1, N_2, N'_2 \in \scN(\cH_{\ge})$ and 
$Z_1, Z_2 \in \scZ(\cH_{=})$,
where $\cH$ $=$ $[\cH_{\ne}, \, \cH_{\ge}, \, \cH_{=}]$ is a system of sign conditions in $\K[v]$.
We pass $N'_1\cdot P$ and $-N'_2 \cdot P$ to the right hand side in 
(\ref{init_inc_ge_lema_cas_par_cas_2}) and 
(\ref{init_inc_le_lema_cas_par_cas_2}), 
we multiply the results and we pass $-N'_1\cdot N'_2 \cdot P^2$ to the left hand side. We
obtain 
$$
S_{1}\cdot S_{2} \cdot P^{2(e_1 + e_2)} +  N_3 + N'_1\cdot N'_2 \cdot P^2 + Z_3 = 0
$$
where $N_3 =  N_1\cdot S_2 \cdot P^{2e_2} + N_2 \cdot S_1 \cdot P^{2e_1}  + N_1 \cdot N_2$ and $Z_3 = 
Z_1 \cdot S_2 \cdot P^{2e_2} + Z_2 \cdot S_1 \cdot P^{2e_1} + Z_1 \cdot N_2 + Z_2 \cdot N_1  + Z_1 \cdot Z_2$.
This proves the claim since
$S_1 \cdot S_2 \cdot P^{2(e_1 + e_2)}  \in \scM((\cH_{\ne}\cup\{P\})^2)$, 
$
N_3 + N'_1 \cdot N'_2 \cdot P^2 
\in \scN(\cH_{\ge})$
and $Z_3 \in \scZ(\cH_{=})$.
The degree bound follows easily. 
\end{proof}

\begin{lemma} \label{CasParCas_3}
Let $P \in \K[u]$. Then
$$
\vdash \ \ \ P > 0  \ \, \vee \ \, P < 0 \  \, \vee \ \, P = 0.
$$

If we have initial incompatibilities in $\K[v]$ where $v\supset u$ with monoid part 
$S_{1} \cdot P^{2e_1},  S_{2} \cdot P^{2e_2}$ and $S_{3}$
and degree in $w \subset v$ bounded by $\delta_{w, 1}$,
$\delta_{w, 2}$ and $\delta_{w, 3}$, 
the final incompatibility has monoid part 
$S_{1} \cdot S_{2} \cdot S_3^{2(e_1 + e_2)}$
and degree in $w$ bounded by 
$\delta_{w, 1} + \delta_{w, 2} + 2(e_1 + e_2)(\delta_{w, 3} - \deg_w P)$.
\end{lemma}

\begin{proof}{Proof.}
Follows from Lemmas \ref{CasParCas_1} and \ref{CasParCas_2}.
\end{proof}

\begin{lemma}\label{lem:multiple_case_by_case}
Let $P_1, \dots, P_m \in \K[u]$. Then
$$\vdash \ \ \ \bigvee_{J \subset \{1, \dots, m\}}
\Big(  \bigwedge_{j \not \in J} P_j \ne 0 ,  \ \bigwedge_{j  \in J} P_j = 0  \Big).   
$$ 

If we have initial incompatibilities in $\K[v]$ where $v\supset u$ 
with monoid part 
$
S_{J}\cdot \prod_{j \not \in J}P_j^{2e_{J, j}}
$,
degree in $w \subset v$ bounded by $\delta_{w}$,
and $e_{J, j} \le e \in \N_*$, 
the final incompatibility has monoid part 
$$
\prod_{J \subset \{1, \dots, m\}}S_{J}^{e'_{J}}
$$
with $e'_J \le 2^{2^{m+1} - m -2}e^{2^m - 1}$
and degree in $w$ bounded by 
$
2^{2^{m+1}  -2}e^{2^m - 1}\delta_{w}
$.
\end{lemma}
 
\begin{proof}{Proof.} Easy to prove by induction on $m$ using Lemma \ref{CasParCas_1}.
\end{proof}

\begin{lemma}\label{lem:multiple_case_by_case_intermediate}
Let $P_1, \dots, P_m \in \K[u]$. Then
$$\bigwedge_{1 \le j \le m} P_j \ne 0 \ \ \ \vdash \ \ \ \bigvee_{J \subset \{1, \dots, m\}}
\Big(  \bigwedge_{j  \in J} P_j > 0 ,  \ \bigwedge_{j  \not \in J} P_j < 0  \Big).   
$$ 

If we have initial incompatibilities in $\K[v]$ where $v\supset u$ 
with monoid part 
$
S_{J} \cdot \prod_{j}P_j^{2e_{J, j}}
$,
degree in $w \subset v$ bounded by $\delta_{w}$,
and $e_{J, j} \le e \in \N$,  
the final incompatibility has monoid part 
$$
\prod_{J \subset \{1, \dots, m\}}S_{J}
\cdot
\prod_{1 \le j \le m}P_j^{2e'_j}
$$
with $e'_J \le 2^me$ and degree in $w$ bounded by 
$2^m\delta_{w}$.
\end{lemma}
 
\begin{proof}{Proof.} 
Easy to prove by induction on $m$ using Lemma \ref{CasParCas_2}.
\end{proof}

\begin{lemma}\label{lem:multiple_case_by_case_with_signs}
Let $P_1, \dots, P_m \in \K[u]$. Then
$$\vdash \ \ \ \bigvee_{J \subset \{1, \dots, m\}
\atop
J' \subset \{1, \dots, m\}\setminus J}
\Big(  \bigwedge_{j  \in J'}  P_j > 0,  \bigwedge_{j  \not \in J \cup J'}
  P_j < 0, \ \bigwedge_{j  \in J} P_j = 0  \Big).   
$$ 

If we have initial incompatibilities in $\K[v]$ where $v\supset u$ 
with monoid part 
$
S_{J,J'} \cdot \prod_{j \not \in J}P_j^{2e_{J, J', j}}
$,
degree in $w \subset v$ bounded by $\delta_{w}$,
and $e_{J, J', j} \le e \in \N_*$, 
the final incompatibility has monoid part 
$$
\prod_{J \subset \{1, \dots, m\} \atop J' \subset \{1, \dots, m\}\setminus J}S_{J,J'}^{e'_{J,J'}}
$$
with $e'_{J,J'} \le 2^{2^{m+1} + m2^m - 2m -2}e^{2^m - 1}$
and degree in $w$ bounded by 
$
2^{2^{m+1} + m2^m -2}e^{2^m - 1}\delta_{w}
$.
\end{lemma}
 
\begin{proof}{Proof.} Follows from Lemmas \ref{lem:multiple_case_by_case} 
and \ref{lem:multiple_case_by_case_intermediate}.
\end{proof}

\subsection{Weak existence} \label{secBasicWE}

Weak inferences are constructions to obtain new incompatibilities from other 
incompatibilities already known. It will be  useful sometimes to introduce 
in the new incompatibilities, new sets of auxiliary variables.
Weak existence is a generalization of 
weak inference which 
enables us to do so. 

\begin{definition}[Weak Existence] \label{defWE} 
Consider disjoint sets of variables 
$u=(u_1,\ldots,u_n)$, $t_0=(t_{0,1},\ldots,t_{0,r_0})$,  $t_1=(t_{1,1},\ldots,t_{1,r_1}), \dots,$ 
$t_m=(t_{m,1},\ldots,t_{m,r_m})$.
Let $\cF(
t_0)$  be a system of sign conditions in $\K[u][t_0]$ and $\cF_1(
t_1) \dots, \cF_m(
t_m)$
be
systems of sign conditions  in $\K[u][t_1], \dots, \K[u][t_m]$. 
A {\em weak existence}
$$
\exists t_0 \ [\; \cF(
t_0) \;] \ \ \ \vdash \ \ \  
\bigvee_{1 \le j \le m} \exists t_j  \ [\; {\cF}_j(
t_j) \;]  
$$
is a construction that, 
given  any system of sign conditions $\cH
$ in $\K[v]$ with $v \supset u$, $v$ disjoint from 
$t_0, t_1, \dots, t_m$, 
 and \emph{initial} incompatibilities  
$$\lda {\cF}_1(
t_1),  \ {\cH}(v) \rda_{\K[v][t_1]}, \dots, \lda {\cF}_m(
t_m),  \ {\cH}(v)\rda_{\K[v][t_m]}$$
produces  
an incompatibility 
$$\lda \cF(
 t_0), \,  \cH(v) \rda_{\K[v][t_0]}$$
called the \emph{final} incompatibility.
\end{definition}

Note that the sets of variables $t_1, \dots, t_m$ 
which appear in the initial incompatibilities have been eliminated in the final incompatibility and
also the set of variables $t_0$ which do not appear in the initial incompatibilities 
has been introduced in the final incompatibility.

Most of the times, it will not be the case that we want to introduce and eliminate sets of variables 
simultaneously.  So,  for instance,
we write
$$
\cF
  \ \ \ \vdash \ \ \  
\bigvee_{1 \le j \le m} \exists t_j  \ [\; {\cF}_j(
t_j) \;] 
$$
for a weak existence in which the sets of variables $t_1, \dots, t_m$ have been eliminated but no 
new set of variables has been introduced.
We also write 
$$
\exists t_0 \ [\; \cF(
t_0) \;]    \ \ \ \vdash \ \ \  
\bigvee_{1 \le j \le m} \ {\cF}_j
$$
for a weak existence in which no sets of variables have been eliminated but a
new set of variables has been introduced.

We illustrate the concept of weak existence with a few lemmas. In general, 
we need to make a careful analysis of 
the degree bounds considering also the auxiliary variables.

\begin{lemma}  \label{lemma_weak_existence_inverse_sign}
Let  $P\in\K[u]$. Then 
$$
P \ne 0 \ \ \  \vdash \ \ \ \exists t \ [\; t\ne 0, \ P \cdot t = 1 \;].
$$

Suppose we have an initial incompatibility in $\K[v][ t]$ where $v \supset u$ and
$t \not \in v$, 
with monoid part $S \cdot t^{2e}$, 
degree in $w \subset v$ bounded by $\delta_w$ 
and degree in $t$ bounded by $\delta_t$. 
Let $\bar \delta_t$ be the smallest even number greater than or 
equal to $\delta_t$.
Then, the final 
incompatibility has monoid part $S \cdot P^{\bar \delta_t - 2e}$ and degree in $w$  bounded by 
$\delta_w + \bar \delta_t \deg_w P$.
\end{lemma}

\begin{proof}{Proof.}
Consider the initial incompatibility in $\K[v][ t]$
\begin{equation}\label{inc:init_inc_lemma_inverse}
S \cdot t^{2e} + \sum_i \omega_i V^2_i(t) \cdot N_i    +
\sum_j
W_j(t) \cdot Z_j + W(t) \cdot (P \cdot t - 1) = 0
\end{equation}
with $S \in {\scM}({{\cH}_{\neq}}^2)$, $\omega_i \in \K$,  $\omega_i > 0$,
$V_i(t) \in \K[v][t]$ and $N_i \in  {\scM}({\cH}_\ge)$ for every $i$,  $W_j(t) \in \K[v][t]$
and $Z_j \in {\cH}_=$ for every $j$ and $W(t) \in\K[v][t]$,
where $\cH = [\cH_{\ne}, \cH_{\ge}, \cH_{=}]$ is a system of sign conditions in $\K[v]$.

For every $i$, let $V_{i0}$ be the remainder of $P^{\frac12 \bar \delta_t} \cdot V(t)$ in the division 
by $Pt-1$ considering $t$ as the main variable; note that $\deg_w V_{i0} \le \deg_w V_i(t) + 
\frac12 \bar \delta_t \deg_w P$. 
Similarly, for every $j$, let $W_{j0}$ be the remainder of $P^{\bar \delta_t} \cdot W_j(t)$ in the division 
by $Pt-1$ considering $t$ as the main variable; 
note that $\deg_w W_{j0} \le \deg_w W_j(t) + 
\bar \delta_t \deg_w P$.

We multiply (\ref{inc:init_inc_lemma_inverse}) by $P^{\bar \delta_t}$
and deduce that there exists $W'(t) \in \K[v][t]$ such that
$$
S \cdot P^{\bar \delta_t - 2e} + \sum_i \omega_i V^2_{i0}  \cdot N_i    +
\sum_j W_{j0} \cdot Z_j + W'(t) \cdot (P \cdot t - 1) = 0.
$$
Looking at the degree in $t$, we have that $W'(t)$ is the zero polynomial. This proves
the claim since $S  \cdot P^{\bar \delta_t - 2e} \in {\scM}
(({\cH}_{\neq} \cup P)^2)$, 
$\sum \omega_i V^2_{i0} \cdot N_i
\in  {\scM}({\cH}_\ge)$ 
and $\sum W_{j0} \cdot Z_j \in {\cH}_=$. The degree bound follows easily. 
\end{proof}

\begin{lemma}\label{lemma_square_root_real_number}
Let $P \in \K[u]$. Then
$$
P \ge 0 \ \ \ \vdash \ \ \ \exists t \  [\; t^2 = P\;].
$$ 

If we have an initial incompatibility in $\K[v][ t]$ where  $v \supset u$ and
$t \not \in v$, 
with monoid part 
$S$, 
degree in $w \subset v$ bounded by $\delta_w$
and  
degree in $t$ bounded by $\delta_t$,
the final 
incompatibility has the same monoid part
and degree in $w$ bounded by 
$
\delta_w + \frac12\delta_t\deg_w P.
$

\end{lemma}
\begin{proof}{Proof.} 
Consider the initial incompatibility in $\K[v][t]$
\begin{equation}\label{eq:sq_root_pos_pol}
S + \sum_{i}  \omega_iV_i^2(t) \cdot N_i + \sum_{j}W_{j}(t) \cdot Z_j + W(t) \cdot (t^2 - 
P) = 0
\end{equation}
with $S \in \scM( \cH_{\ne}^2)$, $\omega_i \in \K,$ $\omega_i >0$, $V_i(t) \in \K[v][ t]$ and $N_i \in
\scM(\cH_{\ge})$ for every $i$, $W_j(t) \in \K[v][t]$ and $Z_j \in \cH_{=}$ for every $j$ and $W(t) \in \K[v][t]$,
where $\cH = [\cH_{\ne}, \cH_{\ge}, \cH_{=}]$ is a system of sign conditions in $\K[v]$. 

For every $i$, let $V_{i1} \cdot t + V_{i0}$ be the remainder of $V_i(t)$ in the division by $t^2 - P$ considering
$t$ as the main variable; note that 
$\deg_w V_{i0} \le \deg_w V_i(t) + \frac14\delta_t\deg_w P$ 
and
$\deg_w V_{i1} \le \deg_w V_i(t) + \frac14(\delta_t-2)\deg_w P$. 
Similarly, for every $j$, let $W_{j1}\cdot t + W_{j0}$ be the remainder of $W_j(t)$ in the division by $t^2 - P$ 
considering
$t$ as the main variable; note that 
$\deg_w W_{j0} \le \deg_w W_j(t) + \frac12 \delta_t\deg_w P$.  

From (\ref{eq:sq_root_pos_pol}) we deduce that exists $W'(t) \in \K[v][ t]$ such that 
$$
S + \sum_{i}  \omega_i( V_{i1} \cdot t + V_{i0} )^2 \cdot  N_i + \sum_{j}(W_{j1} \cdot t + W_{j0}) \cdot Z_j  + W'(t) \cdot (t^2 - 
P) = 0.
$$

We rewrite this equation as 
$$
S + \sum_{i}  \omega_i( V_{i1}^2 \cdot P  + V_{i0}^2) \cdot  N_i + \sum_{j}W_{j0} \cdot Z_j + W'''\cdot t  + W''(t)\cdot(t^2 - 
P) = 0.
$$
for some $W''' \in \K[v]$ and $W''(t) \in \K[v][t]$.

Looking at the degrees in $t$, we have that $W''(t)$ is the zero polynomial; and looking
again at the degree in $t$, we have that then also $W'''$ is the zero polynomial. 
This proves the claim since 
$S \in \scM(\cH_{\ne}^2)$, $\sum  \omega_i( V_{i1}^2 \cdot P  + V_{i0}^2) \cdot  N_i \in \scN(\cH_{\ge} \cup \{P\})$ 
and $\sum W_{j0} \cdot Z_j \in \scZ(\cH_{=})$. 
The degree bound follows easily. 
\end{proof}

\begin{lemma}\label{lemma_square_root_real_number_pos_pos}
Let $P \in \K[u]$. Then
$$
P > 0 \ \ \ \vdash \ \ \ \exists t \  [\;t>0, \  t^2 = P\;].
$$ 

If we have an initial incompatibility in $\K[v][ t]$ where $v \supset u$ and $t \not \in v$, 
with monoid part 
$S \cdot t^{2e}$, 
degree in $w \subset v$ bounded by $\delta_w$
and degree in $t$ bounded by $\delta_t$,
the final 
incompatibility has monoid part
$S^2 \cdot P^{2e}$
and degree in $w$ bounded by 
$
2\delta_w + (\max\{1, 2e\} + \delta_t)\deg_w P.$
\end{lemma}
\begin{proof}{Proof.} 
Consider the initial incompatibility in $\K[v][t]$
\begin{equation}
\label{inc:init_lemma_square_root_real_number_pos_pos}
S \cdot t^{2e} + N_1(t) + N_2(t)t   + Z(t)  + W(t) \cdot (t^2 - P)  = 0
\end{equation}
with
$S\in {\scM}({\cal H}_{\neq}^2)$, $N_1(t), N_2(t) \in {\scN}({\cal H}_\ge )_{\K[v][t]}$, 
$Z(t) \in {\scZ}({\cal H}_= )_{\K[v][t]}$ and $W(t)\in \K[v][ t]$, 
where $\cH$ is a system of sign conditions in $\K[v]$.

We substitute $t = -t$ in (\ref{inc:init_lemma_square_root_real_number_pos_pos}) and we 
obtain 
\begin{equation}\label{inc:aux_lemma_sq_root_real_pos_pos}
\lda 
t < 0, \
t^2 = P, \ 
\cH
\rda _{\K[v][t]}
\end{equation}
with the same monoid part and degree bounds.

Then we apply to 
(\ref{inc:init_lemma_square_root_real_number_pos_pos}) 
and (\ref{inc:aux_lemma_sq_root_real_pos_pos})
the weak inference
$$
 t  \not= 0 \ \ \ \vdash \ \ \  t > 0 \ \,  \vee \ \,  t  < 0.  
$$
By Lemma \ref{CasParCas_2}, we obtain 
$$
\lda 
t \ne 0, \
t^2 = P, \ 
\cH
\rda_{\K[v][t]}
$$
with monoid part $S^2 \cdot t^{4e}$, degree in $w$ bounded by $2\delta_w$ and 
degree in $t$ bounded by $2\delta_t$.
Since the exponent of $t$ in the monoid part is a multiple of $4$, this incompatibility is also an 
incompatibility 
\begin{equation}\label{inc:aux_lemma_sq_root_real_pos_pos_2}
\lda 
t^2 > 0, \
t^2 = P, \ 
\cH
\rda_{\K[v][t]}.
\end{equation}

Then we apply to (\ref{inc:aux_lemma_sq_root_real_pos_pos_2}) the weak inference
$$
P > 0, \ t^2 = P \ \ \ \vdash \ \ \ t^2 > 0.
$$
By Lemma \ref{lemma_sum_of_pos_is_pos}, we obtain 
\begin{equation}\label{inc:aux_lemma_sq_root_real_pos_pos_3}
\lda 
P > 0, \
t^2 = P, \ 
\cH
\rda_{\K[v][t]}
\end{equation}
with monoid part $S^2 \cdot P^{2e}$, 
degree in $w$ 
bounded by $2\delta_w + \max\{1, 2e\}\deg_w P$ and degree in $t$ bounded by $2\delta_t$. 

Finally 
we apply to (\ref{inc:aux_lemma_sq_root_real_pos_pos_3}) the weak inference 
$$
P \ge 0 \ \ \ \vdash \ \ \ \exists t \  [\; t^2 = P\;].
$$ 
By Lemma \ref{lemma_square_root_real_number}, we obtain 
$$
\lda
P > 0, \ \cH
\rda_{\K[v]}
$$
with the same monoid part and degree in $w$ bounded by 
$2\delta_w + (\max\{1, 2e\} +  \delta_t)\deg_w P$, 
which serves as the final incompatibility. 
\end{proof} 

\begin{remark}
 In the preceeding lemmas, we have no case of a weak existence with an existential variable to the left. The first example of such a situation appears later in the paper, when we deal with the Intermediate Value Theorem  in Section \ref{section_ivt}.
\end{remark}

\subsection{Complex numbers}

We introduce  the  conventions we follow to deal with complex
variables  in the context of weak inference, which has been originally defined to be well adapted to a 
real setting. 

\begin{notation}[Complex Variables]\label{real-imaginary}
A complex variable, always named $z$, represents two variables corresponding to its
real and imaginary  parts, always named $a$ and $b$,
so that $z=a+ib$. 
We also use $z$ to denote a set of complex variables and $a$ and $b$ to denote the set
of real and imaginary parts of $z$.

Let $z = (z_1, \dots, z_n)$ and $P \in \K[i][u][ z]$. 
We denote by $P\re \in \K[u][ a, b]$ and $P\im \in \K[u][ a, b]$
the real and imaginary parts of $P$.
The expression
$P = 0$ 
is an abbreviation for
$$P\re = 0, \ P\im =0,$$
and the expression
$P \not=0$ 
is an abbreviation for
$$P\re^2+P\im^2\not=0.$$
\end{notation}

We illustrate the use of complex variables with
some lemmas.

\begin{lemma}\label{lemma_square_root_complex_ne} Let $C, D \in \K[u]$. Then
$$
C + iD \ne 0 \ \ \ \vdash \ \ \ \exists z  \ [\; z \ne 0, \   z^2  = C + iD \;],
$$
where $z$ is a complex variable
(using Notation \ref{real-imaginary})
If we have
an initial incompatibility 
in $\K[v][ a, b]$ where $v \supset u$ and $a, b \not \in v$, with 
monoid part $S  \cdot (a^2 + b^2)^{2e}$,  
degree in $w\subset v$ bounded by $\delta_w$ 
and
degree in $(a,b)$ bounded by $\delta_z$,
the final incompatibility has
monoid part $S^{4} \cdot (C^2 + D^2)^{2(2e+1)}$
and degree in $w$ bounded by 
$4\delta_w + (20 +  24e +  8\delta_z)\max\{\deg_w C, \deg_w D\}.$

\end{lemma}

\begin{proof}{Proof.}
Consider the initial incompatibility in $\K[v][a,b]$
\begin{equation}
\label{inc_dep_quad_aux_ne}
S \cdot (a^2 + b^2)^{2e} + N(a,b)   + Z(a,b)  + W_1(a,b) \cdot (a^2 - b^2 - C)  +  
W_2(a,b) \cdot (2a \cdot b-D)  = 0
\end{equation}
with
$S\in {\scM}({\cal H}_{\neq}^2)$, $N(a,b) \in {\scN}({\cal H}_\ge )_{\K[v][a,b]}$, 
$Z(a,b) \in {\scZ}({\cal H}_= )_{\K[v][a,b]}$ and $W_1(a,b), W_2(a,b) \in \K[v][ a, b]$, 
where $\cH$ is a system of sign conditions in $\K[v]$.

We substitute $b = -b$ in (\ref{inc_dep_quad_aux_ne}) and we obtain 
\begin{equation}\label{inc:aux_lemma_sq_root_complex}
 \lda z \ne 0, \ 
z^2 = C - iD, \ \cH 
\rda_{\K[v][a,b]}
\end{equation}
with the same monoid part and degree bounds. 

Then we apply to (\ref{inc_dep_quad_aux_ne}) and (\ref{inc:aux_lemma_sq_root_complex}) 
the weak inference
$$ (2a \cdot b)^2 = D^2 \ \ \ \vdash \ \ \  2a \cdot b = D \ \, \vee \ \, 2a \cdot b = -D.$$
By Lemma \ref{lem_prod_zero_at_least_one_fact_zero}, we obtain 
\begin{equation}\label{eq:aux_lemma_square_root_complex_number}
\lda 
z \ne 0, \ a^2 - b^2 = C, \
(2a \cdot b)^2 = D^2, \
\cH
\rda_{\K[v][a,b]}
\end{equation}
with monoid part $S^2 \cdot (a^2 + b^2)^{4e}$, 
degree in $w$ bounded by $2\delta_w$ and degree in $(a, b)$ bounded by $2\delta_z$.

We consider a new auxiliary variable $t$. Taking into account the identities
$$
\begin{array}{rcl}
 a^2 - b^2 - C&=&\Big(a^2 - \frac12 (t+C) \Big) - \Big(b^2 - \frac12(t-C) \Big),\\
 (2a \cdot b)^2-D^2&=&\Big(a^2 - \frac12 (t+C) \Big) \cdot 4b^2 +  
\Big(b^2 - \frac12(t-C)\Big)\cdot2(t+C) + (t^2 - C^2 - D^2),
\end{array}
$$
we apply to (\ref{eq:aux_lemma_square_root_complex_number}) the weak inference
$$
a^2 = \frac12(t+C), \ b^2 = \frac12(t-C), \ t^2 = C^2 + D^2  \ \ \  \vdash \ \ \ 
a^2 - b^2 = C, \ (2a\cdot b)^2 = D^2. 
$$
By Lemma \ref{lem:comb_lin_zero_zero}, 
we obtain 
\begin{equation}\label{eq:aux_lemma_square_root_complex_number_2}
\lda 
z \ne 0, \
a^2 = \frac12 (t+C), \
b^2 = \frac12(t-C), \ 
t^2 = C^2 + D^2,  \
\cH
\rda
_{\K[v][a,b,t]}
\end{equation}
with monoid part $S^2\cdot (a^2 + b^2)^{4e}$, degree in $w$
bounded by $2\delta_w + 2\deg_w C$, 
degree in $(a,b)$ bounded by $2\delta_z$ and degree in $t$ bounded by $2$.

Then we apply to (\ref{eq:aux_lemma_square_root_complex_number_2}) the weak 
inference
$$
t \ne 0, \ a^2 = \frac12(t+C), \ b^2 = \frac12(t-C)  \ \ \  \vdash \ \ \ 
z \ne 0. 
$$
By Lemma \ref{lemma_sum_ne_e_e} we obtain 
\begin{equation}\label{eq:aux_lemma_square_root_complex_number_2_aux}
\lda 
t \ne 0, \
a^2 = \frac12 (t+C), \
b^2 = \frac12(t-C), \ 
t^2 = C^2 + D^2,  \
\cH
\rda
_{\K[v][a,b,t]}
\end{equation}
with monoid part $S^2 \cdot t^{4e}$, degree in $w$
bounded by $2\delta_w + (2+4e)\deg_w C$, 
degree in $(a,b)$ bounded by $2\delta_z$ and degree in $t$ bounded by $2 + 4e$.

Then we successively apply  to (\ref{eq:aux_lemma_square_root_complex_number_2_aux}) the
weak inferences
$$
\begin{array}{rcl}
t+C \ge 0 & \ \, \vdash \, \  & \exists a \ [\;  a^2 = \frac12(t+C) \;], \\[3mm]
t-C \ge 0 & \vdash & \exists b \ [\; b^2 = \frac12(t-C) \;]. 
\end{array}
$$
By Lemma \ref{lemma_square_root_real_number}, we obtain
\begin{equation}\label{eq:aux_lemma_square_root_complex_number_3}
\lda
t \ne 0, \ t + C \ge 0, \ t-C \ge 0, \ t^2 = C^2 + D^2, \ \cH
\rda_{\K[v][t]}
\end{equation}
with monoid  part $S^2\cdot t^{4e}$, 
degree in $w$ bounded by 
$2 \delta_w + (2 + 4e + 2\delta_z)\deg_w C$, 
and degree in $t$ bounded by $2 + 4e + 2\delta_z$.

Finally we successively apply to (\ref{eq:aux_lemma_square_root_complex_number_3})   the
weak inferences
$$
\begin{array}{rcl}
t > 0, \ t^2 - C^2 \ge 0 &   \vdash & t+C \ge 0,\ t-C \ge 0, \\[3mm]
D^2 \ge 0, \ t^2 = C^2+ D^2 & \vdash & t^2 - C^2 \ge  0, \\[3mm]
& \vdash & D^2 \ge 0, \\[3mm]
C^2 + D^2 > 0 & \vdash & \exists t \ [\; t > 0,  \   t^2 = C^2 + D^2  \;].\\[3mm]
\end{array}
$$
By Lemmas \ref{lemma_sum_prod_pos_then_pos}, 
\ref{lemma_sum_of_pos_and_zer_is_pos} (item \ref{lemma_sum_of_pos_and_zer_is_pos:2}), 
\ref{lemma_basic_sign_rule_1} (item \ref{lemma_basic_sign_rule:2})
and 
\ref{lemma_square_root_real_number_pos_pos}, 
we obtain an incompatibility in $\K[v]$
$$
\lda
C^2 + D^2 > 0, \
\cH 
\rda_{\K[v]}
$$
with monoid part $S^4\cdot (C^2 + D^2)^{2(2e+1)}$ and degree in $w$ bounded by 
$4\delta_w + (20 + 24e +  8\delta_z)$ $\max\{\deg_w C, \deg_w D\}$. 
Note that this incompatibility is also an incompatibility 
\begin{equation}
\lda
C^2 + D^2 \ne 0, \
\cH 
\rda_{\K[v]}
\end{equation}
with the same degree bound, which serves as the final incompatibility.
\end{proof}

\begin{lemma}\label{lemma_square_root_complex} Let $C, D \in \K[u]$. Then
$$
\vdash \ \ \ \exists z  \ [\;   z^2  = C + iD\;],
$$
where $z$ is a complex variable
(using Notation \ref{real-imaginary}).

If we have
an initial incompatibility in  $\K[v][ a, b]$ where $v \supset u$ and $a, b \not \in v$, with 
monoid part $S$,  
degree in $w\subset v$ bounded by $\delta_w$ 
and
degree in $(a,b)$ bounded by $\delta_z$,
the final incompatibility has
monoid part $S^{8}$
and degree in $w$ bounded by 
$
8\delta_w + (20 + 8\delta_z)\max\{\deg_w C, \deg_w D\}.
$

\end{lemma}

\begin{proof}{Proof.} Consider the initial incompatibility in $\K[v][a,b]$
\begin{equation}
\label{inc_dep_quad_aux}
S + N(a,b)   + Z(a,b)  + W_1(a,b) \cdot (a^2 - b^2 - C)  +  
W_2(a,b)(2a \cdot b-D)  = 0
\end{equation}
with
$S\in {\scM}({\cal H}_{\neq}^2)$, $N(a,b) \in {\scN}({\cal H}_\ge )_{\K[v][a,b]}$, 
$Z(a,b) \in {\scZ}({\cal H}_= )_{\K[v][a,b]}$ and $W_1(a,b), W_2(a,b) \in \K[v][ a, b]$, 
where $\cH$ is a system of sign conditions in $\K[v]$. 

We proceed by case by case reasoning. First we consider the case $C^2+ D^2  \ne 0$.
We apply to (\ref{inc_dep_quad_aux}) the weak inference
$$
C^2 + D^2 \ne 0 \ \ \ \vdash \ \ \ \exists z \ [\; z \ne 0, \ z^2 = C + iD\;].
$$
By Lemma \ref{lemma_square_root_complex_ne} we obtain 
\begin{equation}\label{inc:aux_case1_sq_root_comp_number_aux}
\lda
C^2 + D^2 \ne 0, \
\cH 
\rda_{\K[v]}
\end{equation}
with monoid part $S^4 \cdot (C^2 + D^2)^2$ and degree in $w$ bounded by 
$4\delta_w + (20 + 8\delta_z)$ $\max\{\deg_w C,$ $\deg_w D\}$. 

We consider then the case $C^2 + D^2 = 0$. 
We evaluate $a = b = 0$ in  (\ref{inc_dep_quad_aux}) and we apply the weak inference
$$
C^2 + D^2 = 0 \ \ \ \vdash \ \ \  C =0, \ D = 0.
$$
By Lemma \ref{sos_non_pos_disjunct}, we obtain
\begin{equation} \label{inc:aux_case2_sq_root_comp_number}
\lda C^2 + D^2 = 0, \ \cH \rda_{\K[v]}
\end{equation}
with monoid part $S^2$ and 
degree in 
$w$ bounded by $2\delta_w + 2\max\{\deg_w C, \deg_w D\}$.

Finally we apply to (\ref{inc:aux_case1_sq_root_comp_number_aux}) and (\ref{inc:aux_case2_sq_root_comp_number}) the weak inference
$$  \vdash \ \ \ C^2 + D^2 \ne 0 \ \, \vee \ \, C^2 + D^2 = 0.$$ 
By Lemma \ref{CasParCas_1}, we obtain 
$$ 
\lda
\cH
\rda
_{\K[v]}
$$
with monoid part $S^8$
and degree in $w$ bounded by
$
8\delta_w + (20 + 8\delta_z)\max\{\deg_w C, \deg_w D\}    
$, which serves as the final incompatibility. 
\end{proof}

\begin{lemma}\label{lemma_complex_root_quadratic} 
Let $E
 =  y^2 + G \cdot y + H 
 \in \K[i][u][y]$. Then  
$$  
\vdash \ \ \ \exists z  \ [\; E( z) = 0\;],
$$
where $z$ is a complex variable
(using Notation \ref{real-imaginary}).

If we have
an initial incompatibility in $\K[v][ a, b]$ where  $v \supset u$ and $a, b \not \in v$, with 
monoid part $S$,  
degree in $w\subset v$ bounded by $\delta_w$  
and degree in $(a,b)$ bounded by $\delta_z$,
the final incompatibility has
monoid part 
$
S^{8}
$
and degree in $w$ bounded by 
$8\delta_w + (40 + 24 \delta_z)\max\{\deg_w  G, \deg_w  H\}.$
\end{lemma}

\begin{proof}{Proof.}
Consider the initial incompatibility in $\K[v][ a, b]$
\begin{equation}\label{inc_dep_quad}
S + N(a,b)  + Z(a,b)  + W_1(a,b) \cdot E\re (a,b)  +  
W_2(a,b) \cdot E\im (a,b)  = 0
\end{equation}
with
$S\in {\scM}({\cal H}_{\neq}^2)$, $N(a,b) \in {\scN}({\cal H}_\ge )_{\K[v][a,b]}$, 
$Z(a,b) \in {\scZ}({\cal H}_= )_{\K[v][a,b]}$ and $W_1(a,b), W_2(a,b) \in \K[v][ a, b]$, 
where $\cH$ is a system of sign conditions in $\K[v]$.

Let  $C =  \frac14G\re ^2 - \frac14G\im ^2 - H\re \in \K[u]$ 
and $D =  \frac12 G\re G\im  - H\im \in \K[u]$.  
Then we have 
$$
\begin{array}{rcl}
 E\re (a, b) &= &
 a^2 - b^2 +  G\re  \cdot a  -  G\im  \cdot  b + H\re  = \Big(a + \frac12G\re \Big)^2 - \Big(b + \frac12G\im \Big)^2 - C,\\
 E\im (a, b) &=&
2a \cdot b +  G\im  \cdot a +  G\re  \cdot b   +  H\im  = 2\Big(a + \frac12 G\re \Big) \cdot \Big(b + \frac12 G\im \Big) - D.
\end{array}
$$
We substitute  $a = a - \frac12G\re $ and $b = b - 
\frac12G\im $ in (\ref{inc_dep_quad}) and   
we obtain 
\begin{equation} \label{inc:lemma_quad_comp_root_aux}
 \lda
z^2 = C + i D, \ \cH
\rda
_{\K[v][a,b]}
\end{equation}
with monoid part $S$, degree in $w$
bounded by $\delta_w + \delta_z \deg_wG$ and 
degree in $(a,b)$ bounded by $\delta_z$.

Finally we apply to (\ref{inc:lemma_quad_comp_root_aux}) the weak inference
$$ \vdash \ \ \ \exists\, z \ [\;z^2 = C + iD\;].$$
By Lemma \ref{lemma_square_root_complex}, we obtain 
$$
\lda
\cH
\rda
_{\K[v]}
$$
with monoid part $S^8$
and degree in $w$
bounded by 
$8\delta_w + (40 + 24 \delta_z)\max\{\deg_w G, \deg_w H\}$, which serves as the final incompatibility.
\end{proof}

\subsection{Identical polynomials}

We introduce the notation we use to deal with polynomial identities in the weak inference 
context.

\begin{notation}[Identical Polynomials] Let $P
 = \sum_{0 \le h \le p}C_h\cdot y^h, Q
  = \sum_{0 \le h \le p}D_h\cdot y^h 
\in \K[u][ y]$. The 
expression $P
\equiv Q
$ is an abbreviation for 
$$
\bigwedge_{0\leq h\leq p} C_h = D_h.
$$ 
\end{notation}
Note that $P
 \equiv Q
 $ is a conjunction of polynomial equalities in 
$\K[u]$.

We illustrate the use of this notation with a few lemmas.

\begin{lemma}\label{lem_sust_equiv_greater}
Let $P
, Q
 \in \K[u][y]$ with $\deg_y P  = \deg_y Q$. Then 
$$P
 \equiv Q
 , \ Q
  > 0 \ \ \   \vdash \ \ \ P
   > 0.$$  

If we have an initial incompatibility in $\K[v]$ where $v \supset (u,y)$, with monoid part 
$S \cdot P
^{2e}$
and degree in $w\subset v$ bounded by $\delta_w$,
the final incompatibility has monoid part 
$S \cdot Q
^{2e}$
and degree in $w$ bounded by 
$$
\delta_w + \max\{1, 2e\}\big(\max\{\deg_w P
,    \deg_w  Q
\} - \deg_w P
\big).
$$
\end{lemma}
\begin{proof}{Proof.}
Follows from Lemmas \ref{lemma_basic_sign_rule_1} (item \ref{lemma_basic_sign_rule:4}) and 
\ref{lemma_sum_of_pos_is_pos}. 
\end{proof}

\begin{lemma}\label{lem:two_linear_factor}
Let $P
\in \K[u][y]$ with $\deg_y P \ge 2$. Then
$$ P(t_1) = 0,  \  {\rm Quot}(P,y-t_1)(t_2) = 0    \  \ \  \vdash \ \ \ P
 \equiv 
(y-t_1) \cdot (y-t_2) \cdot {\rm Quot}(P, (y - t_1)(y-t_2)).$$

If we have an initial incompatibility in $\K[v]$ where 
$v \supset (u,t_1,t_2)$  with monoid part 
$S$  and  degree in $w\subset v$ bounded by $\delta_w$,
the final incompatibility has the same monoid part 
and degree in $w$ bounded by 
$$ 
\delta_w + \max\{\deg_w (t_1 \cdot \Quot(P, y-t_1)(t_2)), \deg_w P(t_1) \}
- \deg_w (-t_1 \cdot \Quot(P, y-t_1)(t_2) + P(t_1)).$$
\end{lemma}
\begin{proof}{Proof.} 
Because of the  identity in $\K[u][t_1,t_2, y]$
$$
P
 = (y-t_1) \cdot (y-t_2) \cdot \Quot(P, (y-t_1)(y-t_2)) 
+ 
\Quot(P, y-t_1)(t_2)\cdot y -t_1 \cdot \Quot(P, y-t_1)(t_2) + P(t_1),
$$
the lemma follows from Lemma \ref{lem:comb_lin_zero_zero}.
\end{proof}

\begin{lemma}\label{quadratic_factor}
Let $P
\in \K[u][y]$ with $\deg_y P \ge 2$. Then
$$ P(z) = 0,  \ b \ne 0 \ \ \ \vdash \ \ \ P
 \equiv ((y-a)^2  + b^2) \cdot 
{\rm Quot}(P, (y-a)^2  + b^2).$$

If we have an initial incompatibility in $\K[v]$ where 
$v \supset (u,a,b)$  with monoid part 
$S$  and  degree in $w\subset v$ bounded by $\delta_w$,
the final incompatibility has monoid part 
$S \cdot b^2$ and degree in $w$ bounded by 
$\delta_w + \deg_w b^2 + \deg_w P$.
\end{lemma}

\begin{proof}{Proof.}
Because of the  identity in $\K[u][a,b, y]$
$$
P
= ((y-a)^2 + b^2) \cdot \Quot(P, (y-a)^2 + b^2) 
+ 
\frac{P\im (a,b)}{b}y 
+ \frac{b P\re (a,b) -a  \cdot P\im (a,b)}{b},
$$
the initial incompatibility is of type 
\begin{equation} \label{inc:dep_lemma_quad_factor}
S + N + Z + W_1\frac{P\im (a,b)}{b} + W_2 \frac{b  \cdot  P\re (a,b) -a  \cdot P\im (a,b)}{b} = 0
\end{equation}
with 
$S\in {\scM}({\cal H}_{\neq}^2)$, $N \in {\scN}({\cal H}_\ge )$, 
$Z \in {\scZ}({\cal H}_=)$ and $W_1, W_2 \in \K[v]$, 
where $\cH$ is a system of sign conditions in $\K[v]$. 

We multiply (\ref{inc:dep_lemma_quad_factor})  by $b^2$ and we 
obtain an incompatibility
\begin{equation} \label{inc:dep_lemma_quad_factor_2}
\lda
b \ne 0, \ b \cdot P\im(a,b) = 0, \ b^2 \cdot P\re(a,b) - a \cdot b \cdot P\im(a,b) = 0, \ \cH
\rda
_{\K[v]}
\end{equation}
with monoid part 
$S \cdot b^2$ and degree in $w$ bounded by 
$\delta_w + \deg_w b^2$. 

Finally we apply to (\ref{inc:dep_lemma_quad_factor_2}) the weak inference
$$
P(z) = 0 \ \ \ \vdash \ \ \ b \cdot P\im(a,b) = 0, \ b^2 \cdot P\re(a,b) - a \cdot b \cdot P\im(a,b) = 0.
$$
By Lemma \ref{lem:comb_lin_zero_zero}, we obtain an incompatibility
$$
\lda
P(z) = 0, \ b \ne 0, \ \cH
\rda
_{\K[v]}
$$
with the same monoid part
and, after some analysis, degree in $w$ bounded by $\delta_2 + \deg_w b^2 + \deg_w P$, 
which serves as
the final incompatibility.
\end{proof}

\begin{notation}
\label{ResR}
 We denote
$${\rm R}
(z,z')=\Res_y((y - a)^2 + b^2, (y - a')^2 + b'^2)$$
where $\Res_y$ is  the 
resultant polynomial in the variable $y$.
Note that
$$
{\rm R}(
z,z') = ((a-a')^2+(b-b')^2) \cdot ((a-a')^2+(b+b')^2).
$$ 
\end{notation}

\begin{lemma}\label{lem:resultant_zero_then_eq}
$$
{\rm R}(
z,z')=0
\ \ \ \vdash
\ \ \ (y - a)^2 + b^2 \equiv (y - a')^2 + {b'}^2.
$$  

If we have an initial incompatibility in $\K[v]$ where 
$v \supset (a,b, a', b')$  with monoid part 
$S$  and  degree in $w\subset v$ bounded by $\delta_w$,
the final incompatibility has monoid part 
$S^4$ and degree in $w$ bounded by 
$$4\Big(\delta_w + \max\{ \deg_w a-a', \deg_w b- b' \} - \min\{\deg_w a-a', \deg_w b- b'\}\Big).$$
\end{lemma}

\begin{proof}{Proof.}
Consider the initial incompatibility  
\begin{equation} \label{inc:init_inc_lemma_eq_res_cuad}
\lda a - a' = 0, \ a^2 + b^2 - {a'}^2 - {b'}^2 = 0, \ \cH \rda_{\K[v]}
\end{equation}
where $\cH$ is a system of sign conditions in $\K[v]$. 
On the one hand, we successively apply to (\ref{inc:init_inc_lemma_eq_res_cuad})
the weak inferences
$$
\begin{array}{rcl}
a^2 - {a'}^2 = 0, \ b^2 - {b'}^2 = 0 & \ \, \vdash \, \   & a^2 + b^2 - {a'}^2 - {b'}^2 = 0, \\[3mm]
a-a' = 0 & \vdash & a^2 - {a'}^2 = 0, \\[3mm]
b-b' = 0 & \vdash & b^2 - {b'}^2 = 0,\\[3mm]
(a-a')^2+(b-b')^2 = 0 & \vdash & a-a' = 0, \ b-b' = 0. \\[3mm]
\end{array}
$$
By Lemmas \ref{lemma_sum_of_pos_and_zer_is_pos}
(item \ref{lemma_sum_of_pos_and_zer_is_pos:1}), 
\ref{lemma_basic_sign_rule_1} (item \ref{lemma_basic_sign_rule:4}) and
\ref{sos_non_pos_disjunct} we obtain an incompatibility
\begin{equation}\label{inc:aux_inc_lemma_eq_res_cuad}
\lda 
(a-a')^2+(b-b')^2 = 0, \ \cH
\rda
_{\K[v]}
\end{equation}
with monoid part $S^2$ and degree in $w$ bounded by 
$2(\delta_w + \max\{ \deg_w a-a', \deg_w b- b' \} - \min\{\deg_w a-a', \deg_w b- b'\})$.
On the other hand, in a similar way we obtain from (\ref{inc:init_inc_lemma_eq_res_cuad})
an incompatibility
\begin{equation}\label{inc:aux_inc_lemma_eq_res_cuad_2}
\lda 
(a-a')^2+(b+b')^2 = 0, \ \cH
\rda 
_{\K[v]}
\end{equation}
with the same monoid part and degree bound. 
Since
$$
{\rm R}(
z,z')= ((a-a')^2+(b-b')^2) \cdot ((a-a')^2+(b+b')^2),
$$  
the proof is finished by applying to (\ref{inc:aux_inc_lemma_eq_res_cuad}) and 
(\ref{inc:aux_inc_lemma_eq_res_cuad_2}) the weak inference
$$
{\rm R}(
z,z')=0
\ \ \
\vdash \ \ \
(a-a')^2+(b-b')^2 = 0  \ \, \vee \ \,
(a-a')^2+(b+b')^2 = 0. 
$$
By Lemma \ref{lem_prod_zero_at_least_one_fact_zero}, we obtain an incompatibility
$$ 
\lda 
{\rm R}(
z,z')=0, 
\
\cH \rda
_{\K[v]}
$$ with monoid part $S^4$ and degree in $w$ bounded by 
$$4\Big(\delta_w + \max\{ \deg_w a-a', \deg_w b- b' \} - \min\{\deg_w a-a',
\deg_w b- b'\}\Big),$$
which serves as the 
final incompatibility. 
\end{proof}

\subsection{Matrices}

We introduce the notation we use to  deal with matrix identities in the context of weak inference.

\begin{notation}[Identical Matrices]
Let ${\bm A} = (A_{ij})_{1 \le i,j \le p},$ $\bm B = (B_{ij})_{1 \le i,j \le p} 
\in \K[u]^{p \times p}$. The expression $\bm A \equiv \bm B$ is an abbreviation for  
$$
\bigwedge_{1 \le i \le p, \atop 1 \le j \le p}{A}_{ij} = {B}_{ij}.
$$
We denote by $\bm 0$ the matrix with all its entries equal to $0$. 
\end{notation}

We illustrate the use of this notation with two lemmas.

\begin{lemma} \label{lemma_transit_equiv_matric} 
Let ${\bm A}, \bm B \in \K[u]^{p \times p}$. Then 
$$
{\bm A} \equiv {\bm 0}, \  {\bm B} \equiv {\bm 0} 
\ \ \  \vdash \ \ \   \bm A + \bm B \equiv {\bm 0}.
$$

If we have an initial incompatibility in $\K[v]$ where $v \supset u$
with monoid part $S$ and  degree in 
$w\subset v$ bounded by $\delta_w$,
the final incompatibility has the same monoid part 
and degree in $w$ bounded by 
$$
\delta_w + 
\max\{\max\{\deg_w A_{ij}, \deg_w B_{ij}\} - \deg_w A_{ij} + B_{ij} \, | \, 1 \le i \le p, \, 1 \le j \le p \}.
$$
\end{lemma}
\begin{proof}{Proof.}
Follows from Lemma \ref{lem:comb_lin_zero_zero}. \end{proof}

\begin{lemma} \label{lemma_prod_equiv_matric}
Let $ {\bm A}, {\bm B}, {\bm C}  \in \K[u]^{p \times p}$.
Then
$$
{\bm A} \equiv {\bm 0} \ \ \ \vdash \ \ \   {\bm B} \cdot {\bm A} \cdot {\bm C} \equiv {\bm {0}}. 
$$

If we have an initial incompatibility in $\K[v]$ where $v \supset u$
with  monoid part $S$
and  degree in $w\subset v$ bounded by
$\delta_w$,
the final incompatibility has the same monoid part and degree in $w$ bounded by
$\delta_w + \deg_w {\bm B} + \deg_w{\bm A}  + \deg_w{\bm C}$.
\end{lemma}
\begin{proof}{Proof.} Follows from Lemma \ref{lem:comb_lin_zero_zero}.
\end{proof}


\section{Intermediate Value Theorem}\label{section_ivt}
\setcounter{equation}{0}

In this section we prove a weak existence version of the Intermediate Value Theorem for polynomials
(Theorem \ref{ivt})
and  
we apply it to prove the weak existence of a real root for a polynomial of odd degree (Theorem \ref{odd_degree_real_root_existence}).

The only result extracted from Section \ref{section_ivt} used in the rest of the paper  is 
the last result of the section, which is 
Theorem 
 \ref{odd_degree_real_root_existence} (Real Root of an Odd Degree Polynomial as a weak existence), and is used three times in Section \ref{section_fta}.

\subsection{Intermediate Value Theorem}\label{IVT_subsection}

We define the following auxiliary function, which plays a key role in the estimates
of the growth of degrees in the construction of incompatibilities related to the Intermediate Value Theorem. 

\begin{definition}
\label{defg1}
Let ${\rm g}_1: \N \times \N \to \N$, 
$${\rm g}_1\{k, p\} = 2^{3\cdot2^k}p^{k+1}.$$
We extend the definition of ${\rm g}_1$ with ${\rm g}_1\{-1,0\} = 2$.
\end{definition}

\begin{Tlemma}\label{lemma:aux_fun_g_1} For every $(k, p) \in \N \times \N$, 
$$4p{\rm g}_1\{k-1,k\}{\rm g}_1\{k,p\} \le {\rm g}_1\{k+1,p\}.$$
\end{Tlemma}
\begin{proof}{Proof.} Easy. \end{proof}

\begin{theorem}[Intermediate Value Theorem as a weak existence]\label{ivt}
Let $P
 = \sum_{0 \le h \le p}C_h \cdot y^h\in \K[u][ y]$. Then
$$
\exists (t_1,t_2)  \  [\; C_p \ne 0, \   P(t_1)\cdot P(t_2) \le 0 \;] 
\ \ \ \vdash  \ \ \
\exists t \ [\; P(t) = 0 \;].
$$

If we have
an initial incompatibility in 
$\K[v][t]$  where $v \supset u$ and  $t, t_1, t_2 \not \in v$, 
with 
monoid part $S$, 
degree in $w \subset v$ bounded by 
$\delta_w$ 
and 
degree in $t$ bounded by $\delta_t$,
the final incompatibility has
monoid part
$S^{e}\cdot C_p^{2f}$
with 
$e \le {\rm g}_1\{p-1,p\}$, $f \le {\rm g}_1\{p-1,p\} \delta_t$, 
degree in $w$
bounded by  
$
{\rm g}_1\{p-1,p\}(\delta_w + \delta_t \deg_w P)$
and the degree in $(t_1,t_2)$ bounded by 
${\rm g}_1\{p-1,p\} \delta_t$.

\end{theorem}

Note that the degree estimates obtained are doubly exponential in the degree of $P$ with respect to $y$.

The proof is based on an induction on the degree of $P$ with respect to $y$,
which is an adaptation of the proof by Artin \cite{Artin} that if a field is real
(i.e. -1 is not a sum of squares) its extension by an irreducible polynomial of odd degree
is also real.

\begin{proof}{Proof:}
Consider the initial incompatibility in $\K[v][t]$
\begin{equation}\label{orig_inc_ivt}
S + \sum_{i}  \omega_iV_i^2(t) \cdot N_i + \sum_{j}W_{j}(t)\cdot Z_j + Q(t)\cdot P(t) = 0
\end{equation}
with $S \in \scM( \cH_{\ne}^2)$, $\omega_i \in \K, \omega_i >0$, $V_i(t) \in \K[v][ t]$ and $N_i \in
\scM(\cH_{\ge})$ for every $i$, $W_j(t) \in \K[v][t]$ and $Z_j \in \cH_{=}$ for every $j$ and $Q(t) \in \K[v][t]$,
where $\cH = [\cH_{\ne}, \cH_{\ge}, \cH_{=}]$ is a system of sign conditions in $\K[v]$.

The proof proceeds by induction on $p$. 
For $p = 0$, 
$P(t) = C_0$ and $P(t_1)\cdot P(t_2) = C_0^2$. We evaluate $t = 0$ in (\ref{orig_inc_ivt}),
we pass the term $Q(0)\cdot C_0$ to the right hand side, we square both
sides and we pass $Q^2(0)\cdot C^2_0$ back to the left hand side.
We take the result as the final incompatibility.

Suppose now $p \ge 1$. 
If $Q(t)$ is the zero polynomial, we evaluate $t=0$ in (\ref{orig_inc_ivt}) and we take the result 
as the final incompatibility. From now, we suppose that $Q(t)$ is not the zero polynomial
and therefore, $\delta_t \ge p$. 
We denote by $\bar \delta_t$ the smallest even number
greater than or equal to $\delta_t$.   
For every $i$, let $\tilde V_i(t) \in \K[v][t]$ be 
the remainder of 
$C_p^{\frac 12\bar \delta_t}\cdot V_i(t)$ 
in the division by $P(t)$ 
considering $t$ as the main variable; then $\deg_w \tilde V_i(t) \le \deg_w V_i(t) + 
\frac12 \bar \delta_t \deg_w P$. 
Similarly, for every $j$, let $\tilde W_j(t) \in \K[v][t]$ be 
the remainder of 
$C_p^{\bar \delta_t}\cdot W_j(t)$ 
in the division by $P(t)$ 
considering $t$ as the main variable; then $\deg_w \tilde W_j(t) \le 
\deg_w W_j(t) + 
\bar \delta_t \deg_w P$.

We multiply (\ref{orig_inc_ivt}) by $C_p^{\bar \delta_t}$ and
we deduce that exists $Q'(t) \in \K[v][t]$  such that
\begin{equation}\label{main_inc_ivt}
S \cdot C_p^{\bar \delta_t} + \sum_i \omega_i \tilde V_i^2(t)\cdot N_i + \sum_j  \tilde W_j(t)\cdot Z_j+
Q'(t)\cdot P(t) = 0.
\end{equation}
Since the degree in $w$ 
of 
$S\cdot C_p^{\bar \delta_t}$, $\tilde V_i^2(t)\cdot N_i$ for every $i$ and
$\tilde W_j(t) \cdot Z_j$ for every $j$
is bounded by 
$\delta_w + \bar\delta_t \deg_w P$,  the degree in $w$ of 
$Q'(t)\cdot P(t)$ is also bounded by the same quantity.

If 
$Q'(t)$ is the zero polynomial, 
we evaluate $t = 0$ in (\ref{main_inc_ivt}) and 
take the result as the final incompatibility.
In particular, for $p=1$, $\deg_t \tilde V_i(t) = 0$ for every $i$ and 
$\deg_t \tilde W_i(t) = 0$ for every $j$; looking at the degree in $t$  in (\ref{main_inc_ivt}), we deduce that $Q'(t)$ is the zero polynomial
and we are done.

From now on, we 
suppose $p \ge 2$ and that $Q'(t)$ is not the zero polynomial. 
Let $q= \deg_t \tilde Q'(t)$; 
looking again at the degree in $t$  in (\ref{main_inc_ivt})
we have $q \le p-2$. Let
$Q'(t) = \sum_{0 \le \ell \le q}D_\ell \cdot t^\ell$ and, for $0 \le k \le q+1$, 
$Q'_{k-1}(t) = \sum_{0 \le \ell \le k-1}D_\ell \cdot t^\ell$.
We will prove, by a new induction on $k$, that for $0 \le k \le q+1$, 
we have 
$$ \lda  C_p \ne 0, \
Q'_{k-1}(t_1)\cdot Q'_{k-1}(t_2)\le 0,   \ 
\bigwedge_{k \le \ell \le q } D_\ell = 0 , \
 \cH
\rda
_{\K[v][ t_1, t_2]}$$
of type
\begin{equation} \label{eq_induc_ivt}
S^{e_k}\cdot C_p^{2f_k}
+
 N_{k,1}(t_1, t_2) - N_{k,2}(t_1, t_2)\cdot Q'_{k-1}(t_1)\cdot Q'_{k-1}(t_2) + Z_k(t_1, t_2)
+  \sum_{k \le \ell \le q}D_\ell \cdot R_{k,\ell}(t_1, t_2)  = 0
\end{equation}
with 
$N_{k,1}(t_1, t_2), N_{k,2}(t_1, t_2) \in \scN(\cH_{\ge})_{\K[v][ t_1, t_2]}$, 
$Z_k(t_1, t_2) \in \scZ(\cH_{=})_{\K[v][ t_1, t_2]}$, $R_{k, \ell}(t_1, t_2)$ $\in$ $\K[v][ t_1, t_2]$ for every $\ell$,  
$e_k \le {\rm g}_1\{k,p\}-2$,  $f_k \le ({\rm g}_1\{k,p\}-2)\delta_t$, 
degree in $w$ 
bounded by 
$({\rm g}_1\{k,p\}-4)(\delta_w + \delta_t \deg_w P)$ and
degree in $(t_1,t_2)$ bounded by $({\rm g}_1\{k,p\}-4)\delta_t$.

For $k = 0$, we simply evaluate $t = 0$ in (\ref{main_inc_ivt}).
Suppose now that we have an equation like (\ref{eq_induc_ivt}) 
for some $0 \le k \le q$. 
We will obtain an equation like (\ref{eq_induc_ivt}) for $k+1$.

\begin{itemize}

\item We rewrite (\ref{main_inc_ivt}) in this way:
$$
S\cdot C_p^{\bar \delta_t} + \sum_i \omega_i \tilde V_i(t)^2 \cdot N_i   + \sum_j \tilde W_j(t) \cdot Z_j 
+  P(t)\cdot \sum_{k+1 \le \ell \le q} D_\ell \cdot  t^\ell +
 P(t)\cdot Q'_k(t)
 = 0
$$
to obtain
\begin{equation}\label{inc:aux_th_IVT_00} \lda   C_p \ne 0, \
\bigwedge_{k+1 \le \ell \le q}D_\ell = 0, \ 
 Q'_{k}(t) = 0, \
\cH
\rda_{\K[v][ t]}\end{equation}
with degree in $w$ bounded by $\delta_w + \bar \delta_t \deg_w P$ and degree in $t$ bounded by $2(p-1)$. 
Since $k < p$, by the inductive hypothesis on $p$, we have a procedure 
to obtain from (\ref{inc:aux_th_IVT_00}) an incompatibility 
\begin{equation}\label{inc:aux_th_IVT_1}
\lda   C_p \ne 0, \
D_k \ne 0,  \  
Q'_{k}(t_1)\cdot Q'_{k}(t_2)\le 0, \  
\bigwedge_{k+1 \le \ell \le q}D_\ell = 0, \ 
\cH
\rda
_{\K[v][ t_1, t_2]}
\end{equation}
with monoid part $S^{e'}\cdot C_p^{\bar \delta_t e'}\cdot D_k^{2f'}$ 
with $e' \le {\rm g}_1\{k-1,k\}$,  $f' \le 2{\rm g}_1\{k-1,k\}(p-1)$, 
degree in $w$ bounded by 
${\rm g}_1\{k-1,k\}(\delta_w + \bar \delta_t \deg_w P + 
2(p-1)(\delta_w + \bar\delta_t \deg_w P)) = 
{\rm g}_1\{k-1,k\}(2p-1)(\delta_w + \bar\delta_t \deg_w P)$
and degree in $(t_1,t_2)$ bounded by $2{\rm g}_1\{k-1,k\}(p-1)$.

\item On the other hand, we substitute
$$ 
Q'_{k-1}(t_1)\cdot  Q'_{k-1}(t_2) = Q'_{k}(t_2)\cdot Q'_{k}(t_2) + D_k \cdot
\big( -t_1^k \cdot Q'_k(t_2) - t_2^k\cdot  Q'_k(t_1) + D_k \cdot t_1^k\cdot t_2^k    \big)
$$
in (\ref{eq_induc_ivt})
and we
obtain 
\begin{equation}\label{inc:aux_th_IVT_2}
\lda   C_p \ne 0, \
Q'_{k}(t_1)\cdot Q'_{k}(t_2)\le 0,   \
\bigwedge_{k \le \ell  \le q} D_\ell = 0, \
\cH
\rda 
_{\K[v][ t_1, t_2]}
\end{equation}
with monoid part $S^{e_k}\cdot C_p^{2f_k}$, degree in $w$ bounded by 
${\rm g}_1\{k,p\}(\delta_w + \delta_t \deg_w P)$ and 
degree in $(t_1,t_2)$ 
bounded by $({\rm g}_1\{k,p\}-4)\delta_t + 2k$.

\item Finally we apply 
to
(\ref{inc:aux_th_IVT_1}) and 
(\ref{inc:aux_th_IVT_2})
the weak inference 
$$
\vdash \ \ \ D_k \ne 0 \ \, \vee \ \, D_k = 0.
$$ 
By Lemma \ref{CasParCas_1}, we obtain
$$ \lda  C_p \ne 0, \
 Q'_{k}(t_1)\cdot Q'_{k}(t_2)\le 0, \ 
\bigwedge_{k+1 \le \ell \le q}  D_\ell = 0,  \ 
\cH
\rda
_{\K[v][ t_1, t_2]}$$
with monoid part 
$S^{e_{k+1}}\cdot C_p^{2f_{k+1}}$ with 
$e_{k+1} = e' + 2 e_kf'$
and $f_{k+1} = 
\frac{1}{2} \bar \delta_t e' + 2f_kf'$, 
degree in $w$ bounded by 
$
{\rm g}_1\{k-1,k\}(2p-1)(\delta_w + \bar\delta_t \deg_w P) +
2f'{\rm g}_1\{k,p\}(\delta_w + \delta_t \deg_w P)  
$ and degree in $(t_1,t_2)$ bounded by $2{\rm g}_1\{k-1,k\}(p-1) + 2f'(({\rm g}_1\{k,p\}-4)\delta_t + 2k)$.  
The bounds 
$e_{k+1} \le {\rm g}_1\{k+1,p\}- 2$ and 
$f_{k+1}  \le ({\rm g}_1\{k+1,p\} - 2)\delta_t$
follow using Lemma \ref{lemma:aux_fun_g_1}
since
$${\rm g}_1\{k-1,k\} + 4({\rm g}_1\{k,p\}-2){\rm g}_1\{k-1,k\} (p-1)   \le
4p{\rm g}_1\{k-1,k\}{\rm g}_1\{k,p\} - 2  \le
{\rm g}_1\{k+1,p\}- 2.$$
The degree bounds also follow using Lemma \ref{lemma:aux_fun_g_1} since
$$ 2{\rm g}_1\{k-1,k\}(2p-1) + 4{\rm g}_1\{k-1,k\}{\rm g}_1\{k,p\}(p-1)
\le 
4p{\rm g}_1\{k-1,k\}{\rm g}_1\{k,p\} - 4
\le
{\rm g}_1\{k+1,p\}- 4$$
and
\begin{eqnarray*}
&&2{\rm g}_1\{k-1,k\}(p-1) + 4{\rm g}_1\{k-1,k\}(({\rm g}_1\{k,p\}-4)\delta_t + 2k)(p-1)    
\le  \\
&\le& ( 4p{\rm g}_1\{k-1,k\}{\rm g}_1\{k,p\} - 4)\delta_t \le \\
&\le& ({\rm g}_1\{k+1,p\}-4)\delta_t.
\end{eqnarray*}
\end{itemize}

So, for $k = q+1$, we have 
\begin{equation}\label{final_eq_1_ivt}
S^{e_{q+1}}\cdot C_p^{2f_{q+1}} + N_{q+1,1}(t_1, t_2) + Z_{q+1}(t_1, t_2) = 
N_{q+1,2}(t_1, t_2) \cdot Q'(t_1)\cdot  Q'(t_2). 
\end{equation} 
On the other hand, substituting $t = t_1$ and $t = t_2$ in (\ref{main_inc_ivt}) 
we have 
\begin{equation}\label{final_eq_2_ivt}
S \cdot C_p^{\bar\delta_t}+ \sum_i \omega_i\tilde V_i(t_1)^2\cdot N_i + \sum_j \tilde W_j(t_1)\cdot Z_j = 
- Q'(t_1)\cdot P(t_1) 
\end{equation}
and
\begin{equation}\label{final_eq_3_ivt}
S\cdot C_p^{\bar\delta_t} + \sum_i \omega_i \tilde V_i(t_2)^2\cdot N_i + \sum_j  \tilde W_j(t_2)\cdot Z_j = 
- Q'(t_2)\cdot P(t_2). 
\end{equation}
Multiplying (\ref{final_eq_1_ivt}), (\ref{final_eq_2_ivt}) and (\ref{final_eq_3_ivt})  
and passing terms to the left hand side
we obtain
\begin{equation}\label{eq:final_4_ivt}
S^{e_{q+1} + 2}\cdot C_p^{2(f_{q+1} + \bar\delta_t)} + 
N(t_1, t_2) -  N_{q+1,2}(t_1, t_2)\cdot {Q'}^2(t_1)\cdot {Q'}^2(t_2)\cdot P(t_2)\cdot P(t_2) + Z(t_1, t_2) = 0
\end{equation} 
for some $N(t_1, t_2) \in \scN(\cH_{\ge})_{\K[v][t_1, t_2]}$
and $Z(t_1, t_2) \in \scZ(\cH_{=})_{\K[v][t_1, t_2]}$.
Equation (\ref{eq:final_4_ivt}) serves as the final incompatibility, taking into account that
$e_{q+1} + 2 \le {\rm g}_1\{q+1,p\}$,
$f_{q+1} + \bar\delta_t \le  {\rm g}_1\{q+1,p\}\delta_t$,
the degree in $w$ is bounded by 
$({\rm g}_1\{q+1,p\} - 4)(\delta_w + \delta_t \deg_w P)
+ 2(\delta_w + \bar \delta_t \deg_w P) \le
{\rm g}_1\{q+1,p\}(\delta_w + \delta_t \deg_w P)
$, 
the degree in $(t_1, t_2)$ is bounded by 
$({\rm g}_1\{q+1,p\}-4)\delta_t + 4p-4 \le {\rm g}_1\{q+1,p\}\delta_t$ and ${\rm g}_1\{q+1,p\}\le {\rm g}_1\{p-1,p\}$. 
\end{proof}

\subsection{Real root of a polynomial of odd degree}

Now we prove the weak existence of a real root for a 
monic polynomial of 
odd degree as a consequence of Theorem \ref{ivt} (Intermediate Value Theorem as a weak existence).

\begin{theorem}[Real Root of an Odd Degree Polynomial as a weak existence]\label{odd_degree_real_root_existence}
Let $p$ be an odd number and $P
 = y^p + \sum_{0 \le h \le p-1}C_h \cdot y^h \in \K[u][ y]$. Then
$$ \vdash  \ \ \ \exists t \  [\;  P(t) = 0\;].$$

If we have
an initial incompatibility in $\K[v][ t]$ where $v \supset u$ and $t \not \in v$, with 
monoid part $S$,  
degree in $w \subset v$ bounded by $\delta_w$ 
and 
degree in $t$ bounded by $\delta_t$,
the final incompatibility has
monoid part
$S^{e}$ 
with 
$e \le {\rm g}_1\{p-1,p\}$ and degree in $w$ bounded by 
$
3{\rm g}_1\{p-1,p\}(\delta_w + \delta_t \deg_w P)
$
(see Definition \ref{defg1}).
\end{theorem}

To prove Theorem \ref{odd_degree_real_root_existence} we first give in Lemma \ref{lemPMPP}, for a monic polynomial of odd degree,
 a real value where 
it is positive
and a real value where it is negative. Then, we apply the weak existence version 
of the Intermediate Value Theorem from Theorem \ref{ivt}.

\begin{lemma}\label{lemPMPP}
Let $p$ be an odd number, $P
 = y^p + \sum_{0 \le h \le p-1} C_h \cdot y^h\in \K[u][ y]$ and 
$E = p + \sum_{0 \le h \le p-1} C^2_h \in \K[u]$. Then both $P(E)$ and $-P(-E)$ 
are sums of squares in $\K[u]$ 
multiplied by elements in $\K_+$ 
plus 
an element in $\K_+$. 
\end{lemma}
\begin{proof} {Proof.}
We only prove the claim for $P(E)$  and the 
respective claim for $-P(-E)$ follows by considering the polynomial $-P(-y)
$.

We consider the Horner polynomials of $P$, ${\rm Hor}_0(P) = 1$, 
${\rm Hor}_{i}(P) = C_{p-i} + y\cdot {\rm Hor}_{i-1}(P)$ for $1 \le i \le p$. 
We will prove by induction on $i$ that for $1 \le i \le p$, 
\begin{equation}\label{eq:point_pol_pos}
{\rm Hor}_i(P)(E) = p - i + \sum_{0 \le h \le p-i-1}C_h^2 + N_i + \omega_i
\end{equation}
with $N_i \in \scN(\emptyset) $
 and $\omega_i$ in $\K_+$. 

For $i = 1$ we have
$$
{\rm Hor}_1(P)(E) = C_{p-1} + p + \sum_{0 \le h \le p-1}C_h^2 = 
p-1 + \sum_{0 \le h \le p-2}C_h^2 + \Big(C_{p-1} + \frac12\Big)^2 + \frac34.
$$

Suppose now that we have an equation like (\ref{eq:point_pol_pos}) for some $1 \le i-1 \le p-1$.
Then we have
$$
\begin{array}{rcl}
 {\rm Hor}_{i}(P)(E) &= &C_{p-i} + \Big(p +  \sum_{0 \le h \le p}C_h^2 \Big)\cdot \Big( 
p - i + 1 + \sum_{0 \le h \le p-i}C_h^2 + N_{i - 1} + \omega_{i-1}\Big) =\\[4mm]
&=& p - i + \sum_{0 \le h \le p-i-1}C_h^2 + N_i + \omega_i
\end{array}
$$
by taking
$$
N_i = \Big(p - 1 +  \sum_{0 \le h \le p}C_h^2 \Big)\cdot \Big( 
p - i + 1 + \sum_{0 \le h \le p-i}C_h^2 + N_{i - 1} + \omega_{i-1}
\Big) +  N_{i - 1}  + \Big(C_{p-i} + \frac 12\Big)^2   
$$
and $\omega_i =  \omega_{i-1} + \frac34$.

Finally, since ${\rm Hor}_p(P) = P$, the claim follows by considering equation (\ref{eq:point_pol_pos})
for $i = p$. 
\end{proof}

\begin{proof}{Proof of Theorem \ref{odd_degree_real_root_existence}:}
We apply to the initial incompatibility the weak inference
$$
\exists (t_1,t_2) \  [\;  P(t_1)\cdot P(t_2) \le 0 \;] 
\ \ \ \vdash  \ \ \
\exists t \ [\; P(t) = 0 \;].
$$
By Theorem \ref{ivt} (Intermediate Value Theorem as a weak existence),  we obtain an incompatibility 
with
monoid part
$S^{e}$
with 
$e \le {\rm g}_1\{p-1,p\}$, 
degree in $w$
bounded by  
$
{\rm g}_1\{p-1,p\}(\delta_w + \delta_t \deg_w P)
$ and degree in $(t_1,t_2)$ bounded by ${\rm g}_1\{p-1,p\} \delta_t$. 
Then we simply substitute $t_1 = E$ and $t_2 = -E$ where $E$ is defined as in Lemma
\ref{lemPMPP}. The degree bound follows easily. 
\end{proof}


\section{Fundamental Theorem of Algebra}\label{section_fta}
\setcounter{equation}{0}

In this section, we follow the approach of a famous algebraic proof of the Fundamental Theorem of Algebra due to Laplace
to give a weak existence form of this theorem
(Theorem  \ref{every_pol_has_a_compl_root}).
This approach is based on an induction
on the power of 2 appearing in the degree of the polynomial, 
the base case being the case of polynomials of odd degree.

We then apply Theorem  \ref{every_pol_has_a_compl_root} to obtain a weak disjunction of 
the possible decompositions of a polynomial into irreducible
real factors according to the number of real and complex roots (Theorem  \ref{thLaplace}).
Finally we obtain a weak disjunction of the possible decompositions of a polynomial into irreducible
real factors taking into account multiplicities
(Theorem \ref{thLaplacewithmult}). 

Apart from many results from Section \ref{Weak.inferences}, the only result from Section \ref{section_ivt}
used in this section is Theorem \ref{odd_degree_real_root_existence} 
(Real Root of an Odd Degree Polynomial as a weak existence), and it is used once in the base case of the induction 
in the proof of Theorem \ref{every_pol_has_a_compl_root} (Fundamental Theorem of Algebra as a weak existence), 
once in the proof of Lemma \ref{lem:factors_lin_or_quad} and once in the proof of 
Theorem \ref{thLaplace} (Real Irreducible Factors as a weak existence).

On the other hand, the only result extracted from Section \ref{section_fta} used in the rest of the paper  is  Theorem 
 \ref{thLaplacewithmult} (Real Irreducible Factors with Multiplicities as a weak existence), which is used only once in 
 Section \ref{sect_elim_of_one_var}.

\subsection{Fundamental Theorem of Algebra}

In order to prove
a weak existence version of the Fundamental Theorem of Algebra
in Theorem \ref{every_pol_has_a_compl_root}, 
we need some auxiliary notation, definitions and results.

\begin{notation}\label{notation:n_p}
For $p\in \N_*$, we denote by ${\rm r}\{p\}$ the biggest nonnegative integer $r$ such that 
$2^r$ divides $p$
and by ${\rm n}\{p\}$ the combinatorial number ${{p}\choose{2}}$.   
\end{notation}

Laplace's proof of the Fundamental Theorem of Algebra \cite{Lap} is very well known (see for example \cite{BPRbook_ed1}).
It is based on an 
inductive reasoning on ${\rm r}\{p\}$, where $p$ is the degree of 
the 
polynomial $P \in \R[y]$ for which the existence of a complex root is being proved. 
The result is true for a polynomials of odd degree for which ${\rm r}\{p\}=0$.
An auxiliary polynomial of degree  ${\rm n}\{p\}$ is constructed, and has a complex root by induction,
taking into account that ${\rm r}\{{\rm n}\{p\}\} = {\rm r}\{p\} - 1$.
A complex root of $P$ is then produced by solving a quadratic equation.

Following Laplace's approach, 
we define auxiliary polynomials. 

\begin{definition} Let $p \ge 1$, $c = (c_0, \dots, c_{p-1})$, $y' = (y'_0, \dots, y'_{{\rm n}\{p\}})$ and 
$y'' = (y''_{0,1}, \dots,$ 
$y''_{0,{\rm n}\{p\}},$ $y''_{1,2}, \dots,$ $y''_{1,{\rm n}\{p\}},$ $\dots,$ $y''_{{\rm n}\{p\}-1,{\rm n}\{p\}})$ be 
sets
of variables.  
We denote by $\overline{\K(c)}$ the algebraic closure of $\K(c).$ 
We consider 
\begin{itemize}
\item $P
 = y^p + \sum_{0 \le h \le p-1} c_h \cdot y^h \in \K[c][y]$,  

\item 
for $0 \le k \le {\rm n}\{p\}$, 
$$
Q_k
 = \prod_{1 \le i < j \le p}(y'_k - k(t_i+t_j) - t_it_j) 
\in \K[c][y'_k]
$$
where $t_1, \dots, t_p \in \overline{\K(c)}$ are the roots of $P
$ considering $y$ as the main variable,

\item for $0 \le k < \ell \le {\rm n}\{p\}$, 
$$
R_{k,\ell}
 = {y''_{k,\ell}}^2 - 
\frac{y'_\ell - y'_k}{\ell-k}y''_{k,\ell} + \frac{\ell y'_k - ky'_\ell}{\ell - k} \in \K[y'_k, y'_\ell, y''_{k,\ell}].
$$
\end{itemize}
\end{definition}

\begin{remark}\label{rem:degree_pol_aux_laplace}
For $0 \le k \le {\rm n}\{p\}$, 
$p-1$ of the factors in the definition of $Q_k$ have degree in $t_1$ equal to $1$ and 
the remaining factors have degree in $t_1$ equal to $0$. From this, it can be deduced that
$\deg_c Q_k \le p-1$ and also that $\deg_{(c, y'_k)} Q_k = {\rm n}\{p\}$ (see \cite[Section 2.1]{BPRbook}). 
\end{remark}

\begin{lemma}\label{Laplace} 
We denote by $\overline{\K}$ the algebraic closure of $\K$. 
For any $\gamma \in \overline{\K}^{p} ,\psi' \in \overline{\K}^{{\rm n}\{p\}+1}$ 
and $\psi''\in \overline{\K}^{{{{\rm n}\{p\}+1}\choose {2}}}$,
if $Q_k(\gamma, \psi'_k) = 0$ for $0 \le k \le {\rm n}\{p\}$ and 
$R_{k,\ell}(\psi'_k,\psi'_{\ell}, \psi''_{k,\ell} ) = 0$ for 
$0\le k< \ell \le {\rm n}\{p\}$, then 
$$\prod_{0 \le k < \ell \le {\rm n}\{p\}}P(\gamma, \psi''_{k,\ell})=0.$$
\end{lemma}
\begin{proof}{Proof.}
 For every $0 \le k \le {\rm n}\{p\}$, the condition $Q_k(\gamma, \psi'_k) = 0$ implies
that there exists a pair of  roots $\tau_k, \tau'_k \in \overline{\K}$ of $P(\gamma, y)$ such that
$\psi'_k = k(\tau_k + \tau'_k) + \tau_k \tau'_k$. 
Since there are at most ${\rm n}\{p\}$ different pairs of roots of $P(\gamma, y)$, 
by the pigeon hole principle there exist
indices $(k,\ell)$, $0 \le k < \ell \le {\rm n}\{p\}$ and roots 
$\tau, \tau' \in \overline{\K}$ of $P(\gamma, y)$ such that 
$\psi'_k=k(\tau+\tau')+\tau\tau'$ and $\psi'_\ell=\ell(\tau+\tau')+\tau\tau'$.
Then, we have
$$\tau+\tau'=\frac{ \psi'_\ell- \psi'_k}{\ell-k},\  
\tau\tau'=\frac{\ell \psi'_k - k \psi'_\ell}{\ell - k},
$$
so that the two roots of $R_{k,\ell}(\psi'_k,\psi'_\ell,y''_{k,\ell})$ are $\tau$ and $\tau'$
and therefore $\psi''_{k, \ell}$ is a root of $P(\gamma, y)$, what proves the claim.
\end{proof}

The preceding statement is transformed into an algebraic identity  using Effective
Nullstellensatz (\cite[Theorem 1.3]{Jel}).

\begin{lemma}\label{nullstellensatzLaplace} 
There is an identity in $\K[c][y', y'']$
$$
\begin{array}{rcl}
\displaystyle{\prod_{0 \le k < \ell \le {\rm n}\{p\}}P(c, y''_{k,\ell})^m} & =    &
\displaystyle{\sum_{0 \le k \le {\rm n}\{p\}}W_k(c, {y'},{y''})\cdot Q_k(c, y'_k)} \ + \\[6mm]
& +  &  
\displaystyle{\sum_{0 \le k < \ell \le {\rm n}\{p\}}
W_{k,\ell}(c, {y'},{y''})\cdot R_{k,\ell}(y'_k,y'_\ell,y''_{k,\ell})}
\end{array}
$$ 
such that all the terms have degree in $(c, y', y'')$ bounded by 
${\rm n}\{p\}^{{\rm n}\{p\}+1}2^{{{{\rm n}\{p\}+1}\choose 2}}(1 + {{{\rm n}\{p\}+1}\choose{2}}p)$.
\end{lemma}

\begin{proof}{Proof.}
Consider an auxiliary variable $\bar y$ 
and the polynomials $P^{[h]}(c, y, \bar y)$, $Q_k^{[h]}(c, y'_k, \bar y)$ 
and  $R_{k,\ell}^{[h]}(y'_k,  y'_\ell, y''_{k, \ell}, \bar y)$
obtained respectively from 
$P(c, y)$, $Q_k(c, y'_k)$ 
and  $R_{k,\ell}(y'_k,  y'_\ell,  y''_{k, \ell})$ by homogeneization. 

It is clear from Lemma \ref{Laplace} that 
for any $\gamma \in \overline{\K}^{p}$, 
$\psi' \in \overline{\K}^{{\rm n}\{p\}+1}$,  
$\psi'' \in \overline{\K}^{{{{\rm n}\{p\}+1}\choose {2}}}$ and $\bar \psi \in \overline{\K}$,
if $Q_k^{[h]}(\gamma,  \psi'_k, \bar \psi) = 0$ for $0 \le k \le {\rm n}\{p\}$ and 
$R_{k,\ell}^{[h]}(\psi'_k,  \psi'_{\ell}, \psi''_{k,\ell}, 
\bar \psi) = 0$ for 
$0\le k< \ell \le {\rm n}\{p\}$, then 
$$
\bar \psi
\prod_{0 \le k < \ell \le {\rm n}\{p\}}P^{[h]}(\gamma,   
\psi''_{k,\ell}, \bar \psi)
=0.
$$

Following \cite[Theorem 1.3]{Jel}, we have an identity 
\begin{equation}\label{ident:aux_nullst}
\begin{array}{rcl}
\displaystyle{\bar y^m \cdot
\prod_{0 \le k < \ell \le {\rm n}\{p\}}P^{[h]}(c, 
y''_{k,\ell}, \bar y)^m }
&= &
\displaystyle{\sum_{0 \le k \le {\rm n}\{p\}}
W^{[h]}_k(c,  {y'}, {y''}, \bar y) \cdot
Q_k^{[h]}(c, y'_k, \bar y) }
+  \\[6mm]
& + &   \displaystyle{\sum_{0 \le k < \ell \le {\rm n}\{p\}}
W^{[h]}_{k,\ell}(c, {y'},  {y''}, \bar y)\cdot
R_{k,\ell}^{[h]}(y'_k,  y'_\ell, y''_{k, \ell}, \bar y)}
\end{array}
\end{equation}
with $m = {\rm n}\{p\}^{{\rm n}\{p\}+1}2^{{{{\rm n}\{p\}+1}\choose 2}}$ 
and 
$W^{[h]}_k$ and $W^{[h]}_{k, \ell}$ homogeneous polynomials 
such that all the terms in (\ref{ident:aux_nullst})  have degree in $(c, y', y'', \bar y)$
equal to $m(1 + {{{\rm n}\{p\}+1}\choose{2}}p)$.
The lemma follows by 
evaluating $\bar y = 1$
in (\ref{ident:aux_nullst}). 
\end{proof}

The following  function
plays a key role in the estimates
of the 
degrees in the 
weak inference 
version of the Fundamental Theorem of Algebra.

\begin{definition} \label{defg2}
Using Notation \ref{notation:n_p}, let ${\rm g}_2:\N_* \to {\mathbb R}$, ${\rm g}_2\{p\} = 2^{2^{3(\frac{p}2)^{2^{{\rm r}\{p\}}}}}$. 
\end{definition}

\begin{Tlemma}\label{aux_ineq_tfa}  Let $p \in \N_*$.
\begin{enumerate}
\item \label{aux_ineq_tfa:0} If $p \ge 3$ is an odd number, then $3{\rm g}_1\{p-1,p\} \le {\rm g}_2\{p\}$. 
\item \label{aux_ineq_tfa:1}
 If $p \ge 4$ is an even number, then 
$
\frac3{16}p^{9}m
8^{{{\rm n}\{p\}+1\choose 2}}{\rm g}_2^{{\rm n}\{p\}+1}\{{\rm n}\{p\}\}
\le {\rm g}_2\{p\}
$, where
$m = {\rm n}\{p\}^{{\rm n}\{p\}+1}2^{{\rm n}\{p\}+1\choose 2}$.
\end{enumerate}
\end{Tlemma}

\begin{proof}{Proof.} See Section \ref{section_annex}. 
\end{proof}

\begin{theorem}[Fundamental Theorem of Algebra as a weak existence]\label{every_pol_has_a_compl_root} 
Let $p \ge 1$ and $P
 = y^p +  \sum_{0 \le h \le p-1}C_h \cdot y^h\in \K[u][ y]$. Then
$$
\vdash  \ \ \
\exists z \ [\; P(z) = 0 \;],
$$
where $z=a+ib$ is a complex variable
(see Notation \ref{real-imaginary}).

If we have
an initial incompatibility in  $\K[v][a,b]$ where $v \supset u$ and $a,b \not \in v$, 
with 
monoid part $S$, 
degree in $w \subset v$ bounded by 
$\delta_w$ and
degree in $(a,b)$ bounded by $\delta_z$,
the final incompatibility has
monoid part
$S^{e}$
with 
$e \le {\rm g}_2\{p\}$,
and degree in $w$
bounded by  
$
{\rm g}_2\{p\}(\delta_w + \delta_z \deg_w P)$. 
\end{theorem}

\begin{proof}{Proof.} 
Consider the initial incompatibility in $\K[v][ a, b]$
\begin{equation} \label{inc_dep_fta}
S + N(a,b)  + Z(a,b)  + W_{1}(a,b) \cdot P\re (a,b)  +  
W_{2}(a,b) \cdot P\im (a,b)  = 0
\end{equation}
with
$S\in {\scM}({\cal H}_{\neq}^2)$, $N(a,b) \in {\scN}({\cal H}_\ge )_{\K[v][a,b]}$, 
$Z(a,b) \in {\scZ}({\cal H}_= )_{\K[v][a,b]}$ and $W_{1}(a,b),$ $W_{2}(a,b)$ $\in$ $\K[v][ a, b]$, 
where $\cH$ is a system of sign conditions in $\K[v]$.

The proof proceeds by induction on ${\rm r}\{p\}$. For ${\rm r}\{p\} = 0$, i.e. $p$ is odd,
we evaluate $b = 0$ in 
(\ref{inc_dep_fta}) and, since $P\im (a,b)$ is a multiple of $b$ and $P\re (a,0) = P(a)$, 
we obtain an incompatibility of type
\begin{equation}\label{inc:aux_th_fta_1}
S + N'(a)  + Z'(a)  + W'(a) \cdot P(a) = 0 
\end{equation}
with
$N'(a) \in {\scN}({\cal H}_\ge )_{\K[v][a]}$, 
$Z'(a) \in {\scZ}({\cal H}_= )_{\K[v][a]}$ and $W'(a) \in$ $\K[v][ a]$.  
For $p = 1$, we substitute $a = -C_0$ and we take the result as the final incompatibility. 
For odd $p \ge 3$, we apply to (\ref{inc:aux_th_fta_1}) the weak inference 
$$ \vdash  \ \ \ \exists a \  [\;  P(a) = 0\;].$$
By Theorem \ref{odd_degree_real_root_existence} 
(Real Root of an Odd Degree Polynomial as a weak existence)
we obtain an incompatibility with monoid part
$S^{e}$ with $e \le {\rm g}_1\{p-1,p\}$ and degree in $w$ bounded by 
$3{\rm g}_1\{p-1,p\}(\delta_w + \delta_z\deg_w P)$,
which serves as the final incompatibility taking into account
Lemma  \ref{aux_ineq_tfa} (item \ref{aux_ineq_tfa:0}).

Suppose now ${\rm r}\{p\} \ge 1$, then $p$ is even. If $W_1(a,b)$ and $W_2(a,b)$ in (\ref{inc_dep_fta}) are both the zero polynomial, 
we evaluate $a = 0$ and $b = 0$ in (\ref{inc_dep_fta}) and we take the result
as the final incompatibility. From now, we suppose that
$W_1(a,b)$ and $W_2(a,b)$ are not both the zero polynomial and therefore,  
$\delta_z \ge p$. 

For $p=2$, 
the result follows from Lemma \ref{lemma_complex_root_quadratic}.

So we suppose $p \ge 4$ and , from now on, we denote $n={\rm n}\{p\}$, and  $m=n^{n+1}2^{{{n+1}\choose{2}}}$.

For $0 \le k < \ell \le n$, we substitute $a = a''_{k,\ell}, b = b''_{k,\ell}$ in 
(\ref{inc_dep_fta}) and we apply the 
weak inference
$$
P\re^2(a''_{k,\ell}, b''_{k,\ell}) + P\im^2(a''_{k,\ell}, b''_{k,\ell}) = 0 \ \ \ 
\vdash \ \ \  P(z''_{k,\ell}) = 0. 
$$
By Lemma \ref{sos_non_pos_disjunct}, we obtain  
\begin{equation}\label{inc:lemma_prod_comp_zero_pol_aux_2}
\lda
P\re^2(a''_{k,\ell}, b''_{k,\ell}) + P\im^2(a''_{k,\ell}, b''_{k,\ell}) = 0, \ \cH
\rda
_{\K[v][ a''_{k,\ell}, b''_{k,\ell}]}
\end{equation}
with monoid part $S^2$, degree in $w$ bounded by 
$2(\delta_{w} + \deg_w C_0)$ and degree in $(a''_{k,\ell}, b''_{k,\ell})$ bounded by $2\delta_z$.

Then we apply to the incompatibilities (\ref{inc:lemma_prod_comp_zero_pol_aux_2}) for
$0 \le k < \ell \le n$, each one repeated
$m$
times, the weak inference
$$
\prod_{0 \le k < \ell \le n} \big(  P\re^2(a''_{k,\ell}, b''_{k,\ell}) 
+ P\im^2(a''_{k,\ell}, b''_{k,\ell})  \big)^m = 0 \ \ \
\vdash
\ \ \ \bigvee_{0 \le k < \ell \le n, \atop 1 \le j \le m} 
P\re^2(a''_{k,\ell}, b''_{k,\ell}) + P\im^2(a''_{k,\ell}, b''_{k,\ell}) = 0.
$$
By Lemma \ref{lem_prod_zero_at_least_one_fact_zero},
we obtain 
\begin{equation}\label{inc:lemma_prod_comp_zero_pol_aux_3}
\lda
\prod_{0 \le k < \ell \le n} \big(  P\re^2(a''_{k,\ell}, b''_{k,\ell}) + 
P\im^2(a''_{k,\ell}, b''_{k,\ell})  \big)^m = 0, \ 
\cH
\rda
_{\K[v][ a'', b'']}
\end{equation}
with
monoid part $S^{2m {{n+1}\choose{2}}}$,  
degree in $w$ bounded by 
$2m{n+1\choose 2} (\delta_{w} + \deg_w C_0 )$ and 
degree in $(a''_{k,\ell}, b''_{k,\ell})$ bounded by $2m\delta_z$ for $0 \le k < \ell \le n$.

By Lemma \ref{nullstellensatzLaplace}, we have an identity
$$
\begin{array}{cl}
& \displaystyle{ \prod_{0 \le k < \ell \le n} \big(  P\re^2(a''_{k,\ell}, b''_{k,\ell}) 
+ P\im^2(a''_{k,\ell}, b''_{k,\ell})  \big)^m} = 
\\[4mm]
= &
\displaystyle{ \Big(
\sum_{0 \le k \le n}  
(W_k)\re
\cdot (Q_k)\re - (W_k)\im \cdot (Q_k)\im
+ 
\sum_{0 \le k < \ell \le n} 
(W_{k, \ell})\re\cdot (R_{k, \ell})\re- (W_{k, \ell})\im \cdot (R_{k, \ell})\im
\Big)^2 }
+\\[4mm]
+ & 
\displaystyle{ \Big(
\sum_{0 \le k  \le n}
(W_k)\re \cdot (Q_k)\im + (W_k)\im \cdot (Q_k)\re
+ 
\sum_{0 \le k < \ell \le n}
(W_{k, \ell})\re \cdot (R_{k, \ell})\im + (W_{k, \ell})\im \cdot(R_{k, \ell})\re
\Big)^2 }
\end{array}
$$
and then we apply to (\ref{inc:lemma_prod_comp_zero_pol_aux_3}) the weak inference 
$$
\bigwedge_{0\leq k\leq n} Q_k(C, z'_k) =0, 
\
\bigwedge_{0\le k < \ell \le n}
R_{k,\ell}(z'_k,z'_{\ell}, z''_{k,\ell}) = 0 
\ \ \ \vdash 
$$
$$ \vdash \ \ \
\prod_{0 \le k < \ell \le n} \big(  P\re^2(a''_{k,\ell}, b''_{k,\ell}) 
+ P\re^2(a''_{k,\ell}, b''_{k,\ell})  \big)^m = 0.
$$
By Lemma \ref{lem:comb_lin_zero_zero}, we obtain 
\begin{equation}\label{inc:thm_fta_aux_1}
\lda
\bigwedge_{0\leq k\leq n} Q_k(C, z'_k) =0, 
\
\bigwedge_{0\le k < \ell \le n}
R_{k,\ell}(z'_k,z'_{\ell}, z''_{k,\ell}) = 0 , \
\cH
\rda
_{\K[v][ a',b',a'', b'']}
\end{equation}
with the same 
monoid part, degree in $w$ bounded by 
$$
2m \Big({n+1\choose 2}\delta_{w} + \Big(1 +  {n+1\choose 2}p  \Big)\deg_w P\Big) \le 
 m 
\Big( \frac14p^4\delta_{w} 
+  \frac14p^5 \deg_w P
\Big),
$$ 
degree in 
$(a'_k, b'_k)$ 
bounded by $2m(1 + {n+1\choose 2}p) \le \frac14 mp^5$ for $0 \le k \le n$ and
degree in $(a''_{k,\ell}, b''_{k,\ell})$ bounded by 
$ 2m(1 -p + {n+1\choose 2}p + \delta_z) \le m( \frac14 p^5 + 2\delta_z) $ 
for $0 \le k < \ell \le n$.

Then we fix an arbitrary order $(k_1, \ell_1), \dots, (k_{n+1\choose2}, \ell_{n+1\choose2})$
of all the pairs $(k, \ell)$ with $0 \le k < \ell \le n$ and we 
we successively apply to (\ref{inc:thm_fta_aux_1}) for 
$1 \le h \le {n+1\choose2}$
the weak inference
$$  \vdash \ \ \ 
\exists z''_{k_h,\ell_h}  \ [\; 
R_{k_h,\ell_h}(z'_{k_h}, z'_{\ell_h}, z''_{k_h,\ell_h} ) = 0
\;].$$
By Lemma \ref{lemma_complex_root_quadratic}, we obtain 
\begin{equation}\label{inc:thm_fta_aux_2}
\lda
\bigwedge_{0\leq k\leq n} Q_k(C, z'_k) =0, 
\
\cH
\rda
_{\K[v][ a', b']}
\end{equation}
with monoid part 
$
S^{2m {{n+1}\choose {2}}8^{{{n+1}\choose 2}}}
$ and degree in $w$ bounded by 
$$   
\delta'_w := 
m 
\Big( \frac14p^4\delta_{w} 
+  \frac14p^5 \deg_w P
\Big)8^{{n+1\choose 2}} 
.
$$
In order to obtain a bound for 
the degree in $(a'_k, b'_k)$ of (\ref{inc:thm_fta_aux_2}) for $0 \le k \le n$, 
we do the following analysis. 
Consider a fixed $0 \le k_0 \le n$. 
For $1 \le h \le {n+1\choose2}$,  
$\deg_{(a'_{k_0}, b'_{k_0})} R_{k_h,\ell_h} = 0$ if 
$k_0 \ne k_h, \ell_h$ and $\deg_{(a'_{k_0}, b'_{k_0})}R_{k_h,\ell_h} = 1$ otherwise. 
Again by Lemma \ref{lemma_complex_root_quadratic}, there will be ${{n}\choose {2}}$ values of $h$ for 
which the bound 
for the degree in 
$(a'_{k_0}, b'_{k_0})$ is multiplied by $8$ and $n$ values of $h$ for which the bound 
for the degree in 
$(a'_{k_0}, b'_{k_0})$
is multiplied by $8$ and then increased by 
$40 + m(6 p^5 +48\delta_z)8^{h-1}$. 
It is easy to see that the worst case for the degree bound in 
$(a'_{k_0}, b'_{k_0})$
is when these $n$ values of $h$ are 
$1, \dots, n$, and that, in this case, after the application of the first $h \le n$ weak inferences, 
the degree in $(a'_{k_0}, b'_{k_0})$ of the incompatibility we obtain is bounded by
$$
\frac14mp^58^{h} + 40\Big(\sum_{0 \le j \le h-1}8^j\Big) + 
m(6p^5 +48\delta_z)h8^{h-1}.
$$
From this, we conclude that  the degree in
$(a'_{k}, b'_{k})$ of (\ref{inc:thm_fta_aux_2}) is bounded by 
$$
\frac14mp^58^{n+1\choose2} + 40\Big(\sum_{0 \le j \le n-1}8^j\Big)8^{n\choose2} + 
m(6p^5 +48\delta_z)n8^{{{n+1\choose2}}-1} 
\le m\Big( \frac3{8}p^7 +3p^2\delta_z\Big)8^{n+1\choose2} =: \delta'_{z'}$$
for $0 \le k \le n$.

Finally we successively apply to (\ref{inc:thm_fta_aux_2}) for every $0 \le k \le n$  the weak inference
$$  \vdash \ \ \
\exists z'_{k}  \   [\; Q_{k}(C, z'_k)= 0 \;].$$
Since ${\rm r}\{{\rm n}\{p\}\} = {\rm r}\{p\}-1$, by the inductive hyphotesis, we obtain 
\begin{equation}\label{inc:thm_fta_aux_3}
\lda \cH \rda 
_{\K[v]}
\end{equation}
with monoid part 
$
S^{2m{{n+1}\choose {2}}8^{{{n+1}\choose 2}}{e'}^{n+1}}
$
with 
$e' \le {\rm g}_2\{n\}$.
Also, when applying the weak inference corresponding to the index $k$, 
the bound for the degree in $w$ 
is increased by ${\rm g}_2^{k}\{n\} \delta'_{z'} (p-1) \deg_w P$ 
and then 
multiplied by ${\rm g}_2\{n\}$ (see Remark \ref{rem:degree_pol_aux_laplace}). 
It is easy to see that, after the application of this weak inference, the degree in $w$ of the
incompatibility we obtain is bounded by 
$$
{\rm g}_2^{k+1}\{n\}
(    
\delta'_{w} + 
(k+1) \delta'_{z'} (p-1)\deg_w P
).
$$
Therefore, the degree in $w$ of (\ref{inc:thm_fta_aux_3}) is bounded by 
$$
{\rm g}_2^{n+1}\{n\}
(    
\delta'_{w} + 
(n+1) \delta'_{z'} (p-1)\deg_w P
) \le 
{\rm g}_2^{n+1}\{n\}m
\Big(    
\frac14p^4\delta_{w} +
\frac3{16}p^9\delta_z\deg_w P
\Big)8^{{n+1\choose 2}}.
$$
The incompatibility (\ref{inc:thm_fta_aux_3}) serves as the final incompatibility since
$$
2m{{n+1}\choose {2}}8^{{{n+1}\choose 2}}{\rm g}_2^{n+1}\{n\}
\le \frac14p^4m8^{{{n+1}\choose 2}}{\rm g}_2^{n+1}\{n\} \le
\frac3{16}p^9m8^{{{n+1}\choose 2}}{\rm g}_2^{n+1}\{n\}
\le {\rm g}_2\{p\}
$$
and
$$
{\rm g}_2^{n+1}\{n\}m
\Big(    
\frac14p^4\delta_{w} +
\frac3{16}p^9\delta_z\deg_w P
\Big)8^{{n+1\choose 2}}  
\le
$$
$$
\le
{\rm g}_2^{n+1}\{n\}\frac3{16}p^9m
\Big(    
\delta_{w} 
+ \delta_z \deg_w P \Big)8^{{n+1\choose 2}} \le
{\rm g}_2\{p\}(\delta_w + \delta_z \deg_w P)
$$
using Lemma \ref{aux_ineq_tfa} (item \ref{aux_ineq_tfa:1}).
\end{proof}

\subsection{Decomposition of a polynomial
 into irreducible real factors}

We obtain now a weak disjunction on the possible decompositions of a polynomial into irreducible
real factors. 

We prove first an auxiliary lemma.

\begin{lemma}\label{lem:factors_lin_or_quad}
Let $p \ge 2$ be an even number and  
$P
 = y^p + \sum_{0 \le h \le p-1}C_h \cdot y^h  \in \K[u][ y]$. Then 
$$
  \vdash \ \ \  \exists  (t_1, t_2) \ [\; P
   \equiv (y-t_1) \cdot (y-t_2) \cdot {\rm Quot}(P,(y-t_1)(y-t_2)) \;] 
 \ \, \vee
$$
$$
\vee \ \, 
\exists  z \
[\; P
 \equiv ((y-a)^2  + b^2)\cdot{\rm Quot}(P, (y-a)^2  + b^2),  \ b\ne 0\;],
$$
where $z=a+ib$ is a complex variable
(see Notation \ref{real-imaginary}).

Suppose we have initial incompatibilities in 
$\K[v][ t_1, t_2]$
and $\K[v][ a, b]$
where $v \supset u$ and  $t_1, t_2, a, b \not \in v$, with 
monoid part 
$S_{1}$ and $S_{2}\cdot b^{2e}$ and degree in $w \subset v$ 
bounded by $\delta_{w}$.
Suppose also that the first initial incompatibility has 
degree in $t_1$ and degree in $t_2$ bounded by $\delta_{t}$
and the second initial incompatibility has degree in $(a, b)$ bounded by $\delta_z$.
Then, the final incompatibility has monoid part 
$S_1^{2(e+1)f} \cdot S_2^{f'}$ with 
$f \le {\rm g}_1\{p-2, p-1\}{\rm g}_2\{p\}$
and 
$f' \le {\rm g}_2\{p\}$ and degree in $w$ bounded by 
$$
{\rm g}_2\{p\}\Big(
(1 +  6{\rm g}_1\{p-2,p-1\}(e+1))\delta_{w} +
\Big(3 +  \delta_{z}  + 6{\rm g}_1\{p-2, p-1\}(e+1)(2 +
(p+1)\delta_{t}) \Big) \deg_w P
\Big),
$$
\end{lemma}

\begin{proof}{Proof.}
Consider the initial incompatibilities, 
\begin{equation}\label{inc:lemma_2_real_or_comp_init_1}
\lda
P
 \equiv (y-t_1) \cdot (y-t_2) \cdot {\rm Quot}(P,(y-t_1)(y-t_2)), \ \cH
\rda
_{\K[v][ t_1,t_2]}
\end{equation}
and 
\begin{equation}\label{inc:lemma_2_real_or_comp_init_2}
\lda
P
 \equiv ((y-a)^2  + b^2) \cdot {\rm Quot}(P, (y-a)^2  + b^2),  \ b\ne 0, \ \cH
\rda
_{\K[v][a,b]},
\end{equation}
where $\cH$ is a system of sign conditions in $\K[v]$.

We successively apply to (\ref{inc:lemma_2_real_or_comp_init_1}) the weak inferences
$$
\begin{array}{rcl}
P(t_1) = 0, \ {\rm Quot}(P, y-t_1)(t_2) = 0 & \ \, \vdash \, \ & 
P
\equiv (y-t_1) \cdot (y-t_2) \cdot {\rm Quot}(P, (y-t_1)(y-t_2)), \\[3mm]
& \vdash & \exists t_2 \ [\; {\rm Quot}(P, y-t_1)(t_2) = 0 \;]. \\[3mm]
\end{array}
$$
By Lemma \ref{lem:two_linear_factor} and  
Theorem \ref{odd_degree_real_root_existence} (Real Root of an Odd Degree Polynomial as a weak existence), we obtain 
\begin{equation}\label{inc:lemma_2_real_or_comp_aux_1}
\lda
P(t_1) = 0, \ \cH
\rda
_{\K[v][t_1]}
\end{equation}
with monoid part $S_1^{e'}$ with $e' \le {\rm g}_1\{p-2, p-1\}$ and, after some analysis,  degree in $w$ bounded by 
$3{\rm g}_1\{p-2, p-1\}(\delta_{w}  + (1 + \delta_{t})\deg_w P)$ and degree in $t_1$ bounded by 
$3{\rm g}_1\{p-2, p-1\}( 1 + p\delta_{t})$. 

Then we substitute $t_1 = a$ in (\ref{inc:lemma_2_real_or_comp_aux_1}) and, 
taking into account that 
$P\re(a,b) - P(a)$ is a multiple of $b$,  we apply the weak inference 
$$
P(z) = 0,  \ b = 0 \ \ \   \vdash  \ \ \ P(a) = 0. 
$$
By Lemma \ref{lem:comb_lin_zero_zero}, we obtain 
\begin{equation}\label{inc:lemma_2_real_or_comp_aux_2}
\lda
P(z) = 0, \ b=0,  \ \cH
\rda
_{\K[v][ a, b]}
\end{equation}
with the same monoid part and bound for the degree in $w$ and degree in $(a,b)$ bounded by 
$3{\rm g}_1\{p-2, p-1\}( 1  + p\delta_{t})$.

On the other hand, we apply to (\ref{inc:lemma_2_real_or_comp_init_2}) the weak inference 
$$
P(z) = 0,  \ b \ne 0 \ \ \ \vdash  \ \ \ P 
\equiv ((y-a)^2  + b^2) \cdot {\rm Quot}(P, (y-a)^2  + b^2).
$$
By Lemma \ref{quadratic_factor}, we obtain 
\begin{equation}\label{inc:lemma_2_real_or_comp_aux_3}
\lda
P(z) = 0, \ b\ne 0,  \ \cH
\rda
_{\K[v][ a, b]}
\end{equation}
with monoid part $S_2 \cdot b^{2(e+1)}$,  degree in $w$ bounded by $\delta_{w} + \deg_w P$ and 
degree in $(a, b)$ bounded by $\delta_{z} + 2$.

Then we apply to (\ref{inc:lemma_2_real_or_comp_aux_3}) and (\ref{inc:lemma_2_real_or_comp_aux_2}) 
the weak inference
$$
\vdash \ \ \ b \ne 0 \ \, \vee \, \ b = 0.
$$ 
By Lemma \ref{CasParCas_1}, we obtain 
\begin{equation}\label{inc:lemma_2_real_or_comp_aux_4}
\lda
P(z) = 0,   \ \cH
\rda
_{\K[v][ a, b]}
\end{equation}
with monoid part $S_1^{2(e+1)e'}\cdot S_2$, degree in $w$ bounded by 
$$\delta_{w} + \deg_w P 
+ 6{\rm g}_1\{p-2, p-1\}(e+1)(\delta_{w}  + (1 + \delta_{t})\deg_w P)$$ 
and degree in $(a, b)$ bounded by
$$\delta_{z} + 2 
+ 6{\rm g}_1\{p-2, p-1\}(e+1)(1 + p\delta_{t}).$$

Finally we apply to (\ref{inc:lemma_2_real_or_comp_aux_4}) the weak inference 
$$
\vdash \ \ \ \exists z \ [\;  P(z) = 0 \;].
$$
By Theorem \ref{every_pol_has_a_compl_root} (Fundamental Theorem of Algebra as a weak existence), we obtain 
$$ \lda \cH \rda_{\K[v]}$$
with monoid part 
$S_1^{2(e+1)e'f'}\cdot S_2^{f'}$ with $f' \le {\rm g}_2\{p\}$ and degree in $w$ bounded by 
$$
{\rm g}_2\{p\}\Big(
(1 + 6{\rm g}_1\{p-2, p-1\}(e+1))\delta_{w} +
\Big( 3 + \delta_{z} + 6{\rm g}_1\{p-2, p-1\}(e+1)(2 +
(p+1)\delta_{t}) \Big) \deg_w P
\Big),
$$
which serves as the final incompatibility. 
\end{proof}

We define a new auxiliary function.

\begin{definition}
Let ${\rm g}_3: \N \to {\mathbb R}$, ${\rm g}_3\{p\} = 2^{ 2^{3(\frac{p}2)^p + 1}}$. 
\end{definition}

\begin{Tlemma}\label{tlem:aux_comp_fact} Let $p \in \N_*$. 
\begin{enumerate}
\item \label{it:tlem_aux_comp_fact_1} If $p \ge 3$ is an odd number, then 
$3(2p+1){\rm g}_1\{p-1, p\}{\rm g}_3\{p-1\} \le {\rm g}_3\{p\}$. 
\item \label{it:tlem_aux_comp_fact_2}  If $p \ge 4$ is an even number, 
then $6p^3{\rm g}_1\{p-2, p-1\}{\rm g}_2\{p\}{\rm g}_3^2\{p-2\} \le {\rm g}_3\{p\}$. 
\end{enumerate}
\end{Tlemma} 

\begin{proof}{Proof.} See Section \ref{section_annex}.
\end{proof}

We now prove the weak disjunction on the possible decompositions taking into account only the number of 
real and complex roots.

\begin{theorem}[Real Irreducible Factors as a weak existence] \label{thLaplace} 
Let $p \ge 1$ and $P
= y^p + \sum_{0 \le h \le p-1} C_h \cdot  y^h  \in \K[u][ y]$. 
Then  
$$
\vdash \ \ \  
\bigvee_{ m+2n = p} 
\exists (t_{m},z_n) \ 
\Big[\;
P
\equiv  
\prod_{1 \le j \le m}(y-t_{m,j})
\cdot
\prod_{1 \le k \le n}((y - a_{n,k})^2 + b_{n,k}^2), 
\ 
\bigwedge_{1\le k\le n} b_{n,k} \ne 0 \;\Big], 
$$
where $t_m=( t_{m,1},\dots,t_{m,m})$ 
is a set of variables and $z_n = (z_{n,1},\dots,z_{n,n})$
is set of complex variables with $z_{n,k}=a_{n,k}+ib_{n,k}$
(see Notation \ref{real-imaginary}).

Suppose we have initial incompatibilities in 
$\K[v][ t_m, a_n, b_n]$ 
where
$v \supset u$ and $t_m,  a_n, b_n$ are disjoint from $v$, with 
monoid part 
$S_{m,n}\cdot \prod_{1 \le k \le n}b_{n,k}^
{2e_{n,k}}$
with $e_{n,k} \le e$,  
degree in $w \subset v$ bounded by $\delta_{w}$, 
degree in $t_{m,j}$ bounded by $\delta_{t}$ for $1 \le j \le m$
and
degree in $(a_{n,k}, b_{n,k})$ bounded by $\delta_{z}$ for $1 \le k \le n$.
Then, the final incompatibility has monoid part 
$\prod_{m + 2n = p}S_{m, n}^
{f_{m,n}}$
with $f_{m,n}
\le (e+1)^{2^{\lfloor \frac{p}2 \rfloor}-1}{\rm g}_3\{p\}$ 
and degree in $w$ bounded by 
$(e+1)^{2^{\lfloor \frac{p}2 \rfloor}-1}{\rm g}_3\{p\}(\delta_w + \max\{\delta_t, \delta_z\}\deg_w P)$.
\end{theorem}

\begin{proof}{Proof.}
Consider for $m, n \in \N$ such that $m + 2n = p$ the initial incompatibility
\begin{equation}\label{inc:dep_th_fact_just_num_r_c}
\lda
P
 \equiv  
\prod_{1 \le j \le m}(y-t_{m,j})
\cdot
\prod_{1 \le k \le n}((y - a_{n,k})^2 + b_{n,k}^2)
, \ 
\bigwedge_{1\le k\le n} b_{n,k} \ne 0, \
\cH
\rda
_{\K[v][ t_m, a_n, b_n]}
\end{equation}
where $\cH$ is a system of sign conditions in $\K[v]$. 
If $\max\{\delta_t, \delta_z\}=0$, the result follows by simply taking any 
of the initial incompatibilities 
as the final incompatibility. So from now we suppose
$\max\{\delta_t, \delta_z\}\ge 1$.

We first prove the result for even $p$ by induction. 
For $p = 2$ the result follows from Lemma \ref{lem:factors_lin_or_quad}.
Suppose now $p \ge 4$.

For $m, n \in \N$ such that $m + 2n = p$ with $m \ge 2$, we apply to 
(\ref{inc:dep_th_fact_just_num_r_c}) the weak inference
$$
P
 \equiv (y-t_{m,1}) \cdot (y-t_{m,2}) \cdot \Quot(P, (y-t_{m,1})(y-t_{m,2})),
$$
$$ 
\Quot(P, (y-t_{m,1}) \cdot (y-t_{m,2})) \equiv 
\prod_{3 \le j \le m}(y-t_{m,j})
\cdot
\prod_{1 \le k \le n}((y - a_{n,k})^2 + b_{n,k}^2)
\ \ \ \vdash
$$
$$
\vdash \ \ \
P
 \equiv
\prod_{1 \le j \le m}(y-t_{m,j})
\cdot
\prod_{1 \le k \le n}((y - a_{n,k})^2 + b_{n,k}^2).
$$ 
which is a particular case of the weak inference in Lemma \ref{lem:comb_lin_zero_zero}. 
After a careful  analysis,  we obtain 
\begin{equation}\label{inc:aux_th_fact_just_num_r_c_1}
\begin{array}{c}
\Big \downarrow \  P
 \equiv (y-t_{m,1}) \cdot (y-t_{m,2}) \cdot \Quot(P, (y-t_{m,1}) \cdot (y-t_{m,2})), \\[4mm]
\Quot(P, (y-t_{m,1}) \cdot (y-t_{m,2})) \equiv 
\displaystyle{
\prod_{3 \le j \le m}(y-t_{m,j})
\cdot
\prod_{1 \le k \le n}((y - a_{n,k})^2 + b_{n,k}^2)}, \\[4mm]
\displaystyle{ \bigwedge_{1\le k\le n} b_{n,k} \ne 0}, \
\cH  \ \Big \downarrow
_{\K[v][ t_m, a_n, b_n]}
\end{array}
\end{equation}
with the same monoid part,  
degree in $w$ bounded by $\delta_{w} + \deg_w P$,
degree in $t_{m,1}$ and in $t_{m,2}$ bounded by $\delta_{t} + p-2$, 
degree in $t_{m,j}$ bounded by $\delta_{t}$ for $3 \le j \le m$ and 
degree in $(a_{n,k}, b_{n,k})$ bounded by $\delta_{z}$ for $1 \le k \le n$.

Then we substitute $t_{m,1} = t_1$ and $t_{m, 2} = t_2$ 
in the incompatibilities (\ref {inc:aux_th_fact_just_num_r_c_1})
and we apply to these incompatibilities the weak inference 
$$
\vdash \ \ \  
\bigvee_{ m+2n = p, \atop m \ge 2} 
\exists (t'_{m},z_n) \ 
\Big[\;
\Quot(P, (y-t_{1}) \cdot (y-t_{2})) \equiv  
\prod_{3 \le j \le m}(y-t_{m,j})
\cdot
\prod_{1 \le k \le n}((y - a_{n,k})^2 + b_{n,k}^2)
,$$
$$  
\bigwedge_{1\le k\le n} b_{n,k} \ne 0 \;\Big] 
$$
where $t'_m = (t_{m,3}, \dots, t_{m, m})$.  Since $\deg_y \Quot(P, (y-t_{1}) \cdot (y-t_{2})) 
= p-2$, by the inductive hyphotesis we obtain  
\begin{equation}\label{inc:aux_th_fact_just_num_r_c_2}
\lda
P
 \equiv (y-t_1) \cdot (y-t_2) \cdot \Quot(P, (y-t_1)(y-t_2)), \ 
\cH
\rda
_{\K[v][ t_1, t_2]}
\end{equation}
with monoid part
$$\prod_{m + 2n = p, \atop m \ge 2} S_{m, n}^{f_{m-2, n}
}$$ 
with $f_{m-2, n} 
\le (e+1)^{2^{ \frac{p-2}2 }-1}{\rm g}_3\{p-2\}$, 
degree in $w$ bounded by 
$$
\delta'_w := (e+1)^{2^{ \frac{p-2}2 }-1}{\rm g}_3\{p-2\}
(\delta_w + (1 +  \max\{\delta_t, \delta_z\})\deg_w P),
$$
and degree in $t_1$ and degree in $t_2$ bounded by 
$$
\delta'_t := (e+1)^{2^{ \frac{p-2}2 }-1}{\rm g}_3\{p-2\}
(
\delta_t + (1 + \max\{\delta_t, \delta_z\})(p-2)
).
$$

On the other hand, we obtain in a similar way, from the initial incompatibilities
(\ref{inc:dep_th_fact_just_num_r_c}) for
$m, n \in \N$ such that $m + 2n = p$ with $n \ge 1$, 
\begin{equation}\label{inc:aux_th_fact_just_num_r_c_3}
\lda
P
 \equiv ((y-a)^2 + b^2) \cdot \Quot(P, (y-a)^2 + b^2), \ b\ne 0, \ 
\cH
\rda
_{\K[v][ a,b]}
\end{equation}
with, defining
$E=\sum_{m + 2n = p, \, n \ge 1}e_{n,1}
f_{m, n-1}$,
monoid part
$$
\prod_{m + 2n = p, \atop n \ge 1} S_{m, n}^{f_{m, n-1}}
\cdot
b^{2 E
}
$$ 
with $f_{m, n-1} 
\le (e+1)^{2^{ \frac{p-2}2 }-1}{\rm g}_3\{p-2\}$, 
degree in $w$ bounded by 
$\delta'_w$
and degree in $(a,b)$ bounded by 
$$
\delta'_z := 
(e+1)^{2^{ \frac{p-2}2 }-1}{\rm g}_3\{p-2\}
(
\delta_z + (1+ \max\{\delta_t, \delta_z\})(p-2)
).
$$

Finally, we apply to (\ref{inc:aux_th_fact_just_num_r_c_2}) and
(\ref{inc:aux_th_fact_just_num_r_c_3}) the weak inference 
$$
  \vdash \ \ \  \exists  (t_1, t_2) \ [\; P
   \equiv (y-t_1) \cdot (y-t_2) \cdot {\rm Quot}(P,(y-t_1) \cdot (y-t_2)) \;] 
 \ \, \vee
$$
$$
\vee \ \, 
\exists  z \
[\; P
 \equiv ((y-a)^2  + b^2) \cdot {\rm Quot}(P, (y-a)^2  + b^2),  \ b\ne 0\;] .
$$
By Lemma \ref{lem:factors_lin_or_quad}, we obtain 
\begin{equation}\label{inc:fin_theorem_comp_fact}
\lda \cH
\rda 
_{\K[v]}
\end{equation}
with
monoid part 
$$
\Big(
\prod_{m + 2n = p, \atop m \ge 2} S_{m, n}^{f_{m-2, n}
}
\Big)^
{2(E+1)f}
\cdot
\Big(\prod_{m + 2n = p, \atop n \ge 1} S_{m, n}^{f_{m, n-1}
}\Big)^{f'},
$$ 
with $f \le {\rm g}_1\{p-2, p-1\}{\rm g}_2\{p\}$ and $f' \le {\rm g}_2\{p\}$.
Therefore, for $m \ge 2$ and $n \ge 1$, we take 
$$
f_{m,n}
= 2f_{m-2, n}
(E+1)
f
+ f_{m, n-1}
f' \le 
(e+1)^{2^{ \frac{p}2 }-1}p{\rm g}_1\{p-2, p-1\}{\rm g}_2\{p\}{\rm g}^2_3\{p-2\} \le 
(e+1)^{2^{ \frac{p}2 }-1}{\rm g}_3\{p\} 
$$
using Lemma \ref{tlem:aux_comp_fact} (item \ref{it:tlem_aux_comp_fact_2}). 
Also we take $f_{0, \frac{p}2} = f_{0, \frac{p}2-1}f' \le 
(e+1)^{2^{ \frac{p}2 }-1}{\rm g}_3\{p\}$ and 
$f_{p, 0} = 2f_{p-2, 0}
(E+1)f \le 
(e+1)^{2^{ \frac{p}2 }-1}{\rm g}_3\{p\}$ in a similar way. 
Again by Lemma \ref{lem:factors_lin_or_quad},
the degree in $w$ of (\ref{inc:fin_theorem_comp_fact}) is bounded by  
\begin{eqnarray*}
&&{\rm g}_2\{p\}\Big(
(1 +  6{\rm g}_1\{p-2, p-1\}(E+1))\delta'_{w} + \\
& & +
\Big(3 +  \delta'_{z}  + 6{\rm g}_1\{p-2, p-1\}(E+1)(2 +
(p+1)\delta'_{t}) \Big) \deg_w P
\Big) \le 
\\
&\le&
(e+1)^{2^{\frac{p}2}-1}6p^3{\rm g}_1\{p-2, p-1\}{\rm g}_2\{p\}{\rm g}_3^2\{p-2\}(\delta_w + \max\{\delta_t, 
\delta_z\}\deg_w P) \le \\
&\le&
(e+1)^{2^{\frac{p}2}-1}{\rm g}_3\{p\}(\delta_w + \max\{\delta_t, 
\delta_z\}\deg_w P)
\end{eqnarray*}
using Lemma \ref{tlem:aux_comp_fact} (item \ref{it:tlem_aux_comp_fact_2}).
Therefore (\ref{inc:fin_theorem_comp_fact}) serves as the final incompatibility.

Now we prove the result for odd  $p$. For $p =1$
note that we only have to consider $m = 1$ and $n = 0$; therefore we can take $e = 0$. 
We simply substitute $t_{1,1} = -C_1$ in 
(\ref{inc:dep_th_fact_just_num_r_c}) and take the result as the final incompatibility.
Suppose now $p \ge 3$. 

For $m, n \in \N$ such that $m + 2n = p$  we apply to 
(\ref{inc:dep_th_fact_just_num_r_c}) the weak inference
$$
P
 \equiv (y-t_{m,1})\cdot \Quot(P, y-t_{m,1}),
$$
$$ 
\Quot(P, y-t_{m,1}) \equiv 
\prod_{2 \le j \le m}(y-t_{m,j})
\cdot
\prod_{1 \le k \le n}((y - a_{n,k})^2 + b_{n,k}^2)
\ \ \ \vdash
$$
$$
\vdash \ \ \
P
 \equiv
\prod_{1 \le j \le m}(y-t_{m,j})
\cdot
\prod_{1 \le k \le n}((y - a_{n,k})^2 + b_{n,k}^2).
$$
which is a particular case of the weak inference in Lemma \ref{lem:comb_lin_zero_zero}.
After a careful analysis, we obtain 
\begin{equation}\label{inc:aux_th_fact_just_num_r_c_5}
\begin{array}{c}
\Big \downarrow \
P(t_{m,1}) = 0, \\[3mm]
\Quot(P, y-t_{m,1}) \equiv 
\prod_{2 \le j \le m}(y-t_{m,j})
\cdot
\prod_{1 \le k \le n}((y - a_{n,k})^2 + b_{n,k}^2), \\[3mm]
\bigwedge_{1\le k\le n} b_{n,k} \ne 0
, \
\cH  \ \Big \downarrow
_{\K[v][ t_m, a_n, b_n]}
\end{array}
\end{equation}
with the same  monoid part, 
degree in $w$ bounded by $\delta_{w} + \deg_w P$,
degree in $t_{m,1}$ bounded by $\delta_t + p-1$,
degree in $t_{m,j}$ bounded by $\delta_t$ for $2 \le j \le m$ and 
degree in $(a_{n,k}, b_{n,k})$ bounded by $\delta_z$ for $1 \le k \le n$.

Then we substitute $t_{m,1} = t$  
in the incompatibilities (\ref {inc:aux_th_fact_just_num_r_c_5})
and we apply to these incompatibilities the weak inference 
$$
\vdash \ \ \  
\bigvee_{ m+2n = p} 
\exists (t'_{m},z_n) \ 
\Big[\;
\Quot(P, y-t) \equiv  
\prod_{2 \le j \le m}(y-t_{m,j})
\cdot
\prod_{1 \le k \le n}((y - a_{n,k})^2 + b_{n,k}^2),
\ 
\bigwedge_{1\le k\le n} b_{n,k} \ne 0 \;\Big] 
$$
where $t'_m = (t_{m,2}, \dots, t_{m, m})$.  Since $\deg_y \Quot(P, y-t)$ is
an even number greater than or equal to $2$, we obtain 
\begin{equation}\label{inc:aux_th_fact_just_num_r_c_6}
\lda P(t) = 0, \ \cH \rda
_{\K[v][t]}
\end{equation}
with monoid part
$\prod_{m + 2n = p}S_{m,n}^{f_{m-1,n}}$ with $f_{m-1, n} \le (e+1)^{2^{\frac{p-1}2}-1}{\rm g}_3\{p-1\}$, 
degree in $w$ bounded by 
$$
\delta''_w := 
(e+1)^{2^{\frac{p-1}2}-1}{\rm g}_3\{p-1\}(\delta_w + (1 + \max\{\delta_t, \delta_z \}) \deg_w P)
$$ 
and
degree in $t$ bounded by 
$$
\delta''_t := 
(e+1)^{2^{\frac{p-1}2}-1}{\rm g}_3\{p-1\}(\delta_t + (1 + \max\{\delta_t, \delta_z \}) (p-1)). 
$$
  
Finally, since $p$ is odd, we apply to (\ref{inc:aux_th_fact_just_num_r_c_6}) 
the weak inference 
$$ \vdash \ \ \ \exists t \ [\; P(t) = 0\;].$$
By 
Theorem \ref{odd_degree_real_root_existence} (Real Root of an Odd Degree Polynomial as a weak existence) we obtain 
\begin{equation} \label{inc:fin_inc_the_fact_odd}
\lda \cH \rda
_{\K[v]}
\end{equation}
with monoid part
$
(\prod_{m + 2n = p}{S_{m,n}
^{f_{m-1,n}}})
^{e'}
$
with $e' \le {\rm g}_1\{p-1, p\}$. Therefore, for every $m$ and $n$, we take $
f_{m,n} = f_{m-1, n}e' \le
(e+1)^{2^{\frac{p-1}2}-1}{\rm g}_1\{p-1, p\}{\rm g}_3\{p-1\} \le (e+1)^{2^{\frac{p-1}2}-1}
{\rm g}_3\{p\}$ using Lemma \ref{tlem:aux_comp_fact} (item \ref{it:tlem_aux_comp_fact_1}). 
Again by Theorem \ref{odd_degree_real_root_existence}, the  
degree in $w$ of (\ref{inc:fin_inc_the_fact_odd}) is bounded by 
$$
\begin{array}{lll}
&3{\rm g}_1\{p-1, p\}(\delta''_w + \delta''_t \deg_w P ) \le \\
  \le &  (e+1)^{2^{\frac{p-1}2}-1}3(2p+1){\rm g}_1\{p-1, p\}{\rm g}_3\{p-1\}(\delta_w + 
\max\{\delta_t, \delta_z\}\deg_w P) \le 
\\
 \le&  (e+1)^{2^{\frac{p-1}2}-1}{\rm g}_3\{p\}(\delta_w + 
\max\{\delta_t, \delta_z\}\deg_w P).
\end{array}
$$
using Lemma \ref{tlem:aux_comp_fact} (item \ref{it:tlem_aux_comp_fact_1}). 
Therefore (\ref{inc:fin_inc_the_fact_odd}) serves as the final incompatibility.
\end{proof}

\subsection{Decomposition of a polynomial
 into irreducible real factors with multiplicities}
 
In order to prove the 
weak inference of the  decomposition into irreducible factors taking multiplicities into account, 
we introduce some notation and definitions. 

\begin{notation} \label{not:aux_fact} Let $m \in \N$. We introduce the following notation: 
For $m \in \N_*$, 
$$\Lambda_{m} = \{ {\bm \mu} = (\mu_1\ge \dots \ge \mu_{\# {\bm \mu}})  \ | \ \mu_i \in {\mathbb N}_* \hbox{ for } 1 \le i \le {\# {\bm \mu}}, 
\ |{\bm \mu}|=\sum_{1 \le i \le {\# {\bm \mu}}} \mu_i = {m}\};$$  $\Lambda_0$ is the set with a single element equal to 
an empty vector. 
\end{notation}

\begin{definition}
\label{defFact}
Let  $p \ge 1$,  $P
 = y^p + \sum_{0 \le h \le p-1} C_h \cdot y^h\in \K[u][y]$,  
$({\bm \mu}, {\bm \nu}) \in \Lambda_{m}\times \Lambda_{n}$ with
${m}+2{n}=p$,
$t = (t_1, \dots, t_{\# {\bm \mu}})$ 
and $z = (z_1, \dots, z_{\# {\bm \nu}})$
a set of complex variables with $z_k=a_k+ib_k$
(see Notation \ref{real-imaginary}).
We define
$${\rm F}^{{\bm \mu},{\bm \nu}}
 = y^p + \sum_{0 \le h \le p-1}{\rm F}_{h}^{{\bm \mu},{\bm \nu}}
 \cdot y^h =  
\prod_{1 \le j \le {\# {\bm \mu}}}(y-t_j)^{\mu_j} 
\cdot
\prod_{1 \le k \le {\# {\bm \nu}}}((y-a_k)^2 + b_k^2)^{\nu_k}
\in \mathbb{Z}[t, a,b][y]. 
$$
Using Notation \ref{ResR},
we  define the system of sign conditions 
$${\rm Fact}(P)^{{\bm \mu},{\bm \nu}}(t, z)$$ in $\K[u][t,a,b]$
describing the decomposition of $P$ into irreducible real factors:
$$ 
P
\equiv 
{\rm F}^{{\bm \mu},{\bm \nu}}
,  \ 
\bigwedge_{ 1 \le j < j'  \le \# {\bm \mu}} t_{j} \ne t_{j'}, \
\bigwedge_{1 \le k  \le \# {\bm \nu}}  b_k \ne 0, 
\bigwedge_{1 \le k < k' \le \# {\bm \nu}}
  {\rm R}(z_{k},z_{k'}) \ne 0.
$$ 
\end{definition}

Before proving the weak disjunction on the possible decompositions taking multiplicities
into account, we define a new auxiliary function. 

\begin{definition}
\label{defg4}
Let ${\rm g}_4: \N \to {\mathbb R}$, ${\rm g}_4\{p\} = 2^{ 2^{3(\frac{p}2)^p + 2} }$. 
\end{definition}

\begin{Tlemma}\label{tlem:aux_comp_fact_with_mult} For every $p \in \N_*$,
$2^{(p^2 - p + 2)2^{\frac12 p^2}}{\rm g}_3\{p\} \le {\rm g}_4\{p\}$. 
\end{Tlemma} 
\begin{proof}{Proof.}
Easy.
\end{proof}

\begin{theorem}[Real Irreducible Factors with Multiplicities as a weak existence] 
\label{thLaplacewithmult} 
Let $p \ge 1 $ and $P
 = y^p + \sum_{0 \le h \le p-1} C_h \cdot y^h \in \K[u][ y]$. Then 
$$
\vdash \ \ \ \bigvee_{ {m} + 2{n} = p  \atop ({\bm \mu},{\bm \nu}) \in 
\Lambda_{m} \times \Lambda_{n}}
\exists (t_{{\bm \mu}}, z_{{\bm \nu}}) \ [\; {\rm Fact}(P)^{{\bm \mu},{\bm \nu}}(t_{{\bm \mu}}, z_{{\bm \nu}}) \;],
$$
where $t_{{\bm \mu}} = (t_{{\bm \mu},1},\dots,t_{{\bm \mu},{\# {\bm \mu}}})$ is a set of variables and
$z_{{\bm \nu}} = (z_{{\bm \nu},1},\dots,z_{{\bm \nu},{\# {\bm \nu}}})$ is a set of complex variables with
$z_{{\bm \nu},k}=a_{{\bm \nu},k}+ib_{{\bm \nu},k}$
(see Notation \ref{real-imaginary}).

Suppose we have initial incompatibilities in 
$\K[v][ t_{{\bm \mu}}, a_{{\bm \nu}},
b_{{\bm \nu}}]$  where
$v \supset u$, and $t_{{\bm \mu}}, a_{{\bm \nu}},
b_{{\bm \nu}}$ are disjoint from $v$, with 
monoid part 
$$
S_{{\bm \mu},{\bm \nu}}
\cdot
\prod_{1 \le j < j' \le {\# {\bm \mu}}}(t_{{\bm \mu},j} - t_{{\bm \mu},j'})^{2e^{{\bm \mu}, {\bm \nu}}_{ {j}, j'}} 
\cdot
\prod_{1 \le k \le {\# {\bm \nu}}}b_{{\bm \nu},k}^{2f^{{\bm \mu}, {\bm \nu}}_{k}}
\cdot 
\prod_{1 \le k < k' \le {\# {\bm \nu}}}
 {\rm R}(
 z_{{\bm \nu}, k},z_{{\bm \nu}, k'} )
^{2g^{{\bm \mu}, {\bm \nu}}_{ {k}, k'}},
$$
with
$e^{{\bm \mu}, {\bm \nu}}_{j, j'} \le e \in \N_*$ for $1 \le j < j' \le {\# {\bm \mu}}$,
$f^{{\bm \mu}, {\bm \nu}}_{ k} \le f \in \N_*$ for $1 \le k \le \# {\bm \nu}$
and 
$g^{{\bm \mu}, {\bm \nu}}_{k, k'} \le g \in \N_*$ for $1 \le k < k' \le \# {\bm \nu}$,
degree in $w \subset v$ bounded by $\delta_{w}$,
degree in $t_{{\bm \mu},j}$ bounded by $\delta_{t} \ge p$ for $1 \le j \le {\# {\bm \mu}}$, 
and 
degree in $(a^{{\bm \nu}}_{k}, b^{{\bm \nu}}_{k})$ bounded by $\delta_{z} \ge p$ for 
$1 \le k \le \# {\bm \nu}$. 
Then, the final incompatibility has monoid part 
$$
\prod_{{m} + 2{n} = p \atop ({\bm \mu}, {\bm \nu}) \in \Lambda_{m} \times \Lambda_{n}} S_{{\bm \mu}, {\bm \nu}}^{h_{{\bm \mu}, {\bm \nu}}
}
$$
with $h_{{\bm \mu}, {\bm \nu}} \le \max\{e, g\}^{2^{\frac12 p^2}}f^{2^{\frac12 p}}{\rm g}_4\{p\}$ 
and degree in $w$ bounded by 
$\max\{e, g\}^{2^{\frac12 p^2}}f^{2^{\frac12 p}}{\rm g}_4\{p\}(\delta_w + \max\{\delta_t, \delta_z \}\deg_w P)$.
\end{theorem}

For the proof of Theorem \ref{thLaplacewithmult}, we need an auxiliary notation and lemma. 

\begin{notation} \label{not:aux_fact2} 
 To  $J \subset \{(j, j') \ | \ 1 \le j < j' \le {m}\}$, we associate the smallest equivalence relation 
$\simeq_J$ 
 on $\{1,\ldots,m\}$ such that
$(j,j')\in J$ implies $j
\simeq_J j'$.
We define ${\bm \mu}_J\in \Lambda_m$ as the non-increasing vector of cardinalities of the equivalence classes for $
\simeq_J$ and ${\cal C}_1, \dots, {\cal C}_{
{\# {\bm \mu}}_J}$  the 
equivalence classes defined by $\simeq
_J$.

Similarly, to $K  \subset \{(k, k') \ | \ 1 \le k < k' \le {n}\}$, 
 we associate the smallest equivalence relation 
$\simeq_K$ 
 on $\{1,\ldots,n\}$ such that
$(k,k')\in K$ implies $k
\simeq_K k'$. We define 
${\bm \nu}_K\in \Lambda_m$ as the non-increasing vector of cardinalities of the equivalence classes for $
\simeq_K$ and ${\cal C}'_1, \dots, {\cal C}'_{
{\# {\bm \nu}}_K}$  the 
equivalence classes defined by $\simeq
_K$.
\end{notation}

\begin{lemma}\label{lem:aux_complete_fact}
Let $p \ge 1$, $P
 = y^p + \sum_{0 \le h \le p-1} C_h \cdot y^h \in \K[u][ y]$, ${m}, {n} \in \N$
with ${m} + 2{n} = p$, $J \subset \{(j, j') \ | \ 1 \le j < j' \le {m}\}$ and 
$K  \subset \{(k, k') \ | \ 1 \le k < k' \le {n}\}$. 
Then 
$$
\exists (t', z') \ \Big[\; 
P 
 \equiv 
\prod_{1 \le j' \le {m} } (y-t'_{j'}) 
\cdot
\prod_{1 \le k' \le {n}} ((y - a'_{k'})^2 + {b_{k'}'}^2)   
,  
\bigwedge_{1 \le j'_1 < j'_2 \le {m}, \atop (j'_{1}, j'_{2}) \in J} 
t'_{j'_1} = {t'_{j'_2}}, \ 
\bigwedge_{1 \le j'_1 < j'_2 \le {m}, \atop (j'_{1}, j'_{2}) \not \in J} 
t'_{j'_1} \ne {t'_{j'_2}}, 
$$
$$
\bigwedge_{1 \le k' \le {n}} b'_{k'} \ne 0, \
 \bigwedge_{1 \le k'_1 < k'_2 \le {n}, \atop (k'_{1}, k'_{2}) \in K} 
 {\rm R}( z'_{k},z'_{k'}) = 0, 
\bigwedge_{1 \le k'_1 < k'_2 \le {n}, \atop (k'_{1}, k'_{2}) \not \in K} 
 {\rm R}(z'_{k},z'_{k'})\ne 0 \; \Big]
\ \ \ \
\vdash$$
$$\vdash  \ \ \ 
\exists (t, z) \ [\;  {\rm Fact}(P)^{{\bm \mu}_J,{\bm \nu}_K}(t, z) \; ], 
$$
where $t' = (t'_1, \dots, t'_{m})$, 
$z' = (z'_1, \dots, z'_{n})$
is a set of complex variables 
with $z'_{k'}=a'_{k'}+ib'_{k'}$
$t = (t_1, \dots, t_{{\# {\bm \mu}}_J})$
and $z = (z_1, \dots, z_{{\# {\bm \nu}}_K})$
is a set of complex variables 
with $z_k=a_k+ib_k$
(see Notation \ref{real-imaginary}).

Suppose we have an initial incompatibility in $\K[v][ t, a, b]$ where $v \supset u$ and
$t, a, b$ are disjoint from  $v$, 
with monoid part 
$$
S\cdot
\prod_{1 \le j < j' \le 
{\# {\bm \mu}}_J}(t_{j} - t_{j'})^{2e_{{j}, j'}} 
\cdot
\prod_{1 \le k \le {\# {\bm \nu}}_K}b_{k}^{2f_{k}}
\cdot
\prod_{1 \le k < k' \le {\# {\bm \nu}}_K}
 {\rm R}(z_{k},z_{k'}),
^{2g_{{k}, k'}},
$$
with
$e_{j, j'} \le e$ for $1 \le j < j' \le 
{\# {\bm \mu}}_J$,
$f_{k} \le f$ for $1 \le k \le
{\# {\bm \nu}}_K$ 
and 
$g_{k, k'} \le g$ for $1 \le k < k' \le 
{\# {\bm \nu}}_K$,
degree in $w \subset v$ bounded by $\delta_{w}$,
degree in $t_{j}$ bounded by $\delta_{t} \ge p$ for $1 \le j \le 
{\# {\bm \mu}}_J$, 
and
degree in $(a_{k}, b_{k})$ bounded by $\delta_{z} \ge p$ 
 for $1 \le k \le
{\# {\bm \nu}}_K$.
Then, the final incompatibility has monoid part 
$$
S^h
\cdot
\prod_{1 \le j'_1 < j'_2 \le {m}, \atop (j'_{1}, j'_{2}) \not \in J} 
(t'_{j'_1} - {t'_{j'_2}})^{2e'_{j'_1, j'_2}}
\cdot
\prod_{1 \le k' \le {n}}{b'_{k'}}^{2f'_{k'}}
\cdot
\prod_{1 \le k'_1 < k'_2 \le {n}, \atop (k'_{1}, k'_{2}) \not \in K} 
 {\rm R}(z'_{k},z'_{k'})
^{2g'_{k'_1, k'_2}}
$$
with $h \le 2^{{n}({n}-1)}$, 
$e'_{j'_1, j'_2} \le 2^{{n}({n}-1)}e$ 
for $1 \le j'_1 < j'_2 \le {m}, (j'_1, j'_2) \not \in J$,
$f'_{k'}  \le  2^{{n}({n}-1)}f$ for $1 \le k' \le {n}$
and 
$g'_{k'_1, k'_2}  \le  2^{{n}({n}-1)}g$ 
for $1 \le k'_1 < k'_2 \le {n}, (k'_1, k'_2) \not \in K$, 
degree in $w$ bounded by $2^{{n}({n}-1)}\delta_w$,
degree in $t'_{j'}$ bounded by $2^{{n}({n}-1)}\delta_t$ for $1 \le j' \le {m}$, 
and
degree in $(a'_{k'}, b'_{k'})$ bounded by $2^{{n}({n}-1)}\delta_z$ 
for $1 \le k' \le {n}$.
\end{lemma}

\begin{proof}{Proof.}
Consider the initial incompatibility
\begin{equation} \label{inc:init_lemma_equal_roots_leads_mult}
\lda 
{\rm Fact}(P)^{{\bm \mu}_J,{\bm \nu}_K}(t, z), \ 
\cH
\rda
_{\K[v][t, a,b]}
\end{equation}
where $\cH$ is a system of sign conditions in $\K[v]$.

First, for $1 \le j \le 
{\# {\bm \mu}}_J$
and $1 \le k \le 
{\# {\bm \nu}}_K
$, we choose $\alpha(j) \in 
{\cal C}_j$ and $\beta(k) \in {\cal C}'_k$
(using Notation \ref{not:aux_fact2})
and we substitute $t_j = t'_{\alpha(j)}$ 
and $(a_k, b_k) = (a'_{\beta(k)}, b'_{\beta(k)})$
in (\ref{inc:init_lemma_equal_roots_leads_mult}). 
Then
we  apply 
the weak inference
$$
P
 \equiv 
\prod_{1 \le j' \le {m} } (y-t'_{j'}) 
\cdot 
\prod_{1 \le k' \le {n}} ((y - a'_{k'})^2 + {b_{k'}'}^2) ,  
\bigwedge_{1 \le j'_1 < j'_2 \le {m}, \atop (j'_{1}, j'_{2}) \in J} 
t'_{j'_1} = {t'_{j'_2}},$$
$$   
\bigwedge_{1 \le k'_1 < k'_2 \le {n}, \atop (k'_{1}, k'_{2}) \in K} 
(y - {a'_{k'_1}})^2 + {b'_{k'_1}}^2 \equiv (y - {a'_{k'_2}})^2 + {b'_{k'_2}}^2
\ \ \ \vdash
$$
$$
\vdash \ \ \
P
 \equiv {\rm F}^{{\bm \mu}_J,{\bm \nu}_K}. 
$$
By Lemma \ref{lem:comb_lin_zero_zero}, we obtain
\begin{equation}\label{inc:aux_lemma_equal_roots_leads_mult}
\begin{array}{c}
\Big \downarrow \
\displaystyle{
P
 \equiv 
\prod_{1 \le j' \le {m} } (y-t'_{j'}) 
\cdot
\prod_{1 \le k' \le {n}} ((y - a'_{k'})^2 + {b_{k'}'}^2)  
\bigwedge_{1 \le j'_1 < j'_2 \le {m}, \atop (j'_{1}, j'_{2}) \in J} 
t'_{j'_1} = {t'_{j'_2}}}, \\[6mm]
\displaystyle{  
\bigwedge_{ 1 \le j <  j' \le 
{\# {\bm \mu}}_J} 
t'_{\alpha(j)} \ne t'_{\alpha(j')}, 
\
\bigwedge_{1 \le k \le
{\# {\bm \nu}}_K} b'_{\beta(k)} \ne 0, \
\bigwedge_{1 \le k'_1 < k'_2 \le {n}, \atop (k'_{1}, k'_{2})  \in K} 
(y - {a'_{k'_1}})^2 + {b'_{k'_1}}^2 \equiv (y - {a'_{k'_2}})^2 + {b'_{k'_2}}^2
,} \\[6mm]
\displaystyle{ 
\bigwedge_{1 \le k < k' \le 
{\# {\bm \nu}}_K} 
 {\rm R}(z'_{\beta(k)}, z'_{\beta(k)})
\ne 0, 
\ \cH
\ \Big \downarrow }
_{\K[v][ t', a', b']}
\end{array}
\end{equation}
with monoid part 
$$
S
\cdot
\prod_{1 \le j <  j' \le
{\# {\bm \mu}}_J} 
(t'_{\alpha(j)} - {t'_{\alpha(j')}})^{2e_{j, j'}}
\cdot
\prod_{1 \le k  \le 
{\# {\bm \nu}}_K}{b'_{\beta(k)}}^{2f_{k}}
\cdot
\prod_{1 \le k < k' \le 
{\# {\bm \nu}}_K} 
 {\rm R}( z'_{\beta(k)},z'_{\beta(k')})
^{2g_{k, k'}}
$$
and, after some analysis, degree in $w$ bounded by $\delta_w$,
degree in $t'_{j'}$ bounded by $\delta_t$ for $1 \le j' \le {m}$, 
and degree in $(a'_{k'}, b'_{k'})$ bounded by $\delta_z$ for $1 \le k' \le {n}$ (using $\delta_t \ge p$ and 
$\delta_z \ge p$). Note that
for $1 \le j < j' \le 
{\# {\bm \mu}}_J$, if $\alpha(j) < \alpha(j')$ then 
$(\alpha(j), \alpha(j')) \not \in J$ and if 
$\alpha(j') < \alpha(j)$ then 
$(\alpha(j'), \alpha(j)) \not \in J$, and a similar fact holds for 
$1 \le k < k' \le 
{\# {\bm \nu}}_K$.

Finally, we successively apply to (\ref{inc:aux_lemma_equal_roots_leads_mult}) 
for $(k'_1, k'_2) \in K$
the weak inference
$$ {\rm R}(z'_{k'_1},z'_{k'_1})
= 0 \ \ \ \vdash
\ \ \ (y - a'_{k'_1})^2 + {b'_{k'_1}}^2 \equiv (y - a'_{k'_2})^2 + {b'_{k'_2}}^2.
$$ 
The proof is easily finished using Lemma \ref{lem:resultant_zero_then_eq}.
\end{proof}

\begin{proof}{Proof of Theorem \ref{thLaplacewithmult}.} Consider for 
$({\bm \mu}, {\bm \nu}) \in \Lambda_{m} \times \Lambda_{n}$ the initial incompatibility
\begin{equation}\label{inc:init_theor_fact_}
\lda
{\rm Fact}(P)^{{\bm \mu},{\bm \nu}}(t_{\bm \mu}, z_{\bm \nu}), \ \cH
\rda 
_{\K[v][ t_{\bm \mu}, a_{\bm \nu}, b_{\bm \nu}]}
\end{equation}
where $\cH$ is a system of sign conditions in $\K[v]$. 

For each ${m}$ and ${n}$, for each $J \subset \{(j, j') \ | \ 1 \le j < j' \le {m}\}$ and 
$K \subset \{(k, k') \ | \ 1 \le k < k' \le {n}\}$, we apply to the incompatibility
(\ref{inc:init_theor_fact_}) corresponding to $({\bm \mu}_J, {\bm \nu}_K)$
(see Notation \ref{not:aux_fact2})
the weak inference 
$$
\exists (t_{m}, z_{n}) \ \Big[\; 
P  \equiv 
\prod_{1 \le j \le {m} } (y-t_{{m},j}) 
\cdot
\prod_{1 \le k \le {n}} ((y - a_{{n},k})^2 + b^2_{{n},k}),
\bigwedge_{1 \le j < j' \le {m}, \atop (j, j') \in J} 
t_{{m}, j} = {t_{{m},j'}}, 
$$ 
$$
\bigwedge_{1 \le j < j' \le {m}, \atop (j, j') \not \in J} 
t_{{m},j} \ne {t_{{m},j'}}, \ 
\bigwedge_{1 \le k \le {n}} b_{{n},k} \ne 0, \ 
 \bigwedge_{1 \le k < k' \le {n}, \atop (k, k') \in K} 
{\rm R}(z_{{n},k},z_{{n},k'})
= 0, 
\bigwedge_{1 \le k < k' \le {n}, \atop (k, k') \not \in K} 
{\rm R}(z_{{n},k},z_{{n},k'})
\ne 0 \; \Big] 
\ \ \ 
\vdash$$
$$\vdash \ \ \
\exists (t_{{\bm \mu}_J}, z_{{\bm \nu}_K}) \ [\;  {\rm Fact}(P)^{{\bm \mu}_J,{\bm \nu}_K}(t_{{\bm \mu}_J}, z_{{\bm \nu}_K}) \; ],
$$
where $t_{m} = (t_{{m},1}, \dots, t_{{m},{m}})$ and 
$z_n = (z_{{n},1}, \dots, z_{{n},{n}})$. By Lemma \ref{lem:aux_complete_fact}
we obtain
\begin{equation}
\label{inc:aux_complet_fact_2} 
\begin{array}{c}
\displaystyle{\Big \downarrow \ 
P  \equiv 
\prod_{1 \le j \le {m} } (y-t_{{m},j})
\cdot
\prod_{1 \le k \le {n}} ((y - a_{{n},k})^2 + b^2_{{n},k}),
\bigwedge_{1 \le j < j' \le {m}, \atop (j, j') \in J} 
t_{{m},j} = {t_{{m},j'}}, } \\[6mm]  
\displaystyle{
\bigwedge_{1 \le j < j' \le {m}, \atop (j, j') \not \in J} 
t_{{m},j} \ne {t_{{m},j'}}, \
\bigwedge_{1 \le k \le {n}} b_{{n},k} \ne 0, } \\[6mm] 
\displaystyle{ \bigwedge_{1 \le k < k' \le {n}, \atop (k, k') \in K} 
 {\rm R}( z_{{n},k},z_{{n},k'})
= 0,
\bigwedge_{1 \le k < k' \le {n}, \atop (k, k') \not \in K} 
{\rm R}(z_{{n},k},z_{{n},k'})
\ne 0, \ 
\cH \ \Big \downarrow}
_{\K[v][ t_{m}, a_{n}, b_{n}]}
\end{array}
\end{equation}
with monoid part 
$$
S_{{\bm \mu}_J, {\bm \nu}_K}^{h_{J, K}}
\cdot
\prod_{1 \le j < j' \le {m}, \atop (j, j') \not \in J} 
(t_{{m},j} - {t_{{m},j'}})^{2e_{J, K, j, j'}}
\cdot
\prod_{1 \le k \le {n}}b_{{n},k}^{2f_{J, K, k}}
\cdot
\prod_{1 \le k < k' \le {n}, \atop (k, k') \not \in K} 
{\rm R}(z_{{n},k},z_{{n},k'})
^{2g_{J, K, k, k'}} 
$$
with $h_{J,K}  \le 2^{{n}({n}-1)}$, 
$e_{J, K,  j, j'} \le 2^{{n}({n}-1)}e$ for $1 \le j < j' \le {m}, (j, j') \not \in J$,
$f_{J, K, k} \le 2^{{n}({n}-1)}f$ for $1 \le k \le n$
and
$g_{J, K, k, k'} \le 2^{{n}({n}-1)}g$ for $1 \le k < k' \le {n},
(k, k') \not \in K$,
degree in $w$ bounded by $2^{{n}({n}-1)}\delta_w$,
degree in $t_{{m},j}$ bounded by $2^{{n}({n}-1)}\delta_t$ for $1 \le j \le {m}$ 
and
degree in $(a_{{n},k}, b_{{n},k})$ bounded by $2^{{n}({n}-1)}\delta_z$ 
for $1 \le k \le n$.

Then, for each ${m}$ and ${n}$,  we apply to incompatibilities (\ref{inc:aux_complet_fact_2}) for every 
$J \subset \{(j, j') \ | \ 1 \le j < j' \le {m}\}$ and 
$K \subset \{(k, k') \ | \ 1 \le k < k' \le {n}\}$, the weak inference
$$
P  \equiv 
\prod_{1 \le j \le {m} } (y-t_{{m},j}) 
\cdot
\prod_{1 \le k \le {n}} ((y - a_{{n},k})^2 + b^2_{{n},k})  
, \ 
\bigwedge_{1 \le k \le {n}} b_{{n},k} \ne 0
\ \ \ \vdash
$$
$$
\vdash \ \ \
\bigvee_{J,\, K} 
\Big(
P  \equiv 
\prod_{1 \le j \le {m} } (y-t_{{m},j}) 
\cdot
\prod_{1 \le k \le {n}} ((y - a_{{n},k})^2 + b^2_{{n},k}),
\bigwedge_{1 \le j < j' \le {m}, \atop (j, j') \in J} 
t_{{m},j} = {t_{{m},j'}}, \ 
\bigwedge_{1 \le j < j' \le {m}, \atop (j, j') \not \in J} 
t_{{m},j} \ne {t_{{m},j'}}, \ 
$$
$$ 
\bigwedge_{1 \le k \le {n}} b_{{n},k} \ne 0, \
\bigwedge_{1 \le k < k' \le {n}, \atop (k, k') \in K} 
{\rm R}(z_{{n},k},z_{{n},k'})
= 0,
\bigwedge_{1 \le k < k' \le {n}, \atop (k, k') \not \in K} 
{\rm R}(z_{{n},k},z_{{n},k'})
\ne 0 
\Big).
$$
By Lemma \ref{lem:multiple_case_by_case} and taking into account 
that there are at most $2^{\frac 12 p(p-1)}$ pairs of subsets $(J,K)$ and 
many different pairs may lead to the same pair of vectors $({\bm \mu}_J, {\bm \nu}_K)$, 
we obtain 
\begin{equation}
\label{inc:aux_complet_fact_3} 
\lda P  \equiv 
\prod_{1 \le j \le {m} } (y-t_{{m},j}) 
\cdot
\prod_{1 \le k \le {n}} ((y - a_{{n},k})^2 + b^2_{{n},k})  
, \ 
\bigwedge_{1 \le k \le {n}} b_{{n},k} \ne 0, \ 
\cH \rda
_{\K[v][ t_{m}, a_{n}, b_{n}]}
\end{equation}
with monoid part 
$$
\prod_{({\bm \mu}, {\bm \nu}) \in \Lambda_{m} \times \Lambda_{n}}S_{{\bm \mu}, {\bm \nu}}^{h'_{{\bm \mu},{\bm \nu}}}
\cdot
\prod_{1 \le k \le {n}}b_{{n},k}^{2f'_{{n},k}}
$$ 
with $h'_{{\bm \mu},{\bm \nu}} \le \max\{e, g\}^{2^{\frac12 p(p-1)}-1}
2^{({n}^2 - {n} + 2)2^{\frac12 p(p-1)} -  2}$
and
$f'_{{n},k} \le \max\{e, g\}^{2^{\frac12 p(p-1)}-1}f2^{({n}^2 - {n} + 2)2^{\frac12 p(p-1)} - 2}$
for $1 \le k \le {n}$, 
degree in $w$ bounded by 
$\max\{e, g\}^{2^{\frac12 p(p-1)}-1}
2^{({n}^2 - {n} + 2)2^{\frac12 p(p-1)} -  2}
\delta_w$, 
degree in $t_{{m},j}$ bounded by 
$\max\{e, g\}^{2^{\frac12 p(p-1)}-1}2^{({n}^2 - {n} + 2)2^{\frac12 p(p-1)} -  2}\delta_t$
for $1 \le j \le {m}$ 
and degree in $(a_{n,k}, b_{n,k})$ bounded by 
$\max\{e, g\}^{2^{\frac12 p(p-1)}-1}2^{({n}^2 - {n} + 2)2^{\frac12 p(p-1)} -  2}\delta_z$. 

Finally, we apply to incompatibilities (\ref{inc:aux_complet_fact_3}) for every ${m}$ and ${n}$ such
that ${m} + 2{n} = p$ the weak inference
$$
\vdash \ \ \  
\bigvee_{ {m}+2{n} = p} 
\exists (t_{{m}},z_{n}) \ 
\Big[\;
P
 \equiv  
\prod_{1 \le j \le {m}}(y-t_{{m},j})
\cdot
\prod_{1 \le k \le {n}}((y - a_{{n},k})^2 + {b_{{n},k}}^2),
\ 
\bigwedge_{1\le k\le {n}} b_{{n},k} \ne 0 \;\Big]. 
$$
By Theorem \ref{thLaplace} (Real Irreducible Factors as a weak existence) and using Lemma \ref{tlem:aux_comp_fact_with_mult}, 
we obtain 
$$ 
\lda \cH \rda 
_{\K[v]}
$$ 
with monoid part 
$$
\prod_{{m} + 2{n} = p \atop ({\bm \mu}, {\bm \nu}) \in \Lambda_{m} \times \Lambda_{n}} S_{{\bm \mu}, {\bm \nu}}^{h_{{\bm \mu}, {\bm \nu}}}
$$
with $h_{{\bm \mu}, {\bm \nu}} \le \max\{e, g\}^{2^{\frac12 p^2}}f^{2^{\frac12 p}}{\rm g}_4\{p\}$ 
and degree in $w$ bounded by 
$\max\{e, g\}^{2^{\frac12 p^2}}f^{2^{\frac12 p}}{\rm g}_4\{p\}(\delta_w + \max\{\delta_t, \delta_z \}\deg_w P)$,
which serves as the final incompatibility.
\end{proof}


\section{Hermite's Theory}  \label{sect_sylv_inertia_law}
\setcounter{equation}{0}

In this section we study Hermite's theory and Sylvester's inertia law
in the context of weak inferences and incompatibilities.
Hermite's theory has two aspects: on one hand, the rank and signature of Hermite's quadratic form determine the number of real roots,
and on the other hand, sign conditions on 
the principal minors of Hermite's quadratic form
also determine its rank and signature.

In Subsection \ref{subsecHerm_fact}, we explain how the rank and signature of Hermite's quadratic 
form is related to real root counting (Theorem \ref{hermite})
and we transform this statement into a 
weak inference of a  
diagonalization formula (Theorem \ref{fact_Her_matrix_fact}).
In Subsection \ref{subsecHerm_bez}, we explain that 
the rank and signature of Hermite's quadratic form are also determined by sign conditions on 
principal minors, which are closely related to subresultants (Theorem \ref{bezoutiansignature})
and we transform this statement into a weak inference 
of a different diagonalization 
formula 
(Theorem \ref{thm:her_through_bez}).
In Subsection \ref{subsecSylv}, we produce an incompatibility for Sylvester's inertia law, 
expressing the impossibility for a quadratic form to have two diagonal forms
with distinct rank and signature (Theorem \ref{thSylv2}).
Finally in Subsection \ref{subsecHerm_incomp}, 
combining results from the preceding  subsections,
 we produce
an incompatibility expressing the impossibility for a polynomial 
to have a
number of real roots in  conflict with the rank and signature of its Hermite's quadratic form
predicted by the signs of its principal minors (Theorem \ref{hermitetrsubresw}). 

In this section we use many results from Section \ref{Weak.inferences}, but it 
is absolutely independent from  the results from Section \ref{section_ivt}  and Section
\ref{section_fta}. 

On the other hand, the only result extracted from Section \ref{sect_sylv_inertia_law} used in the rest of the paper 
is Theorem \ref{hermitetrsubresw} (Hermite's Theory as an incompatibility), which produces an incompatibility used only twice in Section \ref{sect_elim_of_one_var}.

\subsection{Signature of Hermite's quadratic form and real root counting}\label{subsecHerm_fact}

 In this section, $\K$ is as usual an ordered field and $\R$ is a real closed field containing $\K$.
 Moreover, $\D$ is a domain and $\F$ is a  field 
of characteristic $0$
containing $\D$.
A typical example of this situation is the following:  $\K$ is the field of rational numbers, $\R$ the field of
real algebraic numbers, $\D=\K[c]$ the polynomials in a finite number of variables 
with coefficients in $\K$ and $\F$ the corresponding field of fractions.

We now recall the definition of 
Hermite's quadratic form \cite{Her,BPRbook} and its role in real root counting.

\begin{notation}\label{def:rksi}
 For a symmetric matrix ${\bm A}\in \K^{p \times p}$, we denote by ${\rm Si}({\bm A})$ 
and ${\rm Rk}({\bm A})$
the signature and rank of 
${\bm A}$ respectively.
\end{notation}

\begin{definition}[Hermite Quadratic Form]
Let  $P, Q
\in \D[y]$ 
with $\deg P
= p \ge 1$ and $P
$ monic.  
The Hermite's matrix ${\rm Her}(P;Q) \in \D^{p \times p}$
is the matrix defined for $1 \le j_1, j_2 \le p$ by 
$$
{\rm Her}(P;Q)_{j_1, j_2} = 
{\rm Tra}
(Q
\cdot y^{j_1 + j_2 - 2})
$$
where  ${\rm Tra}(A
)$ is the trace of the linear mapping of multiplication by $A
\in \F[y]$
in the $\F$-vector space $\F[y]/P
$.
\end{definition}

\begin{theorem}[Hermite's Theory (1)]
\label{hermite} Let $P, Q
\in \K[y]$ with $= p \ge 1$, $P$ monic. Then
\begin{eqnarray*}
{\rm Rk}({\rm Her}(P;Q)) &=& \#\{\alpha + i\beta  \in\R[i] \, | \, P(\alpha + i\beta)=0, Q(\alpha + i\beta)\not=0\}, \\
{\rm Si}({\rm Her}(P;Q)) &=& \#\{\theta \in\R \, | \, P(\theta)=0, Q(\theta)>0\}-\#\{\theta \in\R \, | \,  P(\theta)=0, Q(\theta)<0\}. 
\end{eqnarray*}
\end{theorem}

Even though this result is well known, we give here a detailed proof of Theorem \ref{hermite}
which we will follow later on to 
obtain a weak inference counterpart of it. 
We introduce first some more auxiliary notation and definitions.

\begin{definition}
\label{notation:sign}
For $\alpha \in \R$, its sign is defined as follows:
$$\begin{cases}
   {\rm sign} (\alpha) = 0 & {\rm if } \ \alpha = 0,\\
     {\rm sign} (\alpha) = 1 & {\rm if } \ \alpha > 0,\\
    {\rm sign} (\alpha) = -1 & {\rm if } \ \alpha < 0.\\
   \end{cases}$$
From now on, for $P\in \K[v],\tau \in \{-1,0,1\}$, we freely use ${\rm sign}(P)= \tau$, to mean
$$\begin{cases}
 P = 0 & {\rm if} \ \tau = 0,\\
    P>0 & {\rm if} \ \tau=1,\\
    P<0 & {\rm if} \ \tau = -1 .\\
   \end{cases}$$
   \end{definition}

   Similarly we define the invertibility of an element of $\R[i]$.
\begin{definition}
\label{inv}
For $\alpha+i\beta \in \R[i]$, its invertibility is defined as follows:
$$\begin{cases}
   {\rm inv} (\alpha+i\beta) = 0 & {\rm if} \ \alpha = 0, \ \beta = 0,\\
     {\rm inv} (\alpha+i\beta) = 1 & {\rm if} \ \alpha^2+\beta^2 \not=0.\\
   \end{cases}$$
From now on, for $P(z)=P\re(a,b)+i P\im(a,b)\in \K[i][v][z],\kappa \in \{0,1\}$, we 
freely use ${\rm inv}(P)= \kappa$, to mean
$$\begin{cases}
 P\re(a,b)= 0, \ P\im(a,b)= 0 & {\rm if} \ \kappa = 0,\\
 P\re(a,b)^2+P\im(a,b)^2 \not=0 & {\rm if} \ \kappa=1.
   \end{cases}$$
\end{definition}

\begin{remark}\label{rem:diag}
 If ${\bm D} \in \K^{p \times p}$ is a diagonal matrix, with diagonal elements $D_1,\ldots,D_p$, 
 \begin{eqnarray*}
{\rm Rk}({\bm D}) &=& \sum_{1 \le i \le p} \inv(D_i), \\
{\rm Si}({\bm D}) &=& \sum_{1 \le i \le p} \sign(D_i). 
\end{eqnarray*}
\end{remark}

\begin{notation}
\begin{itemize}
\item  
For $p \in \N_*$ and $j \in \N$ we denote by ${\rm A}_{p,j} \in \mathbb{Z}[c_0, \dots, c_{p-1}]$ the unique 
polynomial 
such that 
$${\rm A}_{p,j} \Big( {\rm Coef}(y_1,\ldots,y_p) \Big) = 
\sum_{1 \le k \le p}y_k^j \in \mathbb{Z}[y_1, \dots, y_p],$$
where 
 ${\rm Coef}(y_1,\ldots,y_p)$ is the vector whose $j$-th entry, $j=0,\ldots,p-1$,  is
 $$ {\rm Coef}_j(y_1,\ldots,y_p)=(-1)^{p-j}\sum_{K \subset \{1, \dots, p\} \atop |K| = p-j} \prod_{k \in K}y_k.$$
Note that  $\deg {\rm A}_{ p, j} = j$ (see \cite[Proof of Theorem 3, Chapter 7]{CLO}).

\item For $j \in \N$ and $({\bm \mu}, {\bm \nu}) \in \Lambda_{m} \times \Lambda_{n}$,
let
$t=(t_1, \dots, t_{\# {\bm \mu}})$ , 
$a=(a_1, \dots,a_{\#  {\bm \nu}})$, 
$b=(b_1, \dots, b_{\# {\bm \nu}})$ 
be sets of variables ,
$z_i=a_i+b_i$ and $z=(z_1, \dots,z_{\#  {\bm \nu}})$.
We denote
by ${\rm N}_{j}^{{\bm \mu},{\bm \nu}} \in$ $\mathbb{Z}[t,a,b]$
the Newton sum polynomial
$$
{\rm N}_{j}^{{\bm \mu},{\bm \nu}}
 = 
\sum_{1 \le i \le \# {\bm \mu}} \mu_it_i^j
+ \sum_{1 \le k \le \# {\bm \nu}} 2\nu_k (z_k^j)_{\re}.
$$

\end{itemize}

\end{notation}

\begin{remark}\label{rem2:form_herm_matrix}

\begin{itemize}
\item Let $p \in \N_*$, $({\bm \mu}, {\bm \nu}) \in \Lambda_{m} \times \Lambda_{n}$ with ${m} + 2n = p$, 
$t = (t_1, \dots, t_{\# {\bm \mu}})$ is a set of variables
and 
$z = (z_1, \dots, z_{\# {\bm \nu}})$
is a set of complex variables.
Following Definition \ref{defFact}, for $j \in \N$ we have
$$
{\rm A}_{p,j}({\rm F}_{0}^{{\bm \mu},{\bm \nu}}(t, z), \dots, 
{\rm F}_{p-1}^{{\bm \mu},{\bm \nu}}(t, z))
= {\rm N}_{j}^{{\bm \mu},{\bm \nu}}(t, z)
$$
in $\mathbb{Z}[t, a, b]$. 

\item Let $p \in \N_*$, $P
= 
y^p + \sum_{0 \le h \le p-1}\gamma_h  y^h$, 
$Q
 = \sum_{0 \le h \le q}\gamma'_h  y^h \in 
{\D}[y]$.
For  $1 \le j_1, j_2 \le p$, 
$$
 {\rm Her}(P;Q)_{j_1, j_2}
= 
\sum_{0 \le h \le q} \gamma'_h {\rm A}_{p,h + j_1 + j_2 - 2}(\gamma_0, \dots, \gamma_{p-1}) 
$$
(see \cite[Proposition 4.54]{BPRbook}). 
\end{itemize}
\end{remark}

\begin{notation} Let $p \in \N_*$, $({\bm \mu}, {\bm \nu}) \in \Lambda_{m} \times \Lambda_{n}$ with ${m} + 2n = p$, 
$t = (t_1, \dots, t_{\# {\bm \mu}})$ and 
$z = (z_1, \dots, z_{\# {\bm \nu}})$. 

\begin{itemize} 
\item For ${\bm \kappa} \in \{0, 1\}^{\{1, \dots, {\# {\bm \nu}}\}}$, we denote
${\rm Di}_{Q}^{{\bm \mu},{\bm \nu},{\bm \kappa}}(t)$ the diagonal matrix with entries 
$$(\mu_1 Q(t_1), \dots, \mu_{\# {\bm \mu}}Q(t_{\# {\bm \mu}}),
\nu_1 \kappa_1,  
-\nu_1 \kappa_1, \dots, \dots,
\nu_{\# {\bm \nu}} \kappa_{\# {\bm \nu}},  
-\nu_{\# {\bm \nu}} \kappa_{\# {\bm \nu}},
0, \dots, 0).
$$

\item 
We denote by
${\rm V}(t,z)$
the $p\times p$ matrix
$$
{
\left(
\begin{array}{cccccccccccc}
1 & \dots & 1  & 1 & 0 & \dots & \dots & 1 & 0 & 0 & \dots & 0 \cr
t_1 & \dots & t_{\# {\bm \mu}}  & a_1 & b_1 & \dots & \dots & a_{\# {\bm \nu}} & b_{\# {\bm \nu}} & 0 & \dots & 0 \cr
\vdots &  & \vdots  & \vdots & \vdots &  &  & \vdots & \vdots & \vdots &  & \vdots \cr
\vdots &  & \vdots  & \vdots & \vdots &  &  & \vdots & \vdots & \vdots &  & \vdots \cr
\vdots &  & \vdots  & \vdots & \vdots &  &  & \vdots & \vdots & \vdots &  & \vdots \cr
\vdots &  & \vdots  & \vdots & \vdots &  &  & \vdots & \vdots & \vdots &  & \vdots \cr
\vdots &  & \vdots  & \vdots & \vdots &  &  & \vdots & \vdots & \vdots &  & \vdots \cr
\vdots &  & \vdots  & \vdots & \vdots &  &  & \vdots & \vdots & \vdots &  & \vdots \cr
\vdots &  & \vdots  & \vdots & \vdots &  &  & \vdots & \vdots & 0 &  & 0 \cr
\vdots &  & \vdots & \vdots & \vdots &  &  & \vdots & \vdots &
1 & \dots & 0 \cr
\vdots &  & \vdots  & \vdots & \vdots &  &  & \vdots & \vdots & \vdots & \ddots  & \vdots \cr
t_1^{p-1} & \dots & t_{\# {\bm \mu}}^{p-1}  & (z_1^{p-1})_{\rm Re} & (z_1^{p-1})_{\rm Im} & \dots & \dots & (z_{\# {\bm \nu}}^{p-1})_{\rm Re} & (z_{\# {\bm \nu}}^{p-1})_{\rm Im} & 0 & \dots & 1 \cr
\end{array}
\right).
}
$$

\item  For ${\bm \kappa} \in \{0, 1\}^{\{1, \dots, {\# {\bm \nu}}\}}$ and 
$z' = (z'_k)_{\kappa_k = 1}$ 
we denote by 
${\rm Sq}_{{\bm \kappa}}(z')$ the $p \times p$ block diagonal matrix having the 
first ${\# {\bm \mu}}$ diagonal elements equal 
$1$, the next 
${\# {\bm \nu}}$ diagonal blocks of size $2$ equal to 
$$
\left\{ \begin{array}{ll}
\left(\begin{array}{cc}
       a'_k & b'_k  \cr
 -b'_k & a'_k
\end{array} \right) & \hbox{if }\kappa_k= 1,\\[4.5mm]
\hbox{the identity matrix of size }2  & \hbox{if }\kappa_k = 0,
\end{array}\right. 
$$
and the last $p - \# {\bm \mu}- 2\# {\bm \nu}$ diagonal elements equal to $1$. 

\item We denote by ${\rm B}_{{\bm \kappa}}(t, z, z')$ the matrix 
${\rm V}(t,z) \cdot{\rm Sq}_{{\bm \kappa}}(z')$.
\end{itemize}
\end{notation}

\begin{lemma}\label{lemma:detarminant_sim_vandermonde}
$$\det ({\rm V}(t,z)) = 
\prod_{1 \le j < j' \le \# {\bm \mu}}(t_{j'} - t_{j})
\cdot
\prod_{1 \le j \le \# {\bm \mu}, \atop 1 \le k \le \# {\bm \nu}}((a_k-t_j)^2 + b_k^2 ) 
\cdot
\prod_{1 \le k \le \# {\bm \nu}} b_k 
\cdot
\prod_{1 \le k < k' \le \# {\bm \nu}} 
{\rm R}(
z_{k},z_{k'}).
$$   
\end{lemma}
\begin{proof}{Proof.} Easy computation from the formula for the usual Vandermonde determinant.
\end{proof}

We can now give a proof of Theorem \ref{hermite} (Hermite's Theory (1)).

\begin{proof}{Proof of Theorem \ref{hermite}.} 
Consider 
the decomposition of $P$ into irreducible factors in $\R[y]$
$$
P
=
\prod_{1 \le j \le \# {\bm \mu}}(y - \theta_j)^{\mu_j}
\cdot
\prod_{1 \le k \le \# {\bm \nu}}((y- \alpha_k)^2 + \beta_k^2)^{\nu_k},
$$
with $\theta = (\theta_1, \dots, \theta_{\# {\bm \mu}}) \in \R^{\# {\bm \mu}}, \alpha = (\alpha_1, \dots, \alpha_{\# {\bm \nu}}) \in \R^{\# {\bm \nu}}$
and $\beta = (\beta_1, \dots, \beta_{\# {\bm \nu}}) \in \R^{\# {\bm \nu}}$ and
${\bm \kappa} \in \{0,1\}^{\{1, \dots, {\# {\bm \nu}}\}}$ defined by $\kappa_k=1$ if $Q(\alpha_k+i\beta_k) \ne 0$ and 
$\kappa_k=0$ otherwise.

For $1 \le k \le {\# {\bm \nu}}$ with $\kappa_k = 1$, we consider a square root 
$\alpha'_k + i\beta'_k$ of 
$2Q(\alpha_k+i\beta_k)$.
Since $\det({\rm V}(\theta,\alpha + i\beta)) \ne 0$ by Lemma \ref{lemma:detarminant_sim_vandermonde} and 
$\det({\rm Sq}_{{\bm \kappa}}(\alpha' + i\beta')) \ne 0$ by an easy computation, we have that 
$\det({\rm B}_{{\bm \kappa}}(\theta,\alpha + i\beta,\alpha' + i\beta'))\ne 0$.

Using Remark \ref{rem2:form_herm_matrix}, it can be checked that 
$$
{\rm Her}(P;Q) = 
{\rm B}_{{\bm \kappa}}(\theta,\alpha + i\beta,\alpha' + i\beta') \cdot
{\rm Di}_{Q}^{{\bm \mu},{\bm \nu},{\bm \kappa}}
(\theta) \cdot
{\rm B}_{{\bm \kappa}}(\theta,\alpha + i\beta,\alpha' + i\beta')^{\mathrm t}.
$$ The proof concludes then by simply noting that, by Remark \ref{rem:diag},
\begin{eqnarray*}
{\rm Rk}({\rm Di}_{Q}^{{\bm \mu},{\bm \nu},{\bm \kappa}}(\theta)) &=& \#\{\alpha + i \beta \in\R[i] \, 
| \, P(\alpha + i \beta)=0, Q(\alpha + i \beta)\not=0\}, \\
{\rm Si}({\rm Di}_{Q}^{{\bm \mu},{\bm \nu},{\bm \kappa}}(\theta)) &=& \#\{\theta \in\R \, | \, P(\theta)=0, Q(\theta)>0\}-\#\{\theta \in\R \, | \,  P(\theta)=0, Q(\theta)<0\}.
\end{eqnarray*}
\end{proof}

Now we give a weak inference version of Theorem \ref{hermite} (Hermite's Theory (1)), using Definition \ref{defFact}. 
Note that for the 
first time in this paper, the set of variables $w$ in the  
statement of the theorem is not an arbitrary set of variables included in $v$.
This is enough for our purposes and enables us to obtain a more precise result. 
In fact, many times from here on we will make a similar distinction for the set of variables $w$.

\begin{theorem}[Hermite's Theory (1) as a weak existence]\label{fact_Her_matrix_fact} Let $p \ge 1$, $P
= y^p + 
\sum_{0 \le h \le p-1}C_h \cdot y^h \in \K[u][y]$, 
$m$, $n \in \N$ with
${m}+2n=p$, $({\bm \mu},{\bm \nu}) \in \Lambda_{m} \times  \Lambda_{n}$,
$t = (t_1, \dots, t_{\# {\bm \mu}})$, $z = (z_1, \dots, z_{\# {\bm \nu}})$, 
$Q
= \sum_{0 \le h \le q}D_h \cdot y^h \in \K[u][y]$, ${\bm \kappa} \in \{0,  1\}^{\{1, \dots,{{\# {\bm \nu}}}\}}$
and $s({\bm \kappa}) = \#\{k \ | \ 1 \le k \le {\# {\bm \nu}}, \kappa_k = 1 \}$. 
Then
$$
{\rm Fact}(P)^{{\bm \mu},{\bm \nu}}(t,z), \ 
\bigwedge_{1 \le k \le {\# {\bm \nu}}}{\inv}(Q(z_k)) = \kappa_k  \ \ \ \vdash$$
$$
\vdash \ \ \ \exists z' \ 
[\; {\rm Her}(P; Q) \equiv
{\rm B}_{{\bm \kappa}}(t,z,z')\cdot
{\rm Di}_{Q}^{{\bm \mu},{\bm \nu},{\bm \kappa}}
(t)\cdot
{\rm B}_{{\bm \kappa}}(t,z,z')^{\rm t}, \ \det({\rm B}_{\kappa}(t,z,z'))\ne 0 \;]
$$
where $z' = (z'_{k})_{\kappa_k = 1}$. 

Suppose we have an initial incompatibility 
in variables $(v, a', b')$ where  $v \supset (u,t, a, b)$ and  $(a', b')$ are disjoint from $v$,  
with monoid part
$
S \cdot \det({\rm B}_{{\bm \kappa}}(t,z,z'))^{2e}
$, degree in $w$ bounded by $\delta_w$ for some subset of variables $w \subset v$ 
disjoint from $(t, a, b)$, 
degree in $t_j$ bounded by $\delta_t$, 
degree in $(a_k, b_k)$ bounded by $\delta_z$ 
and degree in 
$(a'_k, b'_k)$ bounded by $\delta_{z'}$.
Then 
the final incompatibility has monoid part
$$
S^{2^{2s({\bm \kappa})}}
\cdot
\prod_{1 \le j < j' \le \# {\bm \mu}}
(t_{j'} - t_{j})
^{ 2^{2s({\bm \kappa})+1} e} 
\cdot
\prod_{1 \le k \le \# {\bm \nu}}b_k 
^{2^{2s({\bm \kappa})+1}(2{\# {\bm \mu} + 1)e}}
\cdot
$$
$$
\cdot
\prod_{1 \le k < k' \le \# {\bm \nu}} 
{\rm R}(
z_{k},z_{k'})
^{2^{2s({\bm \kappa})+1}e}
\cdot
\prod_{1 \le k \le \# {\bm \nu}, \atop \kappa_k = 1}
(Q_{\re}^2(z_k) + Q_{\im}^2(z_k))^{2e'_k}
$$
with $e'_k \le 2^{2s({\bm \kappa}) - 2}(2e+1)$,
degree in $w$ bounded by 
$$
2^{2s({\bm \kappa})}\Big(\delta_w  + (2s({\bm \kappa})(3e + \delta_{z'}) + q + 2p +6)
\max\{\deg_w P, \deg_w Q \}\Big),
$$ 
degree in $t_j$ bounded by 
$
2^{2s({\bm \kappa})}(\delta_t + q + 2p - 2)
$ and
degree in $(a_k, b_k)$ bounded by $2^{2s({\bm \kappa})}(\delta_z + (6 + 2(3e + \delta_{z'}))q + 2p - 2)$. 
\end{theorem}

\begin{proof}{Proof.}
We apply to the initial incompatibility the weak inference
$$
{\rm Fact}(P)^{{\bm \mu},{\bm \nu}}(t,z), 
\
\bigwedge_{1 \le j \le \# {\bm \mu}, \atop 1 \le k \le \# {\bm \nu}} (a_k - t_j)^2 + b_k^2 \ne 0, 
\ \bigwedge_{1 \le k \le \# {\bm \nu}, \atop \kappa_k = 1} z'_k \ne 0 
\ \ \ \vdash
\ \ \ \det({\rm B}_{{\bm \kappa}}(t,z, z')) \ne 0. 
$$
By Lemma \ref{lemma_basic_sign_rule_1} (item \ref{lemma_basic_sign_rule:5}) according to 
Lemma \ref{lemma:detarminant_sim_vandermonde}, we obtain an incompatibility with monoid part
$$
S
\cdot
\Big( \prod_{1 \le j < j' \le \# {\bm \mu}}(t_{j'} - t_{j})
\cdot
\prod_{1 \le j \le \# {\bm \mu}, \atop 1 \le k \le \# {\bm \nu}}((a_k-t_j)^2 + b_k^2 )
\cdot
\prod_{1 \le k \le \# {\bm \nu}} b_k 
\cdot
\prod_{1 \le k < k' \le {\# {\bm \nu}}} 
{\rm R}(
z_{k},z_{k'})
\cdot
\prod_{1 \le k \le {\# {\bm \nu}}, \atop \kappa_k = 1} ({a'_k}^2 + {b'_k}^2 )\Big)^{2e}
$$
and the same degree bounds. 

Then we successively apply for $1 \le j \le \# {\bm \mu}$ and $1 \le k \le \# {\bm \nu}$
the weak inferences
$$
\begin{array}{rcl}
(a_k - t_j)^2 + b_k^2 > 0 &  \ \, \vdash \, \  & (a_k - t_j)^2 + b_k^2 \ne 0, \\[3mm]
(a_k - t_j)^2  \ge 0, \ b_k^2 > 0 & \vdash & (a_k - t_j)^2 + b_k^2 > 0, \\[3mm]
& \vdash & (a_k - t_j)^2  \ge 0, \\[3mm]
b_k \ne 0 & \vdash & b_k^2  > 0.
\end{array}
$$
By Lemmas \ref{lemma_basic_sign_rule_1} (items \ref{lemma_basic_sign_rule:1.5},
\ref{lemma_basic_sign_rule:2} and \ref{lemma_basic_sign_rule:3}) and 
\ref{lemma_sum_of_pos_is_pos}, we obtain an incompatibility with monoid part
$$
S\cdot
\prod_{1 \le j < j' \le \# {\bm \mu}}(t_{j'} - t_{j})
^{2e} 
\cdot
\prod_{1 \le k \le \# {\bm \nu}}b_k 
^{2(2\# {\bm \mu} + 1)e}
\cdot
\prod_{1 \le k < k' \le \# {\bm \nu}}
{\rm R}(
z_{k},z_{k'})
^{2e}
\cdot
\prod_{1 \le k \le \# {\bm \nu}, \atop \kappa_k = 1} ({a'_k}^2 + {b'_k}^2 )
^{2e}
$$
and the same degree bounds.

For $1 \le j_1, j_2 \le p$, 
by Remark 
\ref{rem2:form_herm_matrix}, 
we
have 
\begin{eqnarray*}
&&{\rm Her}(P;Q)_{j_1, j_2} 
-
({\rm B}_{{\bm \kappa}}(t,z,z')\cdot
{\rm Di}_{Q}^{{\bm \mu},{\bm \nu}, {\bm \kappa}}(t)\cdot
{\rm B}_{{\bm \kappa}}(t,z,z')^{\rm t})_{j_1, j_2}
=\\
&=&
 \sum_{0 \le h \le q} 
D_h \cdot \Big( {\rm A}_{p, h + j_1 + j_2 - 2}(C_0, \dots, C_{p-1}) - 
{\rm A}_{p, h + j_1 + j_2 -2}({\rm F}_{0}^{{\bm \mu},{\bm \nu}}(t, z), \dots, 
{\rm F}_{p-1}^{{\bm \mu},{\bm \nu}}(t, z))
\Big)
+\\
&+&
\sum_{1 \le k \le {\# {\bm \nu}}, \atop \kappa_k = 0}2\nu_k
\Big(Q(z_k)_{\re}(z_k^{j_1 + j_2 - 2})_{\re} - Q(z_k)_{\im}(z_k^{j_1 + j_2 - 2})_{\im}\Big)  
+\\
&+& 
\sum_{1 \le k \le {\# {\bm \nu}}, \atop \kappa_k = 1}\nu_k
\Big(
\big(2Q(z_k)_{\re} - ({a'_k}^2 - {b'_k}^2) \big)\cdot(z_k^{j_1 + j_2 - 2})_{\re}
- \big(2Q(z_k)_{\im} - 2a'_kb'_k\big)\cdot(z_k^{j_1 + j_2 - 2})_{\im} 
\Big).
\end{eqnarray*}
Therefore, we apply the weak inference
$$
{\rm Fact}(P)^{{\bm \mu},{\bm \nu}}(t, z),  \  
\bigwedge_{1 \le k \le {\# {\bm \nu}}, \atop \kappa_k = 0} 
Q(z_k) = 0, \ 
\bigwedge_{1 \le k \le {\# {\bm \nu}}, \atop \kappa_k = 1} 
{z'_k}^2 = 2Q(z_k)
\ \ \ \vdash $$
$$
\vdash
\ \ \
{\rm Her}(P; Q) \equiv
{\rm B}_{{\bm \kappa}}(t,z,z')\cdot
{\rm Di}_{Q}^{{\bm \mu},{\bm \nu},{\bm \kappa}}(t)\cdot
{\rm B}_{{\bm \kappa}}(t,z,z')^{\rm t}. 
$$
By Lemma \ref{lem:comb_lin_zero_zero}, after some analysis, 
we obtain an incompatibility with the same monoid part, 
degree in $w$ bounded by $\delta_w  + \deg_w Q + (q + 2p - 2)\deg_w P$, 
degree in $t_j$ bounded by $\delta_t + q + 2p - 2$,
degree in $(a_k, b_k)$ bounded by $\delta_z + q + 2p - 2$
and degree in $(a'_k, b'_k)$ bounded by $\delta_{z'}$.

Suppose that $\{k \ | \ 1 \le k  \le {\# {\bm \nu}}, \kappa_k = 1 \} = \{k_1, \dots, k_{s({\bm \kappa})}\}$.
Finally we apply for $1 \le s \le s({\bm \kappa})$ the weak inference
$$
Q(z_{k_s}) \ne 0 \ \ \
\vdash 
\ \ \  \exists z'_{k_s} \ [\; z'_{k_s} \ne 0, \ {z'_{k_s}}^2 = 2Q(z_{k_s})\;]. 
$$
Using Lemma \ref{lemma_square_root_complex_ne},
it is easy to prove by induction on $s$ that, for $1 \le s \le s({\bm \kappa})$, after the 
application of the weak inference corresponding to index $s$, we obtain
an incompatibility with monoid part 
$$
S^{2^{2s}}
\cdot
\prod_{1 \le j < j' \le \# {\bm \mu}}
(t_{j'} - t_{j})
^{ 2^{2s+1} e} 
\cdot
\prod_{1 \le k \le \# {\bm \nu}}b_k 
^{2^{2s+1}(2\# {\bm \mu} + 1)e}
\cdot $$
$$ \cdot
\prod_{1 \le k < k' \le \# {\bm \nu}} 
{\rm R}(
z_{k},z_{k'})
^{2^{2s+1}e}
\cdot
\prod_{1 \le i \le s} 
(Q_{\re}^2(z_{k_i}) + Q_{\im}^2(z_{k_i}))^{2^{2s}e + 2^{2s-2i+1}}
\cdot
\prod_{s+1 \le i \le s({\bm \kappa})} ({a'_{k_i}}^2 + {b'_{k_i}}^2)^{2^{2s+1}e}
,
$$
degree in $w$ bounded by 
$2^{2s}(\delta_w  + \deg_w Q  + (q + 2p - 2)\deg_w P) 
+ (\frac{20}3(2^{2s}-1) + s2^{2s+1}(3e + \delta_{z'}))\deg_w Q$,
degree in $t_j$ bounded by $2^{2s}(\delta_t + q + 2p - 2)$, 
degree in $(a_k, b_k)$ bounded by $2^{2s}(\delta_z + q + 2p - 2)$, 
degree in $(a_{k_i}, b_{k_i})$ bounded by 
$2^{2s}(\delta_z + q + 2p - 2) + 2^{2(s-i)}(20 + 2^{2i+1}(3e + \delta_{z'}))q$ 
for $1 \le i \le s$, 
degree in $(a_{k_i}, b_{k_i})$ bounded by 
$2^{2s}(\delta_z + q + 2p - 2)$
for $s+1 \le i \le s({\bm \kappa})$ and  
degree in $(a'_{k_i}, b'_{k_i})$ bounded by 
$2^{2s}\delta_{z'}$
for $s+1 \le i \le s({\bm \kappa})$. 
Therefore, the incompatibility we obtain after the application of the $s({\bm \kappa})$ weak inferences
serves as the final incompatibility. 
\end{proof}

\subsection{Signature of Hermite's 
quadratic form and signs of principal minors}\label{subsecHerm_bez}

The preceeding method to compute the signature of the Hermite's quadratic form is
based on the factorization of $P$ over a real closed field; therefore, it involves algebraic
numbers. 
We explain now
another way to compute this signature using only operations in the ring of coefficients of $P$ and $Q$,
through the principal minors of the Hermite's matrix.
Most of these results are classical \cite{Gan,BPRbook} but we need them under precise algebraic identity form.
   
\begin{notation}
\label{notation:testcoeff}
\begin{itemize}
\item Let $P, Q
\in \D[y]$ with $\deg P
= p \ge 1$ and $P$ monic.  
For $0 \le j \le p-1$, we denote by
${\rm HMi}_j(P;Q)$  the $(p-j)$-th principal minor of ${\rm Her}(P;Q)$ and by   
${\rm HMi}(P;Q)$ the list $ [{\rm HMi}_{0}(P;Q), \dots, {\rm HMi}_{p-1}(P;Q)]$ in 
$\D$. We additionally define ${\rm HMi}_{p}(P;Q) = 1$. 

\item Given 
a sign condition
$\tau \in \{-1, 0, 1\}^{\{0,\dots,{p-1}\}}$  we 
denote by $d(\tau)$ the strictly
decreasing sequence  $(d_0,\ldots,d_s)$ of natural numbers  
defined by $d_0=p$ and 
$\{d_1, \dots, d_s\} = \{ j \ | \ 0 \le j \le p-1, \tau(j) \ne 0\}$.  
\item For $k \in \N, \varepsilon_k = (-1)^{k(k-1)/2}$.

\end{itemize}

\end{notation}

\begin{theorem}[Hermite's Theory (2)]
\label{bezoutiansignature}
Let $P, Q
\in \K[y]$ with $\deg P
 = p \ge 1$, $P$ monic, $\tau \in \{-1, 0, 1\}^{\{0, \dots, p-1\}}$  be
the  sign condition 
defined by 
$\tau(i)={\rm sign}({\rm HMi}_i(P; Q))$
and $d(\tau) = (d_0, \dots, d_{s})$.
Then 
\begin{eqnarray*}
 {\rm Rk}({\rm Her}(P;Q))&=& p - d_s, \\
 {\rm Si}({\rm Her}(P;Q))&=&
\sum_{1 \le i \le s, \atop d_{i-1}-d_i \, \hbox{\scriptsize {odd}}} 
\varepsilon_{d_{i-1}-d_{i}} 
\tau(d_{i-1})\tau(d_i). 
\end{eqnarray*}
\end{theorem}

As in the previous subsection, even though this result is well known, we give here a detailed proof of Theorem \ref{bezoutiansignature}
which we will follow later on to 
obtain a weak inference counterpart of it. 
First, we introduce some more notations and definitions, in order to make a link between Hermite's matrix and subresultants.

\begin{definition}[Subresultants]\label{def:subres} Let $P
, R \in \D[y]$ with $\deg P
 = p \ge 1$ 
and
$\deg R
 = r < p$. 

\begin{itemize}
\item For $0 \le
  j \le r$, the  Sylvester-Habicht matrix ${\rm SyHa}_j (P, R) \in \D^{(p + r - 2j) \times 
(p + r - j)}$ is the matrix 
  whose rows are the polynomials
  \[ y^{r - j - 1}\cdot  P, \ldots, P, R, \ldots, y^{p - j - 1}\cdot R, \]
  expressed in the monomial basis $y^{p+r-j-1},\ldots,y,1$.
\item For $0 \le j \le r$, the $j$-th subresultant polynomial of $P,R$, ${\rm sResP}_j (P, R)
\in \D[y]$ 
  is the polynomial determinant of ${\rm SyHa}_j (P, R)$, i.e.
  $${\rm sResP}_j (P, R)=\sum_{0 \le i \le j} \det({\rm SyHa}_{j,i}(P, R)) \cdot y^i$$
  where ${\rm SyHa}_{j,i}(P, R)\in \D^{(p + r - 2j) \times 
(p + r - 2j)}$ is the matrix obtained by taking the
  $p+r-2j-1$ first
  columns 
  and the $(p+r-j-i)$-th column of  ${\rm SyHa}_j (P, R)$.

 By convention, we
  extend this definition 
  with
  \begin{eqnarray*}
    {\rm sResP}_p (P, R) & = & P,\\
    {\rm sResP}_{p - 1} (P, R) & = & R,\\
    {\rm sResP}_j (P, R) & = & 0 \qquad \hbox{ for } r < j < p - 1.
  \end{eqnarray*}

\item For $0 \le j \le r$, the $j$-th signed subresultant coefficient of $P$ and $R$, ${\rm sRes}_j (P,
R)\in \D$  is the coefficient of $y^j$ in
${\rm sResP}_j (P, R)$. 

By convention, we extend this definition with  
\begin{eqnarray*}
    {\rm sRes}_p(P, R) & = & 1,\\
    {\rm sRes}_j(P, R) & = & 0 \qquad \hbox{ for } r < j \le p - 1.
  \end{eqnarray*}

\item For $0 \le j \le p$, ${\rm sResP}_j (P, R)$ is said to be defective 
if $\deg {\rm sResP}_j (P, R) < j$ or, equivalently, if ${\rm sRes}_j (P, R)
= 0$.

\item For $0 \le j \le r$, the $j$-th subresultant cofactors of $P, R$,  
${\rm sResU}_j (P, R), {\rm sResV}_j (P, R)
\in \D[y]$ 
are  the determinants of the matrices obtained by taking the first $p + r - 2j - 1$ first columns of 
${\rm SyHa}_j (P, R)$ and a last column equal to $(y^{r-j-1}, \dots, 1, 0, \dots, 0)$
and equal to $(0, \dots, 0, 1, \dots, y^{p-j-1})$, respectively.

By convention we extend these definitions with
  $$\begin{array}{rclrcll}
    {\rm sResU}_p (P, R) & = & 1,   & {\rm sResV}_p (P, R) & = & 0, \\
    {\rm sResU}_{p - 1} (P, R) & = & 0,   & {\rm sResV}_{p-1} (P, R) & = & 1,\\
    {\rm sResU}_j (P, R) & = & 0, &  {\rm sResV}_{j} (P, R) & = & 0 \, \qquad \hbox{ for } r < j < p - 1,\\
    {\rm sResU}_{-1} (P, R) & = & -{\rm sRes}_{0} (P, R)\cdot R,   & {\rm sResV}_{-1} (P, R) 
& = & {\rm sRes}_{0} (P, R) \cdot P. \\
  \end{array}$$

\end{itemize}

\begin{remark}
When $P
$ is monic, 
the definitions of subresultant polynomials, signed subresultant coefficients and subresultant
cofactors, 
are independent of the degree $r < p$ of $R
$ (see for instance \cite{GLRR}). 
Therefore,
we can 
artificially 
consider the degree of $R
$ as $p-1$, 
specialize
its first $p-r - 1$ coefficients as $0$ and obtain the same result. 
\end{remark}

\end{definition}

The connection between the subresultant coefficients and the Hermite's matrix is the following.

\begin{proposition}\label{prop:conn_herm_subr} Let $P
, Q
\in \D[y]$ 
with $\deg P 
= p \ge 1$, $P
$ monic and let 
$R
$ be the 
remainder of $P'
\cdot Q
$ in the division by $P
$. Then for $0 \le j \le p$
$${\rm HMi}_{j}(P;Q)={\rm sRes}_{j}(P,R).$$ 
\end{proposition}

\begin{proof}{Proof.}
See \cite[Lemma 9.26 and Proposition 4.55]{BPRbook}.
\end{proof}

We now explain how to diagonalize Hermite's matrix using an alternative method. 
The first step is to transform it into a block Hankel triangular matrix, using 
subresultants.

\begin{notation} 
\begin{itemize} 
\item Given $\alpha = (\alpha_1, \dots, \alpha_p) \in \D^p$, we denote by 
${\rm HanT}_p(\alpha)\in \D^{p \times p}$ the 
Hankel triangular matrix 
defined for $1 \le i, j \le p$ by ${\rm HanT}_p(\alpha)_{ij} = 0$ if $i+j \le p$ and 
${\rm HanT}_p(\alpha)_{ij} = \alpha_{2p + 1 - i - j}$ if $i+j \ge p+1$.

\item Given $S
 = \sum_{0 \le h \le p}\alpha_hy^h \in \D[y]$, we denote by 
${\rm HanT}_p(S)\in \D^{p \times p}$ the 
Hankel triangular matrix ${\rm HanT}_p(\alpha_1, \dots, \alpha_p)$.
\end{itemize}

\end{notation}

\begin{notation}\label{not:subres_and_more} 
Let $P
, R
\in \D[y]$ 
with $\deg P
 = p \ge 1$ and $\deg R
  = r < p$.
Let 
$d=(d_0,\ldots,d_s)$ be
the sequence of degrees of the non-defective subresultant polynomials of $P$ and $R$ and 
$d_{-1} = p+1$. Note that $d_0 = p$ and $d_1 = r$.

\begin{itemize}
\item For $1 \le i \le s$, let $R_i
 = {\rm sResP}_{d_{i-1} - 1}(P,R) \in \D[y]$. 
By the Structure Theorem for Subresultants (\cite[Theorem 8.30]{BPRbook} ),
$\deg R_i
 = d_i$.

\item For $1 \le i \le s$, 
$$
\begin{array}{rcl}
{\rm sR}_{d_i}&=&{\rm sRes}_{d_{i}}(P,R) \in \D,\\
T_{d_{i-1}-1}&=&{\rm lcoeff}({\rm sResP}_{d_{i-1}-1}(P,R)) \in \D.
\end{array}
$$
We extend this definiton with 
$$
\begin{array}{rcl}
{\rm sR}_{p}&=& 1 \in \D,\\
T_{p}&=& 1 \in \D.
\end{array}
$$

\item For $1 \le i \le s$, 
\begin{eqnarray*}
 \tilde F_{d_{i}-1}
  & = &
{\rm sResU}_{d_{i-2}-1}(P,R) \cdot {\rm sResV}_{d_{i}-1}(P,R)  -
\\[3mm]
&&- {\rm sResU}_{d_{i}-1}(P,R) \cdot {\rm sResV}_{d_{i-2}-1}(P,R) \in \D[y],\\
F_{d_{i}-1}
& = &\frac{1}{{\rm sR}_{d_{i}} \cdot {\rm sR}_{d_{i-1}} \cdot  T_{d_{i-1}-1} \cdot  T_{d_{i-2}-1}}  \cdot 
\tilde F_{d_{i}-1}
 \in \F[y].
\end{eqnarray*}
As seen in the proof of \cite[Proposition 8.36]{BPRbook}, 
$\tilde F_{d_{i}-1}$ is the quotient of $T_{d_{i-1}-1} \cdot {\rm sR}_{d_{i}} \cdot {\rm sResP}_{d_{i-2}-1}(P, R)$
in the division by ${\rm sResP}_{d_{i-1}-1}(P, R)$; therefore, for 
$1 \le i \le s$, $\deg_y \tilde F_{d_{i}-1}= \deg_y F_{d_{i}-1} = d_{i-1} - d_i$,  
$
{\rm lcoeff} (\tilde F_{d_{i}-1}) = {\rm sR}_{d_{i}} \cdot T_{d_{i-2}-1}$
and 
$
{\rm lcoeff} (F_{d_{i}-1}) = \frac{1}{{\rm sR}_{d_{i-1}} \cdot T_{d_{i-1}-1}}.$

\item Let ${\rm HanB}_{P;R} \in \F^{p \times p}$ be a 
block Hankel triangular matrix 
composed by $s$ or $s+1$ blocks according to $d_s = 0$ or $d_s  > 0$.
For $1 \le i \le s$, the $i$-th block of ${\rm HanB}_{P;R} \in \F^{p \times p}$, of size $d_{i-1}-d_i$, is 
${\rm HanT}_{d_{i-1}-d_i}(F_{d_{i}-1})$ and, if $d_s >0$, there is a final $0$ block of size $d_s$.

\item Let us take now $P
$ monic and  let $Q
 \in \D[y]$. Consider $R
 \in \D[y]$ to be the remainder of $P'
 \cdot Q
 $ in the division by $P
 $.
Let ${\rm M}_{P;Q}\in \D^{p \times p}$ be the matrix of the basis ${\cal R} := $
$$\{y^{d_0-d_1-1} \cdot R_1,\ldots,R_1,\ldots, y^{d_{i-1}-d_{i}-1} \cdot R_{i},\ldots,R_{i},\ldots,
 y^{d_{s-1}-d_{s}-1} \cdot R_{s},\ldots, R_{s},y^{d_{s} -1}, \dots, 
1\}
$$
of the subspace of $\F[y]$ of polynomials 
of degree less than $p$, 
in the Horner basis of $P$,  
$${\rm Hor}(P) := \{y^{p-1} + \sum_{0 \le h \le p-2}\gamma_{h+1}y^h, \dots, 1\}.$$

\end{itemize}

\end{notation} 

In order to prove Theorem \ref{bezoutiansignature} (Hermite's Theory (2)), we also use Bezoutians, which 
we recall now. 

\begin{definition}[Bezoutian]\label{def:bezoutian} 
Let $P, R
\in \D[y]$, with $\deg P
 = p \ge 1$ and $\deg R = r < p$.
The Bezoutian of $P$ and $R$ is defined as 
$$
{\rm Bez}(P,R)=\frac{P(x)\cdot R(y)-R(x)\cdot P(y)}{x-y} \in \D[x, y].
$$
If ${\cal B}= \{b_1
,\ldots b_{p}
\}$ is a basis of $\F[y]/P
$,
${\rm Bez}(P,R)$ can be uniquely written as
$$
{\rm Bez}(P,R)=\sum_{1 \le i,j \le p}\alpha_{i,j} \cdot b_i(x)\cdot b_j(y).
$$
The Bezoutian matrix ${\rm Bez}_{\cal B}(P;R) \in \F^{p \times p}$ 
is the symmetric matrix with $(i,j)$-th entry equal to the coefficient $\alpha_{i,j}$
of $b_i(x)\cdot b_j(y)$ in ${\rm Bez}(P,R)$.
\end{definition}

\begin{lemma}\label{lem:fact_herm_mat_through_bez_0} 
Following Notation \ref{not:subres_and_more},
$$
{\rm Bez}_{{\cal R}}(P;R)= {\rm HanB}_{P;R}.
$$
\end{lemma}

\begin{proof} {Proof.} 
Since for any $S
 = \sum_{0 \le h \le p}\alpha_hy^h \in \D[y]$ we have that
$$
\frac{S(x) - S(y)}{x-y} = \sum_{1 \le i \le p} \ \sum_{p+1-i \le j \le p} \alpha_{2p+1-i-j}\cdot x^{p-i}\cdot y^{p-j},
$$
in order to prove the claim we have to prove that 
$$
{\rm Bez}(P, R) = \sum_{1 \le i \le s}\frac{F_{d_i - 1}(x) - F_{d_i - 1}(y)}{x-y}\cdot R_i(x)\cdot R_i(y).
$$
This will be done by induction on $s$, which, by the 
Structure Theorem for Subresultants ( \cite[Theorem 8.30]{BPRbook}) is  equal to the length of the remainder sequence of
$P$ and $R$. 

If $s = 0$ then $R
$ is the zero polynomial and the statement is clear.  
Now suppose that $s \ge 1$, therefore $R
$ is not the zero polynomial,
and let $S
$ be the remainder of 
$P
$ in the division by $R
$. Note that $S
$ is the zero polynomial if and only if $s = 1$. 
We also have that $R
 = R_1
 $ and, 
since 
${\rm sR}_{p} = T_{p} = 1$,  
${F}_{r-1}
$ is the quotient of $P
$ in the division by $R
$, this is to say
$$
P
 = {F}_{r-1}
 \cdot  R_1
  + S
$$
and therefore
\begin{equation}\label{aux_equation_bez}
{\rm Bez}(P, R)  =    \frac{{F}_{r-1}(x) - {F}_{r-1}(y)}{x-y} \cdot R_1(x)\cdot R_1(y) +   {\rm Bez} (R, -S).
\end{equation}

For $s = 1$  
equation (\ref{aux_equation_bez})
proves the claim. Suppose now that $s \ge 2$. 
We define $R'_2
, \dots, R'_s
$, 
${\rm sR}'_{d_1}, \dots, {\rm sR}'_{d_{s}}$, 
${T}'_{d_1}, {T}'_{d_1-1}, \dots, {T}'_{d_{s-1}-1}$
and $F'_{d_2 - 1}
, \dots,  F'_{d_s - 1}
$ as we did in Notation \ref{not:subres_and_more}, 
but this time we consider all definitions depending on the polynomials $R$ and $-S$ instead of $P$ and $R$. 
If $\beta$ is the leading coefficient of $R$, we have 
$$
\begin{array}{rcl}
R_2
 & =  & -\varepsilon_{p-r}\cdot \beta^{p-r+1}\cdot S
 ,  \\[2mm]
{\rm sR}_{d_1}& = &\varepsilon_{p-r}\cdot \beta^{p-r}, \\[2mm]
T_{d_0 - 1} &=& \beta, \\[2mm]
R'_2
 & = & -S
 .
\end{array}
$$
In addition, by
Proposition \cite[8.35]{BPRbook}, there exists $\lambda \in \D, \lambda \ne 0$,
such that 
$$
\begin{array}{rclll}
{\rm sR}_{d_i} & = &  \lambda \cdot {\rm sR}'_{d_i}&   & \hbox{ for } 2 \le i \le s, \\[2mm]
T_{d_{i-1} - 1} & = &  \lambda \cdot T'_{d_{i-1} - 1}&   & \hbox{ for } 3 \le i \le s, \\[2mm]
R_i
 & =&  \lambda \cdot R'_i
  & & \hbox{ for } 3 \le i \le s.
\end{array}
$$ 
From this we first deduce that 
$$
F_{d_2 - 1}  =   \frac{1}{{\rm sR}_{d_1}T_{d_{0} - 1}}\cdot {\rm Quot}( R_1 , R_2) =  
 \frac{1}{\varepsilon_{p-r}\cdot \beta^{p-r+1}}\cdot {\rm Quot}( R , -\varepsilon_{p-r}\cdot \beta^{p-r+1}\cdot S) = 
$$
$$
= \frac{1}{(\varepsilon_{p-r}\cdot \beta^{p-r+1})^2}\cdot {\rm Quot}( R , -S)  = \frac{1}{(\varepsilon_{p-r}\cdot \beta^{p-r+1})^2}\cdot  F'_{d_2 - 1},
$$
second, since $T_{d_{1} - 1}$ and $T'_{d_{1} - 1}$ are the leading coefficients of $R_2$ and $R'_2$ respectively,  that
$$
F_{d_3 - 1} = \frac{1}{{\rm sR}_{d_2}\cdot T_{d_{1} - 1}}\cdot {\rm Quot}( R_2 , R_3) 
=
\frac{1}{\lambda\cdot {\rm sR}'_{d_2}\cdot T'_{d_{1} - 1}}\cdot {\rm Quot}( R'_2 , \lambda \cdot R_3) =
\frac{1}{\lambda^2}\cdot F'_{d_3 - 1},
$$
and finally, that for $4 \le i \le s$, 
$$
F_{d_i - 1} = \frac{1}{{\rm sR}_{d_{i-1}}\cdot T_{d_{i-2} - 1}}\cdot {\rm Quot}( R_{i-1} , R_i) 
=
\frac{1}{\lambda^2 \cdot {\rm sR}'_{d_{i-1}}\cdot T'_{d_{i-2} - 1}}\cdot {\rm Quot}(\lambda \cdot \cdot R'_{i-1} , \lambda \cdot R'_i)  =
 \frac{1}{\lambda^2}\cdot F'_{d_i - 1}.
$$
Therefore for $2 \le i \le s$ we have that
$$\left(F_{d_i - 1}(x) - F_{d_i - 1}(y)\right)\cdot R_i(x)\cdot R_i(y)
= \left(F'_{d_i - 1}(x) - F'_{d_i - 1}(y)\right)\cdot R'_i(x)\cdot R'_i(y).$$

Finally, using equation (\ref{aux_equation_bez}) and the inductive hypothesis, 
  since the length of the remainder sequence of $R$ and $-S$ is $s-1$, 
we have that 
$$
\begin{array}{rcl}
{\rm Bez}(P, R)  & = & \displaystyle  \frac{{F}_{r-1}(x) - {F}_{r-1}(y)}{x-y}\cdot  R_1(x)\cdot R_1(y) +   {\rm Bez} (R, -S) = \\[4mm]
& = & \displaystyle \frac{{F}_{r-1}(x) - {F}_{r-1}(y)}{x-y} \cdot R_1(x)\cdot R_1(y) +
\sum_{2 \le i \le s}\frac{F'_{d_i - 1}(x) - F'_{d_i - 1}(y)}{x-y}\cdot R'_i(x)\cdot R'_i(y) = \\[4mm]
 & = &  \displaystyle \sum_{1 \le i \le s}\frac{F_{d_i - 1}(x) - F_{d_i - 1}(y)}{x-y}\cdot R_i(x)\cdot R_i(y)
\end{array}
 $$
as we wanted to prove. 
\end{proof}

\begin{lemma}\label{lem:fact_herm_mat_through_bez} Following Notation \ref{not:subres_and_more} with $R
$ the 
remainder of $P'
\cdot Q
$ in the division by $P
$,  
$${\rm Her}(P;Q)= {\rm M}_{P;Q} \cdot {\rm HanB}_{P;R} \cdot {\rm M}_{P;Q}^{\rm t}.  $$
\end{lemma}

\begin{proof}{Proof.} The claim follows from Lemma \ref{lem:fact_herm_mat_through_bez_0} and the fact that 
$$
{\rm Her}(P;Q) = {\rm Bez}_{{\rm Hor}(P)}(P;R) = {\rm M}_{P;Q} \cdot {\rm Bez}_{{\cal R}}(P;R) \cdot {\rm M}_{P;Q}^{\rm t}
$$ 
(see \cite[Proposition 9.20 and Proposition 4.55]{BPRbook}).
\end{proof}

We introduce some more definitions to transform the preceeding block Hankel form into a diagonal form.

\begin{definition}
For $p \in {\mathbb N}_*$ and a variable $c$,  
we define the diagonal matrix ${\rm Di}_p
\in {\mathbb Q}[c]^{p \times p}$ as follows: 
\begin{itemize}
\item If $p$ is odd, ${\rm Di}_p
$ has $c$ 
in the first $\frac12(p-1)$ diagonal entries, $\frac{1}{2}c$ in the
next diagonal entry and 
$-c$ in the last $\frac12(p-1)$ diagonal entries.
\item If $p$ is even, ${\rm Di}_p
$ has $c$ 
in the first $\frac12p$ diagonal entries and 
$-c$ in last $\frac12p$ diagonal entries. 
\end{itemize}
We also define for $c=(c_{1},\dots,c_{p})$ 
the matrix ${\rm E}_p
 \in {\mathbb Q}[c]^{p \times p}$  as follows:
\begin{itemize}
\item ${\rm E}_1
 = \left(\begin{array}{c}
              2 
             \end{array}\right)$,

\item ${\rm E}_2
 = \left(\begin{array}{cc}
              c_2 & 0 \cr
	      \frac12 c_1 & c_2
	      \end{array}\right)\left(\begin{array}{cc}
              1 & 1 \cr
	      1 & -1
	      \end{array}\right)$,

\item For odd $p \ge 3$,  ${\rm E}_{p}
 =$ 
$$
\left(\begin{array}{c|ccc|c}
              c_{p} &  & 0 &  & 0\cr \hline
c_{p-1} &  &   &  &  \cr
\vdots &  & c_{p}\cdot \Id  &  & 0 \cr
c_{2} &  &   &  &  \cr \hline
\frac12c_{1} &  & 0 &   & c_{p} \cr
	            \end{array}\right) 
\left(\begin{array}{c|ccc|c}
              1 &  & 0 &  & 1 \cr \hline
 &  &   &  &  \cr
0 &  & \Id  &  & 0 \cr
 &  &   &  &  \cr \hline
1 &  & 0 &   & -1 \cr
	            \end{array}\right)
\left(\begin{array}{c|ccc|c}
               c_{p}^{\frac12(p-3)} &  & 0 &  & 0\cr \hline
 &  &   &  &  \cr
0 &  & {\rm E}_{p-2}(c')
  &  & 0 \cr
 &  &   &  &  \cr \hline
0 &  & 0 &   & c_{p}^{\frac12(p-3)} \cr
	            \end{array}\right)
$$
with $c' = (c_{3}, \dots, c_{p})$.

\item For even $p \ge 4$,  ${\rm E}_{p}
 =$ 
$$
\left(\begin{array}{c|ccc|c}
              c_{p} &  & 0 &  & 0\cr \hline
c_{p-1} &  &   &  &  \cr
\vdots &  & c_{p} \cdot \Id  &  & 0 \cr
c_{2} &  &   &  &  \cr \hline
\frac12c_{1} &  & 0 &   & c_{p} \cr
	            \end{array}\right) 
\left(\begin{array}{c|ccc|c}
              1 &  & 0 &  & 1\cr \hline
 &  &   &  &  \cr
0 &  & \Id  &  & 0 \cr
 &  &   &  &  \cr \hline
1 &  & 0 &   & -1 \cr
	            \end{array}\right)
\left(\begin{array}{c|ccc|c}
              c_{p}^{\frac12(p-2)} &  & 0 &  & 0\cr \hline
 &  &   &  &  \cr
0 &  & {\rm E}_{p-2}(c')  &  & 0 \cr
 &  &   &  &  \cr \hline
0 &  & 0 &   & c_{p}^{\frac12(p-2)} \cr
	            \end{array}\right)
$$
with $c' = (c_{3}, \dots, c_{p})$.
\end{itemize}

Finally, for $S
 =  \sum_{0 \le h \le p}c_h  \cdot  y^h \in {\mathbb Q}[c_0, \dots, c_p ][y]$, we
denote by
 ${\rm E}_p(S) \in {\mathbb Q}[c_1,\ldots,c_p]^{p \times p}$ the matrix 
 ${\rm E}_p(c_1, \dots, c_p)$.

\end{definition}

\begin{lemma}\label{lem:signat_block_tr_han}
\begin{itemize}
\item For odd $p \in \N_*$ the degree of the entries of the matrix ${\rm E}_p
$ is $\frac12(p-1)$, 
$\det({\rm E}_p
) = (-1)^{\frac12(p-1)}2^{\frac12(p+1)}c_{p}^{\frac12p(p-1)}$
and
$$
{\rm HanT}_p
 ={\rm E}_p
 \cdot  {\rm Di}_p \Big(\frac12c_{p}^{2-p}\Big)  \cdot  {\rm E}_p ^{\rm t}.
$$

\item For even $p\in \N_*$ the degree of the entries of the matrix ${\rm E}_p$ is $\frac12p$, 
$\det({\rm E}_p )= (-2)^{\frac12p}c_{p}^{\frac12p^2}$
and
$$
{\rm HanT}_p ={\rm E}_p
  \cdot {\rm Di}_p  \Big(\frac12c_{p}^{1-p}\Big) \cdot  {\rm E}_p^{\rm t}.
$$

\end{itemize}
\end{lemma}

\begin{proof}{Proof.} Easy to prove by induction on $p$. 
\end{proof}

We can prove now Theorem \ref{bezoutiansignature} (Hermite's Theory (2)).

\begin{proof}{Proof of Theorem \ref{bezoutiansignature}.} 
Following Notation \ref{not:subres_and_more}, 
by 
Lemmas  \ref{lem:fact_herm_mat_through_bez} 
and \ref{lem:signat_block_tr_han}, 
it is clear that 
\begin{eqnarray*}
 {\rm Rk}({\rm Her}(P;Q))&=& p - d_s, \\
 {\rm Si}({\rm Her}(P;Q))&=&\sum_{1 \le i \le s, \atop d_{i-1}-d_i\ {\rm odd}} 
\sign({\rm sR}_{d_{i-1}}\cdot T_{d_{i-1}-1} ).
\end{eqnarray*}
By the Structure Theorem for Subresultants (\cite[Theorem 8.30]{BPRbook}), for $1 \le i \le s$, 
$${\rm sR}_{d_{i}}  =  \varepsilon_{d_{i-1} - d_{i}}  \frac{T_{d_{i-1} - 1}^{d_{i-1} - d_{i}}}
{{\rm sR}_{d_{i-1}}^{d_{i-1} - d_{i} -1}}.$$
Therefore, for $1 \le i \le s$ such that $d_{i-1} - d_{i}$ is odd,  
$\sign(T_{d_{i-1}-1}) =  \varepsilon_{d_{i-1}-d_{i}} \sign({\rm sR}_{d_{i}})$. 
The conclusion follows using Proposition \ref{prop:conn_herm_subr}.
\end{proof}

Before proving a related weak inference in Theorem \ref{thm:her_through_bez} (Hermite's Theory (2) as a weak existence), we give some auxiliary definitions.

\begin{definition} \label{def:bezfact}

Let $p, q \in \mathbb{N}$, $p \ge 1$.
Let
$c = (c_0, \dots, c_{p-1})$ be variables representing the coefficients of $P$,
$c' = (c'_0, \dots, c'_q)$ be variables represeting the coefficients of $Q$.
In the following definitions, we always consider $y$ as the main variable. 

\begin{itemize}
\item $P
= y^p + \sum_{0 \le h \le p-1}c_h \cdot y^h \in \K[c][ y]$,

\item $Q
 = \sum_{0 \le h \le q}c'_h \cdot y^h \in \K[c'][ y]$,

\item $R
\in \K[c, c'][ y]$ is the remainder of
$P' \cdot Q$ in the division by $P$,

\item for $0 \le j \le p$,  ${\rm sResP}_{j}
\in \K[c, c'][ y]$ is the
$j$-th subresultant polynomial of $P
$ and $R
$,

\item for $0 \le j \le p$, ${\rm sR}_{j}
\in \K[c,c']$
is the
$j$-th signed subresultant coefficient of $P
$ and $R
$,

\item for $-1 \le j \le p$,  ${\rm sResU}_{j}
\in \K[c, c'][ y]$ and ${\rm sResV}_{j}
\in \K[c, c'][ y]$
are the
$j$-th subresultant cofactors of $P
$ and $R
$.
\end{itemize}

Let now
$\tau \in \{-1, 0, 1\}^{\{0, \dots, p-1\}}$ 
be  a sign condition, 
$d(\tau) = (d_0, \dots, d_s)$ and $d_{-1} = p+1$. 

\begin{itemize}

\item for $0 \le i \le s$,
$T_{ d_{i-1}-1}^\tau
 \in \K[c, c']$ is the coefficient of degree $d_i$ in 
${\rm sResP}_{d_{i-1}-1}
$,

\item for $1 \le i \le s$, $R_{ i}^\tau
 \in \K[c, c'][y]$
is the remainder of  ${\rm sResP}_{d_{i-1} - 1}
$ in the division by $y^{d_i + 1}$,

\item ${\rm M}_{P;Q}^\tau
\in \K[c, c']^{p \times p}$ is the matrix of  
$$\{y^{d_0-d_1-1}\cdot  R_{1}^\tau,\ldots,R_{1}^\tau,\ldots, y^{d_{i-1}-d_{i}-1} \cdot R_{i}^\tau,
\ldots,R_{i}^\tau,\ldots,
 y^{d_{s-1}-d_{s}-1} \cdot R_{s}^\tau,\ldots, R_{s}^\tau,y^{d_{s} -1}, \dots, 
1\}
$$
in the Horner basis of $P$,  $\{y^{p-1} + \sum_{0 \le h \le p-2}c_{h+1} \cdot y^h, \dots, 1\}$,

\item for $1 \le i \le s$,  $\tilde F_{ d_{i}-1}^\tau
=
\sum_j \tilde F_{ d_{i}-1,j}^\tau
\cdot y^j  \in \K[c, c'][ y]$
is  
$${\rm sResU}_{d_{i-2}-1}
\cdot {\rm sResV}_{d_{i}-1}
- {\rm sResU}_{d_{i}-1}
 \cdot {\rm sResV}_{d_{i-2}-1}.
$$ 

\end{itemize}

In order  to avoid dealing with rational functions, we consider variables 
$\ell = (\ell_1, \dots, \ell_s)$ representing the inverses of 
$({\rm sR}_{d_{i}}
)_{1 \le i \le s}$ and
$\ell' = (\ell'_1, \dots, \ell'_s)$  variables representing the inverses of 
$(T_{d_{i-1}-1}^\tau
)_{1 \le i \le s}$.  
We additionally define $\ell_0 = \ell'_0 = 1$. 
We also consider variables
$a=(a_i)_{1 \le i \le s, \, d_{i-1}-d_i \hbox{\scriptsize even}}$
and 
$b=(b_i)_{1 \le i \le s, \, d_{i-1}-d_i \hbox{\scriptsize even}}$
which only purpose 
is to fix the sign of the diagonal elements in the even size blocks in the diagonal matrix  
${\rm Di}_{P;Q}^\tau$ defined below.

\begin{itemize}
\item For $1 \le i \le s$,  
$F_{ d_{i}-1}^\tau
 \in \K[c, c'][\ell,\ell'][ y]$
is 
$$\ell_{i-1} \cdot \ell'_i\cdot y^{d_{i-1} - d_i} + 
\ell_i\cdot \ell_{i-1}\cdot \ell'_i\cdot \ell'_{i-1}\Big(
\sum_{0 \le j \le d_{i-1}-d_i-1} \tilde F_{ d_{i}-1,j}^\tau
\cdot y^j
\Big),
$$
\item ${\rm E}_{ P; Q}^\tau
 \in \K[c, c'][ \ell, \ell']^{p \times p}$ is the block diagonal matrix  
composed by $s$ or $s + 1$ blocks according to $d_s = 0$ or $d_{s} >0$; 
for $1 \le i \le s$ the 
$i$-th  block is the matrix 
${\rm E}_{d_{i-1}-d_i}(F_{ d_i-1}^\tau)$, if $d_{s} >0$ 
the last block is the identity matrix of size $d_s$.
\item ${{\rm E}'}^\tau
 \in \K
 [a,b]^{p \times p}$ is the block diagonal matrix  
composed by $s$ or $s + 1$ blocks according to $d_s = 0$ or $d_{s} >0$; 
for $1 \le i \le s$, the 
$i$-th  block is the identity matrix of size 
$d_{i-1}-d_i$ if $d_{i-1}-d_i$ is odd and the matrix
$$
\left(  
\begin{array}{cccccc}
a_i &  0 & \dots & \dots  & 0 & b_i \cr
0 & \ddots & & & \iddots & 0\cr
\vdots & & a_i & b_i & & \vdots \cr
\vdots & & -b_i & a_i & & \vdots  \cr
0 & \iddots & & & \ddots \cr
-b_i &  0 & \dots & \dots  & 0 & a_i \cr
\end{array}
\right)
$$ of size $d_{i-1}-d_i$
if $d_{i-1}-d_i$ is even, if $d_{s} >0$ 
the last block is the identity matrix of size $d_s$.

\item ${\rm B}_{P;Q}^\tau
 = {\rm M}_{ P;Q}^\tau
\cdot {\rm E}_{ P;Q}^\tau
\cdot {{\rm E}'}^\tau
\in 
\K[c, c'] [\ell, \ell',a,b]^{p  \times p}$.

\item ${\rm Di}_{P;Q}^\tau
 \in \K[c, c'][\ell, \ell']^{p  \times p}$
is the diagonal matrix defined by blocks, 
composed by $s$ or $s + 1$ blocks according to $d_s = 0$ or $d_{s} >0$; 
for $1 \le i \le s$, the 
$i$-th  block is the diagonal matrix 
$$
{\rm Di}_{d_{i-1}-d_i}\Big(\frac12\varepsilon_{d_{i-1}-d_i}
\ell_{i-1}^2 \cdot {\ell'_i}^2 \cdot {\rm sR}_{d_{i-1}}^{2(d_{i-1} - d_i)-1}
\cdot {\rm sR}_{d_i}
\Big)
$$
if $d_{i-1}-d_i$ is odd and the matrix 
$$
{\rm Di}_{d_{i-1}-d_i}\Big(\frac12\Big)
$$ 
if $d_{i-1}-d_i$ is even, if $d_{s} >0$ 
the last block is the zero block of size $d_s$.

\end{itemize}

\end{definition}

\begin{remark}\label{rem:degree_bound_for_hmin} Following Definition \ref{def:subres} and Definition \ref{def:bezfact} and 
taking into account Remark \ref{rem2:form_herm_matrix},  it can be proved that:

\begin{itemize}

\item $\deg_{c}{\rm Her}(P; Q)
\le q + 2p - 2$,  $\deg_{c'} {\rm Her}(P; Q)
\le 1$, then $\deg_{(c, c')} {\rm Her}(P; Q)
\le q + 2p - 1$,

\item $\deg_{(c, c')} R
 \le q+2$,

\item for $0 \le j \le p-1$, $\deg_{(c,c')}{\rm sR}_{j}
\le (p - j)(q+3) - 1$, 
$\deg_{(c,c')}{\rm sR}_p
 = 0$, 

\item for $1 \le i \le s$, $\deg_{(c,c')}R_{i}^\tau
\le (p - d_{i-1} + 1)(q+3) - 1$,

\item $\deg_{(c,c')} {\rm M}_{P;Q}^\tau
 \le p(q+3) - 1$, 

\item for $1 \le i \le s$, $\deg_{(c,c', \ell, \ell')} F_{d_i-1}^\tau
\le 
(2p -  d_{i} - d_{i-2} + 1)(q+3) +2 \le  
(2p -  1)(q+3) +2$,

\item $\deg_{(c,c', \ell, \ell')}{\rm E}_{P;Q}^\tau
\le 
\frac12p( (2p -  1)(q+3) +2 )$, 

\item $\deg_{(a,b)} {{\rm E}'}^\tau
 \le 1$, 

\item 
$\deg_{(c,c', \ell, \ell')}{{\rm Di}}_{P;Q}^\tau
 \le 
4 + 2p(p(q+3)-1)$. 

\end{itemize}
We will use these degree bounds in Lemmas \ref{lem:nulls_ident_structure_theorem} and 
\ref{lem:nulls_ident_herm_bez}; but, in fact, a separate degree analysis on the set of variables $(c, c')$ and each variable $\ell_i$ and $\ell'_i$,
which can be easily done, 
will be needed in Theorem \ref{hermitetrsubresw} (Hermite's Theory as an incompatibility).
 
\end{remark}

We prove two auxiliary algebraic identities, using Effective Nullstellensatz (\cite[Theorem 1.3]{Jel}).

\begin{lemma} \label{lem:nulls_ident_structure_theorem} 
Let $p, q \in \mathbb{N}$, $p \ge 1$, 
$\tau \in \{-1, 0, 1\}^{\{0, \dots, p-1\}}$ 
be
a sign condition, 
$d(\tau) = (d_0, \dots, d_s)$, 
$c = (c_0, \dots, c_{p-1})$, $c' = (c'_0, \dots, c'_q)$,
$1 \le i \le s$ and $e = (p(q+3) - 1)^{d_{i-1} - d_i}$.
Following Definition \ref{def:bezfact}, 
there is an identity in $\K[c, c']$
$$
({\rm sR}_{d_{i-1}}
\cdot {\rm sR}_{d_i}
)^{e} = 
\sum_{d_i +1  \le j \le d_{i-1} - 1}{\rm sR}_{j}
 \cdot W_j
  +  T_{d_{i-1}-1}
  \cdot W
  $$ 
such that all the terms have degree in $(c, c')$ bounded by $2ep(q+3)$.  
\end{lemma}

\begin{proof}{Proof.} 
We denote by $\overline{\K}$ the algebraic closure of $\K$. 
By the Structure Theorem for Subresultants (\cite[Theorem 8.30]{BPRbook}), for any $\gamma \in \overline{\K}^p, \gamma' \in \overline{\K}^{q+1},$
such that 
$$
{\rm sR}_{d_{i-1}}(\gamma, \gamma') \ne 0, \ \bigwedge_{d_{i-1} < j < d_i} {\rm sR}_j(\gamma, \gamma') = 0, \ 
{\rm sR}_{d_{i}}(\gamma, \gamma') \ne 0
$$ 
we have 
$$
T_{d_{i-1}-1}(\gamma, \gamma') \ne 0.
$$
The claim follows from a similar use of \cite[Theorem 1.3]{Jel} as in the proof of Lemma \ref{nullstellensatzLaplace}.
\end{proof}

\begin{lemma} \label{lem:nulls_ident_herm_bez} 
Let $p, q \in \mathbb{N}$, $p \ge 1$, 
$\tau \in \{-1, 0, 1\}^{\{0, \dots, p-1\}}$
be
a sign condition,
$d(\tau) = (d_0, \dots, d_s)$, 
$c = (c_0, \dots, c_{p-1})$, $c' = (c'_0, \dots, c'_q)$, 
$\ell = (\ell_1, \dots, \ell_s)$, 
$\ell' = (\ell'_1, \dots, \ell'_s)$, 
$a=(a_i)_{1 \le i \le s, \, d_{i-1}-d_i \hbox{\scriptsize even}}$,
$b=(b_i)_{1 \le i \le s, \, d_{i-1}-d_i \hbox{\scriptsize even}}$
and $e = 2^{2p}p^{4p}(q+3)^{3p}$.
Following Definition \ref{def:bezfact}, 
for $1 \le j_1, j_2 \le p$, there is an identity in 
$\K[c, c']
[\ell,\ell',a,b]$
\begin{eqnarray*}
&&\Big( {\rm Her}(P; Q)_{j_1, j_2}
- 
\big( {\rm B}_{P;Q}^\tau \cdot{\rm Di}_{P;Q}^\tau  \cdot {{\rm B}_{P;Q}^\tau}^{\rm t}\big)_{j_1, j_2}
(\ell,\ell',a,b) 
\Big)^e = 
\\
&=&
\sum_{0 \le j \le p-1, \atop \tau(j) = 0}{\rm sR}_{j}
 \cdot W_j
(\ell, \ell', a, b) 
 + 
\sum_{1 \le i \le s} (\ell_i \cdot {\rm sR}_{d_i} 
- 1) \cdot W'_i
(\ell, \ell', a, b) 
+
\\
&+&
\sum_{1 \le i \le s} (\ell'_i \cdot T_{d_{i-1}-1}
- 1) \cdot W''_i
(\ell, \ell', a, b) 
+
\\
&+&
\sum_{1 \le i \le s, \atop d_{i-1} - d_i \hbox{\scriptsize even}} 
(a_i^2 - b_i^2 - ({\rm sR}_{d_{i-1}}
\cdot T_{d_{i-1}-1}
)^{d_{i-1} - d_i -1}  ) \cdot W'''_i
(\ell, \ell', a, b) 
+
\\
&+&
\sum_{1 \le i \le s, \atop d_{i-1} - d_i \hbox{\scriptsize even}} a_i \cdot b_i \cdot W''''_i
( \ell, \ell', a, b) 
\end{eqnarray*}
such that all the terms have degree in $(c, c', \ell, \ell', a, b)$ bounded by $e( 4p^2(q+3) + p(q+3) + 5)$.  
\end{lemma}

\begin{proof}{Proof.} 
We denote by $\overline{\K}$ the algebraic closure of $\K$. 
By the Structure Theorem for Subresultants (\cite[Theorem 8.30]{BPRbook}), Lemma \ref{lem:fact_herm_mat_through_bez}
and Lemma \ref{lem:signat_block_tr_han}, 
for any $\gamma \in \overline{\K}^p, \gamma' \in \overline{\K}^{q+1},$
$\lambda, \lambda' \in \overline{\K}^s$,
$\alpha, \beta \in 
\overline{\K}^{\# \{1 \le i \le s, \, d_{i-1}  - d_i \hbox{\scriptsize even} \}}$
such that 
$$
\bigwedge_{0\le j \le p-1} {\rm inv}({\rm sR}_{j}(\gamma, \gamma')) = \tau(j)^2, \  
\bigwedge_{1 \le i \le s} \lambda_i \cdot {\rm sR}_{d_i}(\gamma, \gamma') = 1, \ 
\bigwedge_{1 \le i \le s} \lambda'_i \cdot T_{d_{i-1}-1}(\gamma, \gamma') =  1,$$
$$
\bigwedge_{1 \le i \le s, \atop d_{i-1} - d_i \hbox{\scriptsize even}} 
\alpha_i^2 - \beta_i^2 = ({\rm sR}_{d_{i-1}}(\gamma,\gamma')\cdot T_{d_{i-1}-1}(\gamma,\gamma'))^{d_{i-1} - d_i -1},  \ 
\bigwedge_{1 \le i \le s, \atop d_{i-1} - d_i \hbox{\scriptsize even}} \alpha_i \cdot \beta_i =0,
$$ 
we have ${\rm Her}(P; Q)(\gamma, \gamma') = 
{\rm B}_{ P;Q}^\tau \cdot {\rm Di}_{P;Q}^\tau \cdot {{\rm B}_{P;Q}^\tau}^{\rm t}(\gamma, \gamma', \lambda, \lambda', \alpha, \beta)
\in \overline{\K}^{p \times p}$. Moreover, for $1 \le i \le s$, 
the condition 
$\lambda_i \cdot {\rm sR}_{d_i}(\gamma, \gamma') = 1$ clearly implies 
${\rm inv}({\rm sR}_{d_i}(\gamma, \gamma')) = 1$. 
The claim follows from a similar use of \cite[Theorem 1.3]{Jel} as in the proof of Lemma \ref{nullstellensatzLaplace}.
\end{proof}

From now on, we make a slight abuse of notation, denoting by
${\rm B}_{P;Q}^\tau(\ell, \ell', z)$ the matrix
${\rm B}_{P;Q}^\tau(\ell, \ell', a,b)$ 
where 
$z=a+ib$
is a complex variable,

We prove now the following related weak inference.

\begin{theorem}[Hermite's Theory (2) as a weak existence]\label{thm:her_through_bez}
Let $p \ge 1$, $P
 = y^p + \sum_{0 \le h \le p-1}C_h \cdot y^h \in \K[u][ y]$, $Q
  = 
\sum_{0 \le h \le q}D_h \cdot y^h \in \K[u][ y]$, $\tau \in \{-1, 0, 1\}^{\{0, \dots, p-1\}}$ 
be
a sign condition,  $d(\tau) = (d_0, \dots, d_s)$,
and $d'_i={d_{i-1}} - d_{i}$ for $i = 1, \dots, s$.  Then  
$$
\bigwedge_{0\le i \le p-1}{\rm sign}({\rm HMi}_i(P;Q)) = \tau(i) \ \ \ \vdash
$$
$$
\vdash \ \ \ \exists (\ell, \ell', z)  \ [\; {\rm Her}(P;Q) \equiv 
{\rm B}_{P;Q}^\tau(\ell, \ell', z)  \cdot  {\rm Di}_{P;Q}^\tau(\ell, \ell') \cdot { {\rm B}_{P;Q}^\tau}^{\rm t}(\ell, \ell', z),  \  
\det({\rm B}_{P;Q}^\tau(\ell, \ell', z)) \ne 0, 
$$
$$
\bigwedge_{1 \le i \le s, \atop 
d'_i
\hbox{\scriptsize odd}} 
{\rm sign}(\ell_{i-1}^2 \cdot {\ell_i'}^2 \cdot {\rm HMi}_{d_{i-1}}(P; Q)^{2
d'_i-1} \cdot {\rm HMi}_{d_i}(P; Q)) 
= \tau_{d_{i-1}}\tau_{d_i}\;] 
$$ 
where $\ell = (\ell_1, \dots, \ell_{s}), \ell' = (\ell'_1, \dots, \ell'_{s}), 
z = (z_i)_{1\le i \le s, \,
d'_i \hbox{\scriptsize even}}$.

Suppose we have an initial incompatibility in 
$\K[v][ \ell, \ell', a, b]$ where $v \supset u$ and $(\ell, \ell', a, b)$ are disjoint from $v$, 
with monoid part 
$$
S\cdot \det({\rm B}_{P;Q}^\tau(\ell, \ell', z))^{2e}
\cdot
\prod_{1 \le i \le s, \atop 
d'_i
\hbox{\scriptsize odd}} 
(\ell_{i-1}^2 \cdot {\ell_i'}^2 \cdot {\rm HMi}_{d_{i-1}}(P; Q)^{2
d'_i-1} \cdot {\rm HMi}_{d_i}(P; Q))^{2e_i}
$$
with $e \in \N_*$, $e_i \le e' \in \N_*$,
degree in $w \subset v$ bounded by $\delta_w$, 
degree in $\ell_i$ bounded by an even number $\delta_\ell$, 
degree in $\ell'_i$ bounded by an even number $\delta_{\ell'}$ 
and
degree in $(a_i, b_i)$ bounded by $\delta_z$.
Then the final incompatibility has monoid part 
$$
S^{f} \cdot 
\prod_{1 \le i \le s}{\rm HMi}_{d_{i}}(P;Q)^{2f_i}
$$
with $f \le 2^{3p}p^{4p+2}(q+3)^{3p}$, 
$$
f_i \le 2^{3p-1}p^{4p+2}(q+3)^{3p}(\delta_\ell + p^p(q+3)^p\delta_{\ell'} 
+ 10p^{p+2}(q+3)^{p+1}e + 4pe)
$$ 
and degree in $w$ bounded by 
$$
2^{3p}p^{4p+2}(q+3)^{3p}
\cdot
$$
$$
\cdot \Big( \delta_w 
+
\Big(
p^2(q+3)\delta_\ell + 3p^{p+1}(q+3)^{p+1} \delta_{\ell'} + 4p^2(q+3)\delta_z + 31p^{p+3}(q+3)^{p+2}e
\Big)
\max\{\deg_w P, \deg_w Q\}\Big).
$$ 
\end{theorem}

Note that in the weak inference in Theorem \ref{thm:her_through_bez}, the elements
$\ell_{i-1}^2 \cdot {\ell_i'}^2 \cdot {\rm HMi}_{d_{i-1}}(P; Q)^{2
d'_i-1} \cdot {\rm HMi}_{d_i}(P; Q)$ for $1 \le i \le s$, $d'_i$ odd, are, up to scalars, the only non-constant terms in  
the diagonal matrix ${\rm Di}_{P;Q}^\tau(\ell, \ell')$.

\begin{proof}{Proof.} 
Consider the initial incompatibility 
\begin{equation}\label{inc:init_them_fact_bezout}
\begin{array}{c}
\Big\downarrow \
{\rm Her}(P;Q) \equiv 
{\rm B}_{P;Q}^\tau(\ell, \ell', z)  \cdot {\rm Di}_{P;Q}^\tau(\ell,\ell') \cdot {{\rm B}_{P;Q}^\tau}^{\rm t}(\ell, \ell', z),  \  
\det({\rm B}_{P;Q}^\tau(\ell, \ell', z)) \ne 0, \\[3mm] 
\displaystyle{\bigwedge_{1 \le i \le s, \atop 
d'_i \hbox{\scriptsize odd}} 
{\rm sign}(\ell_{i-1}^2 \cdot {\ell_i'}^2 \cdot {\rm HMi}_{d_{i-1}}(P; Q)^{2
d'_i-1} \cdot {\rm HMi}_{d_i}(P; Q)) 
= \tau(d_{i-1})\tau({d_i})}, \ \cH 
\ \Big\downarrow_{\K[v][ \ell, \ell', a, b]}
\end{array}
\end{equation}
where $\cH$ is a system of sign conditions in $\K[v]$. 
By Proposition \ref{prop:conn_herm_subr},
for $0 \le j \le p$, ${\rm HMi}_{j}(P; Q) = {\rm sR}_j(C_0, \dots, C_{p-1}, D_0, \dots, D_q)$. 

Following Lemma \ref{lem:signat_block_tr_han}, $\det ({\rm B}_{P; Q}^\tau(\ell, \ell', z))$
is equal to
$$
\prod_{1 \le i \le s} T_{d_{i-1}-1}^{
d'_i} 
\cdot 
\prod_{1 \le i \le s, \atop 
d'_i \hbox{\scriptsize odd}} 
(-1)^{\frac12(
d'_i- 1)}2^{\frac12(
d'_i+ 1)}
(\ell_{i-1} \cdot \ell'_i)^{\frac12
d'_i (
d'_i-1)} 
\cdot
\prod_{1 \le i \le s, \atop 
d'_i\hbox{\scriptsize even}} 
(-2)^{\frac12
d'_i}
(\ell_{i-1} \cdot \ell'_i)^{\frac12
{d'_i}^2} \cdot
(a_i^2 + b_i^2)^{\frac12
d'_i}.
$$
Then we apply to (\ref{inc:init_them_fact_bezout}) the weak inference 
$$
\bigwedge_{1 \le i \le s} T_{d_{i-1}-1} \ne 0, \ 
\bigwedge_{1 \le i \le s} \ell_i \ne 0, \ 
\bigwedge_{1 \le i \le s} \ell'_i \ne 0, \ 
\bigwedge_{1 \le i \le s, \atop 
d'_i \hbox{\scriptsize even}} z_i \ne 0  
\ \ \ \vdash 
\ \ \ \det( {\rm B}_{P; Q}^\tau(\ell, \ell', z)) \ne 0.
$$
By Lemma  \ref{lemma_basic_sign_rule_1} (item \ref{lemma_basic_sign_rule:5})
we obtain an incompatibility 
\begin{equation}\label{inc:aux_them_fact_bezout_1}
\begin{array}{c}
\Big\downarrow \
{\rm Her}(P;Q) \equiv 
{\rm B}_{P;Q}^\tau(\ell, \ell', z) \cdot  {\rm Di}_{P;Q}^\tau(\ell, \ell') \cdot {{\rm B}_{P;Q}^\tau}^{\rm t}(\ell, \ell', z),  
 \\[6mm] 
\displaystyle{\bigwedge_{1 \le i \le s} T_{d_{i-1}-1} \ne 0, \ 
\bigwedge_{1 \le i \le s} \ell_i \ne 0, \ 
\bigwedge_{1 \le i \le s} \ell'_i \ne 0, \ 
\bigwedge_{1 \le i \le s, \atop 
d'_i \hbox{\scriptsize even}} z_i \ne 0,}
\\[6mm]
\displaystyle{\bigwedge_{1 \le i \le s, \atop 
d'_i \hbox{\scriptsize odd}} 
{\rm sign}(\ell_{i-1}^2 \cdot {\ell_i'}^2 \cdot {\rm sR}_{d_{i-1}}^{2
d'_i -1} \cdot {\rm sR}_{d_i}) 
= \tau(d_{i-1})\tau(d_i)}, \ \cH 
\ \Big\downarrow_{\K[v][ \ell, \ell', a, b]}
\end{array}
\end{equation}
with  monoid part
$$
S
\cdot
\prod_{1 \le i \le s} T_{d_{i-1}-1}^{2
d'_i e} 
\cdot
\prod_{1 \le i \le s, \atop 
d'_i \hbox{\scriptsize odd}} 
(\ell_{i-1} \cdot \ell'_i)^{
d'_i (
d'_i -1)e + 4e_i} \cdot 
({\rm sR}_{d_{i-1}}^{2
d'_i-1} \cdot {\rm sR}_{d_i})^{2e_i} 
\cdot
\prod_{1 \le i \le s, \atop 
d'_i \hbox{\scriptsize even}} 
(\ell_{i-1} \cdot \ell'_i)^{
{d'_i}^2e} \cdot (a_i^2 + b_i^2)^{
{d'_i} e} 
$$
and the same 
degree bounds. 

Let $\tilde e = 2^{2p}p^{4p}(q+3)^{3p}$. 
We pass in (\ref{inc:aux_them_fact_bezout_1})
all the terms in the ideal generated by 
$\{ ({\rm Her}(P;Q)  -
{\rm B}_{P;Q}^\tau \cdot {\rm Di}_{P;Q}^\tau \cdot {{\rm B}_{P;Q}^\tau}^{\rm t})_{j_1, j_2}
\ | \ 1 \le j_1 \le j_2 \le p \}$ 
to the right hand side, 
we raise both sides to the $(\frac 12 p(p+1)\tilde e)$-th power
and we pass all the terms back to the left hand side. 
It is easy to see that what we obtain
is an incompatibility  
$$
\begin{array}{c}
\displaystyle{\Big\downarrow \
\bigwedge_{1 \le j_1 \le j_2 \le p}
\Big(
{\rm Her}(P;Q)_{j_1, j_2} - 
({\rm B}_{ P;Q}^\tau(\ell, \ell', z) \cdot {\rm Di}_{P;Q}^\tau(\ell, \ell') \cdot  {{\rm B}_{P;Q}^\tau}^{\rm t}(\ell, \ell', z))_{j_1, j_2}
\Big)^{\tilde e} = 0},  
 \\[6mm] 
\displaystyle{
\bigwedge_{1 \le i \le s} T_{d_{i-1}-1} \ne 0, \ 
\bigwedge_{1 \le i \le s} \ell_i \ne 0, \ 
\bigwedge_{1 \le i \le s} \ell'_i \ne 0, \ 
\bigwedge_{1 \le i \le s, \atop 
d'_i \hbox{\scriptsize even}} z_i \ne 0,}
\\[6mm]
\displaystyle{\bigwedge_{1 \le i \le s, \atop 
d'_i \hbox{\scriptsize odd}} 
{\rm sign}(\ell_{i-1}^2 \cdot {\ell_i'}^2 \cdot {\rm sR}_{d_{i-1}}^{2
d'_i -1} \cdot {\rm sR}_{d_i}) 
= \tau(d_{i-1})\tau(d_i)}, \ \cH 
\ \Big\downarrow_{\K[v][ \ell, \ell', a, b]}.
\end{array}
$$
Following Lemma \ref{lem:nulls_ident_herm_bez} and applying Lemma \ref{lem:comb_lin_zero_zero}, 
we obtain an incompatibility
\begin{equation}\label{inc:aux_them_fact_bezout_3}
\begin{array}{c}
\Big\downarrow \displaystyle{
\bigwedge_{1 \le j \le p, \, \tau(j) = 0}{\rm sR}_j = 0, \ 
\bigwedge_{1 \le i \le s} \ell_i \cdot {\rm sR}_{d_i} = 1, \ 
\bigwedge_{1 \le i \le s} \ell'_i \cdot T_{d_{i-1}-1} = 1, }  \\[6mm]
\displaystyle{ \bigwedge_{1 \le i \le s, \atop 
d'_i \hbox{\scriptsize even}} z_i^2 = 
({\rm sR}_{d_{i-1}} \cdot T_{d_{i-1}-1})^{
d'_i - 1}
}, \ 
\displaystyle{\bigwedge_{1 \le i \le s} T_{d_{i-1}-1} \ne 0, \ 
\bigwedge_{1 \le i \le s} \ell_i \ne 0, \  
\bigwedge_{1 \le i \le s} \ell'_i \ne 0}, \\[6mm]
\displaystyle{\bigwedge_{1 \le i \le s, \atop 
d'_i \hbox{\scriptsize even}} z_i \ne 0,} \
\displaystyle{\bigwedge_{1 \le i \le s, \atop 
d'_i \hbox{\scriptsize odd}} 
{\rm sign}(\ell_{i-1}^2 \cdot {\ell_i'}^2 \cdot {\rm sR}_{d_{i-1}}^{2
d'_i -1} \cdot {\rm sR}_{d_i}) 
= \tau(d_{i-1})\tau(d_i)}, \ \cH 
\ \Big\downarrow_{\K[v][ \ell, \ell', a, b]}
\end{array}
\end{equation}
with monoid part 
$$
S^{\frac12 p(p+1)\tilde e}
\cdot
\prod_{1 \le i \le s}T_{d_{i-1}-1}^{
d'_i p(p+1)e\tilde e }
\cdot
\prod_{1 \le i \le s, \atop 
d'_i \hbox{\scriptsize odd}} 
(\ell_{i-1} \cdot \ell'_i)^{\frac12p(p+1)(
d'_i (
d'_i -1)e + 4e_i)\tilde e} \cdot 
({\rm sR}_{d_{i-1}}^{2
d'_i -1} \cdot {\rm sR}_{d_i})^{p(p+1)e_i\tilde e} \\
$$
$$
\cdot
\prod_{1 \le i \le s, \atop 
d'_i \hbox{\scriptsize even}} 
(\ell_{i-1} \cdot \ell'_i)^{ \frac12 p(p+1) 
{d'_i}^2e\tilde e} \cdot \prod_{1 \le i \le s, \atop 
d'_i \hbox{\scriptsize even}} 
(a_i^2 + b_i^2)^{\frac 12 p(p+1)
d'_i e \tilde e} := S_1 \cdot 
\prod_{1 \le i \le s, \atop 
d'_i \hbox{\scriptsize even}} 
(a_i^2 + b_i^2)^{\frac 12 p(p+1)
d'_i e \tilde e},
$$
degree in $w$ bounded by 
$$
\delta'_w := \tilde e \Big( 
\frac 12 p(p+1)
\delta_w  + (4p^2(q+3) + p(q+3) + 5 )\max\{\deg_w P, \deg_w Q \} \Big),  
$$ 
degree in $\ell_i$
bounded by   
$$\delta'_\ell := \tilde e \Big(\frac 12 p(p+1)\delta_\ell + 4p^2(q+3) + p(q+3) + 5\Big)$$
degree in $\ell'_i$
bounded by   
$$\delta'_{\ell'} := \tilde e \Big(\frac 12 p(p+1)\delta_{\ell'} + 4p^2(q+3) + p(q+3) + 5\Big)$$
and degree in $(a_i, b_i)$ bounded by 
$$
\delta'_z := \tilde e \Big(\frac 12 p(p+1)\delta_z + 4p^2(q+3) + p(q+3) + 5\Big).
$$

Then we successively apply to (\ref{inc:aux_them_fact_bezout_3}) for 
$1 \le i \le s$ with $
d'_i$ odd the weak inference
$$
{\rm sign}({\rm sR}_{d_i}) = \tau(i), \ {\rm sign}({\rm sR}_{d_{i-1}}) = \tau(i-1), \ 
\ell_{i-1}^2 > 0, \ {\ell'}_i^2 > 0
\ \ \ \vdash $$
$$ \vdash
\ \ \ 
{\rm sign}({\ell}_{i-1}^2 \cdot {\ell'}_i^2 \cdot {\rm sR}_{d_{i-1}}^{2
d'_i -1} \cdot {\rm sR}_{d_{i-1}}) = \tau(i)\tau(i-1).
$$
By Lemma  \ref{lemma_basic_sign_rule_1} (item \ref{lemma_basic_sign_rule:7}) we obtain 
an incompatibility with the same monoid part and degree bounds. 

Then we successively apply for 
$1 \le i \le s$ with $
d'_i$ even the weak inferences
$$
\begin{array}{rcl}
({\rm sR}_{d_{i-1}} \cdot T_{d_{i-1}-1})^{
d'_i-1} \ne 0 &  \ \, \vdash \, \
&
\exists z_i \ [\; z_i \ne 0, \ z_i^2 = ({\rm sR}_{d_{i-1}} \cdot T_{d_{i-1}-1})^{
d'_i - 1} \;], \\[3mm]
{\rm sR}_{d_{i-1}} \ne 0, \ T_{d_{i-1}-1} \ne 0
& \vdash &
({\rm sR}_{d_{i-1}} \cdot T_{d_{i-1}-1})^{d'_i-1} \ne 0.
\end{array}
$$
Let $\{1 \le i \le s \ | \ 
d'_i \hbox { even} \} = 
\{i_1 < \dots < i_{s'}\}$ and $i_0 = 0$.
Using Lemmas  \ref{lemma_square_root_complex_ne} and \ref{lemma_basic_sign_rule_1} 
(item \ref{lemma_basic_sign_rule:5}), it can be proved
by induction in $r$ that, for $0 \le r \le s'$, after the application of the weak inferences
corresponding to index $r$, we obtain an incompatibility with monoid part
$$
S_1^{4^r} \cdot
\prod_{r+1 \le j \le s'} 
(a_{i_j}^2 + b_{i_j}^2)^{\frac 12 4^rp(p+1)
d'_{i_j} e \tilde e}
\cdot
\prod_{1 \le j \le r} 
({\rm sR}_{d_{i_j-1}} \cdot T_{d_{i_j-1}-1})^{4^{r-j+1}(\frac 12 
\cdot4^{j-1}   p(p+1)
d'_{i_j} e \tilde e
+ 1) 
(d'_{i_j}-1)},
$$
degree in $w$ bounded by 
$$
4^{r} \Big(\delta'_w
 + \big(10   + 
3p^2(p+1)e \tilde e + 4\delta'_z \big)p(q+3)(p-d_{i_r})\max \{\deg_w P, \deg_w Q\}\Big), 
$$ 
degree in $\ell_i$ bounded by 
$4^{r}\delta'_\ell$
and degree in $\ell'_i$ bounded by 
$4^{r}\delta'_{\ell'}$
and degree in $(a_{i_j}, b_{i_j})$ bounded by 
$4^{r}\delta'_z$ for $r+1 \le j \le s'$.  At the end we obtain an incompatibility 
with monoid part 
$$
S^{\frac12  4^{s'} p(p+1)\tilde e }\cdot
\prod_{1 \le i \le s} \ell_i^{2g_i} \cdot {\ell'_i}^{2g'_i} \cdot {\rm sR}_{d_i}^{2h_i} \cdot T_{d_{i-1}-1}^{2h'_i}
$$
with 
$$
h_i \le  \frac12 4^{s'}\Big( 2p^2(p+1)e' + \frac12 p^3(p+1)e\Big)   \tilde e, \quad \quad 
h'_i \le 4^{s'-1}p^2(p+1)(p+2)e\tilde e, 
$$
degree in $w$ bounded by 
$$ 
\delta''_w := 4^{s'} \Big(\delta'_w
 + \big(10   + 
3p^2(p+1)e \tilde e + 4\delta'_z \big)p^2(q+3)\max \{\deg_w P, \deg_w Q\}\Big), 
$$ 
degree in $\ell_i$ bounded by 
$4^{s'}\delta'_\ell$
and degree in $\ell'_i$ bounded by 
$4^{s'}\delta'_{\ell'}$. An explicit bound for $g_i$ and $g'_i$ will not be necessary.

Then we successively apply for $1 \le i \le s$ the weak inferences
$$
\begin{array}{rcl}
{\rm sign}({\rm sR}_{d_i}) = \tau(i) & \vdash &  {\rm sR}_{d_i} \ne 0, \\[3mm]
\ell'_i \ne 0 &  \ \, \vdash \, \
&
{\ell'}_i^2 > 0, \\[3mm]
\ell_i \ne 0
& \vdash &
{\ell}_i^2 > 0, \\[3mm]
T_{d_{i-1}-1}\ne 0 & \vdash & \exists \ell'_i \ [\; \ell'_i \ne 0, \ \ell'_i \cdot T_{d_{i-1}-1} = 1 \;].
\end{array}
$$
By Lemmas \ref{lemma_basic_sign_rule_1} (items \ref{lemma_basic_sign_rule:1.5} and 
\ref{lemma_basic_sign_rule:3})
and \ref{lemma_weak_existence_inverse_sign}
 we obtain 
an incompatibility
\begin{equation}\label{inc:aux_them_fact_bezout_4}
\begin{array}{cc}
\Big \downarrow \ \
\bigwedge_{0\le i \le p-1}{\rm sign}({\rm HMi}_i(P;Q))=\tau(i), \ 
\bigwedge_{1 \le i \le s} \ell_i \cdot {\rm sR}_{d_i} = 1, \\[4mm]   
\bigwedge_{1 \le i \le s} \ell_i \ne 0, \ 
\bigwedge_{1 \le i \le s} T_{d_{i-1}-1} \ne 0, \ 
\cH \  \ {\Big \downarrow}_{\K[v][ \ell]}
\end{array}
\end{equation}
with  monoid part
$$
S^{\frac12  4^{s'} p(p+1)\tilde e }
\cdot
\prod_{1 \le i \le s} \ell_i^{2g_i} \cdot {\rm sR}_{d_i}^{2h_i} \cdot 
T_{d_{i-1}-1}^{2h'_i + 4^{s'}\delta'_{\ell'} - 2g_i},
$$
degree in $w$ bounded by $\delta''_w + s4^{s'}p(q+3)\delta'_{\ell'}\max \{\deg_w P, \deg_w Q\}$
and degree in $\ell_i$ bounded by 
$4^{s'}\delta'_\ell$.

For $1 \le i \le s$, we successively multiply (\ref{inc:aux_them_fact_bezout_4}) 
by the polynomial
$W(C, D)^{2h'_i + 4^{s'}\delta'_{\ell'} - 2g_i}$, where $W(C, D)$ is the 
polynomial from Lemma \ref{lem:nulls_ident_structure_theorem}, and we substitute  
$T_{d_{i-1}-1}\cdot W$ 
in the monoid part of the result using the identity from 
this lemma. We obtain 
\begin{equation}\label{inc:aux_them_fact_bezout_5}
\Big \downarrow \ 
\bigwedge_{0\le i \le p-1}{\rm sign}({\rm HMi}_i(P;Q))=\tau(i)\ 
\bigwedge_{1 \le i \le s} \ell_i \cdot {\rm sR}_{d_i} = 1, \ 
\bigwedge_{1 \le i \le s} \ell_i \ne 0, \  
\cH \ {\Big \downarrow}_{\K[v][ \ell]}
\end{equation}
with  monoid part
$$
S^{\frac12  4^{s'} p(p+1)\tilde e }
\cdot
\prod_{1 \le i \le s} \ell_i^{2g_i} \cdot {\rm sR}_{d_i}^{2h''_i}
$$
with
$$
h''_i \le h_i + 
p^p(q+3)^p 4^{s'-1}(
{p^2(p+1)(p+2)e\tilde e + 2\delta'_{\ell'}}
)
$$
degree in $w$ bounded by 
$$
\delta''_w + 4^{s'}\Big( s p(q+3)\delta'_{\ell'} 
+ p^{p+1}(q+3)^{p+1}(p^2(p+1)(p+2)e\tilde e +  2\delta'_{\ell'})
 \Big) \max \{\deg_w P, \deg_w Q\}
$$
and degree in $\ell_i$ bounded by 
$4^{s'}\delta'_\ell$. 

Finally we successively apply to (\ref{inc:aux_them_fact_bezout_5}) for $1 \le i \le s$ the weak inferences
$$
\begin{array}{rcl}
{\rm sR}_{d_i} \ne 0 
&  \ \, \vdash \, \
&
\exists \ell_i \ [\; \ell_i \ne 0,  \ 
\ell_i \cdot {\rm sR}_{d_i} = 1  \;], \\[3mm]
{\rm sign}({\rm sR}_{d_i}) = \tau(d_i) &
\vdash & {\rm sR}_{d_i} \ne 0. 
\end{array}
$$
By Lemmas 
\ref{lemma_weak_existence_inverse_sign} and
\ref{lemma_basic_sign_rule_1} (item \ref{lemma_basic_sign_rule:1.5})
 we obtain 
$$
\Big \downarrow \
\bigwedge_{0\le i \le p-1}{\rm sign}({\rm HMi}_i(P;Q)) = \tau(i), \ \cH \ {\Big \downarrow}_{\K[v]}
$$
with monoid part 
$$
S^{\frac12  4^{s'} p(p+1)\tilde e }
\cdot
\prod_{1 \le i \le s}{\rm sR}_{d_i}^{2h''_i + 4^{s'}\delta'_\ell - 2g_i}
$$
and degree in $w$ bounded by 
$$
\delta''_w + 4^{s'}\Big( s p(q+3)(\delta'_\ell + \delta'_{\ell'}) 
+ p^{p+1}(q+3)^{p+1}(p^2(p+1)(p+2)e\tilde e +  2\delta'_{\ell'})
 \Big) \max \{\deg_w P, \deg_w Q\}
$$
which serves as the final incompatibility, taking into account that $s' \le \frac{p}2$. 
\end{proof}

\subsection{Sylvester Inertia Law} 
\label{subsecSylv}

Sylvester Inertia Law states that two diagonal reductions
of a quadratic form in an ordered field
have  the same number of positive, negative and null coefficients. 
In order to obtain Sylvester Inertia Law as an incompatibility, we use 
linear algebra \`a la Gram. First, we  introduce some definitions, notation and properties. 
We refer to \cite{DGL} and \cite{LaTi} for further details and proofs.

\begin{definition}\label{defGk}
Let $\A$ be a commutative ring, ${\bm A}\in\A^{m \times n}$ and $k \in \N$. 
\begin{enumerate}

\item The {\em Gram's coefficient} $\Gram_k({\bm A})$ 
is the coefficient
$g_k$ of the polynomial
$$   \det(\mathrm{I}_m  +  y \cdot {\bm A} \cdot {\bm A}^\mathrm{t})  =  g_0  +  g_1  \cdot y  +  \cdots  +  g_m \cdot y^m,
$$
where $y$ is an indeterminate over $\A$. 

\item The matrix  ${\bm A}^{\ddag_k} \in \A^{n \times m}$ is the matrix 
$$
{\bm A}^{\ddag_k} = 
\Big(\sum_{0 \le i \le k-1} (-1)^i\Gram_{k-1-i}({\bm A})\cdot({\bm A}^\mathrm{t}\cdot{\bm A})^i \Big)\cdot 
{\bm A}^\mathrm{t}.  
$$
\end{enumerate}
\end{definition}

Note that $\Gram_k({\bm A})$ is an homogeneous polynomial of degree $2k$ in the entries of ${\bm A}$ and 
the entries of ${\bm A}^{\ddag_k}$ are homogeneous polynomials of degree
$2k-1$ in the entries of ${\bm A}$.
Note also that $\Gram_0({\bm A})= 1$ and
$\Gram_k({\bm A})= 0$ for $k>m$.
For $1 \le k \le m$, $\Gram_k({\bm A})$ is equal to the sum of the squares of all the $k$-minors 
of $\A$.

\begin{notation}
Let $\A$ be a commutative ring, ${\bm A}\in\A^{m \times n}$ and $k \in \N$. We denote by 
$\cD_{k}({\bm A})$ the ideal generated by all the $k$-minors of the matrix 
${\bm A}$.
 \end{notation}

\begin{proposition}
\label{propMoPe} 
Let $\A$ be a commutative ring, ${\bm A}\in\A^{m \times n}$,  $v\in \A^m$, $k \in \N$ and
let ${\bm A | \bm v}$
be the  matrix in $\A^{m \times (n+1)}$ obtained by  
adding ${\bm v}$ as a last column
to ${\bm A}$. Then
$$
 \Gram_k({\bm A}) \cdot {\bm v}  =  {\bm A} \cdot {\bm A}^{\ddag_k} \cdot {\bm v} \quad \mathrm{mod} \; 
\cD_{k+1}({\bm A} | {\bm v}).
$$
Moreover, this equation is given by homogeneous identities of degree $2k$
in the entries of ${\bm A}$ and of degree $1$ in the entries of ${\bm v}$.
\end{proposition}

The following proposition plays a fundamental role to express Sylvester Inertia Law as an 
incompatibility.

\begin{proposition}
\label{thSylv0} 
Let ${\bm v}_1, \dots, {\bm v}_s, {\bm w}_1, \dots, {\bm w}_{t+1} \in \K[u]^{p}$ 
with $s\in \N_*,  t \in \N, s+t = p$, ${\bm A} \in \K[u]^{p \times p}$ 
be
a symmetric matrix, and let ${\bm V}\in \K[u]^{p \times s}$ be the matrix having
${\bm v}_1, \dots, {\bm v}_s$ as columns. Then, there is an incompatibility
$$
\Big \downarrow \Gram_s({\bm V}) \ne 0, \ 
\bigwedge_{1\leq i\leq s} {\bm v}_i^\mathrm{t} \cdot {\bm A} \cdot {\bm v}_i \ge 0, \ 
\bigwedge_{1\leq i< i'\leq s}{\bm v}_i^\mathrm{t}\cdot {\bm A} \cdot {\bm v}_{i'} = 0,
$$
$$ 
\bigwedge_{1\leq j\leq t+1} {\bm w}_j^\mathrm{t} \cdot {\bm A} \cdot {\bm w}_j < 0, \ 
\bigwedge_{1\leq j< j'\leq t+1}{\bm w}_j^\mathrm{t} \cdot {\bm A} \cdot {\bm w}_{j'} = 0 \ \Big \downarrow
$$
with monoid part
$$ 
\Gram_s({\bm V})^{2^{2(t+1)}} \cdot
\prod_{1 \le j \le t+1} 
({\bm w}_{j}^\mathrm{t} \cdot {\bm A} \cdot {\bm w}_j)^{2^{2(t-j)+3} }
$$
and  
degree in $w \subset u$ 
bounded by 
$$
\frac23(2^{2(t+1)}-1)\deg_w {\bm A} + 
\frac43\Big( 2^{2t+1}(3s+2) - 1\Big) 
\max\{\deg_w {\bm v}_i\ | \ 1 \le i \le s \} \cup 
\{ \deg_w {\bm w}_i \ | \ 1 \le j \le t+1 \}.
$$
\end{proposition}

\begin{proof}{Proof.} Let $\cH$ be the system of sign conditions whose 
incompatibility we want to obtain. 
Let 
$\delta_w = \deg_w {\bm A}$ and
$\delta'_w =  \max\{\deg_w {\bm v}_i \ | \ 1 \le i \le s \} \cup 
\{ \deg_w {\bm w}_i \ | \ 1 \le j \le t+1 \}$. 
For $0 \le j \le t+1$, we consider the matrix ${\bm V}_{s+j} \in {\bf A}^{p \times (s+j)}$ 
having the vectors ${\bm v}_1, \dots, {\bm v}_s, {\bm w}_1, \dots, {\bm w}_{j}$ as 
columns. We denote by $G_{s+j}$ the Gram's coefficient $\Gram_{s+j}({\bm V}_{s+j}) \in \K[u]$. 

For $1 \le j \le t+1$, we apply Proposition \ref{propMoPe}  to the matrix ${\bm V}_{s+j-1}$, 
the vector
${\bm w}_{j}$ and the number $s+j-1$. If for $1 \le k \le s+j-1$ we 
denote $H_{s+j-1, k}$ the $k$-th coordinate of the vector 
${\bm V}_{s+j-1}^{\ddag_{s+j-1}} \cdot {\bm w}_{j}$,
we obtain
\begin{equation}\label{eq:sylv_as_inc_minus0}
G_{s+j-1}\cdot{\bm w}_{j} - \sum_{1 \le k \le j-1}H_{s+j-1,s+k}\cdot{\bm w}_{k} =
\sum_{1 \le i \le s}H_{s+j-1,i}\cdot{\bm v}_{i}  
\quad \mathrm{mod} \; 
\cD_{s+j}({\bm V}_{s+j}).
\end{equation}
Next we apply to (\ref{eq:sylv_as_inc_minus0}) the quadratic form associated to ${\bm A}$. After passing some terms to the left
hand side, we obtain for $1 \le j \le t+1$, 
\begin{equation}\label{eq:sylv_as_inc_minus1}
G_{s+j-1}^2\cdot{\bm w}_{j}^{\mathrm{t}}\cdot{\bm A}\cdot{\bm w}_{j}  
+ \sum_{1 \le k \le j-1}H_{s+j-1,s+k}^2\cdot{\bm w}_{k}^{\mathrm{t}}\cdot{\bm A}\cdot{\bm w}_{k} -
\sum_{1 \le i \le s}H_{s+j-1,i}^2\cdot{\bm v}_{i}^{\mathrm{t}}\cdot{\bm A}\cdot{\bm v}_{i} + Z_j = D_{s+j}  
\end{equation}
with $Z_j \in \scZ( \cH_{=})$ and $D_{s+j}$ $\in$ $\cD_{s+j}({\bm V}_{s+j})$.
The degree in $w$ of the first three terms of (\ref{eq:sylv_as_inc_minus1}) and 
the components of $Z_j$ and $D_{s+j}$ is bounded by $\delta_w + (4(s+j)-2)\delta'_w$.

Raising (\ref{eq:sylv_as_inc_minus1}) to the square, we obtain 
\begin{equation} \label{eq:sylv_as_inc_0}
G_{s+j-1}^4\cdot({\bm w}_{j}^{\mathrm{t}}\cdot{\bm A}\cdot{\bm w}_{j})^2  
+ N_j + Z'_j 
= D_{s+j}^2  
\end{equation}
with $N_j \in \scN( \cH_{\ge})$ and $Z'_j \in \scZ( \cH_{=})$.
Let $M_1, \dots, M_\ell \in \K[u]$ be all the $(s+j)$-minors of the matrix ${\bm V}_{s+j}$
and 
consider $Q_1, \dots, Q_\ell \in \K[u]$ such that $D_{s+j} = \sum_{1 \le k \le \ell}M_k\cdot Q_k$. 
Note that for $1 \le k \le \ell$, $\deg_w M_k \le (s+j)\delta'_w$ and 
$\deg_w Q_k \le \delta_w + (3(s+j)-2)\delta'_w$. 
Adding to both sides of (\ref{eq:sylv_as_inc_0}) the sum of squares ${\rm N}(M_1, \dots, M_\ell, Q_1, \dots, Q_\ell)$
defined in Remark \ref{lemCS}, we obtain for $1 \le j \le t+1$, 
\begin{equation}\label{eq:sylv_as_inc}
G_{s+j-1}^4\cdot({\bm w}_{j}^{\mathrm{t}}\cdot{\bm A}\cdot{\bm w}_{j})^2  
+ N'_j + Z'_j 
= G_{s+j}\cdot R_{s+j}  
\end{equation}
with $N'_j \in \scN( \cH_{\ge})$ and $R_{s+j} = 2^{\ell}\sum_{1 \le k \le  \ell}Q_k^2$.  
The degree in $w$ of the first term of (\ref{eq:sylv_as_inc}) and 
the components of $N'_j$ and $Z'_j$ is bounded by $2\delta_w + (8(s+j)-4)\delta'_w$.

We will prove by induction on $h$ that for $1 \le h \le t+1$ we have an identity
\begin{equation}\label{eq:sylv_as_inc_2}
G_{s}^{4^{h}}\cdot \prod_{1 \le j \le h} 
({\bm w}_{j}^\mathrm{t}\cdot {\bm A} \cdot{\bm w}_{j})^{2 \cdot 4^{h-j} }
+ N''_h + Z''_h = G_{s+h}\cdot \prod_{1 \le j \le h}R_{s+j}^{4^{h-j}}
\end{equation}
with $N''_h \in \scN(\cH_{\ge})$, $Z''_h \in \scZ(\cH_{=})$
and degree in $w$ of the first term of (\ref{eq:sylv_as_inc_2}) 
and the components of $N''_h$ and $Z''_h$ bounded by 
$$
\frac23(4^h - 1)\delta_w + 
\frac23\Big( 4^h(3s+2) - 2\Big)\delta'_w.   
$$ 

For $h = 1$, we take equation (\ref{eq:sylv_as_inc}) for $j = 1$. 
Suppose now we have an equation like (\ref{eq:sylv_as_inc_2}) for some 
$1 \le h \le t$. We raise it to the 4-th power and we multiply the result by 
$({\bm w}_{h+1}^\mathrm{t}\cdot {\bm A}\cdot {\bm w}_{h+1})^2$. We obtain 
\begin{equation}\label{eq:sylv_as_inc_3}
G_{s}^{4^{h+1}}\cdot \prod_{1 \le j \le h+1} 
({\bm w}_{j}^\mathrm{t}\cdot {\bm A}\cdot {\bm w}_{j})^{2 \cdot 4^{h+1-j} }
+ N'''_h + Z'''_h = G_{s+h}^4\cdot({\bm w}_{h+1}^\mathrm{t}\cdot {\bm A}\cdot {\bm w}_{h+1})^2\cdot
\prod_{1 \le j \le h}R_{s+j}^{4^{h+1-j}}
\end{equation}
with $N'''_h \in \scN(\cH_{\ge})$ and $Z'''_h \in \scZ(\cH_{=})$. On the other hand, we 
multiply equation (\ref{eq:sylv_as_inc}) for $j = h+1$ by 
$\prod_{1 \le j \le h}R_{s+j}^{4^{h+1-j}}$ and we obtain 
\begin{equation}\label{eq:sylv_as_inc_4}
G_{s+h}^4\cdot({\bm w}_{h+1}^{\mathrm{t}}\cdot{\bm A}\cdot{\bm w}_{h+1})^2 \cdot \prod_{1 \le j \le h}R_{s+j}^{4^{h+1-j}} 
+ N''''_{h+1}+ Z''''_{h+1} 
= G_{s+h+1}\cdot\prod_{1 \le j \le h+1}R_{s+j}^{4^{h+1-j}}
\end{equation}
with $N''''_h \in \scN(\cH_{\ge})$ and $Z''''_h \in \scZ(\cH_{=})$.
Finally, by adding equations (\ref{eq:sylv_as_inc_3}) and (\ref{eq:sylv_as_inc_4}) and simplifying 
equal terms at both sides of the identity, we obtain an equation like (\ref{eq:sylv_as_inc_2})
for $h+1$. The degree bound follows easily. 

Taking into account that 
$G_s = \Gram_s({\bm V})$ and
$G_{s+t+1} = 0$ since ${\bm V}_{s+t+1}$ has only $p = s+t$ rows, 
the proposition follows by considering the incompatibility $\lda \cH \rda$ obtained taking
$h = t+1$ in equation (\ref{eq:sylv_as_inc_2}).
\end{proof}

\begin{lemma}\label{determinant_vs_gram}
Let ${\bm C} \in \K[u]^{p \times p}$, 
$1 \le s \le p$,  $1 \le i_1 < \dots < 
i_s \le p$ and
${\bm v}_1, \dots, {\bm v}_s \in \K[u]^p$ 
be 
the   columns $i_1, \dots, i_s$ of 
${\bm C}$.  Then 
$$
\det(\bm C) \ne 0 \ \ \  \vdash \ \ \  \Gram_s([\bm v_{1} | \dots  | \bm v_{s}]) \ne 0,
$$
where $[\bm v_{1} | \dots  | \bm v_{s}]$ is the matrix in $\K[u]^{p\times s}$ formed
by the vectors $\bm v_{1}, \dots, \bm v_{s}$ as columns.  
If we have an initial 
incompatibility in variables $v \supset u$ with monoid part
$S\cdot \Gram_s([\bm v_{1} | \dots  | \bm v_{s}])^{2e}$
and degree in $w \subset v$ bounded by $\delta_w$,
the final 
incompatibility has monoid part 
$S\cdot \det(\bm C)^{4e}$
and degree in $w$ bounded by 
$\delta_w + 4e(p-s)\deg_w {\bm C}$.
\end{lemma}
\begin{proof}{Proof.} 
By the Generalized Laplace Expansion Theorem, $\det(\bm C)$ is a linear combination of 
the $s$ minors of $[\bm v_{1} | \dots  | \bm v_{s}]$, 
where the coefficients are, up to sign,  $p-s$ minors of the matrix formed with the remaining
columns of ${\bm C}$. 
Then, the lemma follows from 
Lemma \ref{comb_ne_some_elem_ne}.
\end{proof}

We can prove now an incompatibility version of Sylvester Inertia Law.

\begin{theorem}[Sylvester Inertia Law as an incompatibility]\label{thSylv2}
Let ${\bm A} \in \K[u]^{p \times p}$ be a symmetric matrix, 
$\bm B, \bm B' \in \K[u]^{p \times p}$, 
$\bm D$, $\bm D' \in \K[u]^{p \times p}$ 
be
diagonal matrices with 
$(\bm D)_{ii} =  D_{i}$ for $1 \le i \le p$ and 
$(\bm D')_{jj} =  D'_{j}$ for $1 \le j \le p$ 
and $\eta, \eta' \in \{-1, 0, 1\}^p$.
If the number of coordinates in $\eta$ and $\eta'$ equal to  $-1$, $0$ and $1$ 
is not respectively the same, there is an incompatibility
$$
\Big \downarrow {\bm A} \equiv \bm B \cdot \bm D \cdot  \bm B^{\mathrm{t}}, \ 
{\bm A} \equiv \bm B' \cdot \bm D' \cdot \bm B'^{\mathrm{t}}, \
\det(\bm B) \ne 0, \ \det(\bm B') \ne 0, $$
$$
\bigwedge_{1 \le i \le p}\sign(D_{i}) = \eta(i), \ 
\bigwedge_{1 \le j \le p}\sign(D'_{j}) = \eta'(j) \
\Big \downarrow
$$
with monoid part 
$$
\det(\bm B)^{2e}\cdot \det(\bm B)^{2e'} \cdot
\prod_{1 \le i \le p, \atop \eta(i) \ne 0}D_i^{2f_{i}}
\cdot
\prod_{1 \le j \le p, \atop \eta'(j) \ne 0}{D'_j}^{2f'_{j}}
$$
with
$ e, e' \le p2^{2p}$, $f_{ i}, f'_{j} \le 2^{2(p-1)}$
and  
degree in $w \subset u$
bounded by 
$$
2^{2p}\deg_w {\bm A} + 
p^22^{2p+1}
\max\{\deg_w {\bm B}, \deg_w {\bm B'}\} +
2^{2p+1}\max\{\deg_w {\bm D}, \deg_w {\bm D'}\}.
$$
\end{theorem}

\begin{proof}{Proof.}
Let $\delta_w = \deg_w {\bm A}$, $\delta'_w = \max\{\deg_w {\bm B}, \deg_w {\bm B'}\}$
and $\delta''_w = \max\{\deg_w {\bm D}, \deg_w{\bm D'}\}$.
Without loss of generality, we suppose that there are at least $s$ coordinates 
$1\le k_1 < \dots < k_s \le p$
in $\eta$ equal to $0$
or $1$ and at least $t+1$ coordinates $1\le k'_1 < \dots< k'_{t+1} \le p$ in 
$\eta'$ equal to $-1$, with $s\in \N_*, t \in \N$ and  $s + t = p$. 
We take ${\bm v}_1, \dots, {\bm v}_s$ as the 
columns $k_1, \dots, k_s$ of ${\rm Adj}({\bm B})^{\mathrm t}$ and 
${\bm w}_1, \dots, {\bm w}_{t+1}$ as the 
columns $k'_1, \dots, k'_{t+1}$ of ${\rm Adj}({\bm B}')^{\mathrm t}$.

We successively apply to the incompatibility from Proposition \ref{thSylv0}
the weak inferences
$$
\begin{array}{rcl}
\det ({\rm Adj}({\bm B})^{\mathrm t}) \ne 0 & \ \, \vdash \, \ & 
\Gram_s({\bm V}) \ne 0, 
 \\[3mm]
\det ({\bm B}) \ne 0 & \ \, \vdash \, \ &  \det ({\rm Adj}({\bm B})^{\mathrm t}) \ne 0.
 \\[3mm]
\end{array}
$$
Since $\det ({\rm Adj}({\bm B})^{\mathrm t}) = \det({\bm B})^{p-1}$, 
by Lemmas  \ref{determinant_vs_gram}
and \ref{lemma_basic_sign_rule_1} (item \ref{lemma_basic_sign_rule:5}), 
we obtain an incompatibility with monoid part
$$ 
\det({\bm B})^{(p-1)2^{2t+3}}
\cdot
\prod_{1 \le j \le t+1} 
({\bm w}_{j}^\mathrm{t} \cdot {\bm A} \cdot {\bm w}_j)^{2^{2(t-j)+3} }
$$
and degree in $w$ bounded by 
$$
\frac23(2^{2(t+1)}-1)\delta_w + 
(p-1)\Big( 2^{2t+3}\Big(p+\frac23\Big) - \frac43 \Big) \delta'_w .
$$

Then we successively apply for $1 \le i \le s$ and for $1 \le j \le t+1$ the weak inferences
$$
\begin{array}{rcl}
{\bm v}_i^{\mathrm t} \cdot  {\bm A} \cdot {\bm v}_i = \det({\bm B})^2\cdot D_{k_i}, \ 
\det({\bm B})^2 \cdot D_{k_i} \ge 0 & \ \, \vdash \, \ &  {\bm v}_i^{\mathrm t} \cdot {\bm A} \cdot {\bm v}_i \ge 0, \\[3mm]
\det({\bm B})^2  \ge 0, \ D_{k_i} \ge 0 & \ \, \vdash \, \ &  \det({\bm B})^2 \cdot D_{k_i} \ge 0, \\[3mm]
 & \ \, \vdash \, \ &  \det({\bm B})^2  \ge 0, \\[3mm]
{\bm w}_j^{\mathrm t} \cdot  {\bm A} \cdot {\bm w}_j = \det({\bm B}')^2 \cdot D'_{k'_j}, \ 
\det({\bm B}')^2 \cdot D'_{k'_j} < 0  & \ \, \vdash \, \ & 
 {\bm w}_j^{\mathrm t} \cdot {\bm A} \cdot {\bm w}_j < 0, \\[3mm]
\det({\bm B}')^2 > 0, \ D'_{k'_j} < 0   & \ \, \vdash \, \ &  \det({\bm B}')^2 \cdot D'_{k'_j} < 0, \\[3mm]
\det({\bm B}') \ne 0  & \ \, \vdash \, \ &  \det({\bm B}')^2 > 0.
\end{array}
$$
By Lemmas \ref{lemma_basic_sign_rule_1} 
(items \ref{lemma_basic_sign_rule:2}, \ref{lemma_basic_sign_rule:3}, 
\ref{lemma_basic_sign_rule:6} and 
\ref{lemma_basic_sign_rule:7})  
\ref{lemma_sum_of_pos_and_zer_is_pos} (item \ref{lemma_sum_of_pos_and_zer_is_pos:2}), 
and \ref{lemma_sum_of_pos_is_pos},
we obtain an incompatibility with monoid part 
$$ 
\det({\bm B})^{(p-1)2^{2t+3}} 
\cdot
\det({\bm B}')^{\frac43 (2^{2(t+1)}-1)} 
\cdot
\prod_{1 \le j \le t+1} 
{D'_{{k'}_j}}^{ 2^{2(t-j)+3} }
$$
and degree in $w$ bounded by
$$
\frac23(2^{2(t+1)}-1)\delta_w + 
\Big(
p^22^{2t+3} - \frac13p2^{2t+3} - \frac132^{2t+4} + 2ps - \frac83p + \frac43
\Big)\delta'_w 
+
\Big(s + \frac23(2^{2(t+1)}-1)\Big)\delta''_w.
$$

Finally, we successively apply the weak inferences
$$
\begin{array}{rcl}
{\bm A} \equiv \bm B \cdot \bm D \cdot  \bm B^{\mathrm{t}} 
& \ \, \vdash \, \ &
{\rm Adj}(\bm B)\cdot {\bm A} \cdot {\rm Adj}(\bm B)^{\mathrm{t}} \equiv \det(\bm B)^2 \cdot 
\bm D, \\[3mm]
{\bm A} \equiv \bm B' \cdot \bm D' \cdot  \bm B'^{\mathrm{t}} & \ \, \vdash \, \ &
{\rm Adj}(\bm B')\cdot {\bm A} \cdot {\rm Adj}(\bm B')^{\mathrm{t}} \equiv \det(\bm B')^2 \cdot
\bm D'.
\end{array}
$$
By Lemma \ref{lemma_prod_equiv_matric}, we obtain an incompatibility with the same monoid
part and degree in $w$ bounded by 
$$
\frac43(2^{2t+1}+1)\delta_w + 
\Big(
p^22^{2t+3} - \frac13p2^{2t+3} - \frac132^{2t+4} + 2ps + \frac43p + \frac43
\Big)\delta'_w 
+
\Big(s + \frac43(2^{2t+1}+1)\Big)\delta''_w.
$$
which is the incompatibility we wanted to obtain. 
\end{proof}

\subsection{Hermite's quadratic form and Sylvester Inertia Law}\label{subsecHerm_incomp}

In order to obtain the main result of this section, we combine now Sylvester Inertia Law with 
the two methods we have considered to compute the signature of the Hermite's quadratic form.

\begin{notation} 
\label{not2sign}
Let $p \in \N_*$. 
\begin{itemize}

\item For $\tau \in \{-1, 0, 1\}^{\{0,\dots,p-1\}}$ and $d(\tau) = 
(d_0,\ldots,d_s)$, we denote by    
\begin{itemize}
\item ${\rm Rk}_{\rm HMi}(\tau) = p - d_s$, 
\item ${\rm Si}_{\rm HMi}(\tau) = 
\sum\limits_{1 \le i \le s, \atop d_{i-1}-d_i \mathrm{\, odd}} 
\varepsilon_{d_{i-1}-d_{i}} \tau({d_{i-1}) \tau(d_{i})}.$
\end{itemize}

\item For $m, n \in \N$ with $m + 2n = p$, ${\bm \eta} \in 
\{-1, 0, 1\}^{m}$ and 
${\bm \kappa} \in 
\{0, 1\}^{n}$, we denote by
\begin{itemize}
\item ${\rm Rk}_{\rm Fact}({\bm \eta},{\bm \kappa})$ the addition of the number of coordinates in ${\bm \eta}$ 
equal to $-1$ or $1$
and twice the number of coordinates in ${\bm \kappa}$ equal to $1$, 
\item ${\rm Si}_{\rm Fact}({\bm \eta})$ the number of coordinates in ${\bm \eta}$ 
equal to $1$ minus 
the number of coordinates in ${\bm \eta}$ 
equal to
$-1$.
\end{itemize}

\end{itemize}

\end{notation}

Note that ${\rm Rk}_{\rm HMi}(\tau)$ and ${\rm Si}_{\rm HMi}(\tau)$ are respectively
the rank and signature of the matrix ${\rm Her}(P; Q)$ if $\tau$ is the sign condition
satisfied by ${\rm HMi}(P; Q)$. 
Similarly, 
${\rm Rk}_{\rm Fact}({\bm \eta},{\bm \kappa})$ and ${\rm Si}_{\rm Fact}({\bm \eta})$
are respectively
the rank and signature of the matrix ${\rm Her}(P; Q)$ if in the decomposition into real irreducible
factors of $P$, 
${\bm \eta}$ is the sign condition 
satisfied by the real roots of $P$ at $Q$ and 
${\bm \kappa}$ is the invertibility condition satisfied
by the complex non-real roots of $P$ at $Q$.

We define a new auxiliary function. 

\begin{definition}\label{def:func_aux_hilb}
\label{defgH}
Let ${\rm g}_{H}: \N \times \N \to \N$, ${\rm g}_{H}\{p,q\} = 39 \cdot 2^{7p} p^{5p+6}(q+3)^{4p+2}$. 
\end{definition}

In the following theorem, we combine Sylvester Inertia Law with Hermite's Theory as an incompatibility. To do so, 
we use many previously given definitions and notation, namely 
Notation \ref{ResR},
Notation \ref{not:aux_fact}, 
Definition \ref{defFact},
Definition \ref{notation:sign},
Definition \ref{inv},
Notation \ref{notation:testcoeff}
and
Notation \ref{not2sign}.

\begin{theorem}[Hermite's Theory as an incompatibility]\label{hermitetrsubresw} Let $P
, Q
\in \K[u][ y]$ with $\deg_y P = p \ge 1$, 
$\deg_y Q = q$ and 
$P$ monic with respect to $y$. 
For $\tau \in \{-1, 0, 1\}^{\{0, \dots, p-1\}}$, 
$d(\tau) = (d_0, \dots, d_{s})$,
$m+2n=p$, $({\bm \mu}, {\bm \nu}) \in \Lambda_{m} \times \Lambda_{n}$,
${\bm \eta} \in \{-1 ,0, 1\}^{\# {\bm \mu}}$, ${\bm \kappa} \in \{ 0,  1\}^{\# {\bm \nu}}$
such that 
$({\rm Rk}_{\rm HMi}(\tau), {\rm Si}_{\rm HMi}(\tau)) \ne 
({\rm Rk}_{\rm Fact}({\bm \eta}, {\bm \kappa}), {\rm Si}_{\rm Fact}({\bm \eta}))$,  
$t = (t_1, \dots, t_{\# {\bm \mu}})$ and
$z  = (z_1, \dots, z_{\# {\bm \nu}})$, we have
$$
\Big\downarrow
\bigwedge_{0\leq i\leq p-1} \sign({\rm HMi}_i(P; Q)) = \tau(i), \ 
{\rm Fact}(P)^{{\bm \mu},{\bm \nu}}(t,z),
$$ 
$$
\bigwedge_{1\leq j\leq \# {\bm \mu}} {\rm sign}(Q(t_j)) = \eta_j, \
\bigwedge_{1\leq k\leq \# {\bm \nu}} {\inv}(Q(z_k)) = \kappa_k
\Big\downarrow
_{\K[u][ t, a, b]}
$$ 
with monoid part
$$
\prod_{1 \le i \le s}{\rm HMi}_{d_i}(P; Q)^{2g_i}
\cdot
\prod_{1 \le j < j' \le \# {\bm \mu}}(t_{j} - t_{j'})^{2e_{{j}, j'}} 
\cdot
\prod_{1 \le k \le \# {\bm \nu}}b_{k}^{2f_{k}}
\cdot
$$
$$
\cdot 
\prod_{1 \le k < k' \le {\# {\bm \nu}}}
{\rm R}(
z_{k},z_{k'})
^{2g_{{k}, k'}} 
\cdot
\prod_{1 \le j \le \# {\bm \mu}, \atop \eta_j \ne 0}Q(t_j)^{2e'_j}
\cdot
\prod_{1 \le k \le \# {\bm \nu}, \atop \kappa_k \ne 0}(Q^2_{\re}(z_k) + Q^2_{\im}(z_k))^{2f'_k}
$$
with $g_i, e_{j,j'}, f_k, g_{k, k'}, e'_j, f'_k \le {\rm g}_{H}\{p,q\}
$, 
degree in $w \subset u$ bounded by 
$
{\rm g}_{H}\{p,q\}\max\{\deg_w P, \deg_w Q\}
$ and degree in $t_j$ and degree in $(a_k, b_k)$ bounded by $
{\rm g}_{H}\{p,q\}$.
\end{theorem}

\begin{proof}{Proof.}
We evaluate 
$$ {\bm A} = {\rm Her}(P;Q), \ 
{\bm B} = {\rm B}_{{\bm \kappa}}(t,z,z'), \ {\bm D} = 
{\rm Di}_{Q}^{{\bm \mu},{\bm \nu}, {\bm \kappa}}(t),
{\bm B}' = {\rm B}_{P;Q}^\tau(\ell, \ell', z''), \ {\bm D}' =   
{\rm Di}_{P;Q}^\tau(\ell,\ell')
$$
in the incompatibility from Theorem \ref{thSylv2} (Sylvester Inertia Law as an incompatibility),
where 
$z' = (z'_k)_{\kappa_k = 1}$, 
$\ell = (\ell_1, \dots, \ell_{s})$, 
$\ell' = (\ell'_1, \dots, \ell'_{s})$ and $z'' = (z''_i)_{d_{i-1} - d_i \, \hbox{\scriptsize{even}}}$
and we obtain 
\begin{equation}\label{inc:dep_final_inc_chap_5}
\begin{array}{c}
\Big \downarrow \ {\rm Her}(P;Q) \equiv {\rm B}_{{\bm \kappa}}(t,z,z')
\cdot 
{\rm Di}_{Q}^{{\bm \mu},{\bm \nu},{\bm \kappa}}(t)
\cdot {\rm B}_{{\bm \kappa}}^{\rm t}(t,z,z'),
\ \det({\rm B}_{{\bm \kappa}}(t,z,z')) \ne 0,
 \\[5mm]
{\rm Her}(P;Q) \equiv
{\rm B}_{P;Q}^\tau(\ell, \ell', z'')
\cdot
{\rm Di}_{P;Q}^\tau(\ell,\ell')
\cdot
{{\rm B}_{ P;Q}^\tau}^{\rm t}(\ell, \ell', z''), \ 
\det({\rm B}_{P;Q}^\tau(\ell, \ell', z'')) \ne 0,  \\[5mm]
\displaystyle{\bigwedge_{1\leq j\leq \# {\bm \mu}} {\rm sign}(Q(t_j)) = \eta_j}, 
\\[5mm]
\displaystyle{
\bigwedge_{1 \le i \le s, \atop d_{i-1}-d_i \hbox{\scriptsize odd}} 
{\rm sign}(\ell_{i-1}^2 \cdot {\ell_i'}^2 \cdot {\rm HMi}_{d_{i-1}}(P; Q)^{2({d_{i-1}} - d_{i})-1} \cdot {\rm HMi}_{d_i}(P; Q)) 
= \tau(d_{i-1})\tau(d_i)} \ \Big \downarrow
_{\K[u][ t, a, b, a', b', \ell, \ell', a'', b'']}
\end{array}
\end{equation}
with monoid part 
$$
\det({\rm B}_{{\bm \kappa}}(t,z,z'))^{2e_1} \cdot \det({\rm B}_{P;Q}^\tau(\ell, \ell', z''))^{2e_2}\cdot
\prod_{1 \le j \le \# {\bm \mu}, \atop \eta_j \ne 0}Q(t_j)^{2f_{1,j}}
\cdot
$$
$$
\prod_{1 \le i \le s, \atop d_{i-1} - d_i \hbox{\scriptsize{odd}}}
\big(\ell_{i-1}^2 \cdot {\ell_i'}^2 \cdot {\rm HMi}_{d_{i-1}}(P; Q)^{2({d_{i-1}} - d_{i})-1} \cdot {\rm HMi}_{d_i}(P; Q)
\big)^{2f_{2,i}}
$$ 
with $e_1, e_2 \le p2^{2p}$, 
$f_{1,j} \le 2^{2(p-1)}$ and 
$f_{2,i} \le p2^{2(p-1)}$, 
degree in $w$ bounded by $9p^4(q+3)2^{2p}\max\{\deg_w P, \deg_w Q \}$,  
degree in $t_j$ and 
degree in $(a_k, b_k)$ bounded by $p^32^{2p+1}$,
degree in $(a'_k, b'_k)$ bounded by $p^22^{2p+1}$,
degree in $\ell_i$  and
degree in $\ell'_i$ bounded by $5p^32^{2p}$
and
degree in $(a''_i, b''_i)$ bounded by $p^22^{2p+1}$.

Then we apply to (\ref{inc:dep_final_inc_chap_5})  the weak inference 
from 
Theorem \ref{fact_Her_matrix_fact} (Hermite's Theory (1) as a weak existence) and we obtain
\begin{equation}\label{inc:aux_final_inc_chap_5_1}
\begin{array}{c}
\Big \downarrow \ 
\displaystyle{{\rm Fact}(P)^{{\bm \mu},{\bm \nu}}(t,z), 
\ \bigwedge_{1\leq j\leq \# {\bm \mu}} {\rm sign}(Q(t_j)) = \eta_j, \ 
\bigwedge_{1 \le k \le \# {\bm \nu}}{\inv}(Q(z_k)) = \kappa_k},
{\rm B}_{P;Q}^\tau(\ell, \ell', z'')
 \\[6mm]
{\rm Her}(P;Q) \equiv
\cdot
{\rm Di}_{P;Q}^\tau(\ell, \ell')
\cdot
{{\rm B}_{ P;Q}^\tau}^{\rm t}(\ell, \ell', z''), \ 
\det({\rm B}_{P;Q}^\tau(\ell, \ell', z'')) \ne 0,\\[6mm]
\displaystyle{\bigwedge_{1 \le i \le s, \atop d_{i-1}-d_i \hbox{\scriptsize odd}} 
{\rm sign}(\ell_{i-1}^2 \cdot {\ell_i'}^2 \cdot {\rm HMi}_{d_{i-1}}(P; Q)^{2({d_{i-1}} - d_{i})-1} \cdot {\rm HMi}_{d_i}(P; Q)) 
= \tau(d_{i-1})\tau(d_i)}
\Big \downarrow
_{\K[u][ t, a, b, \ell, \ell', a'', b'']}
\end{array}
\end{equation}
with monoid part 
$$
\det({\rm B}_{P;Q}^\tau(\ell, \ell', z''))^{2^{2s({\bm \kappa}) + 1}e_2}
\cdot
\prod_{1 \le i \le s, \atop d_{i-1} - d_i \hbox{\scriptsize{odd}}}
\big(\ell_{i-1}^2 \cdot {\ell_i'}^2 \cdot {\rm HMi}_{d_{i-1}}(P; Q)^{2({d_{i-1}} - d_{i})-1} \cdot {\rm HMi}_{d_i}(P; Q)
\big)^{2^{2s({\bm \kappa}) + 1}f_{2,i}}
\cdot
$$
$$
\cdot
\prod_{1 \le j < j' \le \# {\bm \mu}}
(t_{j'} - t_{j})
^{ 2^{2s({\bm \kappa})+1} e_1} 
\cdot
\prod_{1 \le k \le \# {\bm \nu}}b_k 
^{2^{2s({\bm \kappa})+1}(2\# {\bm \mu} + 1)e_1}
\cdot
\prod_{1 \le k < k' \le {\# {\bm \nu}}} 
{\rm R}(
z_{k},z_{k'})
^{2^{2s({\bm \kappa})+1}e_1}\cdot
$$
$$
\cdot
\prod_{1 \le j \le \# {\bm \mu}, \atop \eta_j \ne 0}Q(t_j)^{2^{2s({\bm \kappa}) + 1}f_{1,j}}
\cdot
\prod_{1 \le k \le {\# {\bm \nu}}, \atop \kappa_k = 1} 
(Q_{\re}^2(z_k) + Q_{\im}^2(z_k))^{2f''_k}
$$
with $f''_k \le 2^{2s({\bm \kappa}) - 2}(2e_1 + 1)$,
degree in $w$ bounded by 
$$2^{2s({\bm \kappa})}
\Big( 2^{2p} ( 9p^4(q+3) + 
2s({\bm \kappa})(3p + 2p^2)) + q + 2p + 6\Big)
\max\{\deg_w P, \deg_w Q\},
$$
degree in $t_j$ bounded by 
$$
2^{2s({\bm \kappa})}(p^32^{2p+1} + q + 2p -2),
$$ 
degree in $(a_k, b_k)$ bounded by 
$$
2^{2s({\bm \kappa})}\Big(2^{2p+1}( p^3 + (3p + 2p^2)q) +  6q +  2p -2\Big),
$$ 
degree in $\ell_i$ and 
degree in $\ell'_i$ bounded by $5\cdot 2^{2s({\bm \kappa})}p^32^{2p}$,
and
degree in $(a''_i, b''_i)$ bounded by $2^{2s({\bm \kappa})}p^22^{2p+1}$; where 
$s({\bm \kappa}) = \#\{k \ | \ 1 \le k \le {\# {\bm \nu}}, \kappa_k = 1\}$.

Finally, we apply to (\ref{inc:aux_final_inc_chap_5_1}) the weak inference 
from 
Theorem
\ref{thm:her_through_bez} (Hermite's Theory (2) as a weak existence) and we obtain 
$$
\Big \downarrow \
\bigwedge_{0\le i \le p-1}{\rm sign}({\rm HMi}_i(P;Q))= \tau(i), \
{\rm Fact}(P)^{{\bm \mu},{\bm \nu}}(t,z), 
$$
$$ 
\bigwedge_{1\leq j\leq \# {\bm \mu}} {\rm sign}(Q(t_j)) = \eta_j, \ 
\bigwedge_{1 \le k \le \# {\bm \nu}}{\inv}(Q(z_k)) = \kappa_k \
{\Big \downarrow}_{\K[u][ t, a, b]}
$$
with monoid part 
$$
\prod_{1 \le i \le s} {\rm HMi}_{d_i}(P; Q)^{2g_i}  
\cdot
\prod_{1 \le j < j' \le \# {\bm \mu}}
(t_{j'} - t_{j})^{ 2^{2s({\bm \kappa})+1} e_1f} \cdot
\prod_{1 \le k \le {\# {\bm \nu}}}b_k 
^{2^{2s({\bm \kappa})+1}(2\# {\bm \mu} + 1)e_1f}
$$
$$ \cdot
\prod_{1 \le k < k' \le {\# {\bm \nu}}} 
{\rm R}(
z_{k},z_{k'})
^{2^{2s({\bm \kappa})+1}e_1f}
\cdot
\prod_{1 \le j \le \# {\bm \mu}, \atop \eta_j \ne 0}Q(t_j)^{2^{2s({\bm \kappa}) + 1}f_{1,j}f}
\cdot
\prod_{1 \le k \le {\# {\bm \nu}}, \atop \kappa_k = 1} 
(Q_{\re}^2(z_k) + Q_{\im}^2(z_k))^{2f''_kf}
$$
with 
\begin{eqnarray*}
g_i  &\le &
2^{5p + 2s({\bm \kappa}) -1}p^{4p+4}(q+3)^{3p}   
(5p  + 
5 p^{p+1}(q+3)^p   
+ 10p^{p+1}(q+3)^{p+1} + 1),\\
f&\le& 2^{3p}p^{4p+2}(q+3)^{3p},
\end{eqnarray*}
degree in $w$ bounded by 
$$
2^{3p + 2s({\bm \kappa})}p^{4p+2}(q+3)^{3p}
\Big(  
2^{2p} \Big(  
2s({\bm \kappa})(3p + 2p^2) + 17p^4(q+3) +  
5p^5(q+3) + 
$$
$$
+ 15p^{p+4}(q+3)^{p+1} 
+ 
 31p^{p+4}(q+3)^{p+2}\Big)
+ q + 2p + 6 
\Big)
\max\{\deg_w P, \deg_w Q\},
$$ 
degree in $t_j$ bounded by 
$$
2^{3p + 2s({\bm \kappa})}p^{4p+2}(q+3)^{3p}(p^32^{2p+1} + q + 2p -2),
$$
and 
degree in $(a_k, b_k)$ bounded by 
$$
2^{3p + 2s({\bm \kappa})}p^{4p+2}(q+3)^{3p}
\Big(2^{2p+1}( p^3 + (3p + 2p^2)q) +  6q +  2p -2\Big).
$$
It can be easily seen that this incompatibility satisfies the required bounds to be the final incompatibility.
\end{proof}

\section{Elimination of one variable} \label{sect_elim_of_one_var}
\setcounter{equation}{0}

The main results of this section are as follows: 
given a family $\mathcal{Q}$ of univariate polynomials depending on parameters, 
first, 
to define an eliminating family ${\rm Elim}(\mathcal{Q})$ of polynomials in the parameters, 
such that the signs of the polynomials in the eliminating family ${\rm Elim}(\mathcal{Q})$ determine 
the realizable sign conditions on $\mathcal{Q}$ 
and second, to translate this statement under weak inference form.

Classical Cylindrical Algebraic Decomposition (CAD) is a well known method for constructing 
an eliminating family, containing subresultants of pairs of polynomials  
of $\mathcal{Q}$ (in the case where the polynomials are all monic with respect to the
main variable). However in classical CAD the properties of the eliminating family
are proved using properties of semi-algebraically connected components of realization of sign conditions.
Since semi-algebraic connectivity is not available in our context, we cannot use CAD.

So we need to provide a new 
elimination method.
 This new 
  elimination
 method
 uses
the fact
that each real root of a polynomial is
uniquely determined by the signs it gives to the derivatives of the polynomial (Thom encoding).
The eliminating family of $\mathcal{Q}$  will consist of principal minors of Hermite matrices of pairs of 
polynomials $Q_1,Q_2$  where $Q_1$ belongs to $\mathcal{Q}$ and $Q_2$ is the product of 
(a small number of) derivatives of $Q_1$ and at most one polynomial in  $\mathcal{Q}$ or its square, according to sign determination.
Since minors of Hermite matrices coincide with subresultants (see
Proposition \ref{prop:conn_herm_subr}), the main difference between classical CAD and the 
new elimination method
presented here is that in the new 
method 
it is not sufficient to consider subresultants of pairs of polynomials in $\mathcal{Q}$.

In order to design our new 
elimination method, we proceed in severeal steps.
In Subsection  \ref{subsecThom} we first recall the Thom encodings, which characterize the real roots of a univariate polynomial by sign conditions on the derivatives and we prove some
weak inferences related to them. 
In Subsection \ref{sect_elim_family_fact_order} we consider a univariate polynomial $P$ depending on parameters and
define a family of eliminating polynomials in the parameters whose signs determine
the Thom encodings of the real roots of $P$ (and the sign of another polynomial at these roots).

In Subsection \ref{sect_factor_family}, we consider a whole family ${\cal Q}$ of univariate polynomials depending on parameters
and define the family  ${\rm Elim}(\mathcal{Q})$ whose signs determine 
the ordered list of real roots of all the polynomials in ${\cal Q}$.
Finally, in Subsection \ref{sect_cond_family}, we  deduce that the signs of
${\rm Elim}(\mathcal{Q})$ determine  
the realizable sign conditions on the family ${\cal Q}$. 
All the results in this section are first explained in usual mathematical terms, then translated into weak inferences.

Apart from many results from Section \ref{Weak.inferences}, the only results from previous sections 
used in this section are Theorem 
\ref{thLaplacewithmult} (Real Irreducible Factors with Multiplicities as a weak existence), which is used only once
in the proof of Theorem 
\ref{weaksigndet} (Fixing the Thom encodings as a weak existence) and  
Theorem \ref{hermitetrsubresw} (Hermite's Theory as an incompatibility), which is used once
in the proof of Theorem \ref{weaksigndet} (Fixing the Thom encodings as a weak existence)
and once in the proof of Theorem \ref{weaksigndetgen}
(Fixing the Thom encodings with a Sign as a weak existence).

On the other hand, the main result of the section, Theorem \ref{elimination_theorem}
(Elimination of One Variable as a weak inference)
which describes under weak inference form the fact that the signs of
${\rm Elim}(\mathcal{Q})$ determine  
the realizable sign conditions on ${\cal Q}$ will be the only result from the rest of the paper 
used in Section \ref{secMainTh}.

\subsection{Thom encoding of real algebraic numbers}
\label{subsecThom}

We start this section with a general definition. 

\begin{definition}\label{def:signcondition}
Let $\mathcal{Q}\subset \K[u]$ with $u=(u_1,\ldots,u_k)$.
A {\rm sign condition on a set $\mathcal{Q}$} is an element of 
$\{-1, 0, 1\}^{\mathcal{Q}}.$
The {realization of a sign condition $\tau$ on $\mathcal{Q}$} is defined  by
$${\rm Real}(\tau, \R)=\{\vartheta \in \R^k \mid \bigwedge_{Q\in \mathcal{Q}} {\rm sign}(Q(\vartheta))=\tau(Q)\}.$$
We use $${\rm sign}(\mathcal{Q})=\tau$$ to mean
 $$\bigwedge_{Q\in \mathcal{Q}} {\rm sign}(Q)=\tau(Q).$$
\end{definition}

It will be 
often
convenient 
to use the following abuse of notation.

\begin{notation}
 If $\tau\in \{1,0,-1\}^{\mathcal{Q}}$ is a sign condition on $\mathcal{Q}$
 and $\mathcal{Q'}\subset{\mathcal{Q}}$, we denote again by $\tau$ the restriction 
 $\tau\mid_{\mathcal{Q'}}$  of $\tau$ to ${\mathcal{Q'}}$.
\end{notation}

Now we recall the Thom encoding of real algebraic numbers \cite{CR} and explaining 
its main properties. We refer to \cite{BPRbook} for proofs.

\begin{definition}\label{def:Thom}
Let $P
= \sum_{0 \le h \le p}\gamma_hy^h \in \K[y]$ with $\gamma_p \ne 0$. 
We denote $\Der(P)$ the list formed by the first $p-1$ derivatives of $P$
and $\Der_+(P)$ the list formed by $P$ and $\Der(P)$. 
A real root $\theta$ of $P$ is uniquely determined by the sign condition on $\Der(P)$ evaluated at  $\theta$, i.e. the list of signs of 
$\Der(P)(\theta)$,  
which is called the {\em Thom encoding} of $\theta$ with respect to $P$.

By a slight abuse of notation, we identify sign conditions on $\Der(P)$ (resp. $\Der_+(P)$), i.e. elements of
$\{1,0,-1\}^{\Der(P)}$ (resp. $\{1,0,-1\}^{\Der_+(P)}$)  with $\{-1,0,1\}^{\{1,\ldots,p-1\}}$ (resp. $\{-1,0,1\}^{\{0,\ldots,p-1\}}$).
 For any sign condition
$\eta$ on $\Der(P)$ or $\Der_+(P)$, we extend its definition with $\eta(p)
= \sign(\gamma_p)$ if needed.
\end{definition}

Thom encoding 
not only 
characterizes 
the real roots of a polynomial, it can 
also be used to order real numbers as follows.

\begin{notation}\label{not_order_thom_encoding}
Let $P
= \sum_{0 \le h \le p}\gamma_hy^h \in \K[y]$. 
For $\eta_1, \eta_2$ 
sign conditions on $\Der_+(P)$, we use the notation $\eta_1 \prec_P \eta_2$
to indicate that $\eta_1 \ne \eta_2$ and, 
if $q$ is the biggest value of $k$ such that $\eta_1(k) \ne \eta_2(k)$, then
\begin{itemize}
 \item 
$\eta_1(q)
< \eta_2(q)
$ and $\eta_1(q+1) = 1$ or
 \item 
$\eta_1(q) > \eta_2(q)$ and $\eta_1(q+1) = -1$.
\end{itemize}
We use the notation $\eta_1 \preceq_P \eta_2$ to indicate
that either $\eta_1 = \eta_2$ or $\eta_1 \prec_P \eta_2$. 
\end{notation}

\begin{proposition}\label{lemma_order_thom_encoding} 
Let $P
 =\sum_{0 \le h \le p}\gamma_hy^h \in \K[y]$ with $\gamma_p \ne 0$
and $\theta_1, \theta_2 \in \R$. If 
$\sign(\Der_+(P)(\theta_1)) \prec_P \sign (\Der_+(P)(\theta_2))$ 
then $\theta_1 < \theta_2$.
\end{proposition}

Let $\theta_1, \theta_2 \in \R$, $\eta_1 = \sign(\Der_+(P)(\theta_1))$ and  
$\eta_2 = \sign(\Der_+(P)(\theta_2))$ with $\eta_1 \ne \eta_2$, and 
let $q$ 
be
as in Notation \ref{not_order_thom_encoding}. 
Note that 
it is not possible that there exists $k$ such that $q < k < p$ 
and $\eta_1(k) = \eta_2(k) = 0$. Otherwise,
$\theta_1$ and $\theta_2$ would be roots of 
$P^{(k)}$ with the same Thom encoding with respect to this polynomial,
and therefore $\theta_1 = \theta_2$, which is impossible since 
$\eta_1 \ne \eta_ 2$. 

Next we recall the mixed Taylor formulas, 
which play a central role
in proving the weak inference version of 
these results.

\begin{proposition} 
[Mixed Taylor Formulas]
\label{prop2.4.4}
Let $P 
= y^p + \sum_{0  \le h \le p-1}\gamma_hy^h \in \K[y]$. 
For
every $\varepsilon \in \{1,-1\}^{\{1,\ldots,p\}}$ with $\varepsilon(1) = 1$,
there exist $N_{\varepsilon,1}, \dots, N_{\varepsilon,p} \in \N_*$
such that 
\begin{equation}\label{mixed_taylor_formula}
P(t_2)=P(t_1)+ \sum_{1 \le k \le p}\varepsilon(k)
\fra{N_{\varepsilon,k}}{k!}P^{(k)}(a_k)\cdot(t_2 - t_1)^k  
\end{equation}
where,  for $1 \le k \le p-1$, $a_k=t_1$ if $\varepsilon(k) = \varepsilon({k+1})$
and $a_k=t_2$ otherwise.  
\end{proposition}

Note that $a_p$ is not defined in (\ref{mixed_taylor_formula}), but this is not important
since $P^{(p)}$ is a constant.
A proof of Proposition \ref{prop2.4.4} can be found in \cite{Lom.b} and also in \cite{War}.

We prove now the weak inference version 
of
the main properties of Thom encoding.

\begin{proposition}\label{thWITL_eq} 
Let $p \ge 1$, $P
 = y^p + \sum_{0 \le h \le p-1}C_h \cdot y^h\in \K[u][ y]$,  
$\eta_1, \eta_2$ 
be
sign conditions on $\Der_+(P)$
such that exists $q$, $0 \le q \le p-1$, with $\eta_1(q)
= \eta_2(q) = 0$ and $\eta_1(k) = \eta_2(k) \ne 0$ for $q+1 \le k \le p-1$.
Then
$$
\sign (\Der_+(P)(t_1)) = \eta_1, \ \sign (\Der_+(P)(t_2)) = \eta_2  
\ \  \ \vdash \ \ \ t_1 = t_2.
$$

If we have an initial incompatibility in variables $v \supset (u, t_1, t_2)$ 
with monoid part 
$S$,
degree in $w$ bounded by $\delta_w$ for some subset of variables $w \subset v$ disjoint from 
$(t_1, t_2)$, degree in $t_1$ and degree in $t_2$ bounded by $\delta_{t}$, 
the final incompatibility has monoid part
$$
S \cdot P^{(q+1)}(t_1)^2 \cdot P^{(q+1)}(t_2)^2,
$$
degree in $w$ bounded by $2\delta_w + 14\deg_w P$ and 
degree in $t_1$ and degree in $t_2$ bounded by $2\delta_{t} + 14(p-q) -8$. 
\end{proposition}

In order to prove Proposition \ref{thWITL_eq}, we will prove first an auxiliary lemma.

\begin{lemma}\label{lem:eq_thom_enc_eq_root} 
Let $p \ge 1$, $P
 = y^p + \sum_{0 \le h \le p-1}C_h \cdot y^h\in \K[u][ y]$,  
$\eta_1, \eta_2$
be
sign conditions on $\Der_+(P)$
such that exists $q$, $0 \le q \le p-1$, with $\eta_1(q)
= \eta_2(q) = 0$ and $\eta_1(k) = \eta_2(k) \ne 0$ for $q+1 \le k \le p-1$.
Then
$$
\sign (\Der_+(P)(t_1)) = \eta_1, \ \sign (\Der_+(P)(t_2)) = \eta_2  
\ \  \ \vdash \ \ \ t_1 \le t_2.
$$

If we have an initial incompatibility in variables $v \supset (u, t_1, t_2)$ 
with monoid part 
$S$,
degree in $w$ bounded by $\delta_w$ for some subset of variables $w \subset v$ disjoint from 
$(t_1, t_2)$, degree in $t_1$ and degree in $t_2$ bounded by $\delta_{t}$, 
the final incompatibility has monoid part
$$
S \cdot P^{(q+1)}(a_1)^2
$$
where $a_1 = t_1$ if $q < p-1$ and $\eta_1(q+1)\eta_1(q+2) = -1$ and $a = t_2$ otherwise, 
degree in $w$ bounded by $\delta_w + 7\deg_w P$, 
and degree in $t_1$ and degree in $t_2$ bounded by $\delta_{t} + 7(p  -q) - 4$.\end{lemma}

\begin{proof}{Proof.}
Consider the initial incompatibility 
\begin{equation}\label{inc:init_prop_thom_enc_eq}
\lda t_1 \le t_2, \ \cH \rda 
\end{equation}
where $\cH$ is a system of sign conditions in $\K[v]$. 
For $q \le k  \le p$ we denote $\eta(k) = \eta_1(k) = \eta_2(k)$. 
If $\eta(q+1) = -1$, we change $P$ by $-P$, $\eta_1$ by $-\eta_1$ and $\eta_2$ by $-\eta_2$; so without 
loss of generality we suppose $\eta(q+1) = 1$.

The mixed Taylor formula (Proposition \ref{prop2.4.4}) for $P^{(q)}$ and 
$\varepsilon =$ 
$[\eta(q+1), -\eta(q+2)$ $\dots,$ $(-1)^{p-q-1}\eta(p)]$
provides us the  identity  
\begin{eqnarray}\label{mixed_tay_for_applied_thom_lemma}
P^{(q)}(t_2) - P^{(q)}(t_1) & = & (t_2 - t_1) \cdot  S_{o}-S_{e}
\end{eqnarray}
where
$$
\begin{array}{rcl}
 S_{o}&=& N_{\varepsilon,1}P^{(q+1)}(a_1) + \sum_{3 \le k \le p-q, \atop k \, \hbox{\scriptsize {odd }}} 
\fra{N_{\varepsilon,k}}{k!}\eta(q+k)P^{(q+k)}(a_k)\cdot(t_2 - t_1)^{k-1}, \\[4mm]
S_{e}&=&\sum_{2 \le k \le p-q, \atop k \, \hbox{\scriptsize {even }}} 
\fra{N_{\varepsilon,k}}{k!}\eta(q+k)P^{(q+k)}(a_k)\cdot(t_2 - t_1)^k.
\end{array}
$$

We successively apply to (\ref{inc:init_prop_thom_enc_eq}) the weak inferences
$$
\begin{array}{rcl}
(t_2 - t_1)\cdot S_{o} \ge 0, \ S_{o} > 0 & \ \, \vdash \, \ & t_1 \le t_2, 
\\[3mm]
(t_2 - t_1)\cdot S_o - S_e = 0, \ S_e \ge 0
& \vdash & 
(t_2 - t_1)\cdot S_o \ge 0, 
\\[3mm]
P^{(q)}(t_1) = 0, \ P^{(q)}(t_2) = 0 
& \vdash & 
(t_2 - t_1)\cdot S_o - S_e = 0.
\end{array}
$$
By Lemmas \ref{lem_prod_le_g_then_fact_le} and \ref{lemma_sum_of_pos_and_zer_is_pos} 
(items \ref{lemma_sum_of_pos_and_zer_is_pos:1} and \ref{lemma_sum_of_pos_and_zer_is_pos:2})
using (\ref{mixed_tay_for_applied_thom_lemma}),  we obtain 
\begin{equation}\label{inc:aux_prop_thom_enc_eq_2} 
\lda  S_o > 0, \ S_e \ge 0, \ P^{(q)}(t_1) = 0, \ P^{(q)}(t_2) = 0,  \
\cH
\rda 
\end{equation}
with 
monoid part $S \cdot S_{o}^2 $, degree in $w$ bounded by $\delta_w + 4\deg_w P$ and  
degree in $t_1$ and degree in $t_2$ bounded by $\delta_{t} + 4(p  -q ) - 2$.

Then we successively apply to (\ref{inc:aux_prop_thom_enc_eq_2}) the weak inferences
$$
\begin{array}{rcl}
\displaystyle{P^{(q+1)}(a_1) > 0, \
\bigwedge_{3 \le k \le p-q, \atop k \, \hbox{\scriptsize {odd }}}
\eta(q+k)P^{(q+k)}(a_k)\cdot(t_2 - t_1)^{k-1}  \ge 0}
& \ \, \vdash \, \ & 
S_o > 0,
\\[3mm]
\displaystyle{\bigwedge_{2 \le k \le p-q, \atop k \, \hbox{\scriptsize {even }}} 
\eta(q+k)P^{(q+k)}(a_k)\cdot(t_2 - t_1)^k \ge 0}
& \vdash & 
S_e \ge 0.
\end{array}
$$
By Lemmas 
\ref{lemma_sum_of_pos_is_pos} and
\ref{lemma_sum_of_pos_and_zer_is_pos} (item  \ref{lemma_sum_of_pos_and_zer_is_pos:2})
we obtain an incompatibility
\begin{equation}\label{inc:aux_prop_thom_enc_eq_3} 
\begin{array}{c} 
\displaystyle{\Big\downarrow \
P^{(q+1)}(a_1) > 0, \
\bigwedge_{3 \le k \le p-q, \atop k \, \hbox{\scriptsize {odd }}}
\eta(q+k)P^{(q+k)}(a_k)\cdot(t_2 - t_1)^{k-1}  \ge 0,} \\[4mm]
\displaystyle{ \bigwedge_{2 \le k \le p-q, \atop k \, \hbox{\scriptsize {even }}} 
\eta(q+k)P^{(q+k)}(a_k)\cdot(t_2 - t_1)^k \ge 0, 
 \ P^{(q)}(t_1) = 0, \ P^{(q)}(t_2) = 0,
 \
\cH \ \Big\downarrow }
\end{array}
\end{equation}
with 
monoid part $S \cdot P^{(q+1)}(a_1)^2$, degree in $w$ bounded by $\delta_w + 7\deg_w P$ and
degree in $t_1$ and degree in $t_2$ bounded by $\delta_{t} + 7(p  -q ) - 4$.

Then we successively apply to (\ref{inc:aux_prop_thom_enc_eq_3}) for odd $k$,  $3 \le k \le p-q$,
the weak inferences
$$
\begin{array}{rcl}
\eta(q+k)P^{(q+k)}(a_k) \ge 0, \ (t_2 - t_1)^{k-1} \ge 0 
 & \ \, \vdash \, \ &
\eta(q+k)P^{(q+k)}(a_k)\cdot(t_2 - t_1)^{k-1} \ge 0, 
\\[3mm]
\sign(P^{(q+k)}(a_k)) = \eta(q+k) & \vdash & \eta(q+k)P^{(q+k)}(a_k) \ge 0,
 \\[3mm]
 & \vdash & (t_2 - t_1)^{k-1} \ge 0,
\end{array}
$$
and for even $k$, $2 \le k \le p-q$, the weak inference
$$
\begin{array}{rcl}
\eta(q+k)P^{(q+k)}(a_k) \ge 0 , \ (t_2 - t_1)^{k} \ge 0 
 & \ \, \vdash \, \ &
\eta(q+k)P^{(q+k)}(a_k)\cdot(t_2 - t_1)^k  \ge 0,
\\[3mm]
\sign(P^{(q+k)}(a_k)) = \eta(q+k) & \vdash & \eta(q+k)P^{(q+k)}(a_k) \ge 0,
\\[3mm]
 & \vdash & (t_2 - t_1)^{k} \ge 0. 
\end{array}
$$
By Lemma \ref{lemma_basic_sign_rule_1} (items 
\ref{lemma_basic_sign_rule:1},  
\ref{lemma_basic_sign_rule:2} and \ref{lemma_basic_sign_rule:6}) we obtain 
$$
\lda \sign (\Der_+(P)(t_1)) = \eta_1, \ \sign (\Der_+(P)(t_2)) = \eta_2,  
 \ \cH \rda
$$
with the same monoid
part and degree bounds. 
\end{proof}

We can prove now Proposition \ref{thWITL_eq}.

\begin{proof}{Proof of Proposition \ref{thWITL_eq}.}
Consider the initial incompatibility 
\begin{equation}\label{inc:init_prop_thom_enc_eq_prop}
\lda t_1 = t_2, \ \cH \rda 
\end{equation}
where $\cH$ is a system of sign conditions in $\K[v]$. 

We successively apply to (\ref{inc:init_prop_thom_enc_eq_prop}) the weak inferences
$$
\begin{array}{rcl}
t_1 \ge t_2, \ t_1 \le t_2 & \ \, \vdash \, \ & t_1 = t_2, \\[3mm] 
\sign (\Der_+(P)(t_1)) = \eta_1, \ \sign (\Der_+(P)(t_2)) = \eta_2  
& \vdash & t_1 \le t_2, \\[3mm]
\sign (\Der_+(P)(t_2)) = \eta_2, \ \sign (\Der_+(P)(t_1)) = \eta_1  
& \vdash & t_2 \le t_1. 
\end{array}
$$
By Lemmas \ref{lemma_greatereq_and_lowereq_then_eq} and \ref{lem:eq_thom_enc_eq_root},
we obtain 
$$
\lda \sign (\Der_+(P)(t_1)) = \eta_1, \ \sign (\Der_+(P)(t_2)) = \eta_2 ,  \  \cH \rda 
$$
with monoid part
$
S\cdot P^{(q+1)}(t_1)^2\cdot P^{(q+1)}(t_2)^2,
$
degree in $w$ bounded by $2\delta_w + 14\deg_w P$ and  
degree in $t_1$ and 
degree in $t_2$ bounded by $2\delta_{t} + 14(p-q) -8$,
which serves as the final incompatibility.
\end{proof}

\begin{proposition}\label{thWITL_lower} Let $p \ge 1, 
P
 = y^p + \sum_{0 \le h \le p-1}C_h \cdot y^h\in \K[u][ y]$, 
$\eta_1, \eta_2$ 
be
sign conditions on $\Der_+(P)$
such that exists $q$, $0 \le q \le p-1$, with $\eta_1(q)
\ne \eta_2(q)$ and $\eta_1(k) = \eta_2(k) \ne 0$ for $q+1 \le k \le p-1$, 
and $\eta_1 \prec_P \eta_2$. 
Then
$$
\sign(\Der_+(P)(t_1)) = \eta_1, \ \sign (\Der_+(P)(t_2)) = \eta_2  
\ \ \ \vdash \ \ \ t_1 < t_2.
$$
If we have an initial incompatibility in variables $v \supset (u, t_1, t_2)$ 
with monoid part 
$
S \cdot (t_2 - t_1)^{2e}
$ with $e \ge 1$, degree in $w$ bounded by  $\delta_w$ for some subset of variables $w \subset v$ disjoint
from $(t_1, t_2)$, degree in $t_1$ and degree in $t_2$ 
bounded by $\delta_{t}$, the final incompatibility has monoid part 
$$
S \cdot P^{(q)}(b)^{2e}
$$
with $b = t_2$ if $\eta_2(q) \ne 0$ and
$b = t_1$ otherwise,  degree in $w$ bounded by $\delta_w + (6e + 2)\deg_w P$ and
degree in $t_1$ and degree in $t_2$ bounded by $\delta_{t} + (6e+2)p$.
\end{proposition}

\begin{proof}{Proof.}
The proof is an adaptation of the proof of Lemma \ref{lem:eq_thom_enc_eq_root}. 
For $q+1 \le k \le p$ we denote  $\eta(k) = \eta_1(k) = \eta_2(k)$. If 
$\eta(q+1) = -1$,  we change 
$P$ by $-P$,  $\eta_1$ by $-\eta_1$ and $\eta_2$ by $-\eta_2$; so without loss of generality
we suppose $\eta(q+1) = 1$. 
We replace the first three weak inferences in the proof of Lemma \ref{lem:eq_thom_enc_eq_root}
by  
$$
\begin{array}{rcl}
(t_2 - t_1)S_{o} > 0, \ S_{o} > 0 & \ \, \vdash \, \ & t_1 < t_2, 
\\[3mm]
(t_2 - t_1)S_o - S_e > 0, \ S_e \ge 0
& \vdash & 
(t_2 - t_1)S_o > 0, 
\\[3mm]
\sign(P^{(q)}(t_1)) = \eta_1(q), \ \sign(P^{(q)}(t_2)) = \eta_2(q) 
& \vdash & 
(t_2 - t_1)S_o - S_e > 0.
\end{array}
$$
In fact, just for the case  $\eta_1(q) = -1$ and $\eta_2(q) = 1$, also the 
weak inference
$$
P^{(q)}(t_1) < 0 \ \ \ \vdash \ \ \ P^{(q)}(t_1) \le 0 
$$
from Lemma \ref{lemma_basic_sign_rule_1} (item \ref{lemma_basic_sign_rule:1}) is also needed between the second and 
third weak inference above. 
By Lemmas \ref{lem_prod_g_g_then_fact_g}, \ref{lemma_sum_of_pos_is_pos}
and possibly \ref{lemma_basic_sign_rule_1} (item \ref{lemma_basic_sign_rule:1}), 
we obtain 
$$\lda  S_o > 0, \ S_e \ge 0, \ 
\sign(P^{(q)}(t_1)) = \eta_1(q), \ \sign(P^{(q)}(t_2)) = \eta_2(q),   \
\cH
\rda 
$$
with 
monoid part $S \cdot P^{(q)}(b)^{2e}$ with $b = t_2$ if $\eta_2(q) = 1$ and
$b = t_1$ otherwise,  degree in $w$ bounded by $\delta_w + 6e\deg_w P$ and  
degree in $t_1$  and degree in $t_2$ bounded by 
$\delta_{t} + 2e(3(p  -q ) - 1)$.

The rest of the proof is as in the proof of Lemma \ref{lem:eq_thom_enc_eq_root}.
\end{proof}

\subsection{Conditions on the parameters fixing the Thom encoding}
\label{sect_elim_family_fact_order}

Given $P
, Q
\in \K[u][y]$, with $P
$ monic in $y$ and  $u=(u_1,\ldots,u_k)$,
our goal is to define a family of polynomials in $\K[u]$ whose signs fix the
Thom encoding of the real roots of $P
$ and the signs of $Q
$ at these roots;
the family composed by the principal minors of Hermite matrices of $P$ and 
products of (a small number of) its derivatives
with $1, Q$ or $Q^2$ has this property by sign determination (see \cite[Theorem 27]{PerR}).

We introduce some notation and definitions.

\begin{notation}\label{not:mult} 
Let $P
\in \K[u][y]$ monic in $y$ with $\deg_y P =p \ge 1$. 

For $\eta \in \{-1,0,1\}^{\Der(P)}$, we denote by $\eta_+ \in \{-1,0,1\}^{\Der_+(P)}$
the extension of $\eta$ to $\Der_+(P)$ given by $\eta_+(0) = 0$. 

For $\eta \in \{-1,0,1\}^{\Der_+(P)}$, the 
number ${\rm mu}(\eta,P)$  
is the smallest index $i$, $0 \le i \le p$,
such that $\eta(i)\not=0$.
Note that if the real root $\theta$ of $P$ has Thom encoding $\eta$,
the multiplicity of $\theta$ as a root of $P$ is ${\rm mu}(\eta,P)$.

For $\eta \in \{-1,0,1\}^{\Der(P)}$, the 
number ${\rm mu}(\eta,P)$  
is ${\rm mu}(\eta_+,P)$.

For a list of distinct sign conditions
${\bm \eta}= [\eta_1 , \dots , \eta_{\#{\bm \eta}}]$ 
on ${\Der(P)}$,  
the  
vector ${\rm vmu}({\bm \eta})$ is the list 
${\rm mu}(\eta_{1},P),\ldots,{\rm mu}(\eta_{{\#{\bm \eta}}},P)$
in non-increasing order. 

We define the order $\prec_{P}^{{\rm mu}}$ on $\{-1, 0, 1\}^{\Der(P)}$, 
given by $\eta_1 \prec_{P}^{{\rm mu}} \eta_2$ if 
${\rm mu}(\eta_{1},P) > 
{\rm mu}(\eta_{2},P)$ or 
${\rm mu}(\eta_{1},P) = 
{\rm mu}(\eta_{2},P)$ and
$\eta_{1, +} \prec_{P} \eta_{2, +}$.
\end{notation}

\begin{definition}\label{def:Thom_system} 
Let  $p \ge 1$,  
$P
 = y^p + \sum_{0 \le h \le p-1} C_h \cdot y^h \in \K[u][y]$,   
$({\bm \mu}, {\bm \nu}) \in \Lambda_m \times \Lambda_n$ with $m+2n = p$,
${\bm \eta}=  [\eta_1 , \dots , \eta_{\#{\bm \eta}}]$  
be
a 
list of distinct
sign conditions on $\Der(P)$ with $\#{\bm \mu} =  \#{\bm \eta}$,  
$t = (t_1, \dots, t_{\#{\bm \mu}})$ and 
$z = (z_1, \dots, z_{\# {\bm \nu}})$.
We define the system of sign conditions 
$${\rm Th}(P)^{{\bm \mu}, {\bm \nu}, {\bm \eta}}(t, z)$$
in $\K[u][ t, a, b]$
as
$$
{\rm Fact}(P)^{{\bm \mu}, {\bm \nu}}(t, z), \
\bigwedge_{1 \le j \le \#{\bm \mu}}  {\rm sign}(\Der(P)(t_j)) = \eta_j.
$$
\end{definition} 

Note that in Definition \ref{def:Thom_system}, since the multiplicity of the real roots of $P$ can be read both 
from ${\bm \mu}$ and ${\bm \eta}$, there should be some restrictions on ${\bm \mu}$ and ${\bm \eta}$
in order that the system ${\rm Th}(P)^{{\bm \mu}, {\bm \nu}, {\bm \eta}}(t, z)$ admits a real solution. 
Nevertheless, we will still need the definition in the general case, with 
the only restriction on ${\bm \mu}$ and ${\bm \eta}$ given by 
$\#{\bm \mu} = \#{\bm \eta}$.

\begin{definition} \label{notaDPTP}
Let $p \ge 1$,  $P
 = y^p + \sum_{0 \le h \le p-1} C_h \cdot y^h\in \K[u][y]$, $Q
  \in \K[u][y]$ and
$i \in \N$. We define
$$\begin{array}{rcl}
 {\rm PDer}_i(P) &=& \Big\{ \displaystyle{\prod_{1\le h \le p-1} (P^{(h)}) ^{\alpha_h}}
\ | \ 
\alpha 
\in {\{0,1,2\}} ^{\{1,\ldots,p-1\}}, \, \#\{h  \ | \ \alpha_h \ne 0\}\le i \Big\}
\subset \K[u][y],
\\[4mm]
{\rm PDer}_{i}(P;Q) &= &\{AB \ | \ A \in {\rm PDer}_{i}(P), \, 
B \in \{Q, Q^2\} \} \subset \K[u][ y], 
\\[4mm]
{\rm ThElim}(P)& = &\displaystyle{\bigcup_{A \in 
{\rm PDer}_{{\rm bit}\{p\}}(P)} } {\rm HMi}(P; A)\subset \K[u],
\\[4mm]
{\rm ThElim}(P;Q)& = & \displaystyle{\bigcup_{A \in {
{\rm PDer}}_{{\rm bit}\{p\}-1}(P;Q)} } {\rm HMi}(P; A)\subset \K[u].
\end{array}$$
\end{definition}

The following two results show the connection between signs conditions on the sets
${\rm ThElim}(P)$ and ${\rm ThElim}(P;Q)$ and the Thom encodings of the real roots of $P$ and 
the sign of $Q$ at these roots.

\begin{theorem}[Fixing the Thom encodings]\label{ThomwithP}
 Let $p \ge 1$,  $P
 = y^p + \sum_{0 \le h \le p-1} C_h \cdot y^h\in \K[u][y]$.
 For every  realizable  sign condition $\tau$ on ${\rm ThElim}(P)$, there exist
unique $({\bm \mu}(\tau),{\bm \nu}(\tau)) \in \Lambda_m \times \Lambda_n$ with $m + 2n = p$, and 
a unique list ${\bm \eta}(\tau)$ 
of distinct
sign conditions on $\Der(P)$ ordered with respect to $\prec_P^{\rm mu}$
 such that for every $\vartheta \in {\rm Real}(\tau, \R)$
there exist 
$\theta  \in \R^{{\#{\bm \mu}(\tau)}}, 
\alpha \in \R^{\# {\bm \nu}(\tau)}, 
\beta  \in \R^{\# {\bm \nu}(\tau)}$ such that
$${\rm Th}(P(\vartheta
))^{{\bm \mu}(\tau), {\bm \nu}(\tau), {\bm \eta}(\tau)}(\theta, \alpha + i\beta).$$
\end{theorem}

\begin{proof}{Proof.}
As said in Theorem \ref{bezoutiansignature} (Hermite's Theory (2)),  
a sign condition $\tau$ on ${\rm ThElim}(P)$ 
determines the rank and signature of ${\rm Her}(P; A)$ for every 
$A \in {\rm PDer}_{{\rm bit}\{p\}}(P)$.
By sign determination \cite[Theorem 27]{PerR}, this 
is enough to determine the 
decomposition of $P$ into ireducible real factors
and the
Thom encodings of the real roots of $P$. 
\end{proof}

\begin{theorem}[Fixing the Thom encodings with a Sign]\label{ThomwithPQ}
  Following the notation of Theorem \ref{ThomwithP}, for every realizable sign condition $(\tau,\tau')$ on ${\rm ThElim}(P)\cup{\rm ThElim}(P;Q)$, 
there exists 
a unique list ${\bm \epsilon}(\tau,\tau')=[\epsilon_1(\tau,\tau'),\ldots, \epsilon_{{\#{\bm \mu}(\tau)}}(\tau,\tau')]$
of signs 
 such that for every $\vartheta \in {\rm Real}((\tau,\tau'), \R)$
 there exist 
$\theta  \in \R^{{\#{\bm \mu}}(\tau)}, 
\alpha \in \R^{\# {\bm \nu}(\tau)}, 
\beta  \in \R^{\# {\bm \nu}(\tau)}$ such that
 $$
{\rm Th}(P(\vartheta
))^{{\bm \mu}(\tau), {\bm \nu}(\tau), {\bm \eta}(\tau)}(\theta, \alpha + i\beta), \ 
\bigwedge_{1 \le j \le \# {\bm \mu}(\tau)} \s(Q(\theta_j)) = \epsilon_j(\tau, \tau').
$$
\end{theorem}

\begin{proof}{Proof :}
The claim follows using Theorem \ref{ThomwithP} 
and the fact that a
sign condition
$\tau'$ on ${\rm ThElim}(P;Q)$ 
additionally determines the rank and signature of ${\rm Her}(P; A)$ for every 
$A \in {\rm PDer}_{{\rm bit}\{p\} - 1}(P; Q)$, and therefore, by sign determination \cite[Theorem 27]{PerR},
the signs of $Q$ at the real roots of $P$.
\end{proof}

Before giving the weak inference versions of Theorems \ref{ThomwithP} and \ref{ThomwithPQ}, we define new auxiliary functions
(see Definitions \ref{defg4} and \ref{def:func_aux_hilb}). 

\begin{definition}\label{def:func_aux_hilb_1_and_2}
\begin{enumerate}
\item Let ${\rm g}_{H,1} : \N_* \to \N$, ${\rm g}_{H,1}\{p\} = {\rm g}_H\{p, 2{\rm bit}\{p\}(p-1)\}$.

\item Let $\tilde{\rm g}_{H,1} : \N_* \to \N$, $\tilde{\rm g}_{H,1}\{p\} = 
{\rm bit}\{p\}2^{2^{\frac12(p-1)p+2} -2}{\rm g}_{H,1}\{p\}^{2^{\frac12(p-1)p}-1}({\rm g}_{H,1}\{p\} + 2)$.

\item Let ${\rm g}_{H,2} : \N_* \times \N \to \N$, ${\rm g}_{H,2}\{p, q\} = 
{\rm g}_H\{p,2({\rm bit}\{p\}-1)(p-1)+2q\}$.

\item Let $\tilde{\rm g}_{H,2} : \N_* \times \N \to {\mathbb R}$, $\tilde{\rm g}_{H,2}\{p,q\} = 
{\rm bit}\{p\}2^{2^{\frac12p^2+2} -2}{\rm g}_{H,2}\{p,q\}^{2^{\frac12p^2}-1}({\rm g}_{H,2}\{p,q\} + 2)$.

\item Let ${\rm g}_5 : \N \times \N \times \N \times \N \times \N \to {\mathbb R}$,
$$ {\rm g}_5\{p, e, f, g, e'\}
={\rm g}_4\{p\}
\max\{e', \tilde{\rm g}_{H,1}\{p\}\}^{2^{\frac32p^2}}
\max\{ e, g, \tilde{\rm g}_{H,1}\{p\}\}^{2^{\frac{1}{2}p^2}}
\max\{ f, \tilde{\rm g}_{H,1}\{p\}\}^{2^{\frac{1}{2}p}}.
$$

\end{enumerate}
 
\end{definition}

\begin{Tlemma}\label{tlem:aux_theorem_fix_Thom} For every 
$(p, e, f, g, e') \in \N_* \times \N \times \N \times \N \times \N$, 
$$
2^{p+ (((p-1)p+2)2^{(p-1)p}  - 2 )(2^{\frac12p^2} + 2^{\frac12p} + 1)}{\rm g}_4\{p\} 
 \max\{e', \tilde{\rm g}_{H,1}\{p\}\}^{(2^{(p-1)p}-1)(2^{\frac{1}{2}p^2}+2^{\frac{1}{2}p}+1) + 1}\cdot
$$
$$
\cdot \max\{ e, g, \tilde{\rm g}_{H,1}\{p\} \}^{2^{\frac{1}{2}p^2}}
\max\{ f, \tilde{\rm g}_{H,1}\{p\} \}^{2^{\frac{1}{2}p}}\le 
$$
$$
\le {\rm g}_5\{p, e, f, g, e'\}.
$$
\end{Tlemma}
\begin{proof}{Proof.} See Section \ref{section_annex}.
\end{proof}

Now, we first give weak inference versions of Theorems \ref{ThomwithP} and \ref{ThomwithPQ}, and then the proofs of them.

\begin{theorem}[Fixing the Thom encodings as a weak existence]\label{weaksigndet} 
Let  $p \ge 1$,  
$P
= y^p + \sum_{0 \le h \le p-1} C_h \cdot y^h$ $\in \K[u][y]$ and   
$\tau$ 
be
 a realizable 
sign condition on ${\rm ThElim}(P)$. Then, using the notation of Theorem \ref{ThomwithP},
 $$
\s({\rm ThElim}(P))=\tau \ \ \ \vdash \ \ \
 \exists (t, z) \ 
[\; {\rm Th}(P)^{{\bm \mu}(\tau), {\bm \nu}(\tau), {\bm \eta}(\tau)}(t, z) \;]
$$
where $t =   (t_1, \dots, t_{\#{\bm \mu}(\tau)})$ and
$z =   (z_1, \dots, z_{\# {\bm \nu}(\tau)})$.

Suppose we have an initial incompatibility 
in  $\K[v][ t, a, b]$, 
where
$v \supset u$, and $t, a, b$ are disjoint from $v$, with monoid part
$$
S \cdot 
\prod_{1 \le j < j' \le \# {\bm \mu}(\tau)}(t_{j} - t_{j'})^{2e_{{j}, j'}} 
\cdot
\prod_{1 \le k \le \# {\bm \nu}(\tau)}b_{k}^{2f_{k}}
\cdot 
$$
$$
\cdot
\prod_{1 \le k < k' \le \# {\bm \nu}(\tau)}
{\rm R}( z_{k}, z_{k'})
^{2g_{{k}, k'}} 
\cdot
\prod_{1 \le j  \le \# {\bm \mu}(\tau), \, 1 \le h \le p-1,\atop
\eta_j(\tau)(h) \ne 0
} 
P^{(h)}(t_j)^{2e'_{j,h}}
$$
with
$e_{j, j'} \le e$, 
$f_{k} \le f$, 
$g_{k, k'} \le g$, 
$e'_{j,h} \le e'$, 
degree in $w$ bounded by $\delta_{w}$ for some subset of variables $w \subset v$,
degree in $t_{j}$ bounded by $\delta_{t}$ 
and 
degree in $(a_{k}, b_{k})$ bounded by $\delta_{z}$. Then 
the final incompatibility has monoid part 
$$
S^h \cdot 
\prod_{H \in {\rm ThElim}(P), \atop \tau(H) \ne 0} H^{2h'_H} 
$$
with
$h, h'_H \le {\rm g}_5\{p, e, f, g, e'\}$,
and degree in $w$ bounded by 
$$
{\rm g}_5\{p, e, f, g, e'\}\Big(
\max\{\delta_w , 
\tilde{\rm g}_{H,1}\{p\}\deg_w P
\} + 
\max\{\delta_t , \delta_z ,
\tilde{\rm g}_{H,1}\{p\}\}\deg_wP
\Big).
$$
\end{theorem}

\begin{theorem}[Fixing the Thom encodings with a Sign as a weak existence] \label{weaksigndetgen} 
Let  $p \ge 1$,  
$P
 = y^p + \sum_{0 \le h \le p-1} C_h\cdot y^h \in \K[u][y]$, $Q 
 \in \K[u][y]$ with
$\deg_y Q = q$ and  $\tau$ and $\tau'$   
be
sign conditions on 
${\rm ThElim}(P)$ and ${\rm ThElim}(P;Q)$
respectively such that $(\tau, \tau')$ is a realizable sign condition on 
${\rm ThElim}(P) \cup {\rm ThElim}(P;Q)$. Then using the notation of Theorem \ref{ThomwithPQ},
$$
\s({\rm ThElim}(P;Q))= \tau', \
{\rm Th}(P)^{{\bm \mu}(\tau), {\bm \nu}(\tau), {\bm \eta}(\tau)}(t, z)
 \ \ \ 
\vdash \ \ \ 
\bigwedge_{1 \le j \le \# {\bm \mu}(\tau)} \s(Q(t_j)) = \epsilon_j(\tau, \tau')$$
where $t =   (t_1, \dots, t_{\# {\bm \mu}(\tau)})$ and 
$z = (z_1, \dots, z_{\# {\bm \nu}(\tau)})$.

Suppose we have an initial incompatibility in $\K[v]$, where $v \supset (u, t, a, b)$, 
with monoid part
$$
S \cdot
\prod_{1 \le j  \le \# {\bm \mu}(\tau), \atop \epsilon_j(\tau, \tau') \ne 0} 
Q(t_j)^{2h_{j}},
$$
with
$h_{j} \le h$, 
degree in $w$ bounded by $\delta_{w}$ for some subset of variables $w \subset v$ disjoint from 
$(t, a, b)$,
degree in $t_{j}$ bounded by $\delta_{t}$ 
and
degree in $(a_{k}, b_{k})$ bounded by $\delta_{z}$.
Then, the final incompatibility has monoid part 
$$
S^{h'}\cdot 
\prod_{H \in {\rm ThElim}(P;Q), \atop \tau'(H) \ne 0} H^{2h'_H} 
\cdot
\prod_{1 \le j < j' \le \# {\bm \mu}(\tau)}(t_{j} - t_{j'})^{2e_{{j}, j'}} 
\cdot
\prod_{1 \le k \le \# {\bm \nu}(\tau)}b_{k}^{2f_{k}}
\cdot
$$
$$
\cdot
\prod_{1 \le k < k' \le \# {\bm \nu}(\tau)}
{\rm R}(z_{k}, z_{k'})
^{2g_{{k}, k'}} 
\cdot 
\prod_{1 \le j  \le \# {\bm \mu}(\tau),
1 \le h \le p-1, \atop \eta_j(\tau)(h) \ne 0}
P^{(h)}(t_j)^{2e'_{h,j}}
$$
with
$
h' \le 
2^{(p + 2)2^{p}-2p-2}
\max\{h, \tilde{\rm g}_{H,2}\{p,q\}\}^{2^{p}-1}$, 
$h'_H, 
e_{j, j'},  
f_{k},
g_{k, k'}, 
 \le  
2^{(p + 2)2^{p}-2} \max\{
h, \tilde{\rm g}_{H,2}\{p,q\}\}^{2^{p}-1}\tilde {\rm g}_{H,2}\{p,q\}$,
degree in $w$ bounded by 
$$
2^{(p + 2)2^{p}-2}
\max\{
h, \tilde{\rm g}_{H,2}\{p,q\}\}^{2^{p}-1}
\max\{
\delta_w, \tilde{\rm g}_{H,2}\{p,q\}\max\{\deg_w P, \deg_wQ\} \},
$$
degree in $t_{j}$ bounded by 
$$
2^{(p + 2)2^{p}-2}
\max\{h, \tilde{\rm g}_{H,2}\{p,q\}\}^{2^{p}-1}
\max\{
\delta_t, \tilde{\rm g}_{H,2}\{p,q\}\}
$$
and
degree in $(a_{k}, b_{k})$ bounded by 
$$
2^{(p + 2)2^{p}-2}
\max\{h, \tilde{\rm g}_{H,2}\{p,q\}\}^{2^{p}-1}
\max\{
\delta_z, \tilde{\rm g}_{H,2}\{p,q\}\}.
$$
\end{theorem}

\begin{proof}{Proof of Theorem \ref{weaksigndet}.}
Consider the initial incompatibility 
\begin{equation} \label{inc:init_theorem_weaksigndet}
\Big\downarrow \,
{\rm Th}(P)^{{\bm \mu}(\tau), {\bm \nu}(\tau), {\bm \eta}(\tau)}(t, z), \
\cH \, \Big\downarrow_{\K[v][ t, a, b] }
\end{equation} 
where $\cH$ is a system of sign conditions in $\K[v]$. 

In order to proceed by case by case reasoning,
our first aim is to obtain incompatibilities 
$$\Big\downarrow \,
\sign({\rm ThElim}(P)) = \tau, \ {\rm Th}(P)^{{\bm \mu}(\tau), {\bm \nu}(\tau), {\bm \eta}}, \
\cH \, \Big\downarrow_{\K[v][ t, a, b] }$$
for every list of sign condition ${\bm \eta} = [\eta_1, \dots, \eta_{\#{\bm \mu}(\tau)}]$ on 
$\Der(P)$, including those ${\bm \eta}$ such that the system 
${\rm Th}(P)^{{\bm \mu}(\tau), {\bm \nu}(\tau), {\bm \eta}}(t, z)$
has obviously no solution because of some real root of $P$ having two different multiplicities according to 
${\bm \mu}(\tau)$ and ${\bm \eta}$.

We consider first the case that 
$ {\bm \eta}$ can be obtained from $ {\bm \eta}(\tau)$  
through  permutations of elements corresponding to real  roots 
with the same multiplicity. In this case, 
by simply renaming variables within the set of variables $t$ 
in (\ref{inc:init_theorem_weaksigndet}),
we  obtain 
\begin{equation} \label{inc:aux_theorem_weaksigndet_1}
\lda \,
{\rm Th}(P)^{{\bm \mu}(\tau), {\bm \nu}(\tau), {\bm \eta}}(t, z), \
\cH \, \rda
_{\K[v][ t, a, b]}
\end{equation} 
with the same monoid part up to permutations within  $t$ 
and the same degree bounds. 

We consider now the case that ${\bm \eta}$ cannot be obtained from $ {\bm \eta}(\tau)$ 
through  permutations as above. Let ${\bm \kappa}= [\kappa_1, \dots, \kappa_{\# {\bm \nu}(\tau)}]$
be a list of invertibility conditions on $\Der(P)$. 
By 
Theorem \ref{ThomwithP} (Fixing the Thom encodings) there exists $\alpha \in \{0, 1, 2\}^{1, \dots, p-1}$ 
with $\#\{h  \ | \ \alpha_h \ne 0\} \le {\rm bit}\{p\}$
such that $Q = \prod_{1 \le h \le p-1}(P^{(h)})^{\alpha_h} \in {\rm PDer}_{{\rm bit}\{p\}}(P)$ 
verifies
$$
({\rm Rk}_{\rm HMi}(\tau) , 
{\rm Si}_{\rm HMi}(\tau)) \ne 
({\rm Rk}_{\rm Fact}({\bm \eta}^\alpha, {\bm \kappa}^\alpha), 
{\rm Si}_{\rm Fact}({\bm \eta}^\alpha)),
$$
where ${\bm \eta}^\alpha$ is the list of 
sign conditions satisfied by $Q$ 
on $t$ 
when ${\bm \eta}$ is the 
list of 
sign conditions satisfied by $\Der(P)$ 
on $t$ 
and 
${\bm \kappa}^\alpha$ is defined analogously. 
By Theorem \ref{hermitetrsubresw} (Hermite's Theory as an incompatibility)
there is an incompatibility
\begin{equation}\label{inc:aux_theorem_weaksigndet_2}
\begin{array}{c}
\Big\downarrow  \
\sign({\rm ThElim}(P)) = \tau, \ 
{\rm Fact}(P)^{{\bm \mu}(\tau), {\bm \nu}(\tau)}(t,z), \\[4mm]
\displaystyle{\bigwedge_{1\leq j\le \# {\bm \mu}(\tau)}  {\rm sign}(Q(t_j)) =  \eta_j^{\alpha},  \
\bigwedge_{1\leq k \le \# {\bm \nu}(\tau)} {\rm inv}(Q(z_k)) =  \kappa_k^{\alpha}  }
\ \Big\downarrow
_{\K[u][ t, a, b]}
\end{array}
\end{equation} 
with monoid part 
$$
\prod_{H \in {\rm HMi}(P; Q), \atop \tau(H) \ne 0}H^{2\tilde g_H}
\cdot
\prod_{1 \le j < j' \le \# {\bm \mu}(\tau)}(t_{j} - t_{j'})^{2\tilde e_{{j}, j'}} 
\cdot
\prod_{1 \le k \le \# {\bm \nu}(\tau)}b_{k}^{2\tilde f_{k}}
\cdot
$$
$$
\cdot
\prod_{1 \le k < k' \le \# {\bm \nu}(\tau)}
{\rm R}( z_{k}, z_{k'})
^{2\tilde g_{{k}, k'}} 
\cdot 
\prod_{1 \le j \le \# {\bm \mu}(\tau), \atop \eta^{\alpha}_j \ne 0}Q(t_j)^{2\tilde e'_j}
\prod_{1 \le k \le \# {\bm \nu}(\tau), \atop \kappa^{\alpha}_k \ne 0}(Q^2_{\re}(z_k) + Q^2_{\im}(z_k))^{2\tilde f'_k}
$$
with 
$\tilde g_H, \tilde e_{j, j'}, \tilde f_k, \tilde g_{k, k'}, 
\tilde e'_j, \tilde f'_k \le {\rm g}_{H,1}\{p\}$, 
degree in $w$ bounded by $2{\rm bit}\{p\}{\rm g}_{H,1}\{p\}\deg_w P$ and 
degree in $t_j$ and
degree in $(a_k, b_k)$ bounded by ${\rm g}_{H,1}\{p\}$.

Since the sign and invertibility of a product is determined by 
the sign and invertibility of each factor,
by applying to (\ref{inc:aux_theorem_weaksigndet_2}) the weak inferences in 
Lemmas \ref{lemma_basic_sign_rule_1} (items \ref{lemma_basic_sign_rule:4},
\ref{lemma_basic_sign_rule:5} and
\ref{lemma_basic_sign_rule:7}) and \ref{lem:comb_lin_zero_zero},  we obtain   
\begin{equation}\label{inc:aux_theorem_weaksigndet_3}
\lda \sign({\rm ThElim}(P)) = \tau, \  
{\rm Th}(P)^{{\bm \mu}(\tau), {\bm \nu}(\tau), {\bm \eta}}(t, z), \
\bigwedge_{1 \le k  \le \# {\bm \nu}(\tau)} {\rm inv}(\Der(P)(z_k)) = \kappa_k
 \rda
 _{\K[u][ t, a, b]}
\end{equation} 
with  monoid part 
$$
\prod_{H \in {\rm HMi}(P; Q), \atop \tau(H) \ne 0}H^{2\tilde g_H}
\cdot
\prod_{1 \le j < j' \le \# {\bm \mu}(\tau)}(t_{j} - t_{j'})^{2\tilde e_{{j}, j'}} 
\cdot
\prod_{1 \le k \le \# {\bm \nu}(\tau)}b_{k}^{2\tilde f_{k}}
\cdot
\prod_{1 \le k < k' \le \# {\bm \nu}(\tau)}
{\rm R}( z_{k}, z_{k'})
^{2\tilde g_{{k}, k'}} 
\cdot
$$
$$
\cdot 
\prod_{1 \le j  \le \# {\bm \mu}(\tau), \, 1 \le h \le p-1, \atop \eta^{\alpha}_j \ne 0} 
P^{(h)}(t_j)^{2\alpha_h \tilde e'_{j}}
\cdot
\prod_{1 \le k  \le \# {\bm \nu}(\tau), \, 1 \le h \le p-1, \atop \kappa^{\alpha}_k \ne 0} 
({P^{(h)}_{\re}}(z_k)^2  + {P^{(h)}_{\im}}(z_k)^2 )^{2\alpha_h \tilde f'_{k}},
$$
degree in $w$ bounded by $2{\rm bit}\{p\}({\rm g}_{H,1}\{p\}+1)\deg_w P$,
and degree in $t_j$ 
and
degree in $(a_k, b_k)$ bounded by 
${\rm g}_{H,1}\{p\}$.
Note that Lemma \ref{lem:comb_lin_zero_zero} is used for the weak inference saying that,
for $1 \le k \le \#{\bm \nu}(\tau)$, ${\rm inv}(Q(z_k)) = 0$ when the invertibility of some factor of $Q$ at $z_k$ is $0$.

Then we successively apply to (\ref{inc:aux_theorem_weaksigndet_3}) 
the weak inferences
$$
\sum_{1 \le k \le \# {\bm \nu}(\tau),
1 \le h \le p-1, \atop \kappa_k(h) = 0} 
({P^{(h)}_{\re}}(z_k)^2 + {P^{(h)}_{\im}}(z_k)^2 )= 0
\ \ \ \vdash
\ \ \ 
\bigwedge_{1 \le k \le \# {\bm \nu}(\tau),
1 \le h \le p-1, \atop \kappa_k(h) = 0} 
P^{(h)}_{\re}(z_k) = 0, \ P^{(h)}_{\im}(z_k) = 0
$$
and
$$
\bigwedge_{1 \le k \le \# {\bm \nu}(\tau),
1 \le h \le p-1, \atop \kappa_k(h) = 0} 
{P^{(h)}_{\re}}(z_k)^2 + {P^{(h)}_{\im}}(z_k)^2 = 0
\ \ \ \vdash
\ \ \
\sum_{1 \le k \le \# {\bm \nu}(\tau),
1 \le h \le p-1, \atop \kappa_k(h) = 0} 
({P^{(h)}_{\re}}(z_k)^2 + {P^{(h)}_{\im}}(z_k)^2 )= 0.
$$
By Lemmas \ref{sos_non_pos_disjunct} and \ref{lemma_sum_of_pos_and_zer_is_pos}
(item \ref{lemma_sum_of_pos_and_zer_is_pos:1}) we obtain
\begin{equation}\label{inc:aux_theorem_weaksigndet_3_bis_bis}
\begin{array}{c}
\displaystyle{\Big \downarrow \
\sign({\rm ThElim}(P)) = \tau, \  
{\rm Th}(P)^{{\bm \mu}(\tau), {\bm \nu}(\tau), {\bm \eta}}(t, z),
}
\\[5mm]
\displaystyle{\bigwedge_{1 \le k \le \# {\bm \nu}(\tau),
1 \le h \le p-1, \atop \kappa_k(h) \ne 0} 
P^{(h)}_{\re}(z_k)^2 + P^{(h)}_{\im}(z_k)^2 \ne 0,} \\[5mm]
\displaystyle{\bigwedge_{1 \le k \le \# {\bm \nu}(\tau),
1 \le h \le p-1, \atop \kappa_k(h) = 0} 
P^{(h)}_{\re}(z_k)^2 + P^{(h)}_{\im}(z_k)^2 = 0} \
\Big \downarrow
_{\K[u][ t, a, b]}
\end{array}
\end{equation}
with monoid part 
$$
\prod_{H \in {\rm HMi}(P; Q), \atop \tau(H) \ne 0}H^{4\tilde g_H}
\cdot
\prod_{1 \le j < j' \le \# {\bm \mu}(\tau)}(t_{j} - t_{j'})^{4\tilde e_{{j}, j'}} 
\cdot
\prod_{1 \le k \le \# {\bm \nu}(\tau)}b_{k}^{4\tilde f_{k}}
\cdot
\prod_{1 \le k < k' \le \# {\bm \nu}(\tau)}
{\rm R}(z_{k}, z_{k'})
^{4\tilde g_{{k}, k'}} 
\cdot
$$
$$
\cdot 
\prod_{1 \le j  \le \# {\bm \mu}(\tau), 1 \le h \le p-1, \atop \eta^{\alpha}_j \ne 0} 
P^{(h)}(t_j)^{4\alpha_h \tilde e'_{j}}
\cdot 
\prod_{1 \le k  \le \# {\bm \nu}(\tau),1 \le h \le p-1, \atop {\bm \kappa}^{\alpha}_k \ne 0} 
({P^{(h)}_{\re}}(z_k)^2  + {P^{(h)}_{\im}}(z_k)^2 )^{4\alpha_h \tilde f'_{k}},
$$
degree in $w$ bounded by $(4{\rm bit}\{p\}({\rm g}_{H,1}\{p\}+1) + 2)\deg_w P$,
degree in $t_j$ bounded by $2{\rm g}_{H,1}\{p\}$ 
and
degree in $(a_k, b_k)$ bounded by 
$2({\rm g}_{H,1}\{p\} + p - 1)$.

Then we fix ${\bm \eta}$ and we apply 
to incompatibilities  (\ref{inc:aux_theorem_weaksigndet_3_bis_bis}) 
for ${\bm \eta}$ and every ${\bm \kappa}$,
the weak inference, 
$$
\vdash \ \ \ 
\bigvee_{K \in \mathcal{K} 
} 
\Big(\bigwedge_{(k,h) \not \in K} {P^{(h)}_{\re}}(z_k)^2  + {P^{(h)}_{\im}}(z_k)^2  \ne 0,
\bigwedge_{(k,h) \in K} {P^{(h)}_{\re}}(z_k)^2  + {P^{(h)}_{\im}}(z_k)^2  = 0\Big)
$$
where
$$\mathcal{K}=\{K\mid K \subset \{1 \le k \le \# {\bm \nu}(\tau)\} \times \{1 \le h \le p-1\}\}.$$
By Lemma \ref{lem:multiple_case_by_case} 
we obtain 
\begin{equation}\label{inc:aux_theorem_weaksigndet_3_bis}
\lda {\rm sign}({\rm ThElim}(P)) = \tau, \ 
{\rm Th}(P)^{ {\bm \mu}(\tau), {\bm \nu}(\tau), {\bm \eta}}(t,z) \rda
_{\K[u][ t, a, b]}
\end{equation}
with monoid part 
$$
\prod_{H \in {\rm ThElim}(P), \atop \tau(H) \ne 0}H^{2\hat g_H}
\cdot
\prod_{1 \le j < j' \le \# {\bm \mu}(\tau)}(t_{j} - t_{j'})^{2\hat e_{{j}, j'}} 
\cdot
\prod_{1 \le k \le \# {\bm \nu}(\tau)}b_{k}^{2\hat f_{k}}
\cdot
$$
$$
\cdot
\prod_{1 \le k < k' \le \# {\bm \nu}(\tau)}
{\rm R}(z_{k}, z_{k'})
^{2 \hat g_{{k}, k'}} 
\cdot
\prod_{1 \le j  \le \# {\bm \mu}(\tau),
1 \le h \le p-1, \atop \eta_j(h) \ne 0}
P^{(h)}(t_j)^{2\hat  e'_{j,h}}
$$
with $\hat g_H, \hat e_{j, j'}, 
\hat f_{k},
\hat g_{k, k'}, 
\hat e'_{j,h} \le 
\tilde{\rm g}_{H,1}\{p\}$,
degree in $w$ bounded by 
$
\tilde{\rm g}_{H,1}\{p\}\deg_w P$ and degree in $t_j$ and degree in 
$(a_k, b_k)$ bounded by 
$
\tilde{\rm g}_{H,1}\{p\}$.

Now we have already obtained the necessary incompatibilities for every ${\bm \eta}$. 
Then we apply to incompatibilities 
(\ref{inc:aux_theorem_weaksigndet_1}) and 
(\ref{inc:aux_theorem_weaksigndet_3_bis}) 
the weak inference
$$
\vdash \ \ \ 
\bigvee_{(J,J')\in \mathcal{J}}
\Big(\bigwedge_{(j,h) \in J'} P^{(h)}(t_j) > 0, \ 
\bigwedge_{(j,h) \not \in  
J \cup J'}
P^{(h)}(t_j) < 0, \
\bigwedge_{(j,h) \in J} P^{(h)}(t_j) = 0\Big)
$$
where
$$ \mathcal{J} =
\{(J,J') \mid J \subset \{1 \le j \le \# {\bm \mu}(\tau)\} \times \{1 \le h \le p-1\}, J' \subset \{1 \le j \le \# {\bm \mu}(\tau)\} \times \{1 \le h \le p-1\} \setminus J\}.
$$
By Lemma  \ref{lem:multiple_case_by_case_with_signs} 
we obtain
\begin{equation}\label{inc:aux_theorem_weaksigndet_4}
\lda {\rm sign}({\rm ThElim}(P)) = \tau, \ {\rm Fact}(P)^{{\bm \mu}(\tau), {\bm \nu}(\tau)}(t,z),   \ \cH \rda
_{\K[v][ t, a, b]}
\end{equation}
with monoid part 
$$
S^{\hat h'}\cdot
\prod_{H \in {\rm ThElim}(P), \atop \tau(H) \ne 0}H^{2\hat g'_H}
\cdot
\prod_{1 \le j < j' \le \# {\bm \mu}(\tau)}(t_{j} - t_{j'})^{2 \hat e'_{{j}, j'}} 
\cdot
\prod_{1 \le k \le \# {\bm \nu}(\tau)}b_{k}^{2\hat f'_{k}}
\cdot
\prod_{1 \le k < k' \le \# {\bm \nu}(\tau)}
{\rm R}(z_{k}, z_{k'})
^{2 \hat g'_{{k}, k'}} 
$$
with 
$\hat h'
\le f'_0
$,
$
\hat g'_H \le
f'_0\tilde{\rm g}_{H,1}\{p\} 
$,
$
\hat e'_{j, j'} \le
f'_0\max\{e, \tilde{\rm g}_{H,1}\{p\}\} 
$,
$
\hat f'_{k} \le 
f'_0\max\{f, \tilde{\rm g}_{H,1}\{p\}\} 
$,
$
\hat g'_{k, k'} \le 
f'_0\max\{g, \tilde{\rm g}_{H,1}\{p\}\} 
$,
degree in $w$ bounded by 
$
f'_0
\max\{\delta_w , 
\tilde{\rm g}_{H,1}\{p\}\deg_w P
\}
$,
degree in $t_j$ bounded by
$
f'_0\max\{\delta_t , 
\tilde{\rm g}_{H,1}\{p\}
\}
$
and degree in $(a_k, b_k)$ bounded by 
$
f'_0\max\{\delta_z, 
\tilde{\rm g}_{H,1}\{p\}
\}
$,
where
$$f'_0 = 
2^{((p-1)p+2)2^{(p-1)p}  - 2 }
\max\{e', \tilde{\rm g}_{H,1}\{p\}\}^{2^{(p-1)p}-1}. 
$$
We rename variables $t$ and $z$ in (\ref{inc:aux_theorem_weaksigndet_4}) 
as $t_{{\bm \mu}(\tau)}$ and $z_{{\bm \nu}(\tau)}$ respectively.

Our next aim is to obtain incompatibilities
$$
\lda {\rm sign}({\rm ThElim}(P)) = \tau, \ {\rm Fact}(P)^{{\bm \mu},{\bm \nu}}(t_{{\bm \mu}},z_{{\bm \nu}}), \ \cH  \rda
_{\K[v][ t_{{\bm \mu}}, a_{{\bm \nu}}, b_{{\bm \nu}}]}
$$
for every $({\bm \mu}, {\bm \nu}) \in \cup_{m + 2n = p} \Lambda_{m} \times \Lambda_{n}$, where
$t_{\bm \mu} = (t_{{\bm \mu}, 1}, \dots, t_{{\bm \mu}, {\# {\bm \mu}}})$
and $z_{\bm \nu} = (z_{{\bm \nu}, 1}, \dots, z_{{\bm \nu}, {\# {\bm \nu}}})$, in order to 
be able to apply 
Theorem \ref{thLaplacewithmult} (Real Irreducible Factors with Multiplicities as a weak existence).
For $({\bm \mu}(\tau), {\bm \nu}(\tau))$, we already have incompatibility
(\ref{inc:aux_theorem_weaksigndet_4}), so now we suppose 
$({\bm \mu}, {\bm \nu}) \ne ({\bm \mu}(\tau), {\bm \nu}(\tau))$.

By Theorem \ref{ThomwithP} (Fixing the Thom encodings)
for every
$ {\bm \eta}$  
list of  
sign conditions on $\Der(P)$ and 
${\bm \kappa}$   
list of invertibility conditions on $\Der(P)$,
there exists $\alpha \in \{0, 1, 2\}^{1, \dots, p-1}$ 
with $\#\{h  \ | \ \alpha_h \ne 0\} \le {\rm bit}\{p\}$
such that 
$Q = \prod_{1 \le h \le p-1}(P^{(h)})^{\alpha_h} \in {\rm PDer}_{{\rm bit}\{p\}}(P)$ 
verifies
$$
({\rm Rk}_{\rm HMi}(\tau) , 
{\rm Si}_{\rm HMi}(\tau)) \ne 
({\rm Rk}_{\rm Fact}({\bm \eta}^\alpha, {\bm \kappa}^\alpha), 
{\rm Si}_{\rm Fact}({\bm \eta}^\alpha)).
$$
Proceeding as before, we obtain
\begin{equation}\label{inc:aux_theorem_weaksigndet_5}
\lda {\rm sign}({\rm ThElim}(P)) = \tau, \ {\rm Fact}(P)^{{\bm \mu},{\bm \nu}}(t_{{\bm \mu}},z_{{\bm \nu}})  \rda
_{\K[u][ t_{{\bm \mu}}, a_{{\bm \nu}}, b_{{\bm \nu}}]}
\end{equation}
with monoid part 
$$
\prod_{H \in {\rm ThElim}(P), \atop \tau(H) \ne 0}H^{2\hat g''_H}
\cdot
\prod_{1 \le j < j' \le m}(t_{{\bm \mu}, j} - t_{{\bm \mu}, j'})^{2 \hat e''_{{j}, j'}} 
\cdot
\prod_{1 \le k \le n}b_{{\bm \nu}, k}^{2\hat f''_{k}}
\cdot
\prod_{1 \le k < k' \le n}
{\rm R}(z_{k}, z_{k'})
^{2 \hat g''_{{k}, k'}} 
$$
with 
$
\hat g''_H, 
\hat e''_{j, j'},
\hat f''_{k},
\hat g''_{k, k'} \le f'_0\tilde{\rm g}_{H,1}\{p\}$,  
degree in $w$ bounded by 
$
f'_0\tilde{\rm g}_{H,1}\{p\}\deg_w P
$,
and degree in $t_{{\bm \mu}, j}$ and  
degree in $(a_{{\bm \nu}, k}, b_{{\bm \nu}, k})$ bounded by 
$
f'_0\tilde{\rm g}_{H,1}\{p\}
$.

Finally, we apply to  incompatibility (\ref{inc:aux_theorem_weaksigndet_4}) 
and incompatibilities (\ref{inc:aux_theorem_weaksigndet_5}) 
for every $({\bm \mu}, {\bm \nu}) \ne
({\bm \mu}(\tau), {\bm \nu}(\tau))$
the weak inference
$$
\vdash \ \ \ \bigvee_{ m + 2n = p
\atop
({\bm \mu},{\bm \nu}) \in 
\Lambda_{m} \times \Lambda_{n}}
\exists (t_{{\bm \mu}}, z_{{\bm \nu}}) \ [\; {\rm Fact}(P)^{{\bm \mu},{\bm \nu}}(t_{{\bm \mu}}, z_{{\bm \nu}}) \;].
$$
By Theorem \ref{thLaplacewithmult} (Real Irreducible Factors with Multiplicities as a weak existence), taking into account that
$\# \cup_{m + 2n = p} \Lambda_{m} \times \Lambda_{n} \le 2^p$, 
and using Lemma \ref{tlem:aux_theorem_fix_Thom}, we obtain 
$$
\lda {\rm sign}({\rm ThElim}(P)) = \tau, \  \cH \rda
_{\K[v]}
$$
with monoid part
$$
S^h
\cdot
\prod_{H \in {\rm ThElim}(P), \atop \tau(H) \ne 0} H^{2h'_H} 
$$
with 
\begin{eqnarray*}
h &\le& {\rm g}_4\{p\}
{f'_0}^{2^{\frac{1}{2}p^2}+2^{\frac{1}{2}p}+1}
\max\{ e, g, \tilde{\rm g}_{H,1}\{p\} \}^{2^{\frac{1}{2}p^2}}
\max\{ f, \tilde{\rm g}_{H,1}\{p\} \}^{2^{\frac{1}{2}p}} \le \\[2mm]
& \le & {\rm g}_5\{p, e, f, g, e'\},\\
h'_H &\le& 
2^p{\rm g}_4\{p\} \tilde{\rm g}_{H,1}\{p\}
{f'_0}^{2^{\frac{1}{2}p^2}+2^{\frac{1}{2}p}+1}
\max\{ e, g, \tilde{\rm g}_{H,1}\{p\} \}^{2^{\frac{1}{2}p^2}}
\max\{ f, \tilde{\rm g}_{H,1}\{p\} \}^{2^{\frac{1}{2}p}}  \le \\[2mm]
& \le & {\rm g}_5\{p, e, f, g, e'\}, 
\end{eqnarray*}
and degree in $w$ bounded by 
$$
\begin{array}{rrl}
& {\rm g}_4\{p\}
{f'_0}^{2^{\frac{1}{2}p^2}+2^{\frac{1}{2}p}+1}
\max\{ e, g, \tilde{\rm g}_{H,1}\{p\} \}^{2^{\frac{1}{2}p^2}}
\max\{ f, \tilde{\rm g}_{H,1}\{p\} \}^{2^{\frac{1}{2}p}}
\cdot & \\
& \cdot
\Big(
\max\{\delta_w , 
\tilde{\rm g}_{H,1}\{p\}\deg_w P
\} + 
\max\{\delta_t , \delta_z ,
\tilde{\rm g}_{H,1}\{p\}\}\deg_wP
\Big)
&\le \\
\le & {\rm g}_5\{p, e, f, g, e'\}
\Big(
\max\{\delta_w , 
\tilde{\rm g}_{H,1}\{p\}\deg_w P
\} + 
\max\{\delta_t , \delta_z ,
\tilde{\rm g}_{H,1}\{p\}\}\deg_wP
\Big),
\end{array}
$$
which serves as the final incompatibility.
\end{proof}

\begin{proof}{Proof of Theorem \ref{weaksigndetgen}.} 
We simplify the notation by renaming ${\bm \mu} = {\bm \mu}(\tau)$,  
${\bm \nu} = {\bm \nu}(\tau)$ and
$ {\bm \eta} =  {\bm \eta}(\tau)$.
Consider the initial incompatibility 
\begin{equation}\label{inc:init_inc_sec_the_chap_six}
\Big \downarrow \  \bigwedge_{1 \le j \le \# {\bm \mu}
} \s(Q(t_j)) = \epsilon_j(\tau, \tau')
, \  {\cal H} \ {\Big \downarrow}_{\K[v]}
\end{equation}
where ${\cal H}$ is a system of sign conditions in $\K[v]$. 

Once again, our aim is to proceed by case by case reasoning. 
Let ${\bm \epsilon} = [\epsilon_1, \dots, \epsilon_{\# {\bm \mu}}]$ be a 
list of 
sign conditions
on $Q$ with ${\bm \epsilon} \ne {\bm \epsilon}(\tau, \tau')$, 
${\bm \kappa} = [\kappa_1, \dots, \kappa_{\# {\bm \nu}}]$ 
be
a list of invertibility conditions
on $\Der(P)$
and ${\bm \rho} = [\rho_1, \dots, \rho_{\# {\bm \nu}}]$
be
 a list of invertibility conditions
on $Q$. 
By Theorem \ref{ThomwithPQ} (Fixing the Thom encodings with a Sign)
there exist $\alpha \in \{0, 1, 2\}^{1, \dots, p-1}$ 
with $\#\{h  \ | \ \alpha_h \ne 0\} \le {\rm bit}\{p\}-1$ and
$\beta \in \{1,2\}$
such that $\tilde Q = (\prod_{1 \le h \le p-1}(P^{(h)})^{\alpha_h})Q^{\beta} \in 
{\rm PDer}_{{\rm bit}\{p\}-1}(P;Q)$ 
verifies
$$
({\rm Rk}_{\rm HMi}(\tau) , 
{\rm Si}_{\rm HMi}(\tau)) \ne 
({\rm Rk}_{\rm Fact}( {\bm \eta}
^\alpha{\bm \epsilon}^\beta, {\bm \kappa}^\alpha {\bm \rho}^\beta), 
{\rm Si}_{\rm Fact}( {\bm \eta}
^\alpha{\bm \epsilon}^\beta)),
$$
where $ {\bm \eta}
^\alpha{\bm \epsilon}^\beta$ is the list of 
sign conditions satisfied by $\tilde Q$ 
on $t$ 
if ${\rm Th}(P)^{ {\bm \mu}
,{\bm \nu}, {\bm \eta}
}(t,z)$ 
holds and ${\bm \epsilon}$ is the 
list of 
sign conditions satisfied by $Q$ on $t$ and 
${\bm \kappa}^\alpha{\bm \rho}^\beta$ is defined analogously. 
By Theorem \ref{hermitetrsubresw} (Hermite's Theory as an incompatibility)
there is an incompatibility
\begin{equation}\label{inc:aux_theorem_weaksigndet_sec_2}
\begin{array}{c}
\Big\downarrow  \
\sign({\rm ThElim}(P;Q)) = \tau', \ 
{\rm Fact}(P)^{{\bm \mu},{\bm \nu}}(t,z), \\[4mm]
\displaystyle{\bigwedge_{1\leq j\le \# {\bm \mu}}  {\rm sign}(\tilde Q(t_j)) = 
\eta_j
^\alpha\epsilon_j^\beta,  \
\bigwedge_{1\leq k \le \# {\bm \nu}
} {\rm inv}(\tilde Q(z_k)) = \kappa_k^\alpha \rho_k^\beta }
\ \Big\downarrow_{\K[u][ t, a, b]} \end{array}
\end{equation} 
with monoid part 
$$
\prod_{H \in {\rm HMi}(P; \tilde Q), \atop \tau'(H) \ne 0}H^{2\tilde g_H}
\cdot
\prod_{1 \le j < j' \le \# {\bm \mu}}(t_{j} - t_{j'})^{2\tilde e_{{j}, j'}}
\cdot
\prod_{1 \le k \le \# {\bm \nu}}b_{k}^{2\tilde f_{k}}
\cdot
$$
$$
\cdot
\prod_{1 \le k < k' \le \# {\bm \nu}}
{\rm R}( z_{k}, z_{k'})
^{2\tilde g_{{k}, k'}} 
\cdot
\prod_{1 \le j \le \# {\bm \mu}
, \atop \eta_j^\alpha\epsilon_j^\beta \ne 0}
\tilde Q(t_j)^{2\tilde e'_j}
\cdot
\prod_{1 \le k \le \# {\bm \nu}, \atop  \kappa_j^\alpha\rho_k^\beta \ne 0}
(\tilde Q^2_{\re}(z_k) + \tilde Q^2_{\im}(z_k))^{2\tilde f'_k}
$$
with
$\tilde g_H, 
\tilde e_{j, j'}, 
\tilde f_k, 
\tilde g_{k, k'}, 
\tilde e'_j, \tilde f'_k \le {\rm g}_{H,2}\{p,q\}$, 
degree in $w$ bounded by 
$
2{\rm bit}\{p\}{\rm g}_{H,2}\{p,q\}
\max\{\deg_w P, \deg_wQ\}
$
and
degree in $t_j$ 
and
degree in $(a_k, b_k)$ bounded by ${\rm g}_{H,2}\{p,q\}$.

Since the sign and invertibility of a product is determined by 
the sign and invertibility of each factor,
by applying to (\ref{inc:aux_theorem_weaksigndet_sec_2}) the weak inferences in 
Lemmas \ref{lemma_basic_sign_rule_1} (items \ref{lemma_basic_sign_rule:4},
\ref{lemma_basic_sign_rule:5} and
\ref{lemma_basic_sign_rule:7}) and \ref{lem:comb_lin_zero_zero} (used as in the proof of Theorem \ref{weaksigndet}),  we obtain 
\begin{equation}\label{inc:aux_theorem_weaksigndet_second_3}
\begin{array}{c}
\Big \downarrow 
\sign({\rm ThElim}(P;Q)) = \tau', \ 
{\rm Th}^{ {\bm \mu}, {\bm \nu}, {\bm \eta}}(t, z), \ 
\displaystyle{\bigwedge_{1 \le j \le \# {\bm \mu}} \s(Q(t_j)) = \epsilon_j,}
 \\[4mm]
\displaystyle{\bigwedge_{1 \le k \le \# {\bm \nu}} {\rm inv}({\Der}(P)(z_k)) = \kappa_k, \
\bigwedge_{1 \le k \le \# {\bm \nu}} {\rm inv}(Q(z_k)) = \rho_k}
\Big \downarrow
_{\K[u][ t, a, b]}
\end{array}
\end{equation} 
with  monoid part 
$$
\prod_{H \in {\rm HMi}(P; \tilde Q), \atop \tau'(H) \ne 0}H^{2\tilde g_H}
\cdot
\prod_{1 \le j < j' \le \# {\bm \mu}}(t_{j} - t_{j'})^{2\tilde e_{{j}, j'}} 
\cdot
\prod_{1 \le k \le \# {\bm \nu}}b_{k}^{2\tilde f_{k}}
\cdot
\prod_{1 \le k < k' \le \# {\bm \nu}}
{\rm R}( z_{k}, z_{k'})
^{2\tilde g_{{k}, k'}} 
\cdot
$$
$$
\cdot
\prod_{1 \le j  \le \# {\bm \mu}, \atop \eta_j^{\alpha}\epsilon^{\beta}_j \ne 0}
\Big( 
\prod_{1 \le h \le p-1} P^{(h)}(t_j)^{2\alpha_h \tilde e'_{j}}
\Big)
\cdot
Q(t_j)^{2 \beta \tilde e'_j}
\cdot
$$
$$
\cdot
\prod_{1 \le k  \le \# {\bm \nu}, \atop \kappa_k^{\alpha}\rho^{\beta}_k \ne 0} 
\Big(
\prod_{1 \le h \le p-1}
({P^{(h)}_{\re}}(z_k)^2  + {P^{(h)}_{\im}}(z_k)^2 )^{2\alpha_h \tilde f'_{k}}
\Big)
\cdot
(Q^2_{\re}(z_k)  + Q^2_{\im}(z_k))^{2\beta \tilde f'_k},
$$
degree in $w$ bounded by 
$
2{\rm bit}\{p\}({\rm g}_{H,2}\{p,q\} + 1)
\max\{\deg_w P, \deg_wQ\}
$
and
degree in $t_j$ 
and
degree in $(a_k, b_k)$ bounded by 
${\rm g}_{H,2}\{p,q\}$.

Then we successively apply to (\ref{inc:aux_theorem_weaksigndet_second_3}) 
the weak inferences
$$
\sum_{1 \le k \le \# {\bm \nu}
,
1 \le h \le p-1, \atop \kappa_k(h) = 0} 
({P^{(h)}_{\re}}(z_k)^2 + {P^{(h)}_{\im}}(z_k)^2 ) 
+
\sum_{1 \le k \le \# {\bm \nu}
, \atop \rho_k = 0 } 
 (Q^2_{\re}(z_k) + Q^2_{\im}(z_k)) 
= 0
\ \ \ \vdash
$$
$$
\vdash \ \ \ 
\bigwedge_{1 \le k \le \# {\bm \nu}
,
1 \le h \le p-1, \atop \kappa_k(h) = 0} 
({P^{(h)}_{\re}}(z_k) = 0, \  {P^{(h)}_{\im}}(z_k) = 0 ),   
\bigwedge_{1 \le k \le \# {\bm \nu}
, \atop \rho_k = 0 } 
(Q_{\re}(z_k) =0, \ Q_{\im}(z_k) = 0)
$$
and
$$
\bigwedge_{1 \le k \le \# {\bm \nu},
1 \le h \le p-1, \atop \kappa_k(h) = 0} 
{P^{(h)}_{\re}}(z_k)^2 + {P^{(h)}_{\im}}(z_k)^2 =0 , \  
\bigwedge_{1 \le k \le \# {\bm \nu}
, \atop \rho_k = 0 } 
 Q^2_{\re}(z_k) + Q^2_{\im}(z_k) = 0
\ \ \ \vdash
$$
$$
\vdash  \ \ \
\sum_{1 \le k \le \# {\bm \nu}
,
1 \le h \le p-1, \atop \kappa_k(h) = 0} 
({P^{(h)}_{\re}}(z_k)^2 + {P^{(h)}_{\im}}(z_k)^2 ) 
+  
\sum_{1 \le k \le \# {\bm \nu}
, \atop \rho_k = 0 } 
 (Q^2_{\re}(z_k) + Q^2_{\im}(z_k)) 
= 0.
$$
By Lemmas \ref{sos_non_pos_disjunct} and \ref{lemma_sum_of_pos_and_zer_is_pos}
(item \ref{lemma_sum_of_pos_and_zer_is_pos:1}) we obtain 
\begin{equation}\label{inc:aux_theorem_weaksigndet_second_3_bis_bis}
\begin{array}{c}
\displaystyle{
\Big \downarrow
\sign({\rm ThElim}(P;Q)) = \tau',  \
{\rm Th}^{ {\bm \mu}, {\bm \nu}, {\bm \eta}}(t, z), \,
\bigwedge_{1 \le j \le \# {\bm \mu}
} \s(Q(t_j)) = \epsilon_j, 
}
\\[6mm]
\displaystyle{ 
\bigwedge_{1 \le k \le \# {\bm \nu}
,
1 \le h \le p-1, \atop \kappa_k(h) \ne 0} 
{P^{(h)}_{\re}}(z_k)^2 + {P^{(h)}_{\im}}(z_k)^2 \ne 0,} 
\displaystyle{ 
\bigwedge_{1 \le k \le \# {\bm \nu}
,
1 \le h \le p-1, \atop \kappa_k(h) = 0} 
{P^{(h)}_{\re}}(z_k)^2 + {P^{(h)}_{\im}}(z_k)^2 =0} ,\\[6mm]
\displaystyle{
\bigwedge_{1 \le k \le \# {\bm \nu}
, \atop \rho_k \ne 0 } 
 Q^2_{\re}(z_k) + Q^2_{\im}(z_k) \ne  0, 
\
\bigwedge_{1 \le k \le \# {\bm \nu}
, \atop \rho_k = 0 } 
Q^2_{\re}(z_k) + Q^2_{\im}(z_k) = 0
\Big \downarrow_{\K[u][ t, a, b]}}
\end{array}
\end{equation}
with monoid part 
$$
\prod_{H \in {\rm HMi}(P; \tilde Q), \atop \tau'(H) \ne 0}H^{4\tilde g_H}
\cdot 
\prod_{1 \le j < j' \le \# {\bm \mu}
}(t_{j} - t_{j'})^{4\tilde e_{{j}, j'}} 
\cdot
\prod_{1 \le k \le \# {\bm \nu}
}b_{k}^{4\tilde f_{k}}
\cdot
\prod_{1 \le k < k' \le \# {\bm \nu}
}
{\rm R}( z_{k}, z_{k'})
^{4\tilde g_{{k}, k'}} 
\cdot
$$
$$
\cdot
\prod_{1 \le j  \le \# {\bm \mu}, \atop \eta_j^{\alpha}\epsilon^{\beta}_j \ne 0}
\Big( 
\prod_{1 \le h \le p-1} P^{(h)}(t_j)^{4\alpha_h \tilde e'_{j}}
\Big)
\cdot
Q(t_j)^{4 \beta \tilde e'_j}
\cdot
$$
$$
\cdot
\prod_{1 \le k  \le \# {\bm \nu}, \atop \kappa_k^{\alpha}\rho^{\beta}_k \ne 0} 
\Big(
\prod_{1 \le h \le p-1}
({P^{(h)}_{\re}}(z_k)^2  + {P^{(h)}_{\im}}(z_k)^2 )^{4\alpha_h \tilde f'_{k}}
\Big)
\cdot
(Q^2_{\re}(z_k)  + Q^2_{\im}(z_k))^{4\beta \tilde f'_k},
$$
degree in $w$ bounded by 
$
(4{\rm bit}\{p\}({\rm g}_{H,2}\{p,q\} + 1) + 2)
\max\{\deg_w P, \deg_wQ\}
$,
degree in $t_j$ bounded by $2{\rm g}_{H,2}\{p,q\}$
and
degree in $(a_k, b_k)$ bounded by 
$2({\rm g}_{H,2}\{p,q\} + \max\{p-1, q\})$.

Then we fix ${\bm \epsilon}$ and we apply 
to incompatibilities (\ref{inc:aux_theorem_weaksigndet_second_3_bis_bis})
for ${\bm \epsilon}$ and every ${\bm \kappa}$ and ${\bm \rho}$,
the weak inference
$$
\vdash \ \ \ 
\bigvee_{K\in \mathcal{K},\atop
K' \in \mathcal{K}'} 
\Big(
\bigwedge_{(k, h) \not \in K'} {P^{(h)}_{\re}}(z_k)^2  + {P^{(h)}_{\im}}(z_k)^2  \ne 0, \ 
\bigwedge_{(k,h) \in K'} {P^{(h)}_{\re}}(z_k)^2  + {P^{(h)}_{\im}}(z_k)^2 = 0, 
$$
$$
\bigwedge_{k \not \in K} Q^2_{\re}(z_k)  + Q^2_{\im}(z_k)  \ne 0, \ 
\bigwedge_{k \in K} Q^2_{\re}(z_k)  + Q^2_{\im}(z_k) = 0
\Big),
$$
where 
$$
\mathcal{K} = \{K\mid K \subset \{1 \le k \le \# {\bm \nu}\}\} \quad \hbox{and} \quad
\mathcal{K}' = \{K'\mid K' \subset \{1 \le k \le \# {\bm \nu}\} \times \{1, \dots, p-1 \}\}.
$$
By Lemma \ref{lem:multiple_case_by_case} 
we obtain 
\begin{equation}\label{inc:aux_theorem_weaksigndet_second_3_bis}
\Big \downarrow \ \sign({\rm ThElim}(P;Q)) = \tau', \
{\rm Th}^{{\bm \mu}, {\bm \nu}, {\bm \eta}}(t, z), \,
\bigwedge_{1 \le j \le \# {\bm \mu}
} \s(Q(t_j)) = \epsilon_j \ {\Big \downarrow}_{\K[u][ t, a, b]}
\end{equation}
with monoid part 
$$
\prod_{H \in {\rm ThElim}(P; Q), \atop \tau'(H) \ne 0}H^{2\hat g_H}
\cdot
\prod_{1 \le j < j' \le \# {\bm \mu}
}(t_{j} - t_{j'})^{2\hat e_{{j}, j'}} 
\cdot
\prod_{1 \le k \le \# {\bm \nu}
}b_{k}^{2\hat f_{k}}
\cdot
$$
$$
\cdot
\prod_{1 \le k < k' \le \# {\bm \nu}
}
{\rm R}(z_{k}, z_{k'})
^{2 \hat g_{{k}, k'}} 
\cdot 
\prod_{1 \le j  \le \# {\bm \mu}, \, 1 \le h \le p-1, \atop \eta_j(h) \ne 0}
P^{(h)}(t_j)^{2\hat  e'_{j,h}}
\cdot
\prod_{1 \le j  \le \# {\bm \mu},  \atop \epsilon_j \ne 0}
Q(t_j)^{2\hat e'_j}
$$
with 
$
\hat g_H, 
\hat e_{j, j'}, 
\hat f_{k}, 
\hat g_{k, k'}, 
\hat e'_{j,h},
\hat e'_{j}  \le \tilde{\rm g}_{H,2}\{p,q\}$, 
degree in $w$ bounded by 
$\tilde{\rm g}_{H,2}\{p,q\}
\max\{\deg_w P, \deg_wQ\}
$ and
degree in $t_j$  and 
degree in $(a_k, b_k)$ bounded by 
$\tilde{\rm g}_{H,2}\{p,q\}$.

Finally, we apply 
to incompatibilities (\ref{inc:init_inc_sec_the_chap_six})
and (\ref{inc:aux_theorem_weaksigndet_second_3_bis}) for every ${\bm \epsilon} \ne {\bm \epsilon}(\tau, \tau')$
the weak inference
$$
\vdash \ \ \ 
\bigvee_{J \subset \{1, \dots, \# {\bm \mu}\}
\atop J' \subset \{1, \dots,  {\bm \mu}\} \setminus J} 
\Big(\bigwedge_{j \in J'} Q(t_j) > 0, \ 
\bigwedge_{j  \not \in J \cup J'} Q(t_j) < 0,  \
\bigwedge_{j \in J} Q(t_j) = 0\Big).
$$
By Lemma  \ref{lem:multiple_case_by_case_with_signs} 
we obtain 
$$
\lda {\rm sign}({\rm ThElim}(P;Q)) = \tau', \ 
{\rm Th}(P)^{{\bm \mu}, {\bm \nu}, {\bm \eta}},    \ \cH \rda
_{\K[v]}
$$
with monoid part 
$$
S^{h'}\cdot
\prod_{H \in {\rm ThElim}(P;Q), \atop \tau'(H) \ne 0} H^{2h'_H} 
\cdot
\prod_{1 \le j < j' \le \# {\bm \mu}
}(t_{j} - t_{j'})^{2e_{{j}, j'}} 
\cdot
\prod_{1 \le k \le \# {\bm \nu}
}b_{k}^{2f_{k}}
$$
$$
\cdot
\prod_{1 \le k < k' \le \# {\bm \nu}
}
{\rm R}(z_{k}, z_{k'})
^{2g_{{k}, k'}} 
\cdot  
\prod_{1 \le j  \le \# {\bm \mu}
, 1 \le h \le p-1, \atop \eta_j(h) \ne 0}
P^{(h)}(t_j)^{2e'_{h,j}}
$$
with
\begin{eqnarray*}
h'&\le &
2^{(p + 2)2^{p} - 2p - 2}
\max\{
h, \tilde{\rm g}_{H,2}\{p,q\}
\}^{2^{p}-1},\\
h'_H, 
e_{j, j'},
f_{k},
g_{k, k'},
e'_{h,j}
&\le& 2^{(p + 2)2^{p}  - 2}
\max\{
h, \tilde{\rm g}_{H,2}\{p,q\}
\}^{2^{p}-1}\tilde{\rm g}_{H,2}\{p,q\},
\end{eqnarray*}
and degree in $w$ bounded by 
$$
2^{(p + 2)2^{p}-2}
\max\{
h, \tilde{\rm g}_{H,2}\{p,q\}\}^{2^{p}-1}
\max\{
\delta_w, \tilde{\rm g}_{H,2}\{p,q\}\max\{\deg_w P, \deg_wQ\} \},
$$
degree in $t_{j}$ bounded by 
$$
2^{(p + 2)2^{p}-2}
\max\{h, \tilde{\rm g}_{H,2}\{p,q\}\}^{2^{p}-1}
\max\{
\delta_t, \tilde{\rm g}_{H,2}\{p,q\}\}
$$
and
degree in $(a_{k}, b_{k})$ bounded by 
$$
2^{(p + 2)2^{p}-2}
\max\{h, \tilde{\rm g}_{H,2}\{p,q\}\}^{2^{p}-1}
\max\{
\delta_z, \tilde{\rm g}_{H,2}\{p,q\}\},
$$
which serves as 
the final incompatibility. 
\end{proof}

We finish this subsection with the following remark, which will be used in Subsection \ref{sect_factor_family}.

\begin{remark}\label{rem:degree_and_number_thelim} Following Definition \ref{notaDPTP},  there are
$$
\sum_{0 \le j \le i} {{p-1}\choose{j}}2^j \le 2p^{i}
$$
elements in ${\rm PDer}_i(P)$.
Therefore, there are 
at most $2p^{{\rm bit}\{p\} + 1}$ elements in 
${\rm ThElim}(P)$ and, by Remark \ref{rem:degree_bound_for_hmin}, 
their degrees in $u$ are bounded by 
$$
p\Big(
(2(p-1){\rm bit}\{p\} + 2p - 2)\deg_u P + 2{\rm bit}\{p\}\deg_u P
\Big) \le 2p^2( {\rm bit}\{p\}+1)\deg_u P.
$$
Similarly, there are at most
$4p^{{\rm bit}\{p\}}$  elements in ${\rm ThElim}(P;Q)$ and, again by 
Remark \ref{rem:degree_bound_for_hmin}, 
their degrees in $u$ are bounded by 
\begin{eqnarray*}
&&p\Big(
(2(p-1)({\rm bit}\{p\}-1) + 2q + 2p - 2)\deg_u P
+
2({\rm bit}\{p\}- 1)\deg_u P + 2 \deg_u Q 
\Big) =
\\
&=&
p\Big(
( 2p{\rm bit}\{p\} + 2q - 2)  \deg_u P + 2\deg_u Q
\Big).
\end{eqnarray*}
\end{remark}

\subsection{Conditions on the parameters fixing the real root order on a family}
\label{sect_factor_family}

Consider now a finite family $\cQ$ of polynomials in $\K[u][y]$ monic in the variable $y$, with $u=(u_1,\ldots,u_k)$. 
Our aim is to define a family ${\rm Elim}({\mathcal Q})\subset \K[u]$  such  that
the list of realizable 
sign conditions on ${\rm Elim}({\mathcal Q})$ 
fixes the factorization
and relative order between the real roots of all polynomials in ${\cal Q}$.

\begin{definition}\label{def:sign_det_sect_three_first} 
Let $\cQ$ be
a finite family of polynomials in $\K[u][y]$ monic in the variable $y$. 
We denote by 
$$
 \Der_+({\cal Q}) = \bigcup_{P \in {\cal Q}}\Der_+(P) \subset \K[u][y].
$$

We define 
$$ {\rm Elim}({\mathcal Q}) =
\displaystyle{ 
\bigcup_{ P\in {\cal Q}}
\Big(  {\rm ThElim}(P) \
\bigcup 
\ \bigcup_{Q \in  \Der_+(\cQ) \setminus \Der_+(P)} {\rm ThElim}(P;Q) \Big)
} \subset \K[u].
$$
\end{definition}

In order to prove that the family ${\rm Elim}({\mathcal Q})$ satisfies the required property, 
we introduce some notation and defintions. 

\begin{notation}\label{tableofrealroots}
Let $\cQ$ be
a finite family of polynomials in $\K[u][y]$ monic in the variable $y$. 
We define the set ${\rm H}({\cal Q})$, whose elements  give a description 
of the total list of real roots of ${\cal Q}$.
 An element of  ${\rm H}({\cal Q})$
 is
a list
$ {\bm \eta} =  [\eta_1,   \dots , \eta_r]$ 
of distinct
sign conditions on $\Der_+({\cal Q})$ such that
\begin{itemize}
\item 
for every $1 \le j \le r$, there exists $P \in \cQ$ such that  $\eta_j(P) = 0$.

\item for every $1 \le j \le r$ and every $P \in \cQ$ such that $\eta_j(P) = 0$, 
$\eta_{j'} \prec_{P} \eta_{j}$ for $ 1 \le j' < j$ and 
$\eta_{j} \prec_{P} \eta_{j'}$ for $ j < j' \le r$.

\item for every $1 \le j < j' \le r$ and every $P \in \cQ$, 
$\eta_{j} \preceq_P \eta_{j'}$.
\end{itemize}

For $ {\bm \eta} \in {\rm H}({\cal Q})$ and $P\in \cQ$ we define ${\bm \eta}(P)$ as the 
(possibly empty)
ordered 
sublist 
of 
${\bm \eta}\mid_{\Der(P)}$ containing ${\eta_j}\mid_{\Der(P)}$ for those $1 \le j  \le r $ such that $\eta_j(P)=0$.

Given  $ {\bm \eta} \in {\rm H}({\cal Q})$,
we define the set ${\rm N}({\cal Q}, {\bm \eta})$, whose elements give a description 
of the multiplicity of the complex roots of the polynomials in ${\cal Q}$, given the description
$ {\bm \eta}$ of their real roots, by
$$
{\rm N}({\cal Q}, {\bm \eta})=
\prod_{P\in \cQ} \ \Lambda_{\frac12(\deg_y P-|{\rm vmu}( {\bm \eta}(P))|)}.
$$
(cf Notation \ref{not:mult}).
\end{notation}

Note that every choice of $\vartheta\in \R^k$
defines an element $ {\bm \eta}$ of ${\rm H}({\cal Q})$ and an element ${\bm \nu}$ of ${\rm N}({\cal Q},{\bm \eta})$ by 
considering the list of signs of $\Der_+({\cal Q}(\vartheta))$ at the roots $\theta_1,\ldots,\theta_r$ of the polynomials in ${\cal Q}(\vartheta) 
\subset \K[y]$ 
as well as the vectors of multiplicities of their complex roots.

\begin{definition}\label{def:sign_det_sect_three_sec} 
Let $\cQ$ be
a finite family of polynomials in $\K[u][y]$ monic in the variable $y$ and
$ {\bm \eta}\in {\rm H}({\cal Q}), {\bm \nu} \in {\rm N}({\cal Q}, {\bm \eta})$ with $ {\bm \eta} = [\eta_1, \dots, \eta_r]$, 
$t =   (t_1, \dots, t_r)$, 
$t_P$
be
 the vector formed by those $t_j$ 
whose indices appear in ${\bm \eta}(P)$ in the order $\prec^{\rm mu}_P$,
$z_P = (z_{P,1}, \dots, z_{P,
\# {\bm \nu}(P)})$
for  $P \in {\cal Q}$ and
$z = (z_P)_{P \in {\cal Q}}$.
We define the system of sign conditions
$${\rm OFact}({\cal Q})^{ {\bm \eta},{\bm \nu}}(t, z)$$ in $\K[u][ t, a, b]$ 
describing the 
decompostion into irreducible real factors 
and the relative order between the real roots
of all polynomials in $\cQ$:  
$$
\bigwedge_{P\in {\mathcal Q}}{\rm Fact}(P)^{
{\rm vmu}({\bm \eta}(P)), {\bm \nu}(P)}(t_P, z_P), 
\ \bigwedge_{1  \le j < j' \le r} t_{j} <  t_{j'}.
$$
\end{definition}

The folowing result show the connection between a sign condition on the set 
${\rm Elim}({\mathcal Q})$ and the order between the real roots of the family 
${\mathcal Q}$.

\begin{theorem}[Fixing the Ordered List of the Roots]\label{OThom}
 For every realizable sign condition $\tau$ on ${\rm Elim}({\cal Q})$, there exist 
 $ {\bm \eta}(\tau) \in {\rm H}({\cal Q}),$ ${\bm \nu}(\tau) \in {\rm N}({\cal Q}, {\bm \eta}(\tau))$
 such that for every $\vartheta\in {\rm Real}(\tau, \R)$
  there exist 
 $\theta \in \R^{\#{\bm \eta}(\tau)}$, 
$\alpha \in \R^s, 
\beta   \in \R^s$
with $s = \sum_{P \in \cQ}{\# {\bm \nu}(\tau)
}$  such that
$${\rm OFact}({\cal Q}
(\vartheta)
)^{ {\bm \eta}(\tau),{\bm \nu}(\tau)}(\theta, \alpha+i\beta).$$ 
\end{theorem}

\begin{proof}{Proof.}
By usual properties of Thom encoding \cite[Proposition 2.28]{BPRbook}  and sign determination \cite[Theorem 27]{PerR}
a 
sign condition $\tau$ on ${\rm Elim}({\cal Q})$ 
determines 
the 
decomposition into irreducible real factors 
and the relative order between the real roots
of all polynomials in $\cQ$. 
\end{proof}

Before giving a weak inference form of Theorem \ref{OThom}, we define new auxiliary functions
(see Definitions \ref{defg4} and \ref{def:func_aux_hilb_1_and_2}).

\begin{definition} \begin{enumerate}
                    \item Let $\tilde{\rm g}_{H,3} : \N_* \to {\mathbb R}$, $\tilde{\rm g}_{H,3}\{p\} = \tilde{\rm g}_{H, 2}\{p, p\}$.
\item Let ${\rm g}_{6} : \N \times \N \times \N \times \N \times \N \to {\mathbb R}$,
$$\begin{array}{rcl}
{\rm g}_6\{p,s,e,f,g\}&=&\Big({\rm g}_4\{p\}g^{2^{\frac12p^2}}f^{2^{\frac12p}}
\Big)^{\frac{2^{s(\frac32p^2 + 2)}-1}{2^{\frac32p^2 + 2}-1}}
2^{(p+4)(2^{s(s-1)p^2}-1)2^{s(\frac32p^2 + 2)}} \cdot \\
& &\cdot
\max\{(ps-1)e + s-1, \tilde{\rm g}_{H,3}\{p\} \}
^{2^{s(s-1)p^2 + s(\frac32p^2 + 2)}-1}.
\end{array}
$$
                     \end{enumerate}
\end{definition}

We now give a weak inference form of Theorem \ref{OThom}.

\begin{theorem}[Fixing the Ordered List of the Roots as a weak existence]\label{weaksigndettab} 
Let $p \ge 1$,  $\cQ$
be 
a family of $s$ polynomials in $\K[u][y] \setminus \K$, monic in the variable $y$ with $\deg_y P \le p$ for 
every $P \in \cQ$, 
and 
$\tau$
be
  a realizable 
sign condition 
on ${\rm Elim}({\cal Q})$. Then
$$
\s({\rm Elim}({\cal Q}))=\tau \ \ \ \vdash \ \ \   \exists (t, z)\ 
[\; {\rm OFact}({\cal Q})^{ {\bm \eta}(\tau),{\bm \nu}(\tau)}(t, z) \;]
$$
where  
$t = (t_1, \dots, t_r)$ with $r = \#{\bm \eta}(\tau)$, $z_P = (z_{P, 1}, \dots, z_{P, \#{\bm \nu}(\tau)(P)})$ for $P \in {\cal Q}$ 
and $z = (z_P)_{P \in {\cal Q}}$.

Suppose we have an initial incompatibility in variables $(v, t, a, b)$, where
$v \supset u$, and $t, a, b$ are disjoint from $v$, with monoid part
$$
S \cdot 
\prod_{1 \le j < j' \le r}(t_{j} - t_{j'})^{2e_{{j}, j'}} 
\cdot
\prod_{P \in \cQ,
\atop 1 \le k \le 
\# {\bm \nu}(\tau)(P)
}b_{P,k}^{2f_{P,k}}
\cdot
\prod_{P \in \cQ,
\atop
1 \le k < k' \le 
\# {\bm \nu}(\tau)(P)
}
{\rm R}( z_{P,k}, z_{P,k'})
^{2g_{P, {k}, k'}},
$$
with 
$e_{j, j'} \le e  \in \N_*$,
$f_{P,k} \le f \in \N_*$,
$g_{P, k, k'} \le g \in \N_*$,
degree in $w$ bounded by $\delta_w$ for some subset of variables $w \subset v$, 
degree in $t_j$ bounded by $\delta_t$
and
degree in $(a_{P,k}, b_{P,k})$ bounded by $\delta_z$.
Then the final incompatibility has monoid part
$$
S^h \cdot
\prod_{H \in {\rm Elim}(\cQ), \atop \tau(H) \ne 0} H^{2h'_H} 
$$
with 
$
h, h'_H \le 
{\rm g}_6\{p,s,e,f,g\}\max\{(ps-1)e + s-1, 
\tilde{\rm g}_{H,3}\{p\}
\}
$
and degree in $w$ bounded by 
$$
\begin{array}{rll}
& {\rm g}_6\{p,s,e,f,g\}\cdot
\Big(\max\Big\{ 2^{ps(s-1)}(\delta_w + (ps(ps-1)(3e+1)+14)\deg_w\cQ), 
\tilde{\rm g}_{H,3}\{p\}
\deg_w \cQ
 \Big\} \\
+ & \max\Big\{
2^{ps(s-1)}(\delta_t + ((ps-1)(6e+2) + 15)p-8) + p,
2^{ps(s-1)}\delta_z + p,
\tilde{\rm g}_{H,3}\{p\}
\Big\}\deg_w\cQ\Big),
\end{array}$$
where 
$\deg_w \cQ = \max\{\deg_w P \ | \ P \in \cQ\}$.
\end{theorem}

\begin{proof}{Proof.}
We simplify the notation by renaming
$ {\bm \eta}(\tau) =  {\bm \eta}$,
and
${\bm \nu}(\tau) = {\bm \nu}$. 
Consider the initial incompatibility 
\begin{equation}\label{inc:init_inc_mixed_theorem}
\lda 
{\rm OFact}({\cal Q})^{ {\bm \eta},{\bm \nu}}(t, z),  \
\cH
\rda_{\K[v][ t, a, b]}
\end{equation}
where $\cH$ is a system of sign conditions in $\K[v]$. 

For $1 \le j < j' \le r$ there exists a polynomial $P$ in $\cQ$
such that $\eta_{j}(P) = 0$ and $\eta_{j} \prec_P \eta_{j'}$. 
We successively apply to (\ref {inc:init_inc_mixed_theorem}) for each such pair $(j, j')$ 
the weak inference
$$
t_{j} < t_{j'} \ \ \ \vdash \ \ \ t_{j} \ne t_{j'}
$$
if it is the case that exists $Q \in \cQ$ with ${\rm mu}(\eta_j, Q) > 0$
and ${\rm mu}(\eta_{j'}, Q) > 0$
and
$$
\s(\Der_+(P)(t_{j})) = \eta_{j}, \ 
\s(\Der_+(P)(t_{j'})) = \eta_{j'}
\ \ \ \vdash \ \ \ t_{j} < t_{j'}
$$
in every case. 
By Lemma \ref{lemma_basic_sign_rule_1} (item \ref{lemma_basic_sign_rule:1.5}) and 
Proposition \ref{thWITL_lower} we obtain 
\begin{equation}\label{inc:aux_inc_mixed_theorem_1}
\begin{array}{c}
\displaystyle{
\Big \downarrow
 \bigwedge_{P\in {\mathcal Q}} 
P
 \equiv 
{\rm F}^{{\rm vmu}(\bm \eta(P)), {\bm \nu}(P)}(t_P, z_P)
 ,  \ 
\bigwedge_{P\in {\mathcal Q}, \atop 1 \le k  \le 
\# {\bm \nu}(P)
}
b_{P,k} \ne 0, 
}
\\[4mm]
\displaystyle{
\bigwedge_{P\in {\mathcal Q}, \atop 1 \le k < k' \le 
\# {\bm \nu}(P)}
{\rm R}(z_{k}, z_{k'})
\ne 0, \
\bigwedge_{1  \le j \le r} \s(\Der_+(\cQ)(t_{j})) = \eta_{j}}, \
\cH
\Big \downarrow_{\K[v][ t, a, b]}
\end{array}
\end{equation}
with monoid part
$$
S \cdot
\prod_{1 \le j \le r, \,
Q \in \Der_+(\cQ), \atop \eta_j(Q) \ne 0}Q(t_j)^{2e_{Q, j}} 
\cdot
\prod_{P \in \cQ,
\atop
1 \le k \le 
\# {\bm \nu}(P)}b_{P,k}^{2f_{P,k}}
\cdot 
\prod_{P \in \cQ,
\atop
1 \le k < k' \le 
\# {\bm \nu}
(P)
}
{\rm R}(z_{P,k}, z_{P,k'})
^{2g_{P, {k}, k'}} 
$$
with $e_{Q, j} \le (r-1)e$, degree in $w$ bounded by $\delta_w + r(r-1)(3e+1)\deg_w \cQ$, 
degree in $t_j$ bounded by $\delta_t + (r-1)(6e+2)p$ and 
degree in $(a_{P, k}, b_{P, k})$ bounded by $\delta_z$.

For each $1 \le j \le r$, suppose that 
$\cQ_j$ is the list of polynomials $P$ in $\cQ$ such that
${\rm mu}(\eta_j, P) > 0$, $t_j$ is the $\alpha(j,P)$-th element in $t_P$ for 
$P\in \cQ_j$ and $P_{\gamma(j)}$ the first element of $\cQ_j$.
Conversly, suppose that for $P \in \cQ$ and $1 \le j' \le 
\# {\bm \eta}(P)$, 
the $j'$-th element in $t_P$ is $t_{\beta(P, j')}$.
We consider new variables $t'_P = (t'_{P, 1}, \dots, t'_{P,
\#  {\bm \eta} (P)})$ 
for every $P \in \cQ$ and we substitute 
$t_j$ by $t'_{P_{\gamma(j)}, \alpha(j,P_{\gamma(j)})}$ in (\ref{inc:aux_inc_mixed_theorem_1}) for $1 \le j 
\le r$. For
each $P \in \cQ$, let $\tilde t_P$ be the result obtained in each $t_P$ 
after these substitutions. 
Then we 
apply  the weak inference
$$
\bigwedge_{1 \le j \le r, \atop
P \in \cQ_j\setminus \{P_{\gamma(j)}\}} 
t'_{P_{\gamma(j)}, \alpha(j,P_{\gamma(j)})} =
t'_{P, \alpha(j,P)},  \
\bigwedge_{P \in \cQ} 
P
 \equiv 
{\rm F}^{{\rm vmu}(\bm \eta(P)), {\bm \nu}(P)}(t'_P, z_P)
  \ \ \ \vdash 
$$
$$
\vdash \ \ \  
\bigwedge_{P \in \cQ} P
\equiv 
{\rm F}^{{\rm vmu}(\bm \eta(P)), {\bm \nu}(P)}(\tilde t_P, z_P).
$$
By Lemma \ref{lem:comb_lin_zero_zero} we obtain 
\begin{equation}\label{inc:aux_inc_mixed_theorem_2}
\begin{array}{c}
\displaystyle{
\Big \downarrow
\bigwedge_{1 \le j \le r, \atop
P \in \cQ_j\setminus \{P_{\gamma(j)}\}} 
t'_{P_{\gamma(j)}, \alpha(j,P_{\gamma(j)})} = t'_{P, \alpha(j,P)}, \ 
\bigwedge_{P\in {\mathcal Q}} 
P
\equiv 
{\rm F}^{{\rm vmu}(\bm \eta(P)),{\bm \nu}(P)}(t'_P, z_P),  }
\\[5mm]
\displaystyle{
\bigwedge_{P\in {\mathcal Q}, \atop 1 \le k  \le 
{\# {\bm \nu}}(P)}  b_{P,k} \ne 0, \
\bigwedge_{P\in {\mathcal Q}, \atop 1 \le k < k' \le
{\# {\bm \nu}}(P)}
{\rm R}(z_{P,k}, z_{P,k'})
\ne 0,
} \\[5mm]
\displaystyle{
\bigwedge_{P\in {\mathcal Q}, \atop 1 \le j \le 
{\# {\bm \eta}}(P)}
\s(\Der_+(\cQ)(t'_{P, j})) = \eta_{\beta(P, j)}, \
\cH}
{\Big \downarrow}
_{\K[v][ (t'_P)_{P \in Q}, a, b]}
\end{array}
\end{equation}
with monoid part
$$
S \cdot
\prod_{P \in \cQ,
\atop
1 \le j \le 
{\# {\bm \eta}}(P)}
\prod_{Q \in \Der_+(\cQ), \atop \eta_{\beta(P,j)}(Q) \ne 0}Q(t'_{P,j})^{2e_{P,  Q, j}} 
\cdot
\prod_{P \in \cQ,
\atop 1 \le k \le 
{\# {\bm \nu}}(P)}
b_{P,k}^{2f_{P,k}}
\cdot
\prod_{P \in \cQ,
\atop
1 \le k < k' \le 
{\# {\bm \nu}}(P)}
{\rm R}( z_{P,k}, z_{P,k'})
^{2g_{P, {k}, k'}} 
$$
with $e_{P, Q, j} \le (r-1)e$, degree in $w$ bounded by 
$\delta_w + r(r-1)(3e+1)\deg_w \cQ$, 
degree in $t'_{P,j}$ bounded by $\delta_t + ((r-1)(6e+2) + 1)p$ and 
degree in $(a_{P, k}, b_{P, k})$ bounded by $\delta_z$. 
For simplicity we rename $t'_P$ as $t_P$ for every $P \in \cQ$ and $(t'_P)_{P \in Q}$ as $t$.

Then we successively apply to (\ref{inc:aux_inc_mixed_theorem_2}) for $1 \le j \le r$ 
and $P \in {\cal Q}_j \setminus \{P_{{\gamma(j)}}\}$ the weak inference 
$$
\s(\Der_+(P_{\gamma(j)})(
t_{P_{\gamma(j)}, \alpha(j,P_{\gamma(j)})}
)) = \eta_{j}
, \ 
\s(\Der_+(P_{\gamma(j)})(t_{P, \alpha(j,P)})) = \eta_{j}
\ \ \ \vdash
$$
$$
\vdash
\ \ \ t_{P_{\gamma(j)}, \alpha(j,P_{\gamma(j)})} = t_{P, \alpha(j,P)}.
$$
By Proposition \ref{thWITL_eq}, we obtain 
\begin{equation}\label{inc:aux_inc_mixed_theorem_3}
\begin{array}{c}
\displaystyle{
\Big \downarrow \bigwedge_{P\in {\mathcal Q}} 
P
 \equiv 
{\rm F}^{{\rm vmu}(\bm \eta(P)), {\bm \nu}(P)}(t_P, z_P),  \ 
\bigwedge_{P\in {\mathcal Q}, \atop 1 \le k  \le 
{\# {\bm \nu}}(P)}  b_{P,k} \ne 0, 
}
\\[4mm]
\displaystyle{
\bigwedge_{P\in {\mathcal Q}, \atop 1 \le k < k' \le 
{\# {\bm \nu}}(P)}
{\rm R}(z_{P,k}, z_{P,k'})
\ne 0,
}
\displaystyle{
\bigwedge_{P\in {\mathcal Q}, \atop 1 \le j \le 
\# {\bm \eta}(P)} 
\s(\Der_+(\cQ)(t_{P, j})) = \eta_{\beta(P, j)}, \
\cH
\Big \downarrow
_{\K[v][ t, a, b]}
}\end{array}
\end{equation}
with monoid part
$$
S \cdot
\prod_{P \in \cQ,
\atop
1 \le j \le 
\# {\bm \eta}(P)}
\prod_{Q \in \Der_+(\cQ), \atop \eta_{\beta(P,j)}(Q) \ne 0}Q(t_{P,j})^{2e'_{P,  Q, j}} 
\cdot 
\prod_{P \in \cQ,
\atop 1 \le k \le 
{\# {\bm \nu}}(P)}
b_{P,k}^{2f_{P,k}}
\cdot 
\prod_{P \in \cQ,
\atop
1 \le k < k' \le 
{\# {\bm \nu}}(P)}
{\rm R}( z_{P,k}, z_{P,k'})
^{2g_{P, {k}, k'}} 
$$
with $e'_{P, Q, j} \le (r-1)e + s-1 =: e'$, degree in $w$ bounded by 
$2^{r(s-1)}(\delta_w + (r(r-1)(3e+1)+14)\deg_w \cQ)$, 
degree in $t_{P,j}$ bounded by $2^{r(s-1)}(\delta_t + ((r-1)(6e+2) + 15)p-8)$
and 
degree in $(a_{P, k}, b_{P, k})$ bounded by $2^{r(s-1)}\delta_z$.

Then we apply to (\ref{inc:aux_inc_mixed_theorem_3})
the weak inference
$$
\bigwedge_{P \in {\cal Q}} P
\equiv 
{\rm F}^{{\rm vmu}(\bm \eta(P)), {\bm \nu}(P)}(t_P, z_P)
 \ \ \ \vdash \ \ \ 
\bigwedge_{P \in {\cal Q}, \atop 1 \le j \le 
\# {\bm \eta}(P)}
P(t_{P, j}) = 0. 
$$
By Lemma \ref{lem:comb_lin_zero_zero} we obtain 
\begin{equation}\label{inc:aux_inc_mixed_theorem_4}
\begin{array}{c}
\displaystyle{
\Big \downarrow \bigwedge_{P\in {\mathcal Q}}  P
 \equiv 
{\rm F}^{{\rm vmu}(\bm \eta(P)), {\bm \nu}(P)}(t_P, z_P)
,  \ 
\bigwedge_{P\in {\mathcal Q}, \atop 1 \le k  \le 
{\# {\bm \nu}}_{ P}
}  b_{P,k} \ne 0, \ 
\bigwedge_{P\in {\mathcal Q}, \atop 1 \le k < k' \le 
{\# {\bm \nu}}(P)}
{\rm R}( z_{P,k}, z_{P,k'})
\ne 0,
}
\\[7mm]
\displaystyle{
\bigwedge_{P \in {\cal Q}, \, Q \in \Der_+(\cQ) \setminus\{P\}, \atop 1 \le j \le \# {\bm \eta}( P) }
\s(Q(t_{P, j})) = \eta_{\beta(P, j)}(Q)}
, \
\cH
\Big \downarrow_{\K[v][ t, a, b]}
\end{array}
\end{equation}
with the same monoid part,
degree in $w$ bounded by 
$
\delta'_w := 2^{r(s-1)}(\delta_w + (r(r-1)(3e+1)+14)\deg_w \cQ),
$ 
degree in $t_{P,j}$ bounded by 
$
\delta'_t := 2^{r(s-1)}(\delta_t + ((r-1)(6e+2) + 15)p-8) + p
$ 
and 
degree in $(a_{P, k}, b_{P, k})$ bounded by 
$
\delta'_z := 2^{r(s-1)}\delta_z + p.
$

Now we fix an arbitrary order $(P_1, Q_1), \dots, (P_m, Q_m)$ in the set 
$\{(P, Q) \in \cQ \times \Der_+(\cQ) \ | \ Q \not \in \Der_+(P)\}$, note that
$m \le s(s-1)p$. For $1 \le i \le m$, we 
successively apply to (\ref{inc:aux_inc_mixed_theorem_4}) 
the weak inference
$$
\s({\rm ThElim}(P_{i}, Q_{i})) = \tau
, \ 
{\rm Th}(P_{i})^{{\rm vmu}(\bm \eta(P_{i})), {\bm \nu}(P_{i}), \bm \eta(P_{i})}(t_{P_{i}}, z_{P_{i}}) 
\ \ \ \vdash
$$
$$
\vdash
\ \ \ 
\bigwedge_{1 \le j \le 
{\# {\bm \eta}}( P_i)} 
\s(Q_{i}(t_{P_{i}, j})) = \eta_{\beta(P_{i}, j)}
.
$$
Using 
Theorem \ref{weaksigndetgen} (Fixing the Thom encodings with a Sign as a weak existence), 
it can be proved by induction on $i$  that for $1 \le i \le m$, after
the application of the $i$-{th} weak inference,  we obtain 
an incompatibility in $\K[v][ t, a, b]$
with monoid part
$$
S^{\tilde h_i}
\cdot
\prod_{H \in {\rm Elim}(\cQ), \atop \tau(H) \ne 0} H^{2\tilde h'_{H, i}} 
\cdot
\prod_{P \in \cQ,
\atop 1 \le j \le {\# {\bm \eta}(P)}
}
\Big(
\prod_{
1 \le h \le \deg_y P - 1,  \atop \eta_{\beta(P, j)}(P^{(h)}) \ne 0}P^{(h)}(t_{P, j})^{2\tilde e'_{P, j, h, i}}  
\cdot
\prod_{Q \in \Der_+(\cQ) \setminus \Der_+(P), \atop \eta_{\beta(P,j)}(Q) \ne 0}Q(t_{P,j})^{2\tilde e''_{P,  Q, j, i}} 
\Big)
\cdot
$$
$$
\cdot
\prod_{P \in \cQ,
\atop
1 \le j < j' \le 
{\# {\bm \eta}}(P)}(t_{P, j} - t_{P, j'})^{2\tilde e_{P,j, j', i}}
\cdot
\prod_{P \in \cQ,
\atop 1 \le k \le 
{\# {\bm \nu}}(P)}b_{P,k}^{2\tilde f_{P,k, i}}
\cdot 
\prod_{P \in \cQ,
\atop
1 \le k < k' \le 
{\# {\bm \nu}}(P)}
{\rm R}( z_{P,k}, z_{P,k'})
^{2\tilde g_{P, {k}, k', i}} 
$$
with 
\begin{eqnarray*}
\tilde e''_{P,  Q, j, i}& \le& 2^{((p+2)2^p - 2p - 2)\frac{2^{i p}-1}{2^p-1}}
\max\{e', \tilde{\rm g}_{H,3}\{p\} \}^{2^{i p}},\\
\tilde h_{i} &\le& 2^{((p+2)2^p - 2p - 2)\frac{2^{i p}-1}{2^p-1}}
\max\{e', \tilde{\rm g}_{H,3}\{p\} \}^{2^{i p}-1},\\
\tilde h'_{H, i}, \tilde e_{P, j, j', i} &\le&
2^{(p+2)2^p - 2}(2^{(p+2)2^p - 2p - 2}+1)^{\frac{2^{i p}-1}{2^p-1}-1}
\max\{e', \tilde{\rm g}_{H,3}\{p\} \}^{2^{i p}-1}\tilde{\rm g}_{H,3}\{p\},
\\
\tilde e'_{P, j, h, i} &\le&
2^{(p+2)2^p-1}(2^{(p+2)2^p - 2p - 2}+1)^{\frac{2^{i p}-1}{2^p-1}-1}
\max\{e', \tilde{\rm g}_{H,3}\{p\} \}^{2^{i p}},\\
\tilde f_{P, k, i} &\le&
2^{(p+2)2^p - 1}(2^{(p+2)2^p - 2p - 2}+1)^{\frac{2^{i p}-1}{2^p-1}-1}
\max\{e', \tilde{\rm g}_{H,3}\{p\} \}^{2^{i p}-1}\tilde{\rm g}_{H,3}\{p\}f,
\\
\tilde g_{P, k, i} &\le&
2^{(p+2)2^p - 1}(2^{(p+2)2^p - 2p - 2}+1)^{\frac{2^{i p}-1}{2^p-1}-1}
\max\{e', \tilde{\rm g}_{H,3}\{p\} \}^{2^{i p}-1}\tilde{\rm g}_{H,3}\{p\}g,
\end{eqnarray*}
degree in $w$ bounded by 
$
2^{((p+2)2^p - 2)\frac{2^{i p}-1}{2^p-1}}
\max\{e', \tilde{\rm g}_{H,3}\{p\} \}^{2^{i p}-1}
\max 
\{
\delta'_w,  
\tilde{\rm g}_{H,3}\{p\}\deg_w \cQ
\},
$
degree in $t_{P,j}$ bounded by 
$
2^{((p+2)2^p - 2)\frac{2^{i p}-1}{2^p-1}}
\max\{e', \tilde{\rm g}_{H,3}\{p\} \}^{2^{i p}-1}
\max 
\{
\delta'_t,  
\tilde{\rm g}_{H,3}\{p\}
\}
$
and 
degree in $(a_{P,k}, b_{P, k})$ bounded by 
$
2^{((p+2)2^p - 2)\frac{2^{i p}-1}{2^p-1}}
\max\{e', \tilde{\rm g}_{H,3}\{p\} \}^{2^{i p}-1}
\max 
\{
\delta'_z,  
\tilde{\rm g}_{H,3}\{p\}\}.
$
Therefore, at the end we obtain an incompatibility
\begin{equation}\label{inc:aux_inc_mixed_theorem_5}
\lda
\s({\rm Elim}(\cQ)) = \tau, \ 
\bigwedge_{P\in {\mathcal Q}} 
{\rm Th}(P)^{ {\rm vmu}({\bm \eta}(P)), {\bm \nu}(P), {\bm \eta}(P)}(t_P, z_P), \
\cH
\rda_{\K[v][ t, a, b]}
\end{equation}
with monoid part
$$
S^{\tilde h} \cdot
\prod_{H \in {\rm Elim}(\cQ), \atop \tau(H) \ne 0} H^{2\tilde h'_H} 
\cdot
\prod_{P \in \cQ,
\atop
1 \le j < j' \le 
{\# {\bm \eta}}(P)
}(t_{P, j} - t_{P, j'})^{2\tilde e_{P,j, j'}}
\cdot
\prod_{P \in \cQ,
\atop
1 \le k \le 
{\# {\bm \nu}}(P)}b_{P,k}^{2\tilde f_{P,k}}
\cdot
$$
$$
\cdot
\prod_{P \in \cQ,
\atop
1 \le k < k' \le 
{\# {\bm \nu}}(P)}
{\rm R}(z_{P,k}, z_{P,k'})
^{2\tilde g_{P, {k}, k'}} 
\cdot
\prod_{P \in \cQ,
\atop
1 \le j \le 
{\# {\bm \eta}}(P)} 
\prod_{
1 \le h \le \deg_y P - 1,  \atop \eta_{\beta(P, j)}(P^{(h)}) \ne 0}P^{(h)}(t_{P, j})^{2\tilde e'_{P, j, h}}  
$$
with 

\begin{eqnarray*}
\tilde h,
\tilde h'_{H},
\tilde e_{P, j_1, j_2}, \tilde e'_{P,  j, h}
&\le&
2^{(p+4)(2^{s(s-1)p^2}-1)}
\max\{e', \tilde{\rm g}_{H,3}\{p\} \}^{2^{s(s-1)p^2}},\\
\tilde f_{P, k} &\le&
2^{(p+4)(2^{s(s-1)p^2}-1)}
\max\{e', \tilde{\rm g}_{H,3}\{p\} \}^{2^{s(s-1)p^2}}
f,\\
\tilde g_{P, k_1, k_2} &\le&
2^{(p+4)(2^{s(s-1)p^2}-1)}
\max\{e', \tilde{\rm g}_{H,3}\{p\} \}^{2^{s(s-1)p^2}}g,
\end{eqnarray*}
degree in $w$ bounded by 
$
2^{(p+4)(2^{s(s-1)p^2}-1)}
\max\{e', \tilde{\rm g}_{H,3}\{p\} \}^{2^{s(s-1)p^2}-1}
\max 
\{
\delta'_w,  
\tilde{\rm g}_{H,3}\{p\}\deg_w \cQ
\},
$
degree in $t_{P,j}$ bounded by 
$
2^{(p+4)(2^{s(s-1)p^2}-1)}
\max\{e', \tilde{\rm g}_{H,3}\{p\} \}^{2^{s(s-1)p^2}-1}
\max 
\{
\delta'_t,  
\tilde{\rm g}_{H,3}\{p\}\},
$
and 
degree in $(a_{P,k}, b_{P, k})$ bounded by 
$
2^{(p+4)(2^{s(s-1)p^2}-1)}
\max\{e', \tilde{\rm g}_{H,3}\{p\} \}^{2^{s(s-1)p^2}-1}
\max 
\{
\delta'_z,  
\tilde{\rm g}_{H,3}\{p\}\}.
$

Finally we fix an arbitrary order $P_1, \dots, P_s$ in $\cQ$ and for $1 \le i \le s$ we 
successively apply to (\ref{inc:aux_inc_mixed_theorem_5}) 
the weak inference
$$
\s({\rm ThElim}(P_i) ) = \tau
\ \ \
\vdash
\ \ \ \exists (t_{P_i}, z_{P_i}) \ [\; \Thom^{ {\rm vmu}({\bm \eta}(P_i)),{\bm \nu}(P_i), {\bm \eta}(P_i)}(t_{P_i}, z_{P_i}) \;].
$$
Using Theorem \ref{weaksigndet} (Fixing the Thom encodings as a weak existence), it can be proved by induction on $i$
that for $1 \le i \le s$, after the application of the $i$-th weak inference, 
we obtain an incompatibility in $\K[v][ t_{P_{i + 1}}, \dots,$ $t_{P_{s}},$  
$a_{P_{i + 1}}, b_{P_{i + 1}}, \dots,$ $a_{P_{s}}, b_{P_{s}}]$
with monoid part 
$$
S^{h_{i}}
\cdot
\prod_{H \in {\rm Elim}(\cQ), \atop \tau(H) \ne 0} H^{2 h'_{H, i}} 
\cdot
\prod_{i + 1 \le i' \le s,
\atop 
1 \le j < j' \le 
{\# {\bm \eta}}(P_i)}(t_{P_i, j} - t_{P_i, j'})^{2e_{P_i,j, j', i}}
\cdot
\prod_{i + 1 \le i' \le s,
\atop
1 \le k \le 
{\# {\bm \eta}}(P_{i'})}b_{P_{i'},k}^{2f_{P_{i'},k, i}}
\cdot
$$
$$
\cdot
\prod_{i + 1 \le {i'} \le s,
\atop
1 \le k < k' \le 
{\# {\bm \nu}}(P_{i'})}
{\rm R}( z_{P,k}, z_{P,k'})
^{2g_{P_{i'}, {k}, k', i}} 
\cdot
\prod_{i + 1 \le {i'} \le s,
\atop
1 \le j \le 
{\# {\bm \eta}}(P_{i'})} 
\prod_{1 \le h \le \deg_y P_{i'} - 1,  \atop \eta_{\beta(P_{i'}, j)}(P^{(h)}) \ne 0}P^{(h)}(t_{P_{i'}, j})^{2 e'_{P_{i'}, j, h, i}}   
$$
with, denoting
$$G_i:=
\Big(
{\rm g}_4\{p\}g^{2^{\frac12p^2}}f^{2^{\frac12p}}
\Big)^{\frac{2^{i(\frac32p^2 + 2)}-1}{2^{\frac32p^2 + 2}-1}}
2^{(p+4)(2^{s(s-1)p^2}-1)2^{i(\frac32p^2 + 2)}}
\max\{e', \tilde{\rm g}_{H,3}\{p\} \}
^{2^{s(s-1)p^2 + i(\frac32p^2 + 2)}-1},
$$
\begin{eqnarray*}
h_i, h'_{H, i}, e_{P_{i'}, j, j', i}, e'_{P_{i'}, j, h, i} &\le& G_i \max\{e', \tilde{\rm g}_{H,3}\{p\} \}
\\
f_{P_{i'},k, i} &\le& G_i \max\{e', \tilde{\rm g}_{H,3}\{p\} \} f,
\\
g_{P_{i'}, k, k', i} &\le&
G_i \max\{e', \tilde{\rm g}_{H,3}\{p\} \}g,
\end{eqnarray*}
degree in $w$ bounded by 
$
G_i
\Big(\max\{\delta'_w, 
\tilde{\rm g}_{H,3}\{p\}\deg_w \cQ
 \} + 
\max\{\delta'_t,   \delta'_z, 
\tilde{\rm g}_{H,3}\{p\}\}\deg_w\cQ\Big),
$
degree in $t_{P_{i'}, j}$ bounded by 
$
G_i
\max\{\delta'_t, \tilde{\rm g}_{H,3}\{p\}\}
$
and degree in $(a_{P_{i'}, k}, _{P_{i'}, k})$ bounded by 
$
G_i
\max\{\delta'_z,
\tilde{\rm g}_{H,3}\{p\}\}.
$
Therefore, at the end we obtain 
$$
\lda  \s({\rm Elim}(\cQ)) = \tau, \
\cH
\rda 
_{\K[v]}
$$
with monoid part
$$
S^h \cdot
\prod_{H \in {\rm Elim}(\cQ), \atop \tau(H) \ne 0} H^{2h'_H} 
$$
with the respective bounds replacing $i$ by $s$,
which serves as the final incompatibility. 
\end{proof}

We finish this subsection with the following remark, which will be used in Section \ref{secMainTh}.

\begin{remark}\label{rem:degree_and_number_elim}  Let $\cQ = \{P_1, \dots, P_s\} \subset \K[u][y]$
with $P_i$ monic in the variable $y$ and $\deg_y P_i \le p$ for $1 \le i \le s$.
Following Definition \ref{def:sign_det_sect_three_first}, by 
Remark \ref{rem:degree_and_number_thelim}  there are at most 
$$
4s^2p^{{\rm bit}\{p\} + 1}$$
elements in  ${\rm Elim}(\cQ)$ and
their degrees in $u$ are bounded by 
$$
2p^2( {\rm bit}\{p\}+1)\max\{\deg_u P_i \ | \ 1 \le i \le s\} \le 
4p^3\max\{\deg_u P_i \ | \ 1 \le i \le s\}.
$$
\end{remark}

\subsection{Realizable sign conditions on a family of polynomials}
\label{sect_cond_family}

From the family ${\rm Elim}({\mathcal Q})\subset \K[u]$  defined in Subsection \ref{sect_factor_family}, 
we deduce now the list of realizable 
sign conditions  on ${\mathcal Q}$.

\begin{theorem}[Elimination of One Variable]\label{defSIGN}
 For every realizable sign condition $\tau$ on ${\rm Elim}({\cal Q})$, there exists a list of 
 sign conditions on ${\cal Q}$ 
 $$
{\rm SIGN}( {\mathcal Q}  \, | \tau )
$$
 such that for every $\vartheta = (\vartheta_1, \dots, \vartheta_k)\in {\rm Real}(\tau, \R)$,
the list of realizable sign conditions on ${\cal Q}(\vartheta)$
 is
${\rm SIGN}( {\mathcal Q}  \, | \tau )$.
\end{theorem}

\begin{proof}{Proof.}
The result is immediate from Theorem  \ref{OThom} (Fixing the Ordered List of the Roots), 
since once the factorization
and relative order between the real roots of all the polynomial in ${\cal Q}$ is fixed,  
the
list of all realizable 
sign conditions on ${\cal Q}$
can be determined 
by looking at the partition of the real line given by the set of real roots. 
\end{proof}

Before stating  the main result of Section \ref{sect_elim_of_one_var}, we introduce an auxiliary 
function.

\begin{definition}\label{def_func_g_5}
Let ${\rm g}_7: \N \times \N \times \N \to {\mathbb R}$, ${\rm g}_7\{p,s,e\} = 
2^{2^{3(\frac{p}2)^p + s^2(\frac32p^2 + 3) + 8}}e^{2^{6p^2s^2}}$.
\end{definition}

\begin{Tlemma}\label{tlemma_almost_final} For every $p, s, e \in \N_*$, 
$$
2^{ps(s-1)+2}e^2s^4p\,{\rm g}_6\{p, s, 2eps + 8(eps)^2, (ps+1)ep + 4e^2p^3s^2, 1\}\max\{8e^2p^3s^3, \tilde{\rm g}_{H,3}\{p\} \}
\le
$$
$$ 
\le {\rm g}_7\{p, s, e\}.
$$
\end{Tlemma}
\begin{proof}{Proof.} See Section \ref{section_annex}.
\end{proof}

The main result of Section \ref{sect_elim_of_one_var} is the following weak inference form of Theorem \ref{defSIGN}.

\begin{theorem}[Elimination of One Variable as a weak inference]\label{elimination_theorem}
Let $p \ge 1$,  $\cQ$ be
a family of $s$ polynomials in $\K[u][y] \setminus \K$, monic in the variable $y$ with $\deg_y P \le p$ for 
every $P \in \cQ$ and 
$\tau$ be a realizable 
sign condition 
on ${\rm Elim}({\cal Q})$.
Then
$$
\s({\rm Elim}({\cal Q}))=\tau \ \ \ \vdash \ \ \  
\bigvee_{\sigma \in {\rm SIGN}( {\mathcal Q}  \, | \tau )}
\s({\mathcal Q})= \sigma.
$$

Suppose we have initial incompatibilities with monoid part
$$
S_{\sigma}\cdot 
\prod_{P \in {\cal Q}, \atop \sigma(P) \ne 0}P
^{2e_{P, \sigma}}
$$
with $e_{P, \sigma} \le e \in \N_*$ and 
degree in $w \subset v$ bounded
by $\delta_w$. Then, the final incompatibility 
has monoid part 
$$
\prod_{\sigma \in {\rm SIGN}( {\mathcal Q}  \, | \tau )}
S_{\sigma}^{h_{\sigma}}
\cdot
\prod_{H \in {\rm Elim}({\cal Q}), \atop \tau(H) \ne 0}H^{2h'_H}
$$
with 
$h_{\sigma}, h'_H \le  {\rm g}_7\{p,s,e\}$
and 
degree in $w$ bounded by 
$
{\rm g}_7\{p,s,e\}  \max\{\delta_w ,  \deg_w \cQ\}.
$ 
\end{theorem}

As said before, once the factorization
and relative order between the real roots of all polynomials in ${\cal Q}$ is fixed,  
the
list of all realizable  
sign conditions on ${\cal Q}$
can be determined 
by looking at the partition of the real line given by the set of real roots. 
We prove now  weak inference version of some auxiliary results in this direction.

\begin{proposition} 
\label{propSegLine} 
$$
\vdash \ \ \ 
 y < t_1 \ \vee \ \ y = t_1 \  \vee  ( \ t_1 < y, \ y < t_2)  \
 \vee \ \dots \ \vee  \ ( \ t_{r-1} < y, \  y < t_{r}) \ \vee \ y = t_r  \ \vee \ t_r < y.
$$

Suppose we
have initial incompatibilities in variables $v \supset (t_1, \dots, t_r, y) $ with monoid part 
$$S'_{1}(y-t_1)^{2e_1}, \ S_1, \ S'_{2}(y-t_1)^{2f_1}(y-t_2)^{2e_2},  \
 \dots, \ S'_{r}(y-t_{r-1})^{2f_{r-1}}(y-t_r)^{2e_r},  \ S_r, \  \ S'_{r+1}(y-t_r)^{2f_r}$$
with $e_j \le e$ and $f_j \le e$ and
degree in $w$ bounded by $\delta'_{w, 1},$ $\delta_{w,1},$ $\delta'_{w, 2}, \dots,$ $\delta'_{w, r},$ 
$\delta_{w, r},$ $\delta'_{w, r+1}$ for some subset of variables $w \supset v$.  
Then, the final incompatibility has monoid part 
$$
\prod_{1 \le j \le r+1} S'_j 
\cdot
\prod_{1 \le j \le r} S_j^{2(e_j+f_j)}
$$
and degree in $w$ bounded by 
$\sum_{1 \le j \le r+1}\delta'_{w, j} + 4e \cdot \sum_{1 \le j \le r}\delta_{w,j}$.
\end{proposition}

When $t_1, \dots, t_r$ are not variables but elements in $\K$,  
similar degree estimations are due to Warou \cite{War}.

\begin{proof}{Proof.} 
Consider the initial incompatibilities
\begin{equation}\label{inc:init_part}
\lda y < t_1, \ \cH \rda, \   \dots, \ 
\lda t_{r-1} < y, \ y < t_r, \ \cH \rda, \   \lda y = t_r, \ \cH \rda, \ \lda t_r < y, \ \cH \rda_{\K[v]}
\end{equation}
where $\cH$ is a system of sign conditions in $\K[v]$. 

We proceed by induction on $r$. If $r = 1$, the result follows from 
Lemma \ref{CasParCas_3}. 
Suppose now $r > 1$. 
We apply to the last three initial incompatibilities (\ref{inc:init_part}) the weak inference 
$$ 
\vdash \ \ \
y < t_r  \ \vee \ y = t_r \ \vee \  t_r < y.
$$
By Lemma \ref{CasParCas_3} 
we obtain an incompatibility
\begin{equation}\label{inc:aux_part_1}
\lda t_{r-1} < y, \ \cH \rda_{\K[v]} 
\end{equation}
with monoid part 
$$S'_r \cdot S'_{r+1} \cdot S_r^{2(e_r + f_r)} \cdot (y - t_{r-1})^{2f_{r-1}}$$
and degree in $w$ bounded by 
$\delta'_{w, r} + \delta'_{w, r+1} +  4e \cdot \delta_{w, r}$.
The result follows by applying 
the inductive hypothesis to the remaining initial incompatibilities 
(\ref{inc:init_part}) and 
(\ref{inc:aux_part_1}). 
\end{proof}

\begin{lemma}\label{partition_after_fact_1}
Let $p \ge 1$,  $\cQ$ 
be
a family of $s$ polynomials in $\K[u][y] \setminus \K$, monic in the variable $y$ with 
$\deg_y P \le p$ for 
every $P \in \cQ$, 
$\tau$
be
 a realizable 
sign condition
on ${\rm Elim}({\cal Q})$, $ {\bm \eta}(\tau) = [\eta_1, \dots, \eta_r]$ with $r \ge 1$
and $1 \le j_0 \le r$. 
Then, defining
$\varepsilon_P=(-1)^{\sum_{j_0 +1 \le j' \le r} {\rm mu}(\eta_{j'}, P)}$,
$$ 
\exists  (t, z)  \ [\; 
{\rm OFact}({\cal Q})^{ {\bm \eta}(\tau),{\bm \nu}(\tau)}(t, z), \ y = t_{j_0} \;] 
\ \ \ \vdash
\ \ \ 
\bigwedge_{P \in \cQ, \atop {\rm mu}(\eta_{j_0},P) > 0} P
 = 0, \  
\bigwedge_{P \in \cQ, \atop  {\rm mu}(\eta_{j_0},P) = 0} \s(P
) = \varepsilon_P
$$
where $t = (t_1, \dots, t_r)$, $z = (z_P)_{P \in \cQ}$ and 
$z_P = (z_{P, 1}, \dots, z_{P,
{\# {\bm \nu}}(\tau)(P)})$.

Suppose we have an initial incompatibility in variables $v \supset (u, y)$
with monoid part 
$$
S \cdot 
\prod_{P \in \cQ, \atop {\rm mu}(\eta_{j_0} ,P)  = 0}   P
^{2e_P}
$$
with $e_P \le e \in \N_*$ and  
degree in $w$ bounded by $\delta_w$ for some set of variables $w \subset v$. 
Then, the final incompatibility has 
monoid part 
$$
S \cdot
\prod_{1 \le j' \le r \atop j' \ne j_0}(t_{j_0}-t_{j'})^{2e_{j'}} 
\cdot 
\prod_{P \in \cQ
\atop
1 \le k \le 
{\# {\bm \nu}}(\eta)(P)}b_{P,k}^{2e_{P,k}}
$$
with 
$e_{j'} \le eps$, $e_{P, k} \le ep$, 
degree in $w$ bounded by $\delta_w$, 
degree in $t_j$ bounded by $2epsr$ 
and 
degree in $(a_{P, k}, b_{P, k})$  
bounded by $2ep$.
\end{lemma}

\begin{proof}{Proof.}
We simplify the notation by renaming
$ {\bm \eta}(\tau) =  {\bm \eta}$ and
${\bm \nu}(\tau) =  {\bm \nu} $. 
 Consider the initial incompatibility 
\begin{equation}\label{inc:init_lemma_root_sign}
\Big \downarrow \   \bigwedge_{P \in \cQ, \atop {\rm mu}(\eta_{j_0}, P) > 0} P
 = 0, \  
\bigwedge_{P \in \cQ, \atop {\rm mu}(\eta_{j_0}, P)  = 0} \s(P
) = 
\varepsilon_P, \ \cH  \ {\Big \downarrow}_{\K[v]}
\end{equation}
where $\cH$ is a system of sign conditions in $\K[v]$.

Following the notation from Definition \ref{defFact} and Definition \ref{def:sign_det_sect_three_sec},  
we  apply to (\ref{inc:init_lemma_root_sign}) the weak inference
$$
\bigwedge_{P \in {\cal Q}, \atop {\rm mu}(\eta_{j_0}, P) > 0}
P
 \equiv 
{\rm F}^{ {\rm vmu}( {\bm \eta}(P)),    {\bm \nu}(P)}(t_P,z_P)
,  \ y = t_{j_0} 
\ \ \ \vdash
\ \ \ 
\bigwedge_{P \in {\cal Q}, \atop {\rm mu}(\eta_{j_0}, P) > 0} P
= 0.
$$
By Lemma \ref{lem:comb_lin_zero_zero}, we obtain 
\begin{equation}\label{inc:aux_lemma_root_sign_1}
\begin{array}{c}
\Big \downarrow
\displaystyle{
\bigwedge_{P \in {\cal Q}, \atop {\rm mu}(\eta_{j_0}, P) > 0}
P
 \equiv {\rm F}^{ {\rm vmu}( {\bm \eta}(P)),    {\bm \nu}(P)}(t_P,z_P ), }\
y = t_{j_0}, \\[4mm]
\displaystyle{ \bigwedge_{P \in \cQ, \atop {\rm mu}(\eta_{j_0}, P) = 0} \s(P
) = 
\varepsilon_P, } \ 
\cH \
\Big \downarrow
_{\K[v][ t, a, b]}
\end{array}
\end{equation}
with the same monoid part, 
degree in $w$ bounded by $\delta_w$, 
degree in $t_j$  and 
degree in $(a_{P, k}, b_{P, k})$  bounded by $p$. 

Then we  successively apply to (\ref{inc:aux_lemma_root_sign_1}) for $P \in {\cal Q}$ with ${\rm mu}(\eta_{j_0},P) = 0$
the weak inferences
$$
P
 \equiv {\rm F}^{ {\rm vmu}( {\bm \eta}(P)),    {\bm \nu}(P)}(t_P,z_P
  ), \ 
\s( {\rm F}^{ {\rm vmu}( {\bm \eta}(P)),    {\bm \nu}(P)}(t_P,z_P
) )
=\varepsilon_P
  \ \, \vdash \ \, 
\s(P
) = 
\varepsilon_P$$
and
$$
 \displaystyle{
 \bigwedge_{1 \le j' \le j_0 - 1} t_{j'} < y, \ 
\bigwedge_{j_0 + 1 \le j' \le r} y < t_{j'}, \
\bigwedge_{1 \le k \le 
{\# {\bm \nu}}(P)}(y-a_{P,k})^2 + b_{P,k}^2 > 0 } \ \, \vdash $$
$$ \vdash \ \, 
\sign( 
{\rm F}^{ {\rm vmu}( {\bm \eta}(P)),    {\bm \nu}(P)}(t_P,z_P
) )
 =
 \varepsilon_P, 
$$
and for $P \in {\cal Q}$ with ${\rm mu}(\eta_{j},P) = 0$
and $1 \le k \le 
{\# {\bm \nu}}(P)$
the weak inferences
$$
\begin{array}{rcl}
(y-a_{P,k})^2 \ge 0, \ b^2_{P, k} > 0  & \ \, \vdash \ \, &
(y-a_{P,k})^2 + b_{P,k}^2 > 0, \\[3mm]
& \vdash & (y-a_{P,k})^2 \ge 0,\\[3mm]
b_{P,k} \ne 0 & \vdash & 
b_{P,k}^2 > 0.
\end{array}
$$
By Lemmas \ref{lem_sust_equiv_greater}, \ref{lemma_basic_sign_rule_1} 
(items \ref{lemma_basic_sign_rule:7}, \ref{lemma_basic_sign_rule:2} and \ref{lemma_basic_sign_rule:3})
and \ref{lemma_sum_of_pos_is_pos} 
we obtain 
\begin{equation}\label{inc:aux_lemma_root_sign_2}
\begin{array}{c}
\Big \downarrow
\displaystyle{
\bigwedge_{P \in {\cal Q}}
P
 \equiv {\rm F}^{ {\rm vmu}( {\bm \eta}(P)),    {\bm \nu}(P)}(t_P,z_P
  ) )}, \\[4mm]
\displaystyle{y = t_{j_0}, \ \bigwedge_{1 \le j' \le j_0 - 1} t_{j'} < y, \ 
\bigwedge_{j_0 + 1 \le j' \le r} y < t_{j'}, \
\bigwedge_{P \in \cQ
\atop
1 \le k \le 
{\# {\bm \nu}}(P)} b_{P, k} \ne 0, \ 
\cH 
}
\Big \downarrow
_{\K[v][ t, a, b]}
\end{array}
\end{equation}
with monoid part 
$$
S \cdot
\prod_{1 \le j' \le r, \atop j' \ne j_0}(y-t_{j'})^{2e_{j'}} 
\cdot
\prod_{P \in \cQ
\atop 1 \le k \le 
{\# {\bm \nu}}(P)}
b_{P,k}^{2e_{P,k}}
$$
with 
$e_{j'} \le eps$,  
$e_{P, k} \le ep$, 
degree in $w$ bounded by $\delta_w$, 
degree in $t_j$ bounded by $2eps$ 
(taking into account that ${\rm mu}(\eta_{j_0}, P_0) > 0$ for at least one $P_0 \in \cQ$)
and 
degree in $(a_{P, k}, b_{P, k})$  
bounded by $2ep$ (taking into account that for each $P \in \cQ$, either ${\rm mu}(\eta_{j_0}, P) > 0$ 
or ${\rm mu}(\eta_{j_0}, P) = 0$).

Finally, we successively apply to (\ref{inc:aux_lemma_root_sign_2}) for $1 \le j' \le j_0 - 1$ the weak inference
$$
t_{j'} < t_{j_0}, \ t_{j_0} = y \ \ \  \vdash \ \ \  t_{j'} < y  
$$
and for  $j + 1  \le j' \le r$ the weak inference
$$
y = t_{j_0}, \ t_{j_0} < t_{j'} \ \ \ \vdash \ \ \  y < t_{j'}. 
$$
By Lemma \ref{lemma_sum_of_pos_is_pos} 
we obtain 
$$
\begin{array}{c}
\Big \downarrow
\displaystyle{
\bigwedge_{P \in {\cal Q}}
P
\equiv {\rm F}^{ {\rm vmu}( {\bm \eta}(P)),    {\bm \nu}(P)}(t_P,z_P
) )}, \\[4mm] 
\displaystyle{y = t_{j_0}, \ \bigwedge_{1 \le j' \le j_0 - 1} t_{j'} < t_{j_0}, \ 
\bigwedge_{j_0 + 1 \le j' \le r} t_{j_0} < t_{j'}, \
\bigwedge_{P \in \cQ
\atop
1 \le k \le 
{\# {\bm \nu}}(P)} b_{P, k} \ne 0, \ 
\cH 
}
\Big \downarrow
_{\K[v][ t, a, b]}
\end{array}
$$
with monoid part 
$$
S \cdot
\prod_{1 \le j' \le r, \atop j' \ne j_0}(t_{j_0}-t_{j'})^{2e_{j'}} 
\cdot
\prod_{P \in \cQ
\atop
1 \le k \le 
{\# {\bm \nu}}(P)}b_{P,k}^{2e_{P,k}} 
$$
with 
degree in $w$ bounded by $\delta_w$, 
degree in $t_j$ bounded by $2epsr$
and 
degree in $(a_{P, k}, b_{P, k})$  
bounded by $2ep$,
which serves as the final incompatibility. 
\end{proof}

\begin{lemma}\label{partition_after_fact_2}
Let $p \ge 1$,  $\cQ$ 
be
a family of $s$ polynomials in $\K[u][y] \setminus \K$, monic in the variable $y$ with $\deg_y P \le p$ for 
every $P \in \cQ$, 
$\tau$
be
 a realizable 
sign condition 
on ${\rm Elim}({\cal Q})$, $ {\bm \eta}(\tau) = [\eta_1, \dots, \eta_r]$  with $r > 1$
and $1 \le j_0 \le r-1$. 
Then, 
defining
$\varepsilon_P=(-1)^{\sum_{j_0 +1 \le j' \le r} {\rm mu}(\eta_{j'}, P)}$,
$$ 
\exists  (t, z)  \ [\; 
{\rm OFact}({\cal Q})^{ {\bm \eta}(\tau),{\bm \nu}(\tau)}(t, z), \ t_{j_0} < y, \ y < t_{j_0+1}\;] 
\ \ \ \vdash
\ \ \ 
\bigwedge_{P \in \cQ} \s(P
) = 
\varepsilon_P
$$
where $t = (t_1, \dots, t_r)$, $z = (z_P)_{P \in \cQ}$ and 
$z_P = (z_{P, 1}, \dots, z_{P, 
{\# {\bm \nu}}(\tau)(P)})$.

Suppose we have an initial incompatibility in variables $v \supset (u, y)$
with monoid part 
$$
S \cdot
\prod_{P \in \cQ}   P
^{2e_P}
$$
with $e_P \le e \in \N_*$ and 
degree in $w$ bounded by $\delta_w$ for some set of variables $w \subset v$.
Then, the final incompatibility has 
monoid part 
$$
S \cdot (y - t_{j_0})^{2e_{j_0}} \cdot (y - t_{j_0+1})^{2e_{j_0+1}}
\cdot 
\prod_{1 \le j' \le j_0 - 1}(t_{j_0}-t_{j'})^{2e_{j'}} 
\cdot
\prod_{j_0 + 2 \le j' \le r}(t_{j_0+1}-t_{j'})^{2e_{j'}} 
\cdot 
\prod_{P \in \cQ
\atop 
1 \le k \le 
{\# {\bm \nu}}(P)}b_{P,k}^{2e_{P,k}},
$$
with 
$e_j \le eps$,  
$e_{P, k} \le ep$, 
degree in $w$ bounded by $\delta_w$, 
degree in $t_j$ bounded by $2epsr$ 
and 
degree in $(a_{P, k}, b_{P, k})$  
bounded by $2ep$.
\end{lemma}

\begin{proof}{Proof.} 
We simplify the notation by renaming
$ {\bm \eta}(\tau) =  {\bm \eta}$ and
${\bm \nu}(\tau) = {\bm \nu} 
$. 
Consider the initial incompatibility 
\begin{equation}\label{inc:init_lemma_root_sign_sec}
\Big \downarrow \ \bigwedge_{P \in \cQ} \s(P
) = 
\varepsilon_P, \ \cH \ {\Big \downarrow}_{\K[v]}
\end{equation}
where $\cH$ is a system of sign conditions in $\K[v]$.

We  successively apply to (\ref{inc:init_lemma_root_sign_sec}) 
for $P \in {\cal Q}$ the weak inferences
$$ 
P
 \equiv 
{\rm F}^{ {\rm vmu}( {\bm \eta}(P)),    {\bm \nu}(P)}(t_P,z_P
) ), \
\s\Big( 
{\rm F}^{ {\rm vmu}( {\bm \eta}(P)),    {\bm \nu}(P)}(t_P,z_P
) )
\Big) 
=
\varepsilon_P
 \ \, \vdash \ \, 
\s(P
) = 
\varepsilon_P $$
and
$$
\displaystyle{\bigwedge_{1 \le j' \le j_0 } t_{j'} < y, \ 
\bigwedge_{j_0 + 1 \le j' \le r} y < t_{j'}, \
\bigwedge_{1 \le k \le 
{\# {\bm \nu}}(P)}(y-a_{P,k})^2 + b_{P,k}^2 > 0}  \ \, \vdash 
$$
$$
\vdash \ \, 
\sign( 
{\rm F}^{ {\rm vmu}( {\bm \eta}(P)),    {\bm \nu}(P)}(t_P,z_P
) )
)
 =\varepsilon_P, 
$$
and for $P \in {\cal Q}$ 
and $1 \le k \le 
{\# {\bm \nu}}(P)$
the weak inferences
$$
\begin{array}{rcl}
(y-a_{P,k})^2 \ge 0, \ b^2_{P, k} > 0  & \ \, \vdash \ \, &
(y-a_{P,k})^2 + b_{P,k}^2 > 0, \\[3mm]
& \vdash & (y-a_{P,k})^2 \ge 0,\\[3mm]
b_{P,k} \ne 0 & \vdash & 
b_{P,k}^2 > 0.
\end{array}
$$
By Lemmas \ref{lem_sust_equiv_greater}, \ref{lemma_basic_sign_rule_1} 
(items \ref{lemma_basic_sign_rule:7}, \ref{lemma_basic_sign_rule:2} and \ref{lemma_basic_sign_rule:3})
and \ref{lemma_sum_of_pos_is_pos} 
we obtain 
\begin{equation}\label{inc:aux_lemma_root_sign_sec_2}
\begin{array}{c}
\Big \downarrow
\displaystyle{
\bigwedge_{P \in {\cal Q}}
P
\equiv {\rm F}^{ {\rm vmu}( {\bm \eta}(P)),    {\bm \nu}(P)}(t_P,z_P
), } \\[4mm]
\displaystyle{\bigwedge_{1 \le j' \le j_0} t_{j'} < y, \ 
\bigwedge_{j_0 + 1 \le j' \le r} y < t_{j'}, \
\bigwedge_{P \in \cQ
\atop
1 \le k \le 
{\# {\bm \nu}}(P)} b_{P, k} \ne 0, \ 
\cH 
}
\Big \downarrow
_{\K[v][ t, a, b]}
\end{array}
\end{equation}
with monoid part 
$$
S \cdot
\prod_{1 \le j \le r}(y-t_j)^{2e_j} 
\cdot
\prod_{P \in \cQ
\atop
1 \le k \le 
{\# {\bm \nu}}(P)}b_{P,k}^{2e_{P,k}}
$$
with 
$e_j \le eps$,  
$e_{P, k} \le ep$, 
degree in $w$ bounded by $\delta_w$, 
degree in $t_j$ bounded by $2eps$
and 
degree in $(a_{P, k}, b_{P, k})$  
bounded by $2ep$.

Finally, we successively apply to (\ref{inc:aux_lemma_root_sign_sec_2}) for $1 \le j' \le j_0 - 1$ 
the weak inferences
$$
\begin{array}{rcl}
t_{j'} < t_{j_0}, \ t_{j_0} \le y & \; \  \vdash \; \ &  t_{j'} < y, \\[3mm]
t_{j_0} < y & \vdash & t_{j_0} \le y
\end{array}
$$
and for  $j_0 + 2  \le j' \le r$ the weak inferences
$$
\begin{array}{rcl}
y \le t_{j_0+1}, \ t_{j_0+1} < t_{j'}  & \; \  \vdash \; \ &  y < t_{j'}, \\[3mm]
y < t_{j_0+1} & \vdash & y \le t_{j_0+1}.
\end{array}
$$
By Lemmas \ref{lemma_sum_of_pos_is_pos} and \ref{lemma_basic_sign_rule_1} 
(item \ref{lemma_basic_sign_rule:1})
we obtain 
$$
\begin{array}{c}
\Big \downarrow
\displaystyle{
\bigwedge_{P \in {\cal Q}}
P
 \equiv {\rm F}^{ {\rm vmu}( {\bm \eta}(P)),    {\bm \nu}(P)}(t_P,z_P
  ), } \\[4mm]
\displaystyle{t_{j_0}< y, \ y < t_{j_0+1}, \ 
\bigwedge_{1 \le j' \le j_0 - 1} t_{j'} < t_{j_0}, \ 
\bigwedge_{j_0 + 2 \le j' \le r} t_{j_0 + 1} < t_{j'}, \
\bigwedge_{P \in \cQ
\atop
1 \le k \le 
{\# {\bm \nu}}(P)} b_{P, k} \ne 0, \ 
\cH 
}
\Big \downarrow
_{\K[v][ t, a, b]}
\end{array}
$$
with monoid part 
$$
S \cdot
(y - t_{j_0})^{2e_{j_0}}
\cdot
(y - t_{j_0+1})^{2e_{j_0+1}}
\cdot
\prod_{1 \le j' \le j_0 - 1}(t_{j_0}-t_{j'})^{2e_{j'}} 
\cdot
\prod_{j_0 + 2 \le j' \le r}(t_{j_0+1}-t_{j'})^{2e_{j'}} 
\cdot
\prod_{P \in \cQ
\atop
1 \le k \le 
{\# {\bm \nu}}(P)}b_{P,k}^{2e_{P,k}},
$$
degree in $w$ bounded by $\delta_w$, 
degree in $t_j$ bounded by $2epsr$
and 
degree in $(a_{P, k}, b_{P, k})$  
bounded by $2ep$, 
which serves as the final incompatibility. 
\end{proof}

We state below two more lemmas corresponding to the other cases needed to
analyze the whole partition of the real line given by the set of roots. 
We omit their proofs since they are very similar 
to the proof of Lemma \ref{partition_after_fact_2}.

\begin{lemma}\label{partition_after_fact_3} 
Let $p \ge 1$,  $\cQ$ 
be
a family of $s$ polynomials in $\K[u][y] \setminus \K$, monic in the variable $y$ with $\deg_y P \le p$ for 
every $P \in \cQ$, 
$\tau$ 
be
a realizable 
sign condition 
on ${\rm Elim}({\cal Q})$ and $ {\bm \eta}(\tau) = [\eta_1, \dots, \eta_r]$  with $r \ge 1$.
Then
$$ 
\exists  (t, z)  \ [\; 
{\rm OFact}({\cal Q})^{ {\bm \eta}(\tau),{\bm \nu}(\tau)}(t, z), \ t_{r} < y\;] 
\ \ \ \vdash
 \ \ \ 
\bigwedge_{P \in \cQ} P
 > 0  
$$
where $t = (t_1, \dots, t_r)$, $z = (z_P)_{P \in \cQ}$ and 
$z_P = (z_{P, 1}, \dots, z_{P, 
{\# {\bm \nu}}(\tau)(P)}
)$.

Suppose we have an initial incompatibility in variables $v \supset (u, y)$
with monoid part 
$$
S \cdot \prod_{P \in \cQ}   P
^{2e_P}
$$
with $e_P \le e \in \N_*$ and
degree in $w$ bounded by $\delta_w$ for some set of variables $w \subset v$.
Then, the final incompatibility has 
monoid part 
$$
S \cdot
(y - t_{r})^{2e_{r}} \cdot 
\prod_{1 \le j \le r - 1}(t_{r}-t_j)^{2e_j} 
\cdot
\prod_{P \in \cQ
\atop
1 \le k \le 
{\# {\bm \nu}}(P)}b_{P,k}^{2e_{P,k}},
$$
with 
$e_j \le eps$,  
$e_{P, k} \le ep$, 
degree in $w$ bounded by $\delta_w$, 
degree in $t_j$ bounded by $2epsr$
and 
degree in $(a_{P, k}, b_{P, k})$  
bounded by $2ep$.
\end{lemma}

\begin{lemma}\label{partition_after_fact_4}
Let $p \ge 1$,  $\cQ$
be 
a family of $s$ polynomials in $\K[u][y] \setminus \K$, monic in the variable $y$ with $\deg_y P \le p$ for 
every $P \in \cQ$, 
be
$\tau$ a realizable 
sign condition 
on ${\rm Elim}({\cal Q})$ and $ {\bm \eta}(\tau) = [\eta_1, \dots, \eta_r]$  with $r \ge 1$.
Then
$$ 
\exists  (t, z)  \ [\; 
{\rm OFact}({\cal Q})^{ {\bm \eta}(\tau),{\bm \nu}(\tau)}(t, z), \  y < t_{1}\;] 
\ \ \ 
\vdash \ \ \ 
\bigwedge_{P \in \cQ} \s(P
) = 
(-1)^{\sum_{1 \le j \le r} {\rm mu}(\eta_j,P)}
$$
where $t = (t_1, \dots, t_r)$, $z = (z_P)_{P \in \cQ}$ and 
$z_P = (z_{P, 1}, \dots, z_{P,
{\# {\bm \nu}}(\tau)(P)
})$.

Suppose we have an initial incompatibility in variables $v \supset (u, y)$
with monoid part 
$$
S \cdot   \prod_{P \in \cQ}   P
^{2e_P}
$$
with $e_P \le e \in \N_*$ and 
degree in $w$ bounded by $\delta_w$ for some set of variables $w \subset v$.
Then, the final incompatibility has 
monoid part 
$$
S \cdot (y - t_{1})^{2e_{1}}\cdot 
\prod_{ 2 \le j \le r}(t_{1}-t_j)^{2e_j} 
\prod_{P \in \cQ \atop
1 \le k \le 
{\# {\bm \nu}}(P)}b_{P,k}^{2e_{P,k}},
$$
with 
$e_j \le eps$, 
$e_{P, k} \le ep$, 
degree in $w$ bounded by $\delta_w$, 
degree in $t_j$ bounded by $2epsr$
and 
degree in $(a_{P, k}, b_{P, k})$  
bounded by $2ep$.
\end{lemma}

We introduce an auxiliary definition.

\begin{definition}
\label{defSIGNbis} Let $\tau$ be a relizable sign condition on ${\rm Elim}(\cQ)$ and 
${\bm \eta}(\tau) = [\eta_1, \dots, \eta_r]$. If $\vartheta \in \R^k$, 
$\theta \in \R^r, \alpha \in \R^s, 
\beta  \in \R^s$ with $s=\sum_{P\in {\mathcal Q}}
{\# {\bm \nu}}(\tau)(P)$
verifies
$\s({\rm Elim}(\cQ)(\vartheta)) = \tau$ and 
$
{\rm OFact}({\cal Q}(\vartheta))^{ {\bm \eta}(\tau),{\bm \nu}(\tau)}(\theta,\alpha + i \beta)$,
we denote $\sigma_j$ the sign condition $\sign(\cQ(\vartheta, \theta_j))$ for $1 \le j \le r$
and $\sigma_{(j-1,j)}$ the sign condition $\sign(\cQ)(\iota)$ for any $\iota \in 
(\theta_{j-1},\theta_j)$ for $1 \le j \le r+1$, where $\theta_0 = -\infty$ and 
$\theta_{r+1} = + \infty$. 
\end{definition}

\begin{proposition} \label{weaksigndettab2} 
Let $p \ge 1$,  $\cQ$ 
be
a family of $s$ polynomials in $\K[u][y] \setminus \K$, monic in the variable $y$ with $\deg_y P \le p$ for 
every $P \in \cQ$, 
$\tau$ 
be
a realizable 
sign condition 
on ${\rm Elim}({\cal Q})$, 
$ {\bm \eta}(\tau) = [\eta_1, \dots, \eta_r]$, 
$t = (t_1, \dots, t_r)$ and $z = (z_P)_{P \in \cQ}$ where $z_P = (z_{P, 1}, \dots,
z_{P, 
{\# {\bm \nu}}(\tau)(P)})$.  
Then
$$
\exists \,(t, z) \ [\;   
{\rm OFact}({\cal Q})^{ {\bm \eta}(\tau),{\bm \nu}(\tau)}(t, z) \,] \ \ \ 
\vdash \ \ \  
\bigvee_{\sigma \in {\rm SIGN}( {\mathcal Q}  \, | \tau )}
\s({\mathcal Q})= \sigma.
$$

Suppose we have for 
$\sigma = \sigma_{(0,1)}, \sigma_1, \dots, \sigma_{(r, r+1)}$ 
an initial incompatibility in variables $v \supset (u, y)$ with monoid part
$$
S_\sigma \cdot 
\prod_{P \in {\cal Q}, \atop \sigma(P) \ne 0}P
^{2e_{P, \sigma}}
$$
with $e_{P, \sigma} \le e \in \N_*$ and 
degree in $w$ bounded by 
$\delta_w$ for some subset of variables $w\subset v$. 
Then the final incompatibility has monoid part
$$
\prod_{1 \le j \le r+1} S_{\sigma_{(j-1,j)}}
\cdot
\prod_{1 \le j \le r} S_{\sigma_j}^{e_j}
\cdot
\prod_{1 \le j < j' \le r} (t_{j'} - t_{j})^{2e_{j,j'}} 
\cdot
\prod_{P \in \cQ
\atop
1 \le k \le 
{\# {\bm \nu}}(\tau)(P)}b_{P,k}^{2e_{P,k}}
$$
with 
$e_j \le 4eps$, 
$e_{j,j'} \le 2eps + 8(eps)^2$,
$e_{P,k} \le (ps+1)ep + 4e^2p^3s^2$, 
degree in $w$ bounded by 
$(ps+1+4ep^2s^2)\delta_w$, 
degree in $t_j$ bounded by $2(ps+1+4ep^2s^2)ep^2s^2$ and degree in 
$(a_{P, k}, b_{P, k})$ bounded by 
$2(ps+1+4ep^2s^2)ep$.
\end{proposition}

\begin{proof}{Proof.} 
We consider 
first the case that at least one polynomial in $\cQ$ has a real root, this is to say, $r > 0$.
In this case, the proof is done by applying to the initial incompatibilities the weak inferences in 
Lemmas \ref{partition_after_fact_1}, 
\ref{partition_after_fact_2}, \ref{partition_after_fact_3} and \ref{partition_after_fact_4} and
Proposition \ref{propSegLine}. 

In the case that every polynomial in $\cQ$ has no real root, this is to say, $r = 0$,   
the set of variables $t = (t_1, \dots, t_r)$ is actually empty. 
Moreover, it is clear that ${\rm SIGN}( {\mathcal Q}  \, | \tau )$  has only the element
$1^{\cQ}$, since every $P$ is monic and without real roots. 
We omit the proof since it is very easy. 
\end{proof}

We are finally ready for the proof of the main result of the section.

\begin{proof}{Proof of Theorem \ref{elimination_theorem}.}
Consider the initial incompatibilities
\begin{equation}\label{inc:init_inc_main_the_chap_six}
\lda \s(\cQ) = \sigma, \ \cH \rda
\end{equation}
where ${\cal H}$ is a system of sign conditions in $\K[v]$. 

We apply to (\ref{inc:init_inc_main_the_chap_six}) the weak inference
$$
\exists \,(t, z) \ [\;   
{\rm OFact}({\cal Q})^{ {\bm \eta}(\tau),{\bm \nu}(\tau)}(t, z) \,] \ \ \ 
\vdash \ \ \  
\bigvee_{\sigma \in {\rm SIGN}( {\mathcal Q}  \, | \tau )}
\s({\mathcal Q})= \sigma.
$$
By Proposition \ref{weaksigndettab2}  we obtain 
\begin{equation}\label{inc:aux_inc_main_the_chap_six_1}
\lda  
{\rm OFact}({\cal Q})^{ {\bm \eta}(\tau),{\bm \nu}(\tau)}(t, z), 
\ \cH \rda_{\K[v][ t, a, b]},
\end{equation}
where
$ {\bm \eta}(\tau) = [\eta_1, \dots, \eta_r]$, 
$t = (t_1, \dots, t_r)$, $z = (z_P)_{P \in \cQ}$ with $z_P = (z_{P, 1}, \dots,
z_{P, 
{\# {\bm \nu}}(\tau)(P)})$,
with monoid part
$$
\prod_{1 \le j \le r+1} S_{\sigma_{(j-1,j)}}
\cdot
\prod_{1 \le j \le r} S_{\sigma_j}^{e_j}
\cdot
\prod_{1 \le j < j' \le r} (t_{j'} - t_{j})^{2e_{j,j'}} 
\cdot
\prod_{P \in \cQ
\atop
1 \le k \le 
{\# {\bm \nu}}(\tau)(P)}b_{P,k}^{2e_{P,k}}
$$
with 
$e_j \le 4eps$, 
$e_{j,j'} \le 2eps + 8(eps)^2$,
$e_{P,k} \le (ps+1)ep + 4e^2p^3s^2$, 
degree in $w$ bounded by 
$(ps+1+4ep^2s^2)\delta_w$, 
degree in $t_j$ bounded by $2(ps+1+4ep^2s^2)ep^2s^2$ and degree in 
$(a_{P, k}, b_{P, k})$ bounded by 
$2(ps+1+4ep^2s^2)ep$.

Finally we apply to (\ref{inc:aux_inc_main_the_chap_six_1})
the weak inference 
$$
\s({\rm Elim}({\cal Q}))=\tau \ \ \ \vdash \ \ \   \exists (t, z)\ 
[\; {\rm OFact}({\cal Q})^{ {\bm \eta}(\tau),{\bm \nu}(\tau)}(t, z) \;].
$$
By Theorem \ref{weaksigndettab} (Fixing the Ordered List of the Roots as a weak existence), we obtain 
$$
\lda \s({\rm Elim}({\cal Q}))=\tau, \ \cH \rda
$$
with monoid part 
$$
\prod_{\sigma \in {\rm SIGN}( {\mathcal Q}  \, | \tau )}
S_{\sigma}^{h_{\sigma}}
\cdot
\prod_{H \in {\rm Elim}({\cal Q}), \atop \tau(H) \ne 0}H^{2h'_H}
$$
with 
$$
h_{\sigma}, h'_{H} \le 4eps{\rm g}_6\{p, s, 2eps + 8(eps)^2, (ps+1)ep + 4e^2p^3s^2, 1\}
\max\{8e^2p^3s^3, \tilde{\rm g}_{H,3}\{p\} \},
$$
and degree in $w$ bounded by 
$$
\begin{array}{rll}
& {\rm g}_6\{p, s, 2eps + 8(eps)^2, (ps+1)ep + 4e^2p^3s^2, 1\} \cdot 
\\[3mm]
\cdot & \Big(\max\{ 2^{ps(s-1)}(6ep^2s^2\delta_w + 24e^2p^4s^4\deg_w\cQ), 
\tilde{\rm g}_{H,3}\{p\}\deg_w \cQ
 \} 
+  \\[3mm]
& \max\{
2^{ps(s-1)}56e^2p^4s^4,
\tilde{\rm g}_{H,3}\{p\}\}\deg_w\cQ\Big) \le
\\[3mm]
\le & 
2^{ps(s-1)+1}e^2s^4\tilde{\rm g}_{H,3}\{p\} {\rm g}_6\{p, s, 2eps + 8(eps)^2, (ps+1)ep + 4e^2p^3s^2, 1\} \max\{\delta_w, \deg_w \cQ\},
\end{array}
$$
which serves as the final incompatibility, using Lemma \ref{tlemma_almost_final}.
\end{proof}


\section{Proof of the main theorems}\label{secMainTh}
\setcounter{equation}{0}
\setcounter{theorem}{0}

In this section we prove Theorem \ref{Theoremfinal} (Positivstellensatz with elementary recursive
degree estimates) and Theorem \ref{Theoremfinal17} (Hilbert 17-th problem with elementary recursive
degree estimates),
which are the main results of this paper. The proof proceeds by successive elimination
of the variables, using at each stage Theorem  \ref{elimination_theorem} (Elimination of One Variable as a weak inference).
This is the only result from previous sections which is used in this section.

First, we introduce some notation, 
new auxiliary functions and  
a final auxiliary lemma.

\begin{notation}
For ${\cal Q} \subset \K[x_1,\ldots,x_k]$,
${\rm SIGN}({\cal Q})$ is the set of realizable sign conditions on $\cal Q$ in $\R^k$.
\end{notation}
Note that by Theorem \ref{defSIGN} (Elimination of One Variable), 
$$\displaystyle{{\rm SIGN}({\cal Q})
=
\bigcup_{\tau \in {\rm SIGN}({\rm Elim}({\cal Q}))}{\rm SIGN}({\cal Q} \, | \tau )}.
$$

\begin{definition}
\label{defg8} \begin{itemize}
               \item Let ${\rm g}_{8}: \N \times \N \times \N \times \N \to {\mathbb R}$, 
$$
\begin{array}{c}
{\rm g}_{8}\{d, s, k, i\} = \\[3mm]
= {\rm g}_7\left\{ 
4^{\frac{4^{k-i}-1}{3}}d^{4^{k-i}} , 
s^{2^{k-i}}\max\{2, d\}^{(16^{k-i}-1){\rm bit}\{d\}}, 
2^{2^{\Big(2^{\max\{2,d\}^{4^{k-i}}} + s^{2^{k-i}}\max\{2, d\}^{16^{k-i}{\rm bit}(d)}\Big) }}
\right\}. 
\end{array}
$$

\item Let ${\rm g}_{9}: \N \times  \N \times \N \to {\mathbb R}$, 
$${\rm g}_{9}\{d, k, i\} = {\rm g}_7\left\{ 
4^{\frac{4^{k-i}-1}{3}}d^{4^{k-i}} , 
d^{(16^{k-i}-1){\rm bit}\{d\}}, 
2^{\Big(2^{2^{d^{4^{k-i}}} } -2 \Big)}
\right\}.$$

              \end{itemize}

\end{definition}

\begin{Tlemma}\label{tlem_sect7_1}
\begin{enumerate}

\item \label{tlem_sect7_1:iti} For every $d, s, k, i \in \N_*$ with $1 \le i \le k$, 
$$
{\rm g}_{8}\{d,s,k,i\}
\cdot
2^{2^{\Big(2^{\max\{2,d\}^{4^{k-i}}} + s^{2^{k-i}}\max\{2, d\}^{16^{k-i}{\rm bit}(d)}\Big)}}
\le $$
$$
\le
2^{2^{\Big(2^{\max\{2,d\}^{4^{k-i+1}}} + s^{2^{k-i+1}}\max\{2, d\}^{16^{k-i+1}{\rm bit}(d)}\Big)}}.
$$

\item \label{tlem_sect7_1:itii} For every $d, k, i \in \N_*$ with $1 \le i \le k$ and $d\ge 2$, 
$$
{\rm g}_{9}\{d, k, i\} \cdot 2^{\left(
2^{
2^{d^{4^{k-i}}}
}-2\right)
}
\le 
2^{\left(
2^{
2^{d^{4^{k-i+1}}}
}-2\right)
}.
$$
\end{enumerate}
 
\end{Tlemma}
\begin{proof}{Proof.} See Section \ref{section_annex}.
\end{proof}

Given a set of polynomials ${\cal P}$ and a 
polynomial $\ell$, we 
denote by 
${\cal P} \circ \ell$  
the set of compositions $\{ P \circ \ell \ | \ P \in {\cal P} \}$.
Similarly, if ${\cal F} = [ {\cal F}_{\neq},\,{\cal F}_\ge,\, {\cal F}_=]$ is a system
of sign conditions, we denote by ${\cal F} \circ \ell$ the system 
$[ {\cal F}_{\neq} \circ \ell,\,{\cal F}_\ge \circ \ell,\, {\cal F}_= \circ \ell]$.

We are ready now to prove our main theorems.

\begin {proof}{Proof of Theorem \ref{Theoremfinal}.} 
We define ${\cal P}_{k}$ as $|{\cal F}|$ (see Notation \ref{not:set_of_pol_under_eq_rel}), 
note that without loss of generality we can assume ${\cal F}\subset \K[x] \setminus \K$. 
For $i = k, \dots, 1$, we define
inductively finite families
${\cal Q}_{i} \subset \K[x_1, \dots, x_{i}]$ and
 ${\cal P}_{i-1} \subset \K[x_1, \dots, x_{i-1}]$.
Let
$\ell_{i}: \K[x_1, \dots, x_{i}] \to \K[x_1, \dots, x_{i}]$ be a linear change of variables
such that for every polynomial $P \in {\cal P}_{i}, P \circ \ell_i(x_1, \dots, x_i)$ 
is quasimonic in the variable $x_{i}$;  
we define
\begin{itemize}
\item ${\cal Q}_{i}$ as the family 
obtained by dividing each polynomial 
$P\circ\ell_i(x_1, \dots, x_i)$ in $\cP_{i} \circ \ell_i$
by its 
leading coefficient in
the variable $x_i$,
\item ${\cal P}_{i-1} = {\rm Elim}({\cal Q}_i) \setminus \K$, considering
$(x_1,\ldots,x_{i-1})$ as parameters and $x_i$ as the main variable.
\end{itemize}

Following Remark \ref{rem:degree_and_number_elim},
it can be easily proved by induction that
for $i = k, \dots, 1$, 
$$
\deg \cP_{i} \le 4^{\frac{4^{k-i}-1}{3}}d^{4^{k-i}}.
$$
Also using Remark \ref{rem:degree_and_number_elim}, we will prove
that
$$
\# \cP_{i} \le s^{2^{k-i}}\max\{2, d\}^{(16^{k-i}-1){\rm bit}\{d\}}. 
$$
Indeed, $\# \cP_{k} = s$ and for $i = k, \dots, 2$, 
\begin{eqnarray*}
\# \cP_{i-1} &\le& 4 s^{2^{k-i+1}} \max\{2, d\}^{2 (16^{k-i}-1){\rm bit}\{d\}}
(4^{\frac{4^{k-i}-1}{3}}d^{4^{k-i}})^{{\rm bit}\{4^{\frac{4^{k-i}-1}{3}}d^{4^{k-i}}\} + 1} \le \\
&\le&
s^{2^{k-i+1}}\max\{2, d\}^{2 + 2(16^{k-i}-1){\rm bit}\{d\} + (2\frac{4^{k-i}-1}{3} + 4^{k-i})(2\frac{4^{k-i}-1}{3} + 4^{k-i}{\rm bit}\{d\} + 1)} \le \\
&\le& s^{2^{k-(i-1)}}\max\{2, d\}^{(16^{k-(i-1)}-1){\rm bit}\{d\}}. 
\end{eqnarray*}

For $1 \le i \le k$, we denote by $\ell_{[k, i]}$ the polynomial $\ell_k \circ \dots \circ \ell_i$. 
Let us show by induction in $i = k, \dots, 0$, that for every realizable
sign condition $\sigma$
on ${\cal P}_{i}$ we have an incompatibility 
\begin{equation}\label{inc:final_th}
\lda 
{\rm sign}({\cal P}_{i}) = \sigma,    \  
{\cal F} \circ \ell_{[k, i+1]}
  \rda
\end{equation}
with monoid part
$$
\prod_{H \in \cP_i, \, \sigma(H) \ne 0}H^{2e_H}
$$
with $e_{H}$ bounded by 
$$
2^{2^{\Big(2^{\max\{2,d\}^{4^{k-i}}} + s^{2^{k-i}}\max\{2, d\}^{16^{k-i}{\rm bit}(d)}\Big) }}
$$
for $H \in {\rm Elim}(\cP_i)$ with $\sigma(H) \ne 0$ and degree bounded by
$$
2^{
2^{\Big(2^{\max\{2,d\}^{4^{k-i}}} + s^{2^{k-i}}\max\{2, d\}^{16^{k-i}{\rm bit}(d)} \Big) }
}.
$$

For $i = k$, 
$|{\cal F}|$ and ${\cal P}_{i}$ are 
the same sets of polynomials. 
Moreover, for every strict sign condition $\sigma$ which is realizable for 
$|{\cal F}|$, there 
must be a polynomial $P \in |{\cal F}|$ such that $\sigma(P)$ is incompatible
with the system of sign conditions ${\cal F}$. 
It is easy to check that, in all possible cases, the algebraic identity 
$$P^2 - P^2 = 0$$
serves as the corresponding
 incompatibility
(see Example \ref{exsquare-square}). So $e_H \le 1$ for 
$H \in \cP_i $ with $\sigma(H) \ne 0$ and the degree of the incompatibility (\ref{inc:final_th})
is bounded by $2d$. 

Suppose now that the 
induction hypothesis holds for some value of $i > 0$ and 
let $\tau$ be a realizable strict sign condition on $\cP_{i-1}$. 
For every realizable strict sign condition $\sigma$
on ${\cal P}_{i}$ we compose the incompatibility we have already
by induction hypothesis
with $\ell_i$ to obtain an incompatibility for 
$$\lda 
{\rm sign}({\cal P}_{i}\circ \ell_i) = \sigma , \
{\cal F} \circ \ell_{[k, i]}
\rda$$
with the same bounds for the degree and the exponents in the monoid part as 
(\ref{inc:final_th}). 
We denote $\sigma'$ the strict sign condition on ${\cal Q}_{i}$ 
obtained from a strict $\sigma$ on ${\cal P}_{i}\circ \ell_i$ by replacing $>$ for $<$ and vice versa
when the leading coefficient of the corresponding polynomial in ${\cal P}_{i}\circ \ell_i$ is negative.
It is clear that
$$
{\rm SIGN}({\cal Q}_{i})=
\{\sigma' \ | \ \sigma \in {\rm SIGN}({\cal P}_{i} \circ \ell_i) \}.
$$
So, we have for every 
realizable strict sign condition $\sigma'$
on ${\cal Q}_{i}$ an incompatibility 
\begin{equation}\label{inc:finalfinal}
\lda 
{\rm sign}({\cal Q}_{i}) =  \sigma', \ 
{\cal F} \circ \ell_{[k, i]}
\rda
\end{equation}
with the same bounds as (\ref{inc:final_th}).
We apply to (\ref{inc:finalfinal}) for every 
$\sigma' \in  {\rm SIGN}( {\mathcal Q}_i  \, | \tau )$ the weak inference
$$
\s(\cP_{i-1})=\tau \ \ \ \vdash \ \ \  
\bigvee_{\sigma' \in {\rm SIGN}( {\mathcal Q}_i  \, | \tau )}
\s({\mathcal Q}_i)= \sigma'
$$
of Theorem \ref{elimination_theorem} (Elimination of One Variable as a weak inference). We obtain in this way an incompatibility
$$
\lda 
\s(\cP_{i-1}) = \tau 
, \
{\cal F} \circ \ell_{[k, i]}
\rda
$$
with monoid part
$$
\prod_{H \in \cP_{i-1}, \, \sigma(H) \ne 0}H^{2e'_H}
$$
with $e'_{H}$ bounded by 
${\rm g}_8\left\{d,s,k,i\right\} $
and degree bounded by 
$$
{\rm g}_8\left\{d,s,k,i\right\} 
\cdot
2^{2^{\Big(2^{\max\{2,d\}^{4^{k-i}}} + s^{2^{k-i}}\max\{2, d\}^{16^{k-i}{\rm bit}(d)}\Big)}}.
$$
The claim follows then by Lemma \ref{tlem_sect7_1} (item \ref{tlem_sect7_1:iti}).

Since ${\cal P}_{0} \subset \K$, 
after the inductive procedure described above is finished,
we obtain a single incompatibility 
$$\lda {\cal F} \circ \ell_{[k, 1]}  \rda$$
with degree bounded by
$$
2^{
2^{\Big(2^{\max\{2,d\}^{4^{k}}} + s^{2^{k}}\max\{2, d\}^{16^{k}{\rm bit}(d)} \Big)}
}.
$$
Our result follows then by composing this incompatibility with 
$\ell_{[k, 1]}^{-1}$
which does not change the degree bound. 
\end{proof}

\begin {proof}{Proof of Theorem \ref{Theoremfinal17}.} The sketch of the proof is the 
following: first we proceed as in the proof of 
Theorem \ref{Theoremfinal} (Positivstellensatz with elementary recursive
degree estimates) but obtaining a slightly better bound
which holds for the particular case 
when the original system has only one polynomial. 
Then we proceed as in the proof of Theorem \ref{th:Hilbert_by_Stengle} (Improved Hilbert 17-th problem). 

The initial system $\cF$ we consider is 
$$ 
P\not=0,-P\ge 0
$$
and the initial incompatibility between $\cF$ and $P\ge 0$ is
$$P^2-P^2=0.$$
Note that since $P$ is nonnegative in $\R^k$, $d$ is even and therefore $d \ge 2$.

Proceeding as in the proof
of Theorem \ref{Theoremfinal} and using Lemma \ref{tlem_sect7_1} (item \ref{tlem_sect7_1:itii})
(instead of 
Lemma \ref{tlem_sect7_1} (item \ref{tlem_sect7_1:iti})),  we prove that for $i = k, \dots, 0$, for every realizable
strict sign condition $\sigma$
on ${\cal P}_{i}$ we have an incompatibility 
$$
\lda 
{\rm sign}({\cal P}_{i}) = \sigma,    \  
{\cal F} \circ \ell_{[k, i+1]}
  \rda
$$
with monoid part
$$
\prod_{H \in \cP_i, \, \sigma(H) \ne 0}H^{2e_H}
$$
with $e_{H}$ bounded by 
$$
2^{\left(
2^{
2^{d^{4^{k-i}}}
}-2\right)
}
$$
for $H \in {\rm Elim}(\cP_i)$ with $\sigma(H) \ne 0$ and degree bounded by
$$
2^{\left(
2^{
2^{d^{4^{k-i}}}
}-2\right)
}.
$$

After finishing the inductive procedure and composing with $\ell^{-1}_{[k, 1]}$ as before, 
we obtain a final incompatibility of 
$\cF,$
$$
\lda P \not= 0, -P\geq 0 \rda, 
$$
of
type 
$$
P^{2e}+ N_1 - N_2P=0
$$
with $e \in \N$, $N_1, N_2 \in \scN(\emptyset)$ and degree bounded by  
$$
2^{\left(
2^{
2^{d^{4^{k}}}
}-2\right)
}.
$$
From this we deduce, 
as in the proof of Theorem \ref{th:Hilbert_by_Stengle} (Improved Hilbert 17-th problem),
\begin{equation}\label{inc:final_hilbert}
P=\frac{N_2P^2}{P^{2e}+ N_1}=\frac{N_2P^2(P^{2e}+ N_1)}{(P^{2e}+ N_1)^2}. 
\end{equation}
After expanding the numerator in (\ref{inc:final_hilbert}) we obtain an expression 
$$
P =  \sum_{i}\omega_i\frac{P_i^2}{Q^2}
$$
with $\omega_i \in \K, \omega_i > 0, P_i \in \K[x], Q = P^{2e} + N_1 \in \K[x]$ 
and 
$$
\deg P_i^2 \le  2^{\left(
2^{
2^{d^{4^{k}}}
}-1 \right)} +  d \le 2^{
2^{
2^{d^{4^{k}}}
}
}
$$
for every $i$ and
$$\deg Q^2 \le 
2^{\left(
2^{
2^{d^{4^{k}}}
}-1\right)
} \le 2^{
2^{
2^{d^{4^{k}}}
}
}
.$$
\end{proof}

\section{Annex}\label{section_annex}
\setcounter{equation}{0}

Here we include the proof of technical lemmas from the previous sections. 

\begin{proof}{Proof of Technical Lemma \ref{aux_ineq_tfa}.} 

We first prove item \ref{aux_ineq_tfa:0}. 
$$
3{\rm g}_1\{p-1, p\} = 3\cdot 2^{3\cdot 2^{p-1}}p^{p} \le 2^{3\cdot 2^{p-1} + p^2} \le 
2^{2^{3\frac{p}2} } = 
{\rm g}_2\{p\}.
$$

Now we prove item \ref{aux_ineq_tfa:1}. 
We check separately that the inequality holds 
for $p = 4$ and $p =  6$ and we suppose that $p \ge 8$. Then we have
$$
\frac3{16}  p^9  {\rm n}\{p\}^{{\rm n}\{p\}+1}2^{4{{{\rm n}\{p\}+1}\choose 2}}{\rm g}_2^{{\rm n}\{p\}+1}\{{\rm n}\{p\}\}
\le 
2^{ \frac{1}{2}p^4 + 
\frac{1}{2}p^2 2^{ 3(\frac{ p^2-p   }4)^{2^{{\rm r}\{p\}-1}}  } 
}.$$
The lemma follows since
$$
 \frac{1}{2}p^4 + 
\frac{1}{2}p^2 2^{ 3(\frac{ p^2-p   }4)^{2^{{\rm r}\{p\}-1}}  } 
\le
 p^2 2^{ 3(\frac{ p^2-p   }4)^{2^{{\rm r}\{p\}-1}}  } 
\le
2^{ 3(\frac{ p   }2)^{2^{{\rm r}\{p\}}}  }.
$$
\end{proof}

\begin{proof}{Proof of Technical Lemma \ref{tlem:aux_comp_fact}.} We first prove item \ref{it:tlem_aux_comp_fact_1}. 
$$
3(2p+1){\rm g}_1\{p-1, p\}{\rm g}_3\{p-1\} 
\le 2^{1 + p^2 + 3\cdot2^{p-1}+  2^{3(\frac{p-1}2)^{p-1} + 1}  }
\le 2^{2^{3(\frac{p-1}2)^{p-1} + 3} }
\le {\rm g}_3\{p\}.
$$

Now we prove item \ref{it:tlem_aux_comp_fact_2}. 
$$ 
6p^3 {\rm g}_1\{p-2, p-1\}{\rm g}_2\{p\}{\rm g}_3^2\{p-2\}
\le
2^{p^2 + 3\cdot2^{p-2} + 2^{3(\frac{p}2)^{p}} + 2^{3(\frac{p-2}2)^{p-2} + 2} } \le {\rm g}_3\{p\}.
$$
\end{proof}

\begin{proof}{Proof of Technical Lemma \ref{tlem:aux_theorem_fix_Thom}.}
It is easy to see that it is enough to prove that
$$
2^{p+ (((p-1)p+2)2^{(p-1)p}  - 2 )(2^{\frac12p^2} + 2^{\frac12p} + 1)}
\le \tilde{\rm g}_{H,1}\{p\}^{2^{\frac32p^2} - (2^{(p-1)p}-1)(2^{\frac12p^2} + 2^{\frac12p} + 1) - 1}.
$$
Indeed, since $2^{(2^{\frac12(p-1)p+2}-2)} \le \tilde{\rm g}_{H,1}\{p\}$ and
$2^{\frac32p^2} - (2^{(p-1)p}-1)(2^{\frac12p^2} + 2^{\frac12p} + 1) - 1 \ge 0$, the
lemma follows from  
$$
p+ (((p-1)p+2)2^{(p-1)p}  - 2 )(2^{\frac12p^2} + 2^{\frac12p} + 1)
\le
 ((p-1)p+2)2^{\frac32p^2 - 1}
\le 
$$
$$
\le
(2^{\frac12(p-1)p+2}-2)2^{\frac32p^2 - 1}
\le
(2^{\frac12(p-1)p+2}-2)(2^{\frac32p^2} - (2^{(p-1)p}-1)(2^{\frac12p^2} + 2^{\frac12p} + 1) - 1).
$$
\end{proof}

\begin{proof}{Proof of Technical Lemma \ref{tlemma_almost_final}.} First, it is easy to 
prove that for every $p \in \N_*$ we have that 
$\tilde{\rm g}_{H,3}\{p\} \le 2^{(9p^2 + 14p + 3)2^{\frac12p^2 + 2} - 2}.$ 
Then, 
$$
2^{ps(s-1)+2}e^2s^4p\,{\rm g}_6\{p, s, 2eps + 8(eps)^2, (ps+1)ep + 4e^2p^3s^2, 1\} \max\{8e^2p^3s^3, \tilde{\rm g}_{H,3}\{p\} \} \le 
$$
$$
\le 2^{ps(s-1)+2}e^2s^4p\Big(
{\rm g}_4\{p\}(6e^2p^3s^2)^{2^{\frac12p}}
\Big)^{\frac{2^{s(\frac32p^2 + 2)}-1}{2^{\frac32p^2 + 2}-1}}
2^{(p+4)(2^{s(s-1)p^2}-1)2^{s(\frac32p^2 + 2)}}
\cdot
$$
$$
\cdot
\max\{8e^2p^3s^3, \tilde{\rm g}_{H,3}\{p\}\}
^{2^{s(s-1)p^2 + s(\frac32p^2 + 2)}} 
\le
$$
$$
\le
\Big(
(6p^3)^{2^{\frac12p}}2^{ 2^{3(\frac{p}2)^p + 2} }
\Big)^{\frac{2^{s(\frac32p^2 + 2)}-1}{2^{\frac32p^2 + 2}-1}}
2^{\alpha_1\{p, s\}} 
s^{\beta_1\{p, s\}}
e^{\gamma_1\{p, s\}},
$$
where 
$$\alpha_1\{p, s\} = (p+4)2^{s(s-1)p^2 + s(\frac32p^2 + 2)} + 
((9p^2 + 14p + 3)2^{\frac12p^2 + 2} - 2)2^{s(s-1)p^2 + s(\frac32p^2 + 2)}, $$
$$\beta_1\{p, s\} =4 + 2^{\frac12 p+1}\frac{2^{s(\frac32p^2 + 2)}-1}{2^{\frac32p^2 + 2}-1} 
+ 3 \cdot 2^{s(s-1)p^2 + s(\frac32p^2 + 2) } , $$
and
$$\gamma_1\{p, s\} = 2  + 2^{\frac12 p + 1}\frac{2^{s(\frac32p^2 + 2)}-1}{2^{\frac32p^2 + 2}-1} 
+  2^{s(s-1)p^2 + s(\frac32p^2 + 2) + 1 }
.$$
Then we have
$$
\Big(
(6p^3)^{2^{\frac12p}}2^{ 2^{3(\frac{p}2)^p + 2} }
\Big)^{\frac{2^{s(\frac32p^2 + 2)}-1}{2^{\frac32p^2 + 2}-1}}
2^{\alpha_1\{p, s\}} 
\le
2^{\alpha_2\{p, s\}}$$
and
$$
s^{\beta_1\{p, s\}} \le 2^{\alpha_2'\{p, s\}}
$$
where
$$\alpha_2\{p, s\} = 
2^{s(\frac32p^2 + 2)} + 
2^{3(\frac{p}2)^p + 2 + s(\frac32p^2 + 2)} +
$$
$$
+(p+4)2^{s(s-1)p^2 + s(\frac32p^2 + 2)} 
+((9p^2 + 14p + 3)2^{\frac12p^2 + 2} - 2)2^{s(s-1)p^2 + s(\frac32p^2 + 2)}
$$
and
$$
\alpha'_2\{p, s\} = 
(s-1)(
2^{s(\frac32p^2 + 2)} 
+ 3 \cdot 2^{s(s-1)p^2 + s(\frac32p^2 + 2) }).
$$
But then, 
$$\alpha_2\{p, s\} + \alpha_2'\{p, s\}
\le 
2^{3(\frac{p}2)^p + 2 + s(\frac32p^2 + 2)} 
+ 2^{\frac12p^2 + p + 7 + s(s-1)p^2 + s(\frac32p^2 + 2)}
+ 2^{s(s-1)p^2 + s(\frac32p^2 + 3) } \le
$$
$$
\le
2^{3(\frac{p}2)^p + s^2(\frac32p^2 + 3) + 8}.
$$

On the other hand, 
$$
\gamma_1\{p, s\} \le  2^{s(\frac32p^2 + 2)}
+  2^{s(s-1)p^2 + s(\frac32p^2 + 2) + 1 } \le 2^{6s^2p^2}
$$ 
and the lemma follows. 
\end{proof}

\begin{proof}{Proof of Technical Lemma \ref{tlem_sect7_1}.} We prove item 
\ref{tlem_sect7_1:iti} 
and the proof of item \ref{tlem_sect7_1:itii} can be done in a similar way.
$$
{\rm g}_{8}\{d, s,k, i\}
\cdot
2^{2^{\Big(2^{\max\{2,d\}^{4^{k-i}}} + s^{2^{k-i}}\max\{2, d\}^{16^{k-i}{\rm bit}(d)}\Big)}}
= 
2^{2^{\alpha\{d, s, k\}
}}
2^{2^{
\beta\{d, s, k\}
}}
2^{2^{
\gamma\{d, s, k\}
}} 
$$
where
$$
\alpha\{d, s, k\} = 
3\left( 2^{2\frac{4^{k-i}-1}{3}-1}d^{4^{k-i}} \right)^{2^{2\frac{4^{k-i}-1}{3}}d^{4^{k-i}}} 
+ 
$$
$$
+
s^{2^{k-i+1}}\max\{2, d\}^{2(16^{k-i}-1){\rm bit}\{d\}}\left( \frac32  2^{4\frac{4^{k-i}-1}{3}}d^{2\cdot4^{k-i}}    + 3  \right) 
+
8,
$$
$$
\beta\{d, s, k\} = 
2^{\max\{2,d\}^{4^{k-i}}} + 
$$
$$
+s^{2^{k-i}}\max\{2, d\}^{16^{k-i}{\rm bit}(d)}
+ 6 \cdot 2^{4\frac{4^{k-i}-1}{3}}d^{2\cdot4^{k-i}} s^{2^{k-i+1}}
\max\{2, d\}^{2(16^{k-i}-1){\rm bit}\{d\}}, 
$$
and
$$
\gamma\{d, s, k\} = 
2^{\max\{2,d\}^{4^{k-i}}} + s^{2^{k-i}}\max\{2, d\}^{16^{k-i}{\rm bit}(d)}.
$$

The inequality holds since 

\begin{eqnarray*}
\alpha\{d, s, k\}
&\le&
2^{2^{2(k-i) + 2\frac{4^{k-i}-1}{3} }d^{1 + 4^{k-i}}} + 
s^{2^{k-i+1}}\max\{2, d\}^{2(16^{k-i}-1){\rm bit}\{d\} + 4\frac{4^{k-i}-1}{3}
+ 2\cdot 4^{k-i} + 4}\\
&\le& 2^{\max\{2,d\}^{4^{k-i+1}}} + s^{2^{k-i+1}}\max\{2, d\}^{16^{k-i+1}{\rm bit}(d)} - 1,
\end{eqnarray*}
\begin{eqnarray*}
\beta\{d, s, k\}&\le &
2^{\max\{2,d\}^{4^{k-i+1}}} + 7s^{2^{k-i+1}} 
\max\{2, d\}^{4\frac{4^{k-i}-1}{3} + 2\cdot4^{k-i} +  2(16^{k-i}-1){\rm bit}(d)} \\
&
\le &
2^{\max\{2,d\}^{4^{k-i+1}}} + s^{2^{k-i+1}}\max\{2, d\}^{16^{k-i+1}{\rm bit}(d)} - 2,
\end{eqnarray*}
and
$$
\gamma\{d, s, k\}\le 
2^{\max\{2,d\}^{4^{k-i+1}}} + s^{2^{k-i+1}}\max\{2, d\}^{16^{k-i+1}{\rm bit}(d)} - 2.
$$
\end{proof}

\end{document}